\documentclass[a4paper,10pt]{amsart}
\usepackage{longtable,geometry}
\usepackage{graphicx}
\usepackage[english]{babel}
\usepackage[utf8]{inputenc}
\usepackage[T1]{fontenc}
\usepackage{amsmath,amsfonts,amssymb}
\usepackage{bbm}
\usepackage{pict2e}
\usepackage[dvipsnames]{xcolor}
\usepackage{url}
\usepackage[normalem]{ulem}
\usepackage{array}
\usepackage{appendix}
\usepackage{chngcntr}
\usepackage{calrsfs}

\usepackage{hyperref}

\renewcommand{\paragraph}[1]{\medskip\noindent{\bf$\bullet$  #1. }}

\newenvironment{aligne}[1]{ \everymath={\displaystyle}\begin{array}{#1}}{\end{array}}

\numberwithin{equation}{section} 

\newtheorem{theorem}{Theorem}
\newtheorem{prop}{Proposition}[section]
\newtheorem{corollary}[prop]{Corollary}
\newtheorem{lemma}[prop]{Lemma}

\newtheorem{remark}{Remark}[section]\theoremstyle{remark}
\newtheorem{definition}{Definition}[section]\theoremstyle{definition}
\newtheorem{assumption}{Assumption}
\newtheorem*{notation}{Notation}

\newcommand{\gr}[1]{{\mathbf{#1}}}

\newcommand{\3}{\operatorname{|\hspace{-0.4mm}|\hspace{-0.4mm}|}}

\renewcommand{\div}{{\rm div}}
\newcommand{\ind}{\mathbbm{1}}
\newcommand{\e}{\varepsilon}
\newcommand{\para}{{\shortparallel}}
\newcommand{\V}{{\hat{\mathcal{V}}}}
\newcommand{\LL}{{\mathcal L}}
\newcommand{\ud}{\hspace{1pt}\mathrm{d}}
\newcommand{\uud}{\hspace{1pt}{\rm ^-\!\!\!\!d}}
\newcommand{\step}[1]{\medskip \noindent {\bf\textit{ Step} #1.}}

\DeclareMathOperator{\tr}{{\rm tr}}

\newcommand{\DDelta}{\Delta\hspace{-6pt}\Delta}

\renewcommand{\>}{\rangle}
\newcommand{\<}{\langle}
\newcommand{\sg}{\operatorname{sg}}

\author{Le Bihan, Corentin}
\address[Corentin Le Bihan]{CMLS, \'Ecole Polytechnique, Palaiseau, France}
\email{corentin.le-bihan@polytechnique.edu}
\title{Around the Quantum Lenard-Balescu equation}

\makeindex

\begin{document}
	\maketitle
	
		\begin{abstract}
		In the mean-field regime, a gas of quantum particles with Boltzmann statistics can be described by the Hartree-Fock equation. This dynamics becomes trivial if the initial distribution of  particle is invariant by translation. However, the first correction is given on time of order $O(N)$ by the quantum Lenard--Balescu equation. In the first part of the present article, we justify this equation until time of order $O((\log N)^{1-\delta})$ (for any $\delta\in(0,1)$).
		
		A similar phenomenon exists in the classical setting (with a similar validity time obtained by Duerinckx \cite{Duerinckx}). In a second time, we prove the convergence for dimension $d\geq 2$ of the solutions of the quantum Lenard--Balescu equation to the solutions of its classical counterpart in the semi-classical limit. This problem can be interpreted as a grazing collision limit: the quantum Lenard--Balescu equation looks like a cut-off Boltzmann equation, when the classical one looks like the Landau equation.
	\end{abstract}
	
	\tableofcontents
	
	\section{Introduction}
		\subsection{General presentation}
			Fix $L>0$. We denote $\mathbb{T}_L$ the torus of length $L$,  $(\mathbb{R}/L\mathbb{Z})^d$ (with $d\geq 2$) and $\mathfrak{H}_L$ the Hilbert space $L^2(\mathbb{T}_L)$.	We want to describe $N$ particles in $\mathbb{T}_L$ interacting {\it via} pairwise potential energy. In the following, we denote $\mathbb{Z}^d_L:=(2\pi L \mathbb{Z})^d$ the dual space of $\mathbb{T}_L$.
			
			We fix $\mu$, $L$ and $N$ such that $\mu = NL^{-d}$. The state  $\psi_N(t)\in\mathfrak{H}_L^N$  follows the Schrodinger dynamics
			\begin{equation}\label{eq:Schrodinger}
				i\hbar\partial_t \psi_N(t)= \frac{\hbar^2}{2}\sum_{j=1}^{N}- \triangle_j\psi_N(t) +\frac{1}{\mu}\sum_{j<k}^N V_{j,k}\psi_N(t)
			\end{equation}
			with $\triangle_j$ the Laplacian with respect to the $j$-th variable and $V_{j,k}$ the multiplication by $\mathcal{V}(x_j-x_k)$. For the moment, one can suppose $\mathcal{V}$ smooth, with spherical symmetry and compact support.
			
			With a statistic point of view, the system is described by a density matrix $F_N\in\mathcal{L}^1(\mathfrak{H}_L)$ (the space of trace class operators, see Section 1.9 of \cite{Golse} for the definition and basic properties). An easy way to introduce $F_N$ is to understand it as a Hilbert-Schmidt operator with Kernel $F_N(x_{[1,N]},y_{[1,N]})\in L^2(\mathfrak{H}_L^{2N})$, with $x_{[1,N]}:=(x_1,\cdots,x_N)\in\mathbb{T}^N_L$
			\begin{equation}\label{eq:def matrice}
				F_N(x_{[1,N]},y_{[1,N]}) := \int \psi(x_{[1,N]})\bar{\psi}_N(y_{[1,N]}) \pi\left(\ud\psi\right)	
			\end{equation}
			where $\pi$ is a probability measure on $\{\psi\in L^2(\mathfrak{H}_L^N),\|\psi\|=1\}$.
	
			As particles are exchangeable, the matrix $F_N$ has to verify the following (weak) symmetry condition:
			\begin{equation}
				\forall\sigma\in\mathfrak{S}_N,~F_N(x_{\sigma(1)},\cdots,x_{\sigma(N)},y_{\sigma(1)},\cdots,y_{\sigma(N)}) = F_N(x_1,\cdots,x_N,y_1,\cdots,y_N).
			\end{equation}
	
			\begin{remark}
				We say that $F_N$ is 
				\begin{itemize}
					\item Fermionic if $\pi$ has support included in (with $\e(\sigma)$ the signature of $\sigma\in\mathfrak{S}_n$)
					\begin{gather*}\{\psi\in \mathfrak{H}_L^N,~\psi(x_{\sigma(1)},\cdots,x_{\sigma(N)}) = \epsilon(\sigma)\psi(x_1,\cdots,x_N)\},\\
					{\it id~est}~F_N(x_{1},\cdots,x_{N},y_{\sigma(1)},\cdots,y_{\sigma(N)}) =\epsilon(\sigma) F_N(x_1,\cdots,x_N,y_1,\cdots,y_N),\end{gather*}
					\item Bosonic if $\pi$ has support included in 
					\begin{gather*}\{\psi\in \mathfrak{H}_L^N,~\psi(x_{\sigma(1)},\cdots,x_{\sigma(N)}) = \psi(x_1,\cdots,x_N)\},\\
					{\it id~est}~F_N(x_{1},\cdots,x_{N},y_{\sigma(1)},\cdots,y_{\sigma(N)}) =F_N(x_1,\cdots,x_N,y_1,\cdots,y_N).
					\end{gather*}
				\end{itemize}
			\end{remark}

			Using \eqref{eq:Schrodinger} and \eqref{eq:def matrice}, the density matrix $F_{\hbar,N}(t)$ is solution of the Von Neumann equation
			\begin{equation}\label{eq:Von Neuman}
				i\hbar\partial_t F_{\hbar,N}(t) = \frac{\hbar^2}{2}\sum_{j=1}^{N} \left[-\triangle_j,F_{\hbar,N}(t)\right]+\frac{1}{\mu}\sum_{j<k}^N \left[V_{j,k},F_{\hbar,N}(t)\right], ~ F_{\hbar,N}(t=0):=F_{N,0}.
			\end{equation}

			At time $0$, we make a \emph{chaos assumption}, {\it i.e.} we suppose the factorization of the density matrix:  there exists $F_0\in\mathcal{L}^1(\mathfrak{H}_L)$, non-negative, and with $\tr F_0 = L^d$, such that the density matrix $F_{N,0}$ verifies
			\begin{equation}
				F_{N,0}(x_1,\cdots,x_n,y_1,\cdots,y_n) = \prod_{i= 1}^N L^{-d}F_0(x_i,y_i).
			\end{equation}
			We choose $\tr F_0 = L^d$ in order to have formally $\|F_0\|_{L^\infty}$ of order $1$ (it will be useful when we will take a large box limit $L\to\infty$).
				
			\paragraph{Derivation of an effective equation on the long time scale}
			Define $F^{(n)}{\hbar, N}$ the $n$-particles density matrix
			\begin{equation}
				\forall n\in[1,N],~F_{\hbar,N}^{(n)}(t,x_{[1,n]},y_{[1,n]}):=\int F_{\hbar,N}(t,x_{[1,n]},z_{[n+1,N]},y_{[1,n]},z_{[n+1,N]})\ud{z_{[n+1,N]}}
			\end{equation}
			
			In that setting, one can prove (see \cite{Golse,PPS}) that when $N\to\infty$ with $\mu = NL^{-d}$ (and $L$ fixed), the one particle density matrix $F_{\hbar,N}^{(1)}(t,x_1,y_1)$ converges to $L^{-d}F_\hbar(t,x,y)$, the solution of the solution of Hartree equation
			\begin{equation}
				\left\{\begin{split}
				i\hbar\partial_t F_\hbar(t) &= -\hbar^2\big[\tfrac12\triangle,F_\hbar\big]+\big[\tilde{\mathcal{V}},F_\hbar\big],~F_\hbar(t= 0)= F_0,\\
				\tilde{\mathcal{V}}(t,x) &= \int \mathcal{V}(x-y)F_\hbar(t,y,y)\ud{y}.
				\end{split}\right.
			\end{equation}
			where $\tilde{\mathcal{V}}$ is an effective potential.
	
			In \cite{PPS}, the authors prove a stronger result (propagation of chaos): if at time $0$ the system is distributed with respect to $F_{0,N}:= L^{-Nd}F_0^{\otimes N}$, then $\forall t\geq 0, j\in [1,N]$,
			\begin{equation}
				\sup_{y_{[1,j]}\in\mathbb{T}^j_L}\int 	\left|\left(F_{\hbar,N}^{(j)}-F_\hbar^{\otimes j}\right)(t,x_{[1,j]},x_{[1,j]}+y_{[1,j]})\right|\ud{x_{[1,j]}} \leq C e^{C/\hbar t}\frac{j^2}{\mu}
			\end{equation}
			where the constant $C$ does not depend on $\hbar$. Note that the norm is the $\mathcal{L}^1(\mathfrak{H}_L^j)$-operator norm.
			
			If we choose an initial data invariant by translation ($F_0(x,y)= G(x-y)$), the effective potential becomes constant. Hence, at the first order, the density is constant at the first order.
			
			However, the error terms  $F_{\hbar,N}^{(1)}(t)-F_\hbar(t)$, $F_{\hbar,N}^{(2)}(t)-F_\hbar^{\otimes 2}(t)$ are of order $O(1/\mu)$. Hence, it is natural to guess if one can obtain a corrector on long time (of order $O(\mu)$). 
			
			We introduce  $\Phi_{\hbar,N,L}(t,v)$ the scaled in time Wigner transform of $F_{\hbar,N}^{(1)}$ as
			\begin{equation}
				\forall v\in \mathbb{Z}^d_{L/\hbar},~\Phi_{\hbar,N,L}(t,v):=\frac{L^d}{(2\pi\hbar)^d}\int_{\mathbb{T}_{2L}} F_{\hbar,N}^{(1)}(\mu t,\tfrac{y}{2},-\tfrac{y}{2}) e^{i\frac{v}{\hbar}y}\ud{y}
			\end{equation}
	
			In limit $L,\mu\to\infty$, $L^d\mu = N$, one can prove the formal convergence (see Theorem \ref{thm:limite champ moyen} for a precise statement) of  $\Phi_{\hbar,N,L}(\mu t,v)$ to the solution of the Quantum-Lenard-Balescu equation: for $\uud k := \tfrac 1{(2\pi)^d}\uud k$,
			\begin{gather}
				\forall t\geq 0,~\forall v_1\in\mathbb{R}^d,~\partial_t \Phi_\hbar(t,v_1)=Q^\hbar_{LB}(\Phi_\hbar(t))(v_1).\label{eq:equation de Lenard Balescu Quantique}\tag{QLB}\\
				Q^\hbar_{LB}(\Phi)(v_1)=\frac{2}{\hbar^2}\int\frac{ \hat{\mathcal{V}}(k)^2\delta_{ k\cdot(v_2-v_1-\hbar k)}}{|\e_\hbar(\Phi,k,v_1)|^2}(\Phi(v_2-\hbar k)\Phi(v_1+\hbar k)-\Phi(v_1)\Phi(v_2))\uud k\ud{v_2} \\
				\e_\hbar(\Phi,k,v) = 1+{\hat{\mathcal{V}}(k)}\int \frac{\Phi(v_*)-\Phi(v_*-\hbar k)}{\hbar k\cdot(v_*-v-\hbar k)-i0}\ud{v_*}.
			\end{gather}
	
			We precise that we take the convention for the Fourier transform
			\begin{equation}
				\hat{\mathcal{V}}(k) = \int \mathcal{V}(x)e^{-ikx}\ud x.
			\end{equation}	
	
			This equation looks like the Boltzmann equation, with a dynamical screening $|\e_\hbar(\Phi,k,v_1)|^{-2}$, and verifies formally the same conservation laws and H-theorem
			\[\frac{\ud}{\ud t} \int \Phi_\hbar(t,v_1) \left(\begin{matrix}
			1\\v_1\\|v_1|^2
			\end{matrix}\right)\ud{v_1}=0,~\frac{\ud}{\ud t} \int \Phi_\hbar(t,v_1)\log\Phi_\hbar(t,v_1)\ud{v_1}\leq 0.\]
			Hence, it describes an quantitative irreversibility coming from the mixing of the system.
			
			This equation has been first formally derived by Balescu \cite{Balescu1} for particles in the Bose-Einstein or Fermi-Dirac statistic (for these particles, he found a cubic version of the equation).
			
			\begin{remark}
				Let  $F_N$ be a Bosonic density matrix.
				\begin{itemize}
					\item if there exists $F_1\in \mathcal{L}^1(\mathfrak{H}_L)$ such that $F_N = F_1^{\otimes N}$, then there exists $\psi\in\mathfrak{H}$ such that $F_1(x,y)=\psi(x)\bar{\psi}(y)$,
					\item if $F_1$ is translation invariant, then there exists a family $(c_k)_{k}$ such that
					\[F_1(x,y) = \sum_{k\in\mathbb{Z}_L}c_ke^{ik(x-y)}.\]
				\end{itemize}
				We deduce that if $F_N$ is symmetric and invariant by translation, then there exists $k\in\mathbb{Z}_L^d$ such that
				\[F_{\hbar,N}(x_{[1,N]},y_{[1,N]}) = \frac{1}{L^{Nd}}e^{ik\cdot\sum_{j=1}^N(x_j-y_j)},\]
				which is a fix point of the mean-field dynamic. This shows that our results cannot be applied to Bosons.
				
				Note that the density $F_1^{\otimes N}$ is not fermionic, hence our result cannot be applied to Fermions.
			\end{remark}

			Note that in the weak coupling scaling, a quantum system can also be described by a \emph{collisional kinetic} equation. In that setting, the interaction are given by the short range potential: the density matrix respect the equation
			\begin{equation}
				i\hbar\partial_t F_N = -{\hbar^2}\sum_{i= 1}^N\left[\tfrac12\triangle_i,F_N\right]+\sqrt{\hbar}\sum_{i<j}^N\left[\mathcal{V}\left(\tfrac{x_i-x_j}{\hbar}\right),F_N\right].
			\end{equation}
			with the scaling $N\hbar^3 = 1$ (in dimension $3$) and $L=1$. The equation has been first formally derived by Nordheim \cite{Nordheim}, and Uehling and Uhlenbeck \cite{UU} (see also \cite{Pulvirenti}).
	
			\paragraph{The classical counterpart}	
			A similar phenomenon exists in the classical setting.
			
			Consider a gas of $N$ particles inside the domain $\mathbb{T}_L$, with density $\mu = N/L^d\to\infty$. Denoting their coordinates $(x_1,\cdots,x_N,v_1,\cdots,$ $v_N)$, particles follow the classical Hamiltonian dynamic linked to the energy 
			\begin{equation}
				\mathcal{H}_N:=\sum_{i= 1}^N\frac{|v_i|^2}{2}+\frac{1}{\mu}\sum_{i<j}^N\mathcal{V}(x_i-x_j).
			\end{equation}
			We denote $F_N(t,x_{[1,n]},v_{[1,n]})$ the distribution of particle. It is supposed  initially factorized:
			\[\tilde{F}_N(t=0,x_{[1,n]},v_{[1,n]}):=\tilde{F}_0(x_1,v_1)\cdots \tilde{F}_0(x_N,v_N)\]
			for some probability density $F_0$ on $\mathbb{T}_L\times\mathbb{R}^d$. Then one can show that the first marginal 
			\[\tilde{F}_N^{(1)}(t,x_1,v_1) := \int \tilde{F}_N(t,x_{[1,n]},v_{[1,n]})\ud{x_{[2,n]}}\ud{v_{[2,n]}}\]
			converges when $N$ goes infinity to $\tilde{F}(t,x,v))$, the solution of the Vlasov equation (see Braun et Hepp \cite{BH}, Dobrushin \cite{Dobrushin}, and a review of Golse \cite{Golse2})
			\begin{equation}
				\left\{\begin{split}
				\partial_t \tilde F(t) &= -v\cdot\nabla_x\tilde{F}(t)+\nabla_x\tilde{\mathcal{V}}\cdot\nabla_v\tilde{F}(t),~\tilde{F}(t= 0)= \tilde{F}_0\\
				\tilde{\mathcal{V}}(t,x) &= \int \mathcal{V}(x-y)\tilde{F}(t,y,v)\ud{y}\ud{v}.
				\end{split}\right.,
			\end{equation}
	
			If we suppose that the initial density is invariant by translation, $F_0(x,v):=L^{-d}\Phi_0(v)$, then the effective potential $\tilde{\mathcal{V}}$ becomes constant, and the Vlasov dynamics is trivial. As in the quantum case, one shows that the scaled density	
			\[\Phi_N(t,v_1):=\int \tilde{F}_N(\mu t,x_{[1,n]},v_{[1,n]})\ud{x_{[1,n]}}\ud{v_{[2,n]}}\]
			converges as $N\to\infty$, $NL^d=1$ to $\Phi(t,v_1)$, solution of the (classical) Lenard-Balescu equation
			\begin{gather}
				\label{eq:Lenard--Balescu classique} \tag{CLB}\partial_t \Phi = Q_{LB}^0(\Phi)(v_1)\\
				Q_{LB}^0(\Phi)(v_1):=\nabla_{v_1}\cdot\int B(\Phi,v_1,v_2)\left(\Phi(v_2)\nabla \Phi(v_1)-\Phi(v_1)\nabla \Phi(v_2)\right)\ud{v_2}\\
				B(\Phi,v_1,v_2):= \int  \frac{|\V(k)|^2k\otimes k}{|\e_0(\Phi,k,v)|^2}\delta_{k\cdot(v_1-v_2)}\uud{k} \label{eq:collision kernel}\\
				\e_0(\Phi,k,v) := 1+\V(k)\int \frac{k\cdot\nabla \Phi(v_*)}{k\cdot(v-v_*)-i0}\ud{v_*}.
			\end{gather}
			
			The equation has been derived simultaneously in the 60s by Guernsey \cite{Guernsey, Guernsey2}, Lenard \cite{Lenard} and Balescu \cite{Balescu2}. In the former two derivations, the authors used a truncation of the BBGKY hierarchy. This strategy has been by adapted by Duerinckx \cite{Duerinckx} and Duerinckx--Saint-Raymond \cite{DS} in order to obtain a consistency result (for $L=1$ fixed). Note that in the first result, the author reach the time $O((\log N)^{1-\delta})$ (compare to $O(N)$ of the kinetic time), when in \cite{DS}, the authors reaches the longer time $O(N^\delta)$ in a linear setting (where $\delta\in(0,1)$ is small enough). 
			
			The classical Lenard-Balescu equation has been studied by Duerinckx and Winter in \cite{DW}.
		
			\paragraph{Semi-classical limit of the quantum Lenard--Balescu equation} 	
			In the semi-classical limit $\hbar\to0$, one can proof that the quantum system "converge" to the classical one. More precisely, consider $F_\hbar(t,x,y)$ solution of the Hartree equation with initial data $F_{0,\hbar}(x,y)$. At time $0$, we suppose that the Wigner transform of $F_{0,\hbar}$ defined by
			\begin{equation}
				\tilde{F}_{\hbar,0}(x,v):=\frac{1}{(2\pi\hbar)^d}\int F_{0,\hbar}(x+\tfrac{y}{2},x-\tfrac{y}{2}) e^{i\frac{v}{\hbar}y}\ud{y}
			\end{equation}
			converges to some distribution  $\tilde{F}_0(x,v)$. Then, the Wigner transform of $F_\hbar(t,x,y)$ converge to $\tilde{F}(t,x,v)$, the solution of the Vlasov equation with initial data $F_0(x,v)$.

			One can ask a similar question for the quantum Lenard--Balescu equation: for $\Phi_\hbar(t,v_1)$ the solutions of the quantum Lenard--Balescu \eqref{eq:equation de Lenard Balescu Quantique} with initial data $\Phi_0(v_1)$ equation converge to $\Phi(t,v_1)$, the solution of the classical Lenard--Balescu equation \eqref{eq:Lenard--Balescu classique} with the same initial data.
			
			This problem looks like a grazing collision limit: we want to link a Boltzmann like operator
			\begin{equation*}
				Q^\hbar(\Phi)(v_1)=\int_{\mathbb{S}^{d-1}\times\mathbb{R}^d} \left(\Phi\left(\tfrac{v_1+v_2}{2}+\tfrac{|v_1-v_2|}{2}\sigma\right)\Phi\left(\tfrac{v_1+v_2}{2}-\tfrac{|v_1-v_2|}{2}\sigma\right)-\Phi(v_1)\Phi(v_2)\right) B_\hbar(v_1,v_2,\sigma)\ud{\sigma}\ud{v_2}
			\end{equation*}
			toward a Landau type operator
			\begin{equation*}
				Q^0(\Phi)(v_1)=\nabla_1\cdot\int_{\mathbb{R}^d} \left(\nabla\Phi\left(v_1\right)\Phi\left(v_2\right)-\Phi(v_1)\nabla\Phi(v_2)\right) B_0(v_1,v_2)\ud{v_2}.
			\end{equation*}
			if the collision kernel $B_\hbar(\Phi,\tfrac{v_1+v_2}{2},\tfrac{v_1-v_2}{2},\sigma)$ concentrates on the $\sigma$ with $\sigma\cdot\frac{v_1-v_2}{|v_1-v_2|}\simeq 1$. Here, we use the variable $\sigma := \frac{2\hbar k +v_1-v_2}{|v_1-v_2|}$, which is common in the literature.
			
			The rigorous analysis of the grazing collision limit goes back to the classical result of Arsenev and Buryak in \cite{AB}, and a general setting as been treated by Alexandre and Villani in \cite{AV}. However, the solutions consider by these two authors are very weak (they are called the \emph{renormalized solution}). As the Lenard-Balescu kernel $B_\hbar(\Phi_\hbar(t),v_1,v_2,\sigma)$ depends on the solution $\Phi_\hbar(t)$, one needs a strong notion of solution. 
	
			In \cite{He}, the author proved the strong convergence of the solution of the Boltzmann equation toward the solution of the Landau equation, in the case where the kernel $B_\hbar(\Phi_\hbar(t),v_1,v_2,\sigma)$ is independent of $\Phi_\hbar(t)$ and $\frac{v_1+v_2}{2}$, and that is not integrable with respect to the variable $\sigma$ (these kernels are said \emph{uncutoff}). In that case, the Boltzmann operator acts like a fractional Laplacian, bringing some additional parabolic regularity.
			
			Note also the work of He, Lu, Pulvirenti and Ma around the classical limit of the quantum Boltzmann equation \cite{HLP, HLPZ}.
			
			In the case of Lenard--Balescu kernel, there is \emph{cutoff}: for $\V\in L^1\cap L^\infty(\mathbb{R}^3)$, the integral
			$\int \V(k)^2\delta_{k\cdot(v_2-v_1-\hbar k)}\uud{k} $ is always finite.
			
			\begin{notation}
				As the constants in the inequalities change along the computation, we denote $A\lesssim B$ if there exists a constant $C>0$ independent of $\hbar$ such that $A\leq CB$. If $A\lesssim B$ and $B\lesssim A$, we denote $A\simeq B$. We denote $A=O(B)$ if $A\lesssim B$.\smallskip
				
				For $v_1,v_2,k\in\mathbb{R}^d$, $\hbar>0$ and $g,h:\mathbb{R}^d\to\mathbb{C}$, we denote
				\begin{gather*}
					v_1' :=v_1+\hbar k,~v_2':=v_2-\hbar k,~v_1'' :=v_1-\hbar k,~v_2'':=v_2+\hbar k\\
					\DDelta gh_*:=g(v_1)h(v_2)+g(v_2)h(v_1)-g(v'_1)h(v'_2)-g(v_2')h(v_1')\\
					\forall i\in\{1,2\},~j\in\{~,'\},~g_i^j:=g(v_i^j),\\
					\hat{k}:=\tfrac{k}{|k|},~\ud V:= \ud v_1\ud v_2\uud k.
				\end{gather*}
				
				For $\alpha = (\alpha_1,\cdots,\alpha_d)\in\mathbb{N}^d$, $(\beta^i)_{i\leq d} = (\beta^i_j)_{i,j\leq d}\in\mathbb{N}^{d^2}$ with $\beta^1+\cdots+\beta^d = \alpha$, and $\gr{e}_i := (0,\cdots,0,\underset{i}{1},0,\cdots,0)$,  we denotes
				\begin{gather*}
					\binom{\alpha}{\beta^1\cdots \beta^d} = \frac{\prod_i \alpha_i !}{\prod_{i,j}\beta_i^j!},~|\alpha|=\sum_i \alpha_i\\
					\partial_v^\alpha g(v) =  (\gr{e}_1\cdot\nabla_v)^{\alpha_1}\cdots(\gr{e}_d\cdot\nabla_v)^{\alpha_d} f(v).
				\end{gather*}
			\end{notation}
	
		\subsection{Statement of the results}

			\subsubsection{Short time validity of the quantum Lenard--Balescu equation}
				We begin by justifying the quantum Lenard--Balescu equation. The main goal would be reaching the kinetic time $O(N)$, but this is for the moment out of reach. However, we are able to propose a partial result, up to some time $O((\log N)^{(1-\delta)})$ for any $\delta\in(0,1)$.
			
				We consider an interaction potential $\mathcal{V}$ satisfying the following assumption:
				\begin{assumption}\label{ass:premier potentiel}
					The interaction potential $\mathcal{V}$ is smooth, radial and with compact support.
				\end{assumption}
				
				\begin{theorem}\label{thm:limite champ moyen}
					Consider a density $\Phi_0$ such that 
					\begin{equation}
						\Phi_0\geq 0,~\int \Phi_0(v)\ud{v}= 1~{\rm and}~\|\<v\>^{2d}\<\nabla\>^2\Phi_0\|_{L^\infty}\|\V\|_{\L^\infty}<\frac{1}{2}.
					\end{equation}
					At $t=0$ we set the system such that 
					\begin{gather}\label{eq:def de F0L}
						F_{0,L}(x,y):= \frac{(2\pi\hbar)^d}{L^{2d}\mathcal{Z}_L}\sum_{k\in\mathbb{Z}_L}\Phi_0(\hbar k) e^{ik\cdot(x-y)},\\
						\hat{F}_{N,0} = F_{0,L}^{\otimes N.}
					\end{gather}
					where $\mathcal{Z}_L$ is a normalization constant, and we construct $F_N(t):\mathbb{R}^+\to\mathcal{L}^1(\mathfrak{H}^N_L)$ the solution of the Von Neumann equation \eqref{eq:Von Neuman} with initial data $F_{N,0}$.
					
					We fix the scaling $\mu,L\to\infty$ with $\mu = NL^{-d}$. Then, for any diverging sequence $\tau_N$ such that for some $\delta>0$, $\tfrac{\tau_N}{\hbar\mu} \ll (\log N)^\delta $ and $L\geq 2\hbar\log\tfrac{\tau_N}{L}$, and for any test function $\psi\in\mathcal{C}^\infty_c(\mathbb{R}_t^+\times \mathbb{R}_v^d)$, the scaled Wigner transform $\Phi_{N,L}$ verifies
					\begin{equation}
						\int_0^\infty \frac{(2\pi\hbar)^d}{L^d}\sum_{v\in\mathbb{Z}_{L/\hbar}} \psi\left(t,v\right)\partial_t\Phi_{N,L}(\tfrac{\tau_N}{\mu}t,v) \ud{t} \underset{N\to\infty}{\longrightarrow}\int_0^\infty\int \psi(t,v)Q^\hbar_{LB}(\Phi_0)(v) \ud{v}\ud{t}.
					\end{equation}	
				\end{theorem}
				\begin{remark}
					One can choose the scaling $\mu L^d=N$, with $\tau_n\leq (\log N)^{\frac{1}{2}-\beta}$ for $\beta>0$ and $L \simeq \log\log N$.
				\end{remark}
				The proof is given in Section \ref{sec:val au temps 0}.
				
				\paragraph{Ideas of the proof} We follow mostly the strategy of \cite{Duerinckx} in the classical setting. It is based on the estimation of the "defect of factorization" (also called the \emph{cumulant}, introduce in \eqref{eq:def des cumulants}) of the marginal of the system. One can prove that for short time (of order $O((\log N)^{1-\delta})$), the correction implying three or more particles can be neglected. The remaining system of two equations gives the Lenard--Balescu evolution. \medskip
		
				In \cite{Duerinckx}, the estimations of the cumulants are based on the Glauber calculus. In the present paper, we use an improvement of the cumulant estimates of \cite{PPS}.
				
				Note that in general, obtaining a justification of a collisional kinetic equation from a particle system is a difficult problem. The only full (non-linear, until a kinetic time) results hold in the low density setting, for the hard spheres system \cite{Lanford,King,IP,Spohn3,GST,PSS,DHM}, and for the interacting wave system \cite{DH,DH1}. The other results are either in a \emph{linear setting} (when we follow a particle in a bath at equilibrium \cite{vBHLS,DP,DR,BGS2,Ayi,Catapano1,MS,DL2,BL}) or in the linearized setting (the study of the fluctuations of the empirical measure around the equilibrium \cite{Spohn,Spohn2,BGS,BGSS1,BGSS2,LeBihan}),  or are consistency result (result at time 0) \cite{BPS,Winter,VW,Duerinckx,DS}.
		
			\subsubsection{Property of quantum Lenard--Balescu equation around the equilibrium}
				We study now the property of quantum Lenard--Balescu equation. We note that the Maxwellian  \[M(v):=\tfrac{1}{(2\pi)^{d/2}}\exp(-\tfrac{|v|^2}{2})\]
				is a steady state of the equation. We want to look at a solution close to this equilibrium:
				\begin{equation}
					F(t,v_1) = M(v_1)+\sqrt{M(v_1)} g(t,v_1).
				\end{equation}
				
				The equation solved by $g(t,v_1)$ is 
				\begin{align}\label{eq:QLB-bis}
					\tag{QLB$_\hbar$-bis}
					\partial_t g(t,v_1)& + \mathcal{L}_\hbar g(t,v_1)=\mathcal{Q}_\hbar(g(t),g(t))(v_1)
				\end{align}
				where we denote $F_h:=M+\sqrt{M}h$ and
				\begin{align}
					\mathcal{L}_\hbar g(v_1)&:=\frac{2}{\hbar^{2}}\int\DDelta \left(g\sqrt{M_*}\right) \frac{ \hat{\mathcal{V}}(k)^2\delta_{ k\cdot(v_2-v_1-\hbar k)}}{|\e_\hbar(M,k,v_1)|^2}\sqrt{M_2}\uud k \ud{v_2}\\
					\mathcal{Q}_\hbar(g,h)(v_1):=&-\frac{2}{\hbar^{2}}\int \DDelta(gh_*)\frac{ \hat{\mathcal{V}}(k)^2\delta_{ k\cdot(v_2-v_1-\hbar k)}}{|\e_\hbar(F_h,k,v_1)|^2} \sqrt{M_2}\uud{k}\ud{v_2}\\
					&+\frac{2}{\hbar^{2}}\int\DDelta\left(g\sqrt{M_*}\right)\left(\frac{ \hat{\mathcal{V}}(k)^2\delta_{ k\cdot(v_2-v_1-\hbar k)}}{|\e_\hbar(M,k,v_1)|^2}-\frac{ \hat{\mathcal{V}}(k)^2\delta_{ k\cdot(v_2-v_1-\hbar k)}}{|\e_\hbar(F_h,k,v_1)|^2}\right) \sqrt{M_2}\uud k\ud{v_2},\nonumber
				\end{align}
				where $\mathcal{L}_\hbar$ is the Linearized Boltzmann operator.
	
				It is well known that the operator $\mathcal{L}_\hbar$ is a self-adjoint non-negative operator in $L^2(\mathbb{R}^d)$.
				
				One of the goals of the present paper is to describe the limit as $\hbar$ goes to zero. We want to link the quantum Lenard-Balescu with its classical counterpart
				
				For perturbative solution $F(t):= M+\sqrt{M} g(t)$, the classical equation \eqref{eq:Lenard--Balescu classique} becomes
				\begin{align}\label{eq:CLB-bis}	\tag{CLB}
					\partial_t g(t,v_1)& + \mathcal{L}_0 g(t,v_1)=\mathcal{Q}_0(g(t),g(t))(v_1)
				\end{align}
				where we denote $H:=M+\sqrt{M} h$ and
				\begin{align}
					\mathcal{L}_0 g(v_1)&:=\left(\nabla-\frac{v_1}{2}\right)\cdot\int B(M,v_1,v_2) \left[\sqrt{M_2}\left(\nabla_1+\frac{v_1}{2}\right)g_1-\sqrt{M_1}\left(\nabla_2+\frac{v_2}{2}\right)g_2\right]\sqrt{M_2} \ud{v_2}\\
					\mathcal{Q}_0(g,h)&(v_1):=-\left(\nabla-\frac{v_1}{2}\right)\cdot\int B(H,v_1,v_2) \left[h_2\nabla g_1+g_2\nabla h_1-h_1\nabla g_2-g_1\nabla h_2\right]\sqrt{M_2} \ud{v_2}\\
					+\left(\nabla-\frac{v_1}{2}\right)\cdot\int &\left(B(H,v_1,v_2)-B(M,v_1,v_2)\right) \left[\sqrt{M_2}\left(\nabla_1+\frac{v_1}{2}\right)g_1-\sqrt{M_1}\left(\nabla_2+\frac{v_2}{2}\right)g_2\right]\sqrt{M_2} \ud{v_2},\nonumber
				\end{align}
				with the collision kernel $B(H,v_1,v_2)$ defined in \eqref{eq:collision kernel}.
				
				In the following, we choose a potential $\mathcal{V}$ satisfying the following assumption:
				\begin{assumption}\label{ass: potentiel}
					We set the dimension $d\geq 2$. We define  $\mathcal{V}$ a symmetric ($\mathcal{V}(x)=\mathcal{V}(|x|)$) such that there exists two constants $C>0$ and $\eta> 0$
					\begin{gather}
						|\hat{\mathcal{V}}(k)|^2+|\hat{\mathcal{V}}'(k)|^2\leq \frac{C}{(1+|k|)^{d+3+\eta}}.
					\end{gather}
				\end{assumption}
				
				We introduce the weighted Sobolev space $\mathcal{H}^r\subset L^2(\mathbb{R}^d)$, with norm
				\begin{equation}
					\|g\|_r^2 := \sum_{r'=0}^r\left\|\left(\nabla-\frac{v}{2}\right)^{r'}g\right\|^2<\infty,
				\end{equation}
				where we denote $\|~\|$ the $L^2(\mathbb{R}^d)$. Note that the norm $\|g\|_r$ is equivalent to $\|(\nabla-\frac{v}{2})^rg\|$.
	
				\begin{theorem}\label{thm:Cauchy theory}
					For any $d\geq 2$ and 
					$r\geq 5$
					\begin{itemize}
						\item  \emph{Cauchy problem.} There exists $\hbar_0>0$ and $\delta>0$ such that for any initial data $g_0\in\mathcal{H}^r$ with $\|g_0\|_r\leq \delta$ and any $\hbar<\hbar_0$, there exists a unique $g_\hbar\in\mathcal{C}^0_b(\mathbb{R}^+,\mathcal{H}^r)$ solution of the \eqref{eq:QLB-bis}-equations in the mild form:
					\begin{equation}
						g_\hbar(t) = e^{-t\LL_\hbar}g_0+\int_0^te^{-(t-s)\LL_\hbar}\mathcal{Q}_\hbar(g_\hbar(s),g_{\hbar}(s))\ud s.
					\end{equation}
					
					In addition, one has the following bound:
					\begin{equation}
						\sup_{t\geq 0}\|g_\hbar(t)\|_r\leq \|g_0\|_r.
					\end{equation}
					
					\item \emph{Semi-classical limit.} The sequence of  $(g_\hbar)_{\hbar<\hbar_0}$ of solution of the \eqref{eq:QLB-bis}-equations converges weakly in $\mathcal{C}_b^0(\mathbb{R}^+,\mathcal{H}^r)$ as $\hbar\to0$ to $g_\infty$, mild-solution of the \eqref{eq:CLB-bis}-equation with initial data $g_\infty(t=0)=g_0$:
					\begin{equation}
						g_\infty(t) = e^{-t\LL_0}g_0+\int_0^te^{-(t-s)\LL_0}\mathcal{Q}_0(g_\infty(s),g_{\infty}(s))\ud s.
					\end{equation}
					
					In addition, there exists some positive constant $C>0$ independent of $g_0$ and $\hbar$ such that:
					\begin{equation}
						\forall t\geq 0,~\|g_\hbar(t)-g_\infty(t)\|\leq C 
						\hbar t \|g_0\|_r .
					\end{equation}
					\end{itemize}
				\end{theorem}
				
				The proof is given in Section \ref{sec:cauchy theory}. 
				
				\begin{remark}
					The preceding estimates can be improve to get a similar bound on $\|g_\hbar(t)-g_\infty(t)\|_r$. As the computations become technical, we prefer to skip them.
				\end{remark}
				
				\paragraph{Ideas of the proof}
				As we want to perform a semi-classical limit, we want to obtain a Cauchy theory closed to the one obtain by Duerinckx and Winter \cite{DW} (following the strategy of \cite{Guo}). 
				
				In classical case, Duerinckx and Winter construct a weighted Sobolev norm $\|\cdot\|_0$,
				\[\|g\|_0^2 \simeq \int \frac{|p^\perp_v \nabla g(v)|^2+|g(v)|^2}{1+|v|}+\frac{|(\text{Id}-p^\perp_v) \nabla g(v)|^2}{1+|v|^3}\ud v,~p^\perp_v := \text{Id}-\tfrac{v\otimes v}{|v|^2},\] controlling the non-linear operator. In order to use the dissipation coming from the linear operator $\mathcal{L}_0$, the norm $\|\cdot\|_0$ is bonded from below by $\mathcal{L}_0$: denoting $\pi_0$ the $L^2$-orthogonal projector onto the space $<\sqrt{M},v_1\sqrt{M},\cdots v_d\sqrt{M},|v|^2\sqrt{M}>^\perp$,
				\[\forall g\in\mathcal{C}^\infty_c(\mathbb{R}^d),~\int g\mathcal{L}_0 g\geq \tfrac 1C \|(1-\pi_0)g\|_0^2 \]
				for some constant $C>0$.
				
				In the quantum case, we construct a family of norm $\|~\|_\hbar$
				\[\|g\|_\hbar^2 \simeq \int \frac{ |g(v)|^2}{1+|v|}\ud v+\int\frac{\V(k)^2}{|k|}(g(v+\hbar k)-g(v))^2e^{-\alpha\frac{|v\cdot k|^2}{|k|^2}}\ud v\uud k\]
				for some fix constant $\alpha>0$. In the semi-classical limit,  $\|~\|_\hbar$ converge to $\|~\|_0$ (for $g$ smooth enough, $\|g\|_\hbar\to\|g\|_0$), but we always have $\|f\|_\hbar\leq \tfrac{C}{\hbar}\|f\|$ (for some constant $C>0$).
				
				The first point of the proof is that $\|~\|_\hbar$ does not depend on the constant $\alpha$ (see Corollary \ref{prop:form equivalent norm}). This is crucial in the proof that these norms control the non-linear operators $\mathcal{Q}_\hbar$.
				
				The second point is that, as in the classical setting, there exists a constant $C_*>0$ independent of $\hbar$ such that
				\[\forall g\in\mathcal{C}^\infty_c(\mathbb{R}^d),~\int g\mathcal{L}_\hbar g\geq \tfrac1{C_*} \|(1-\pi_0)g\|_\hbar^2, \]
				(see Proposition \ref{prop:dissipation estimates}). Thus, the dissipation due to the linearised operator is sufficient to control the non-linear operator, and we can close some energy estimates.

				\begin{remark}
					The construction of the constant $C_*$ is based on a compactness argument and is not constructive.
				\end{remark}
				
			\subsubsection{Property of quantum Lenard--Balescu equation away from the equilibrium}
				We study now the property of quantum Lenard--Balescu equation far from equilibrium.
				In the following, we choose a potential $\mathcal{V}$ satisfying the following assumption:
				\begin{assumption}\label{ass: potentiel bis}
					We set the dimension $d\geq 2$. We define  $\mathcal{V}$ a symmetric ($\mathcal{V}(x)=\mathcal{V}(|x|)$) such that there exists two constants $C>0$ and $\eta> 0$
					\begin{gather}
						|\hat{\mathcal{V}}(k)|^2+|\hat{\mathcal{V}}'(k)|^2\leq \frac{C}{(1+|k|)^{ 2d+3+\eta}}.
					\end{gather}
				\end{assumption}
				
				\begin{theorem}\label{thm:Cauchy theory bis}
					For any $d\geq 2$ and 
					$N\geq 4$, $m\geq \tfrac{d}{2}+4$, fix an initial data $\Phi_0\geq 0$ such that $\|\<v\>^m\Phi_0\|_{H^{N+1}}$ and $\sup_{k,v_1}\tfrac{1}{|\e_0(\Phi_0,k,v_1)|^2}$ are bounded.
					
					\begin{itemize}
						\item \emph{Cauchy theory.} There exists $\hbar_0>0$ and $T>0$ such that for any $\hbar<\hbar_0$ there exists a unique $\varphi_\hbar\in\mathcal{C}^0([0,T],H^N(\<v\>^{2m}\ud v))$ solution of the \eqref{eq:equation de Lenard Balescu Quantique}-equations in mild-form:
					\begin{equation}
						\varphi_\hbar(t) = \Phi_0+\int_0^t{Q}_{LB}^\hbar(\varphi_\hbar(s),\varphi_{\hbar}(s))\ud s.
					\end{equation}
					In addition, there exists a constant $C$ depending only on $\Phi_0$ such that
					\begin{equation}
						\sup_{t\in[0,T]}\|\<v\>^mg_\hbar(t)\|_{H^N}\leq C.
					\end{equation}
					
					\item \emph{Semi-classical limit.} The sequence of  $(g_\hbar)_{\hbar<\hbar_0}$ of solution of the \eqref{eq:equation de Lenard Balescu Quantique}-equations converges weakly in $\mathcal{C}_b^0([0,T],H^N(\<v\>^{2m}\ud v))$ as $\hbar\to0$ to $g_\infty$, mild-solution of the \eqref{eq:Lenard--Balescu classique}-equation with initial data $\varphi_\infty(t=0)=\Phi_0$:
					\begin{equation}
						\varphi_\infty(t) = \Phi_0+\int_0^t{Q}_0(\varphi_\infty(s),\varphi_{\infty}(s))\ud s.
					\end{equation}
					
					In addition, there exists some positive constant $C>0$ such that:
					\begin{equation}
						\forall t\in[0,T],~\|\<v\>^m(\varphi_\hbar(t)-\varphi_\infty(t))\|_{L^2}\leq C 
						\hbar t.
					\end{equation}
				\end{itemize}
				\end{theorem}			
				
				\paragraph{Ideas of the proof}
				It will be useful to introduce the trilinear operator
				\begin{equation}
					Q_\hbar(g,h,\varpi)(v_1):=\tfrac{2}{\hbar^2}\int (g_1'h_2'-g_1h_2)\varpi(k,v_1)\V(k)^2\ud V
				\end{equation}
				with $Q_{LB}^\hbar(\varphi) = Q_\hbar(\varphi,\varphi,\tfrac{1}{|\e_\hbar(\varphi)|^2})$.
				
				The proof can be split into two steps. First, we construct solutions to \eqref{eq:equation de Lenard Balescu Quantique} on short time interval (with length depending on $\hbar$), via a fix argument in a $H^N(\<v\>^{m}\ud v)$. The second step is the propagation of regularity on time scale independent of $\hbar$ via an energy method.
				
				As we want to understand the evolution of $\|\<v\>^m\partial^\alpha_v \varphi_\hbar(t)\|^2_{L^2}$, we need to estimate 
				\begin{gather*}
					\left\<\partial^\alpha_v \varphi_\hbar, \partial_v^\alpha Q_\hbar(\varphi_\hbar,\varphi_\hbar,\tfrac{1}{|\e_\hbar(\varphi_\hbar)|^2})\right\> = \sum_{\beta_1+\beta_2+\beta_3=\alpha} \binom{\alpha}{\beta_1,\beta_2,\beta_3} \left\<\partial^\alpha_v \varphi_\hbar,  Q_\hbar\big(\partial_v^{\beta_1}\varphi_\hbar,\partial_v^{\beta_2}\varphi_\hbar,\partial_v^{\beta_3}\tfrac{1}{|\e_\hbar(\varphi_\hbar)|^2}\big)\right\>\\
					\text{with }\partial_v^{\beta_3}\tfrac{1}{|\e_\hbar(\varphi_\hbar)|^2} = -2\V(k)^2\Re \left(\tfrac{\overline{\e_\hbar(\varphi_\hbar)}}{|\e_\hbar(\varphi_\hbar)|^4}\int \tfrac{\partial_v^{\beta_3}\varphi_\hbar(v_*)-\partial_v^{\beta_3}\varphi_\hbar(v_*-\hbar k)}{\hbar k\cdot(v_*-v-\hbar k)-i0}\ud{v_*}\right)+\cdots
				\end{gather*}
				where the rest only implies lower order derivatives of $\varphi_\hbar$.
				
				If $\max (|\beta_1|,|\beta_2|,|\beta_3|)\leq N-1$, we can use the naive estimation $\|g(\cdot)-g(\cdot+\hbar k )\|_{L^2}\leq \hbar |k|\|g\|_{H^1}$ and usual Sobolev injection to bound $Q_\hbar\big(\partial_v^{\beta_1}\varphi_\hbar,\partial_v^{\beta_2}\varphi_\hbar,\partial_v^{\beta_3}\tfrac{1}{|\e_\hbar(\varphi_\hbar)|^2}\big)$.
				
				The more involving terms are $(\beta_1,\beta_2,\beta_3)\in\{(\alpha,0,0),(0,\alpha,0),(0,0,\alpha)\}$. These terms need some additional quantitative regularity, in term of the following energy
				\begin{equation*}
					\mathcal{A}_{\hbar,g}(f) := \int \tfrac{(f(v_1+\hbar k)-f(v_1))^2}{\hbar^2}g(v_2)\hat{\mathcal{V}}(k)^2\delta_{k\cdot(v_2-v_1-\hbar k)}\ud V.
				\end{equation*}
								
			We can bound $\mathcal{A}_{\hbar, \varphi_\hbar}(\partial^\alpha \varphi_\hbar)$ from the term $(\beta_1,\beta_2,\beta_3) = (\alpha,0,0)$: using the symmetry $(v_1, v_2, k)     \mapsto (v_1',v_2',-k)$, 
			\begin{align*}
				\left\<\partial^\alpha_v \varphi_\hbar, Q_\hbar(\partial_v^\alpha \varphi_\hbar,\varphi_\hbar,\tfrac{1}{|\e_\hbar(\varphi_\hbar)|^2})\right\>  &=  \tfrac{2}{\hbar^2} \int \partial_v^\alpha \varphi_{\hbar, 1}(\partial_v^\alpha \varphi_{\hbar, 1}'\varphi_{\hbar,2}'-\partial_v^\alpha \varphi_{\hbar, 1}\varphi_{\hbar,2})\tfrac{\V^2(k)}{|\e_\hbar(\varphi_\hbar,k,v_1)|^2}\delta_{k\cdot(v_2-v_1-\hbar k)}\ud V\\
				&=\tfrac{2}{\hbar^2} \int \partial_v^\alpha \varphi_{\hbar, 1}\varphi_{\hbar,2}(\partial_v^\alpha \varphi_{\hbar, 1}'-\partial_v^\alpha \varphi_{\hbar, 1})\tfrac{\V^2(k)}{|\e_\hbar(\varphi_\hbar,k,v_1)|^2}\delta_{k\cdot(v_2-v_1-\hbar k)}\ud V\\
				&=-\tfrac{1}{\hbar^2} \int (\partial_v^\alpha \varphi_{\hbar, 1}'-\partial_v^\alpha \varphi_{\hbar, 1})^2\varphi_{\hbar,2}\tfrac{\V^2(k)}{|\e_\hbar(\varphi_\hbar,k,v_1)|^2}\delta_{k\cdot(v_2-v_1-\hbar k)}\ud V\\
				&~~+\tfrac{1}{\hbar^2} \int ((\partial_v^\alpha \varphi_{\hbar, 1}')^2-(\partial_v^\alpha \varphi_{\hbar, 1})^2)\varphi_{\hbar,2}\tfrac{\V^2(k)}{|\e_\hbar(\varphi_\hbar,k,v_1)|^2}\delta_{k\cdot(v_2-v_1-\hbar k)}\ud V.
			\end{align*}
			The first term looks like $\mathcal{A}_{\hbar, \varphi_\hbar}(\partial^\alpha_v \varphi_\hbar)$ (as $ \Phi_0$ and $\tfrac1{|\e_\hbar(\varphi_\hbar)|^2}\gtrsim 1$). The second term can be bound by the cancellation lemma (Propostion \ref{prop:lemme d'annulation}). This estimation can be understood as a discrete integration by part
			\begin{multline}\label{eq:illustration cancellation lemma}
				\frac{1}{\hbar^2} \int ((\partial_v^\alpha \varphi_{\hbar, 1}')^2-(\partial_v^\alpha \varphi_{\hbar, 1})^2)\varphi_{\hbar,2}\tfrac{\V^2(k)}{|\e_\hbar(\varphi_\hbar,k,v_1)|^2}\delta_{k\cdot(v_2-v_1-\hbar k)}\ud V \\
				= \int \partial_v^\alpha \varphi_{\hbar, 1}^2\varphi_{\hbar,2}\left[\int\tfrac{\V^2(k)}{\hbar^2}\left(\tfrac{\delta_{k\cdot(v_2-v_1)}}{|\e_\hbar(\varphi_\hbar,k,v_1-\hbar k)|^2}-\tfrac{\delta_{k\cdot(v_2-v_1-\hbar k)}}{|\e_\hbar(\varphi_\hbar,k,v_1)|^2}\right)\uud k\right]\ud v_1\ud v_2,
			\end{multline}
			where the integral in $k$ is bounded by $O(\tfrac{\<v_1-v_2\>}{|v_1-v_2|^3})$ \footnote{Note that in the setting of the quantum Boltzmann equation presented in \cite{HLPZ}, the integral in $k$ is bounded by some $\Psi(v_2-v_1)$, with $\Psi\in L^1$. The authors use the symmetry $k\mapsto -k$ of the collision kernel that leads exact cancellation. This symmetry is no longer verified in Lenard--Balescu setting.}, which is not $L^1_{\rm loc}$. Hence, this bound is only useful away from the singularitu, on $\{|{k}\cdot (v_1-v_2)| > \tfrac1K\}$ (with $K>1$ an arbitrary constant).
			
			In order to treat $\{|{k}\cdot (v_1-v_2)| \leq \tfrac1K\}$, we use that 
			\begin{multline}\label{eq:illustration cancellation lemma2}
				\frac{1}{\hbar^2} \int_{\{|{k}\cdot (v_1-v_2)| \leq \tfrac1K\}} ((\partial_v^\alpha \varphi_{\hbar, 1}')^2-(\partial_v^\alpha \varphi_{\hbar, 1})^2)\varphi_{\hbar,2}\tfrac{\V^2(k)}{|\e_\hbar(\varphi_\hbar,k,v_1)|^2}\delta_{k\cdot(v_2-v_1-\hbar k)}\ud V \\
				=\frac12\int_{\{|{k}\cdot (v_1-v_2)| \leq \tfrac1K\}} \frac{(\partial_v^\alpha \varphi_{\hbar, 1}')^2-(\partial_v^\alpha \varphi_{\hbar, 1})^2}{\hbar}\frac{\varphi_{\hbar,2}-\varphi_{\hbar,2}'}{\hbar}\tfrac{\V^2(k)}{|\e_\hbar(\varphi_\hbar,k,v_1)|^2}\delta_{k\cdot(v_2-v_1-\hbar k)}\ud V
			\end{multline}
		 	Then, we remark that (Proposition \ref{cor:equivalence des normes}), for $\V$ verifying Assumption \ref{ass: potentiel bis}, $g\geq 0$ and $\| \<v\>^{\frac d2+4} (g-\Phi_0)\|_{H^1}$ small enough,
				\begin{equation}\label{eq:norme aux differences finis, exemple 1000}
					\|f\|_{L^2}^2+\mathcal{A}_{\hbar,g}(f) \simeq_g \|f\|_{L^2}^2+\int \tfrac{(f(v_1+\hbar k)-f(v_1))^2}{\hbar^2} \tfrac{1}{\<k\>^{d+8}\<v_1\>^{\frac 92}}\ud v_1\uud k.
				\end{equation}
			This allows us to bound \eqref{eq:illustration cancellation lemma2} .
			
		\section{Validity at time $0$ of the Quantum Lenard-Balescu equation}\label{sec:val au temps 0}
			In the first part of the present paper, we present a system from which the quantum Lennard--Balescu equation can be derived (at least at time $0$).

			The strategy is quiet disconnected from the rest of the paper, and is mainly inspired from \cite{PPS} for the estimation of the cumulant, and from \cite{Duerinckx} for the computation of the collision operator.
			
			We denote $\mathcal{L}^1(\mathfrak{H})$ the space of trace-class operator on the Hilbert space $\mathfrak{H}$, and $\mathfrak{H}_L:= L^2(\mathbb{T}_L)$.
			
			We recall that the system is describe by a density matrix $F_N(t)\in\mathcal{L}^1(\mathfrak{H}_L^N)$, self-adjoint, symmetric: for all $\sigma\in\mathfrak{S}_N$ and $(\psi_i)_{i\leq N}\in\mathfrak{H}^N_L$,
			\[\left\< \psi_1\otimes\cdots\psi_N,F_N(t)\psi_1\otimes\cdots\psi_N\right\>= \left\< \psi_{\sigma(1)}\otimes\cdots\psi_{\sigma(N)},F_N(t)\psi_{\sigma(1)}\otimes\cdots\psi_{\sigma({N})}\right\>.\]
			It is solution of the Von-Neuman--Fock equation
			\[\frac{i\mu\hbar}{N}\partial_t F_N(t) = \frac{\mu\hbar^2}{N}\sum_{j=1}^{N} \left[-\triangle_j,F_N(t)\right]+\frac{1}{N}\sum_{j<k}^N \left[V_{j,k},F_N(t)\right].\]
			Here, $\triangle_j$ is the Laplacian with respect to the $j$-th variable, and $V_{j,k}$ the multiplication by $\mathcal{V}(x_j-x_k)$.
			
			At time $0$, we fix
			\begin{gather*}
				F_{0,L}(x,y):= \left(\frac{2\pi\hbar}{L^2}\right)^d\frac{1}{\mathcal{Z}_L}\sum_{k\in\mathbb{Z}_L}\Phi_0(\hbar k) e^{ik\cdot(x-y)},\\
				\hat{F}^N(t= 0) = F_{0,L}^{\otimes N.}
			\end{gather*}
	
			\subsection{{\it A priori} estimation of the cumulant}
				In the following we introduce the notations: for $f\in\LL^1(\mathfrak{H}_L)$ and $\sigma\subset[1,N]$
				\[f^\sigma:=\bigotimes_{k\in\sigma}f_k\in\LL^1(\mathfrak{H}^\sigma_L),\]
				and $\tr_\sigma$ the trace with respect the particles in $\sigma\subset[n]$:
				\[\forall A\in\LL^1(\mathfrak{H}^N),B\in\LL(\mathfrak{H}^\sigma),~\tr\left(\left(B\otimes\ind^{[1,n]\setminus\sigma}\right)A\right)=\tr\left( B\tr_\sigma A\right).\]
				
				We use the notation $[m,n] = \{m,m+1,\cdots,n\}$ and $[n]:=[1,n]$.
				
				That defines the $k$-th marginal of the density matrix $F_N\in\LL^1(\mathfrak{H}^N)$:
				\[F_N^{[k]}:=\tr_{[1,k]}F_N.\]
				
				If the system start from a factorized initial data $F_{N,0}= F_0^{\otimes N}$, one can expect that the system remains almost factorized. As the interaction creates some correlation between particles it will never remains fully factorized and we introduce the family of density matrix $(G_{N,n})_n\in\prod_{n\geq 1} \LL^1(\mathfrak{H}_L^n)$ to measure the defect of tensorisation:
				\begin{equation}\label{eq:def des cumulants}
					G_{N}^{[n]}(t) = \sum_{\sigma\subset[n]} (-1)^{\sigma} G_{N}^{[n]\setminus \sigma}(t)\otimes (F_N^{(1)}(t))^{\sigma}.
				\end{equation}
				Using the $G^N_\omega$, one can decompose the marginal as
				\begin{equation}
					F_N^{[n]}(t) = \sum_{\sigma\subset[1,n]}  G_{N}^{[1,n]\setminus \sigma}(t)\otimes (F_{N}^{(1)}(t))^{\sigma}.
				\end{equation}
				
				We will prove the following result, which is similar to Theorem 2.2 of \cite{PPS}.
				\begin{prop}\label{prop:estimation a priori}
					Let $(F_0^N)\in\prod_N\LL^1(\mathfrak{H}_L^N)$ a family of initial data such that for all $N\geq 1$, 
					\begin{equation}
						\forall n\in[1,N], \|G_{N}^{[n]}(t=0)\|_{\LL_s^1(\mathfrak{H}_L^n)}\leq A^n\left(\frac{n}{\sqrt{N}}\right)^{n/2}
					\end{equation}
					
					Then there exists a constant $C>0$ independant of $C^0$ such that for any $\beta\in(0,1/9)$, 
					\[\forall t>0 \text{ such that }\frac{Nt}{\mu\hbar}\leq\frac{\beta\log N}{C\|\mathcal{V}\|_\infty},~ \forall n\geq 1,~\|G_{N}^{[n]}(t)| \leq(CA)^n\left(\frac{n^{1-\beta}}{N^{\frac{1-3\beta}{2}}}\right)^n.\]
				\end{prop} 
				\begin{proof}
					We define $\tilde{F}(t)$ the solution of the Hartree equation
					\[\left\{\begin{split}
					\frac{iN\hbar}{\mu}\partial_t \tilde{F}(t) &= \frac{N\hbar^2}{\mu}\left[-\triangle,\tilde{F}(t)\right]+\tr_1\left[V_{1,2},\tilde{F}\otimes\tilde{F}(t)\right]\\
					\tilde{F}(t=0)&= F_N^{(1)}(t=0)
					\end{split}\right.\]
					
					We introduce the family of matrices $E_{[1,n]}^N(t)$ defined by 
					\begin{equation}
						E_N^{[n]}(t) = \sum_{\sigma\subset[n]} (-1)^{\sigma} F_N^{[n]\setminus \sigma}(t)\otimes (\tilde{F}(t))^{\sigma}.
					\end{equation}
					
					It has been proved in \cite{PPS} that for $t>0$, there exists a constant $C>0$ (independent of $N$ and $L$) such that for all $t>0$,
					\begin{equation}\label{eq:estimation des deux premiers cumulants}
						\|F_N^{[1]}(t)-\tilde{F}(t)\|_{\mathcal{L}^1}\leq \frac{C}{N}e^{C\frac{tN}{\mu \hbar}},~\|E^{\{1,2\}}_N(t)\|_{\mathcal{L}^1}\leq \frac{C}{N}e^{C\frac{tN}{\mu \hbar}}.
					\end{equation}
					
					We deduce immediately that for $\beta\in(0,1)$, $\forall t$ such that $	\frac{tN}{\mu\hbar}\leq \frac{\beta\log N}{C\|\mathcal{V}\|_\infty}$
					\begin{equation*}
						\|G_N^{[2]}(t)\|_{\mathcal{L}^1}\leq \|\tilde{F}\otimes\tilde{F}-F_N^{\{1\}} F_N^{\{2\}}\|_{\mathcal{L}^1}+\|E_N^{[2]}\|_{\mathcal{L}^1} \leq \frac{C}{N}e^{C\frac{\|\mathcal{V}\|_{L^\infty} Nt}{\mu\hbar}}\leq \frac{C}{N^{1-\beta}}.
					\end{equation*}
					
					We can now look at the $G_N^{[n]}$. They are solution of the following hierarchy (its derivation is similar to the Appendix of \cite{PPS})
					
					\begin{multline}
						\frac{i\hbar}{L^d}\partial_t G_N^{[n]} = \Bigg(\frac{\hbar^2}{L^d} K_n+\frac{1}{N}\sum_{j<k} T_{j,k}\Bigg)G_N^{[n]}+D_n({f_1})G_N^{[n]}+D^1_n({f_1})G_N^{[n+1]}\\[-5pt]
						+\bar{D}^{-1}_n (F_N^{(1)},G_N^{[2]})G_N^{[n-1]}+D^{-2}_n(F_N^{[1]})G_N^{[n-2]}
					\end{multline}
					where
					\[\begin{aligne}{c}
					K_n := \sum_{j=1}^n [-\tfrac{\triangle_j}{2},\cdot],~T_{j,k}:= [V_{j,k},\cdot],~C_{j,n+1}h_{[n+1]} = \tr_{[1,n]} \Big(T_{j,n+1} h_{[n+1]}\Big)\\[10pt]
					\bar{D}_n^{-1}(f_1,g_2)h_{[n-1]} := \sum_{j\in[n]}\Big(C_{1,2}g_{[2]}\Big)^{\{j\}}h_{[n]\setminus\{j\}}+D^{-1}_{n}(f_1)h_{[n-1]}\\[15pt]
					D^1_{N}=D_{0}^{-2}:=0\\[5pt]
					D^{-1}_1(f_1)h_{\emptyset}:=-\frac{1}{N}C_{1,2}f_1\otimes f_1\\[5pt]
					D^{-2}_2(f_1)h_{\emptyset}:=\frac{1}{N}\Big(T_{1,2}f_1\otimes f_1-\big(C_{1,2}(f_1\otimes f_1)\big)\otimes f_1-f_1\otimes\big(C_{1,2}(f_1\otimes f_1)\big)\Big)
					\end{aligne}\]
					with the similar bound, for $\frac{Nt}{\mu\hbar }\leq \frac{\beta\log N}{C\|\mathcal{V}\|_\infty}$, $\beta\in(0,1/2)$,
					\begin{multline*}
						\|\bar{D}^{-1}_N\|_{\mathcal{L}^1\to\mathcal{L}^1}\leq \|\mathcal{V}\|_{L^\infty}\Big(\frac{4n^2}{N}+n\|g_{[2]}\|_{\mathcal{L}^1}\Big)
						\leq \|\mathcal{V}\|_{L^\infty}\Big(\frac{4n^2}{N}+\frac{n}{N}Ce^{C\|\mathcal{V}\|_\infty\frac{tL^d}{\hbar}}\Big)\\
						\leq C'\|\mathcal{V}\|_{L^\infty}\Big(\frac{n^2}{N}+\frac{n}{N^{1-\beta}}\Big)
						\leq C'\|\mathcal{V}\|_\infty\Big(\frac{n^2}{N}+\Big(\frac{n^2}{N}\Big)^{1-\beta}\Big).
					\end{multline*}
					
					\begin{remark}
						Note that this hierarchy of equation can be simplified considering the \emph{Grand canonical ensemble}: the number of particles is not fixed but is a random (see \cite{PS3}).
					\end{remark} 
					
					In \cite{PPS}, the authors use the following technical lemma, which is proved in Section 3 of \cite{PPS}
					\begin{lemma}
						Let $(g_n(t))_{n\leq N}$, $g_n:\mathbb{R}^+\to \mathcal{E}_n$ be a sequence of function onto a family of Banach spaces $\mathcal{E}_n$, and $\alpha\in(0,1)$ such that
						\[\forall t\geq 0,~\|g_n(t)\|_{\mathcal{E}_n}\leq 2^n,~\forall n\geq 1,~\|g_n(0)\|_{\mathcal{E}_n} \leq A^n\Big(\frac{n^2}{N}\Big)^{\frac{\alpha n}{2}}\]
						\[\partial_t \|h_n\|_{\mathcal{E}_n}\leq n \big(\|h_n\|_{\mathcal{E}_n}+\|h_{n+1}\|_{\mathcal{E}_{n+1}}\big)+\Big(\frac{n^2}{N}+\Big(\frac{n^2}{N}\Big)^\alpha\Big)\big(\|h_{n-1}\|_{\mathcal{E}_{n-1}}+\|h_{n-2}\|_{\mathcal{E}_{n-2}}\big)\]
						
						Then there exists a constant $C$ independent of $A$ such that
						\[\forall n\geq 1,~\|g_n(t)\|_{\mathcal{E}_n} \leq \Big(CAe^{Ct}\Big)^n\Big(\frac{n^2}{N}\Big)^{\frac{\alpha n}{2}}.\]	
					\end{lemma}
					This lemma was used in \cite{PPS} to prove estimation of \eqref{eq:estimation des deux premiers cumulants}\footnote{Note that this lemma has not been explicitly written.}. Applying the lemma on the family $G^{[n]}_N(t)$, one obtain
					\[\forall n\geq 1,~\|G_N^{[n]}(t)\|_{\mathcal{L}^1} \leq \Big(CAe^{\frac{C\|\mathcal{V}\|_{L^\infty}Nt}{\mu\hbar}}\Big)^n\Big(\frac{n^2}{N}\Big)^{\frac{(1-\beta) n}{2}}\leq (CA)^n\left(\frac{n^{1-\beta}}{N^{\frac{1-3\beta}{2}}}\right)^n.\]
				\end{proof}
				
			\subsection{Effective equation for initial data stable by translation}			
				In the classical setting, the Lenard-Balescu equation is a correction of the Vlasov dynamics in the case of uniformly distributed (in $x$) initial data. We have to precise the meaning in the quantum case. 
				
				\begin{definition}
					We define for $z\in\mathbb{T}_L$ the translation operator $\tau_z:\mathfrak{H}^N_L\to\mathfrak{H}^N_L$ by
					\[\forall \psi\in\mathfrak{H}^N_L,~(\tau_z\psi)(x_1,\cdots,x_N)=\psi(x_1+z,\cdots,x_N+z).\]
					
					A density matrix $F^N\in\LL^1(\mathfrak{H}_L^N)$ is said \emph{invariant by translation} if it commutes with all the translation operators. We denote $\LL_{i}^1(\mathfrak{H}_L^N)\subset\LL^1(\mathfrak{H}_L^N)$ the set of operator invariant by translation.
					
					We can write this property for density matrix in the $X$ variable: denoting $F^N(x_{1},\cdots,x_N,y_1,\cdots,y_N)$ its kernel, $F_N$ is invariant by translation if
					\[\forall z\in\mathbb{T}_L,~F_N(x_1+z,\cdots,x_N+z,y_1+z,\cdots,y_N+z) = F_N(x_1,\cdots,x_N,y_1,\cdots,y_N).\]
					
					This property can also be defined in the Fourier Variable:
					Denote $\mathbb{Z}^d_L:= \left(\tfrac{2\pi}{L}\mathbb{Z}\right)^{d}$, and for $(k_{[N]},l_{[N]})\in \mathbb{Z}_L^{2nd}$
					\[\hat{F}_N(k_{[N]},l_{[N]}):= \iint F_N(x_{[N]},y_{[N]})e^{i(-k_{[N]}\cdot x_{[N]}+l_{[N]}\cdot y_{[N]})}dx_{[N]}dy_{[N]}.\]
					Then $F^N$ is invariant by translation if and only if
					\[\sum_{j=1}^N k_j-l_j \neq 0 \Rightarrow \hat{F}_N(k_{[N]},l_{[N]}) = 0.\]
					
					Note that if $F_N$ is invariant by translation, then for any $\omega\subset[N]$, $\tr_{\omega}F^N$ is also invariant by translation.
				\end{definition}
				
				As the Hamiltonian of the system commutes with all the translation, $F^N(t)$ remains invariant by translation if the initial data is. In addition the function $x\mapsto f^N_1(t,x,x)$ is constant and so the effective potential has no effect.
				
				We can now compute the equations verified by $f^N_1$.
				\begin{equation}
					\begin{split}
					i\hbar\partial_t F_N^{(1)} &= [-\tfrac{\hbar^2}{2}\triangle,F_N^{[1]}]+\frac{N-1}{\mu}\tr_1[V_{1,2},F_N^{[2]}]=[-\tfrac{\hbar^2}{2}\triangle,F_N^{[1]}]+ \frac{N-1}{\mu}\tr_1 [V_{1,2},G_N^{[2]}].
					\end{split}
				\end{equation}
				using that for denoting $H^{(n)}(X_n,Y_n)$ the kernel of $H^{(n)}\in\LL^1(\mathfrak{H}_L^n)$, for $F\in\LL^1(\mathfrak{H}_L)$ invariant by translation, and $\mathcal{V}$ has spherical symmetry, $F(x,y)=\tilde{F}(x-y)$ for some $\tilde{F}$ and
				\begin{gather}
					[\triangle,F] = \triangle_x F(x-y)-\triangle_yF(x-y)=0,\\
					\label{eq:annulation de la trace}
					\tr_1[V_{1,2},F\otimes F](x,y)=\int  (\mathcal{V}(x-z)-\mathcal{V}(y-z))\tilde{F}(x-y)\tilde{F}(z-z)\ud{z}=0.
				\end{gather}
				
				To compute the equation verified by $G_N^{(2)}(t)$, we begin by derivating $F_N^{(2)}$. We denote 
				\[\begin{split}
				i\hbar\partial_tF_N^{[2]}=&[\tfrac{\hbar^2}{2}(-\triangle_1-\triangle_2)+\frac{1}{\mu}V_{1,2},F_N^{[2]}]+\frac{N-2}{\mu}\tr_{12} [V_{1,3}+V_{2,3},F_N^{[3]}]\\
				=&[\tfrac{\hbar^2}{2}(-\triangle_1-\triangle_2)+\frac{1}{\mu}V_{1,2},G_N^{[2]}]+F_N^{\{1\}}[-\tfrac{\hbar^2}{2}\triangle_2,F_N^{\{2\}}]+[-\tfrac{\hbar^2}{2}\triangle_1,F_N^{\{1\}}]F_N^{\{2\}}\\
				&+\frac{1}{\mu}[V_{1,2},F_N^{\{1\}}F_N^{\{2\}}]+\frac{N-2}{\mu}\Big(\tr_{12}[V_{2,3},G_N^{\{1,2\}}F_N^{\{3\}}]+\tr_{12}[V_{1,3}, G^{\{1,2\}}_{N}F_N^{\{3\}}]\Big)\\
				&+\frac{N-2}{\mu} \Big(F_N^{\{1\}}\tr_2 \Big[V_{2,3},F_N^{\{2\}}F_N^{\{3\}}+G_N^{\{2,3\}}\Big] +\tr_2 \Big[V_{1,3},F_N^{\{1\}}F_N^{\{3\}}+G_N^{\{1,3\}}\Big]F_N^{\{2\}}\Big) \\
				&+\frac{N-2}{\mu}\Big(\tr_{12}[V_{2,3},G_N^{\{1,3\}}F_N^{\{2\}}]+\tr_{12}[V_{1,3},G_N^{\{2,3\}}F_N^{\{1\}}]+\tr_{12}[V_{1,3}+V_{2,3},G_N^{\{1,2,3\}}]\Big).
				\end{split}\]
				Using that $\tr_{12}[V_{2,3},G_N^{\{1,2\}}F_N^{\{3\}}]$ and $\tr_{12}[V_{1,3},G_N^{\{1,2\}}F_N^{\{3\}}]$ vanish (in the same way that \eqref{eq:annulation de la trace}),
				\begin{align*}
					i\hbar\partial_t G_N^{\{1,2\}} =& \Big[\tfrac{\hbar^2}{2}(-\triangle_1-\triangle_2)+\frac{1}{\mu}V_{1,2},G_N^{\{1,2\}}\Big]+\frac{1}{\mu}\Big[V_{1,2},F_N^{\{1\}}F_N^{\{2\}}\Big]\\
					&+\frac{N-2}{\mu}\Big(\tr_{12}\Big[V_{1,3},g^N_{13}f^N_2\Big]+\tr_{12}\Big[V_{2,3},G_N^{\{1,3\}}F_N^{(2)}\Big]\Big)\\
					&-\frac{1}{\mu}\Big(F^{\{1\}}_N\tr_2\Big[V_{2,3},G_N^{\{2,3\}}\Big]+\tr_1\Big[V_{1,3},G_N^{\{1,3\}}\Big]F_N^{\{2\}}\Big)+\frac{N-2}{\mu}\tr_{12}\Big[V_{1,3}+V_{2,3},G_N^{\{1,2,3\}}\Big].
				\end{align*}
				
				Consider a family of initial data $F^N_0\in\LL_s^1(\mathfrak{H}^N_L)$ invariant by translation and such that the associated families $(G_{N}^{[n]}(t))_{n\leq N}$ vanish at time $0$. Then, using the estimation on $G^{[2]}_N$ and $G^{[3]}_N$ obtained in Proposition \ref{prop:estimation a priori}, we have that for $\tfrac{t}{\hbar}\leq \tfrac{\beta\log N}{C\|\mathcal{V}\|_\infty}$, $F_N^{[1]}$ and $G_{N}^{[2]}$ verify
				\begin{equation}
					\left\{\begin{split}
					i\hbar \mu\partial_t F_N^{[1]} &= \tr_{[1]} [V_{1,2},\mu G_{N}^{[2]}] + O_{\LL^1}\left(N^{-1+\beta}\right),\\
					i\hbar\partial_t \mu G_{N}^{[2]} &=[\hbar^2(-\triangle_1-\triangle_2),\mu G_{N}^{[2]}]+\tr_{[2]}\Big(\Big[V_{1,3},F_N^{\{1\}}G_N^{\{2,3\}}\Big]+\Big[V_{2,3},F_N^{\{2\}}G_N^{\{1,3\}}\Big]\Big) \\
					&\hspace{6cm}+[V_{1,2},F_N^{\{1\}}F_N^{\{2\}}] +O_{\LL^1}\left({N^{\frac{-1+9\beta}{2}}}\right).
					\end{split}\right.
				\end{equation}
				
				Denoting $F^L_0 :=  F^N_1(0)$, we define $F_\mu^L:\mathbb{R}^+\to \LL^1_s(\mathfrak{H}_L)$ and $G_\mu^L:\mathbb{R}^+\to \LL^1_s(\mathfrak{H}^2_L)$ the solutions of the equation
				\begin{equation}\label{eq: equation f^L}
					\left\{\begin{aligne}{c}
						F_\mu^L(t=0)=F_1^N(0),~G^L_\mu=0,\\[5pt]
						i\hbar \mu\partial_t F_\mu^L = \tr_1 [V_{1,2},G_\mu^L],\\[5pt]
						\begin{split}
						i\hbar\partial_t G_\mu^L =[\tfrac{\hbar^2}{2}(-\triangle_1-\triangle_2),G_\mu^L]+\tr_{[2]}\Big(\Big[V_{1,3},F_\mu^{L,\{1\}}G_\mu^{L,\{2,3\}}\Big]+\Big[V_{2,3},F_\mu^{L,\{2\}}&G_\mu^{L,\{1,3\}}\Big]\Big)\\
						&+[V_{1,2},F^L_{0}F^L_{0}].
						\end{split}
						\end{aligne}\right.
				\end{equation}
				
				From a straight forward application Proposition \ref{prop:estimation a priori} and Gronwall Lemma, we obtain
				\begin{prop}\label{prop:estimation une part}
					For all $t$ such that $\tfrac{Nt}{\mu\hbar}\leq \tfrac{\beta\log N}{C\|\mathcal{V}\|_\infty}$,
					\begin{equation}
						\|i\hbar \mu\partial_t (F_\mu^L(t)-F_\mu^{[1]}(t))\|_{\mathcal{L}^1(\mathfrak{H}_L)}+\|G_\mu^L(t)-\mu G_\mu^{[2]}(t)\|_{\mathcal{L}^1(\mathfrak{H}^2_L)} \leq O\left({N^{\frac{-1+9\beta}{2}}}\right).
					\end{equation}
				\end{prop}
				
			\subsection{The Larg Box limit}
				In this section, we want to take the limit $L\to \infty$.
				
				We begin by writing the equation \eqref{eq: equation f^L} in the Fourier variables. As the matrix $F_\mu^{L}$ and $G_\mu^{L}$ are invariant by translation,  
				\[\hat{F}_{\mu}^{L}(k_1,l_1)=0~{\rm if}~k_1\neq l_1,~\hat{G}_{\mu}^{L}(k_1,k_2,l_1,l_2)=0 ~{\rm if}~k_1+k_2\neq l_1+l_2.\]
				
				We introduce 
				\begin{gather}
					\forall v_1\in\mathbb{Z}_{ L/\hbar},~f_{\mu,L}(t,v_1) := \frac{L^{d}}{(2\pi \hbar)^d}\hat{F}_{\mu}^{L}(t,\tfrac{v_1}{\hbar},-\tfrac{v_1}{\hbar}),\\
					g_{\mu,L}(t,v_1,v_2,k):=\frac{L^{2d}}{(2\pi\hbar)^{2d}}\hat{G}_{\mu}^{L}(t,\tfrac{v_1}{\hbar},\tfrac{v_2}{\hbar},\tfrac{v_1}{\hbar}+k,\tfrac{v_2}{\hbar}-k).
				\end{gather}
				
				Using that $\hat{F}_{\mu}^{L}$ and $\hat{G}_{\mu}^{L}$ are Hermitian and positive, one have
				\begin{equation}\label{eq:condition de conjugaison} 
					f_{\mu,L}(t,v_1)\in \mathbb{R}^+,~ g_{\mu,L}(t,v_1,v_2,k) =\overline{g_{\mu,L}(t,v_1+\hbar k,v_2-\hbar k,-k)}. 
				\end{equation}

				With the change of unknown, the initial data becomes
				\begin{equation}
					\forall v_1\in\mathbb{Z}^d_{ L/\hbar},~ f_{0,L}(v_1) = \frac{1}{\mathcal{Z}_L}f_0(v_1).
				\end{equation}
				
				\begin{prop}
					Denoting for any function  $\Delta_kh(v)=h(v)-h(v-\hbar k)$, and $E(v):=\frac{|v|^2}{2}$,
					\begin{equation}
						\left\{\begin{aligne}{c}
						\hbar \mu i\partial_t {f}_{\mu,L}(v_1) = -2 ~\frac{(2\pi\hbar)^{d}}{L^{2d}}\!\! \sum_{k\in\mathbb{Z}^d_L,\,v_2\in\mathbb{Z}^d_{L/\hbar}}{\V}(k)({g}_{\mu,L}(v_1,v_2,k)-{g}_{\mu,L}(v_1+\hbar k,v_2-\hbar k,-k)),\\[16pt]
						\hbar\partial_t {g}_{\mu,L}(t)+i(K_2+B^L) {g}_{\mu,L}(t) = i A^L\\[7pt]
						A^L(v_1,v_2,k):= \V(k)	({f}_{0,L}(v_1){f}_{0,L}(v_2)-{f}_{0,L}(v_1+\hbar k){f}_{0,L}(v_2-\hbar k))
					\end{aligne}\right.
					\end{equation}
					where $K_2$ is the kinetic part, an unbounded multiplicative operator
					\begin{equation}
						K_2 g(v_1,v_2,k):=  \big(\Delta_{-k}E(v_1)+\Delta_{k}E(v_2)\big) g(v_1,v_2,k)
					\end{equation}
					and $B^L$ is a bounded operator on (note that it is even compact)
					\begin{equation}
						B^L g(v_1,v_2,k):=  {\V}(k)\Bigg(\Delta_{-k}f_{0,L}(v_1)\frac{(2\pi)^d}{\hbar^dL^d}\sum_{\tilde{v}_1\in\mathbb{Z}_{L/\hbar}}g(\tilde{v}_1,v_2,k)+ \Delta_{k}f_{0,L}(v_2)\frac{(2\pi)^d}{\hbar^dL^d}\sum_{\tilde{v}_2\in\mathbb{Z}_{L/\hbar}}g(v_1,\tilde{v}_2,k)\Bigg)
					\end{equation}
				\end{prop}
				\begin{proof}
					We recall that
					\begin{gather*}
						i\hbar \mu\partial_t F_\mu^L = \tr_1 [V_{1,2},G_\mu^L],\\
						\begin{split}
						i\hbar\partial_t G_\mu^L =[\tfrac{\hbar^2}{2}(-\Delta_1-\Delta_2),G_\mu^L]+\tr_{[2]}\Big(\Big[V_{1,3},F_\mu^{L,\{1\}}G_\mu^{L,\{2,3\}}\Big]+\Big[V_{2,3},F_\mu^{L,\{2\}}&G_\mu^{L,\{1,3\}}\Big]\Big)+[V_{1,2},F^L_{0}F^L_{0}].
						\end{split}
					\end{gather*}
					
					We set $\Psi:=[V_{1,2},G_\mu^L]$. Its kernel is
					\begin{align*}
						\Psi&(x,y)=\int_{\mathbb{T}_L} \left[\mathcal{V}(x-z)-\mathcal{V}(z-y)\right]{G}_\mu^L(x,z,y,z)\ud{z}\\
						&=\frac{1}{L^{5d}}\sum_{k,\ell_1,\ell_2,p_1,p_2\in\mathbb{Z}^d_L} \int_{\mathbb{T}_L} \left[e^{i(k(x-z)+\ell_1x+\ell_2z-p_1y-p_2z)}-e^{i(k(z-y)+\ell_1x+\ell_2z-p_1y-p_2z)}\right]\V(k)\hat{G}_\mu^L(\ell_1,\ell_2,p_1,q_2)\ud{z}\\
						&=\frac{1}{L^{4d}}\sum_{k,\ell_1,\ell_2,p_1\in\mathbb{Z}^d_L} \left[e^{i((k+\ell_1)x-p_1y)}\hat{G}_\mu^L(\ell_1,\ell_2,p_1,\ell_2-k)-e^{i(\ell_1x-(k+p_1)y)}\hat{G}_\mu^L(\ell_1,\ell_2,p_1,p_1+k)\right]\V(k)\\
						&=\frac{1}{L^{4d}}\sum_{k,\ell_1,\ell_2,p_1\in\mathbb{Z}^d_L} \left[\hat{G}_\mu^L(\ell_1-k,\ell_2,p_1,\ell_2-k)-\hat{G}_N^L(\ell_1,\ell_2,p_1-k,\ell_2+k)\right]\V(k)e^{i(\ell_1x-p_1y)}.
					\end{align*}
					We deduce that 
					\[\hat{\Psi}(\ell_1,\ell_1) = \frac{1}{L^{2d}}\sum_{k,\ell_2\in\mathbb{Z}^d_L}\left[\hat{G}_N^L(\ell_1+k,\ell_2-k,\ell_1,\ell_2)-\hat{G}_\mu^L(\ell_1,\ell_2,\ell_1+k,\ell_2-k)\right]\V(k)\]
					This gives the evolution equation of $f_{\mu,L}$. The evolution equation of $g_{\mu,L}$ can be computed in the same way.
				\end{proof}
				
				We define the couple $(f_\mu,g_\mu)$ on $\mathbb{R}^d$ and $\mathbb{R}^{3d}$ as the  solution of the equation
				\begin{equation}
					\left\{\begin{aligne}{c}
					\hbar N \partial_t {f}_\mu(v_1) = -4\Im ~\int \uud k\ud v_2\V(k){g}_\mu(v_1,v_2,k),\\[10pt]
					\hbar\partial_t {g}_\mu(t)+i(K_2+B) {g}_\mu(t) = i A,\\[7pt]
					{f}_\mu(t= 0)=\Phi_0,~{g}_{\mu}(t= 0)=0,,
					\end{aligne}\right.
				\end{equation}
				where we denote 
				\begin{align}
					A(v_1,v_2,k)&:= \V(k)	(f_0(v_1)f_0(v_2)-f_0(v_1+\hbar k)f_0(v_2-\hbar k)),\\
					B h(v_1,v_2,k)&:=  {\V(k)}\Bigg(\Delta_{-k}f_0(k_1)\int \ud\tilde{v}_1 h(\tilde{v}_1,v_2,k)+ \Delta_{k}f_0(v_2)\int d\tilde{v}_1h(v_1,\tilde{v}_2,k)\Bigg).
				\end{align}
				
				We want to compare the $({f}_{\mu,L}(t),{g}_{\mu,L}(t))$ with $({f}_\mu(t),{g}_\mu(t))$. The space $\LL^1(\mathfrak{H}_L)$ and $\LL^1(\mathfrak{H}^{2d}_L)$ can be injected respectively in $\ell^1_{v_1}(\mathbb{Z}^d_{L/\hbar})$ and $\ell_k^\infty(\mathbb{Z}^d_L,\ell^1_{v_1,v_2}(\mathbb{Z}^2_{L/\hbar}))$ (see Section 1.5 of \cite{Golse}):
				\begin{gather*}
					\|f\|_{\ell_{v_1}^1}:=\frac{(2\pi\hbar)^d}{L^d}\sum_{v_1\in\mathbb{Z}^d_{L/\hbar}} |{f}(v_1)|\leq \left\|F\right\|_{\LL^1_{s,i}(\mathfrak{H}_L)},\\
					\|g\|_{\ell^\infty_k(\ell_{v_1,v_2}^1)}:=\sup_{k\in\mathbb{Z}^d_L}\frac{(2\pi\hbar)^{2d}}{L^{2d}}\sum_{v_1,v_2\in\mathbb{Z}^d_{L/\hbar}}|\hat{g}(v_1,v_2,k)|\leq  \|G\|_{\LL^1_{s,i}(\mathfrak{H}^2_L)}.
				\end{gather*}
				
				We want to find two Banach spaces, $\mathcal{E}_1$ containing both $L^1_{v_1}(\mathbb{R}^d)$ and the $\ell^1_{v_1}(\mathbb{Z}^d_{L/\hbar})$ for any $L$, and $\mathcal{E}_2$ containing both $L^\infty_{k}(\mathbb{R}^d,L^1_{v_1,v_2}(\mathbb{R}^{2d}))$ and the  $\ell_k^\infty(\mathbb{Z}_L,\ell^1_{v_1,v_2}(\mathbb{Z}^{2d}_{L/\hbar}))$ for any $L$. For $\mathcal{E}_1$ we take the set of bounded radon measure on $\mathbb{R}^d$
				with total variation norm, with injection
				\[\begin{aligne}{c}
				\pi^\infty_1:f\in L^1(\mathbb{R}^d)\mapsto f(v_1)\ud v_1,~~\pi^L_1:f^L\in\ell^1(\mathbb{Z}_L)\mapsto \frac{(2\pi\hbar)^{d}}{L^d}\sum_{v_1\in\mathbb{Z}^d_{L/\hbar}} f^L(v_1)\delta_{v_1}.
				\end{aligne}\]
				For $\mathcal{E}_2$ we consider the space of radon measure $h$ on $\mathbb{R}^{3d}$ such that the norm
				\[\|h\|_{\mathcal{E}_2} := \sup_{\substack{k\in\mathbb{R}^d\\ \e>0}}\frac{1}{|B_k(\e)|}\|h(v_1,v_2,\tilde{k})\|_{TV((B_k(\e))_k\times\mathbb{R}^d_{v_1}\times\mathbb{R}^d_{v_2})}\]
				is finite (where $|B_k(\e)|:=\{k' \text{ with } |k'-k|<\e\}$), with the injection
				\begin{gather*}
					g\in L^\infty_k(L^1_{v_1,v_2})\mapsto g(k,v_1,v_2)\uud{k}\ud{v_1}\ud{v_2},~\\
					g^L\in\ell_k^\infty(\ell^1_{v_1,v_2}(\mathbb{Z}^3_L)\mapsto \frac{(2\pi\hbar)^{2d}}{L^{3d}}\sum_{\substack{v_1,v_2\in\mathbb{Z}^d_{L/\hbar},\\k\in\mathbb{Z}^d_L}} g^L(v_1,v_2,k)\delta_{v_1,v_2,k}.
				\end{gather*}
				
				Denoting, for $h(v_1,v_2,k)$ a radon measure on $\mathbb{R}^{3d}$, (here the brackets designate the duality product)
				\[\left<h(v_1,v_2,k)\right>_{v_1}:(v_2,k)\mapsto \left<h(v_1,v_2,k),1_{v_1}\right>_{v_1}.\]
				One can now extend the operator $B:L_k^\infty(L^1_{v_1,v_2})\to\ell^1_{v_1}(\mathbb{Z}^d_{L/\hbar})$ and $B^L:\ell_k^\infty(\ell^1_{v_1,v_2})\to\ell^1_{v_1}(\mathbb{Z}^d_{L/\hbar})$ as operators $\mathcal{E}_2\to\mathcal{E}_1$ by
				\begin{gather*}
					B^L h(v_1,v_2,k):=  \V(k)\bigg(\pi_1^L\left(\Delta_{-k}f^L_0\right)\!\!(v_1) \left<h(\tilde{v}_1,v_2,k)\right>_{\tilde{v}_1}+ \pi_1^L\left(\Delta_{k}f_0^L\right)\!\!(v_2)\left< h(v_1,\tilde{v}_2,k)\right>_{\tilde{v}_2}\bigg)\\
					B h(v_1,v_2,k):=  \V(k)\bigg(\pi_1^\infty\left(\Delta_{-k}f_0\right)\!\!(v_1) \left<h(\tilde{v}_1,v_2,k)\right>_{\tilde{v}_1}+ \pi_1^\infty\left(\Delta_{k}f_0\right)\!\!(v_2)\left< h(v_1,\tilde{v}_2,k)\right>_{\tilde{v}_2}\bigg)
				\end{gather*}
				
				\begin{prop}\label{prop:estimation grosse boite}
					Fix an initial data  $f_0$ with $\|\<v\>^{d+1}\<\nabla\> f_0\|_{L^\infty}<\infty$ .
					
					For any parameter $N,\hbar$ and $L$,
					\[\mu\hbar\|\partial_t(\pi^\infty_1{f}_\mu)-\pi^L_1f_{\mu,L}\|_{TV}\lesssim \|\pi_2^\infty g_{\mu}(t)-\pi_2^L g_{\mu,L}(t)\|_{\mathcal{E}_2} \lesssim \frac{\|\<v\>^{d+1}\<\nabla\> f_0\|_{L^\infty}}{L/\hbar}e^{\frac{4\|\V\|_{L^\infty}t}{\hbar}}.\]
				\end{prop}
				\begin{proof} 
					We begin by the following lemma which can be directly deduced from the definition of the operator $B$ and $B^L$.
					\begin{lemma}
						The operator $B$ and $B^L$ are bounded in $\mathcal{E}_2\to\mathcal{E}_2$, with the following bound
						\[\|B\|_{\mathcal{E}_2\to\mathcal{E}_2}\leq \frac{4}{(2\pi)^d}\|\V\|_{L^\infty}\|f_0\|_{TV},~\|B^L\|_{\mathcal{E}_2\to\mathcal{E}_2}\leq \frac{4}{(2\pi)^d}\|\V\|_{L^\infty}\|f^L_0\|_{TV},\]
						\[\|B-B^L\|_{\mathcal{E}_2\to\mathcal{E}_2}\leq \frac{4}{(2\pi)^d}\|\V\|_{L^\infty}\|f_0-f^L_0\|_{TV}.\]
					\end{lemma}
					
					We have
					\[\partial_t(g_{\mu,L}-g_\mu)+iK_2(g_{\mu,L}-g_\mu)=-iB^L(g_{\mu,L}-g_L)-i(B-B^L)g_\mu +i(A-A^L).\]
					Using that $iK_2$ is the generator of an unitary group on $\mathcal{E}_2$ and that $\| f\|_{TV} = \| f^L\|_{TV} =1$ (we begin with two probability measure),
					\[\|A-A^L\|_{\mathcal{E}_2}\leq 4\|\V\|_{L^\infty}\|\pi_1^\infty f^0-\pi_1^L f_0^L\|_{TV},\]
					we have by a Gronwall lemma 
					\[\|\pi_2^\infty g_{\mu}(t)-\pi_2^L g_{\mu,L}(t)\|_{\mathcal{E}_2}\leq \|\pi_1^\infty f^0-\pi_1^L f_0^L\|_{TV}\exp\left({\frac{4\|V\|_{L^1}t}{(2\pi)^d\hbar}}\right).\]
					Because $\mathcal{V}$ verifies Assumption \ref{ass: potentiel},
					\[\mu\hbar\|\partial_t(\pi^\infty_1{f}_\mu-\pi^L_1f_{\mu,L})\|_{TV}\lesssim\|\pi_2^\infty g_{\mu}(t)-\pi_2^L g_{\mu,L}(t)\|_{\mathcal{E}_2}\lesssim \|\pi_1^\infty f^0-\pi_1^L f_0^L\|_{TV}\exp\left({\frac{4\|\mathcal{V}\|_{L^1}t}{(2\pi)^d\hbar}}\right).\]
					
					We need to  $\|\pi_1^\infty f^0-\pi_1^L f_0^L\|_{TV} $. We can decompose $\mathcal{Z}_L$ as
					\[\mathcal{Z}_L = 1+\frac{(2\pi\hbar)^d}{ L^d}\sum_{v\in\mathbb{Z}^d_{L/\hbar}} \left(f_0(v)-\int_{[-\frac{1}{2},\frac{1}{2}]^d}f_0(v+\tfrac{2\pi}{L/\hbar}\ell)d\ell \right)\]
					and finally,
					\begin{align*}
						\|\pi_2^\infty g_{\mu}(t)-\pi_2^L g_{\mu,L}(t)\|_{\mathcal{E}_2} &=\frac{(2\pi\hbar)^d}{L^d}\sum_{v\in\mathbb{Z}^d_{L/\hbar}} \left|\frac{f_0(v)}{\mathcal{Z}_L}-\int_{[-\frac{1}{2},\frac{1}{2}]^d}f_0(v+\tfrac{2\pi}{L/\hbar}w)\ud w \right| \\
						&\leq \left|1-\frac{1}{\mathcal{Z}_L}\right|+ \frac{(2\pi\hbar)^d}{L^d}\sum_{v\in\mathbb{Z}^d_{L/\hbar}} \left|f_0(v)-\int_{[-\frac{1}{2},\frac{1}{2}]^d}f_0(v+\tfrac{2\pi}{L/\hbar}w)\ud{w} \right|\\
						&\lesssim 
						\frac{\hbar^d}{L^d}\sum_{v\in\mathbb{Z}^d_{L/\hbar}} \frac{\|\<x\>^{d+1}\<\nabla\> f_0\|_{L^\infty}}{L/\hbar\<v\>^{d+1}}\lesssim \frac{\|\<x\>^{d+1}\<\nabla\> f_0\|_{L^\infty}}{L/\hbar}.
					\end{align*} 			
				\end{proof}
				
			\subsection{Computation of the collision term}
				In this final section, we study the limit of the $(f_\mu,g_\mu)$ as $\mu\to\infty$. The first step is to apply the Laplace transform in time
				
				\begin{prop}
					For any smooth test function $\varphi \in \mathcal{C}^\infty_c(\mathbb{R}^+\times\mathbb{R}^d)$, 
					\begin{multline}
						\mu\int_0^\infty\left<\partial_t\tilde{f}_\mu(\tau_Nt,v_1)\varphi(t,v_1)\right>_{\!\!v_1}\!\!\ud{t} 	\\
						=\int_{\mathbb{R}}\frac{\ud{\alpha}}{i\pi\hbar}\left<\frac{\tilde{\varphi}(\alpha,v_1)}{1+i\alpha} \left(T\left(\tfrac{\hbar(i-\alpha)}{\tau_N},k,v_1\right)-T\left(\tfrac{\hbar(i-\alpha)}{\tau_N},-k,v'_1\right)\right)\right\>_{\!\!k,v_1}
					\end{multline}
					where $\tilde{\varphi}(\alpha,v_1)$ is the Laplace transform in time of $\varphi$ (with $\<f(k)\>_k := \int f(k)\uud k$).
					\[\tilde{\varphi}(\alpha,v_1) := \frac{1}{2\pi}\int_0^\infty e^{t(1+i\alpha)}\varphi(t,v_1)\ud{t}\]
					and
					\begin{equation}
						T(\omega,k,v_1):=\V(k)\left<\left((K_2+B)-\omega\right)^{-1}A(v_1,v_2,k)\right>_{\!\!v_2}.
					\end{equation}
				\end{prop}
				
				Note that the preceding proposition is justified by the existence of the resolvent $\left(K_2+B-\frac{\hbar}{\tau_N}(i-\alpha)\right)^{-1}$. 
				\begin{proof}
					The proof is the same than Lemma 5.2 of \cite{DS}.
					
					We recall that $g_N$ is solution of 
					\begin{gather*}
						\hbar\partial_t {g}_\mu(t)+i(K_2+B) {g}_\mu(t) = i A\\
						{g}_{\mu}(t= 0)=0.
					\end{gather*}
					
					Using the Duhamel form of the evolution equation of $g_\mu$, and that $\delta(t) = \frac{1}{2\pi}\int e^{i\alpha t}\ud{\alpha}$ in the sense of distribution, 
					\begin{align*}
						\int_0^\infty g_\mu(\tau_Nt)\varphi(t)\ud{t}&=\frac{i}{\hbar}\int_0^\infty\int_0^{\tau_N t} e^{-\frac{i}{\hbar}(K_2+B)(t-t_1)}A\ud{t_1}~\varphi(t)\ud{t}\\
						&=\frac{i}{\hbar}\int_{(\mathbb{R}^+)^3}e^{-(\frac{1}{\tau_N}+\frac{i}{\hbar}(K_2+B))t_2}Ae^{-\frac{t_2}{\tau_N}}e^{t}\varphi(t)\delta_{t_1+t_2=\tau_Nt}\ud{t_1}\ud{t_2}\ud{t}\\
						&=\frac{i}{2\pi\hbar\tau_N}\int_{\mathbb{R}}\int_{(\mathbb{R}^+)^3}e^{-(\frac{1+i\alpha}{\tau_N}+\frac{i}{\hbar}(K_2+B))t_2}Ae^{-\frac{1+i\alpha}{\hbar\tau_N}t_1}e^{(1+i{\alpha})t}\varphi(t)\ud{t_1}\ud{t_2}\ud{t} \ud{\alpha}\\
						&=\frac{1}{2\pi}\int_{\mathbb{R}}\left(\frac{\hbar}{\tau_N}(\alpha-i)+(K_2+B)\right)^{\!\!\!-1}\hspace{-7pt}A~~\frac{\tilde{\varphi}(\alpha)}{1+i\alpha}\ud{\alpha}
					\end{align*}
					
					Using the symmetry property of $g_\mu(t,v_1,v_2,k)$, one has
					\[\mu \partial_t {f}_\mu(v_1) = -\frac{2}{i\hbar} ~\left< \V(k)\left({g}_\mu(v_1,v_2,k)-{g}_\mu(v'_1,v'_2,-k)\right)\right>_{v_2,k},\]
					where $v_1':= v_1+\hbar k$ and $v_2':=v_2-\hbar k$.
					
					The result is obtain by the combination of the both identities.
				\end{proof}
				
				\begin{prop}
					The function $T(\omega,k,v_1)$ verifies the following bounds and convergence results: for any $\omega\in\mathbb{C}\setminus\mathbb{R}$, 
					\begin{gather}
						\left|T(\omega,k,v_1)-T(\omega,-k,v'_1)\right|\lesssim {\V}^2(k)\|\<v\>^{2}\<\nabla\>^2f_0\|_{L^\infty}^2 \log\left(1+\frac{|\Re \omega|}{|\Im\omega|}\right)^2,\\
						\lim_{\substack{\omega\to 0\\ \Im\omega>|\Re \omega|}}\frac{T(\omega,k,v_1)-T(\omega,-k,v'_1)}{2\pi i} =\frac{ \V^2(k)\left< (f_0(v_2')f_0(v_1')-f_0(v_1)f_0(v_2))\delta_{\Delta_kE(v_2)+\Delta_{-k}E(v_1)}\right>_{v_2}}{(2\pi)^d|\tilde{\e}(k,i0-\Delta_{-k}E(v_1))|^2},
					\end{gather}
				\end{prop}
				
				\begin{proof}	
					We decompose $(K_2+B)$ as $L_1+L_2$, where $L_1$,$L_2$ are defined by
					\begin{align}
						L_1 h(v_1,v_2,k) &:= \Delta_{-k}E(v_1) h(v_1,v_2,k) +{\V(k)}\Delta_{-k}f_0(v_1)\left<h(\tilde{v}_1,v_2,k)\right>_{\tilde{k}_1}\\
						L_2 h(v_1,v_2,k) &:= \Delta_{k}E(v_2) h(v_1,v_2,k) +\V(k)\Delta_{k}f_0(v_2)\left<h(v_1,\tilde{v}_2,k)\right>_{\tilde{k}_2}.
					\end{align}
					
					The operators $L_1$ and $L_2$ are easier to study than $K_2+B$, which can be recover by 
					\begin{equation}
						((K_2+B^L)-\omega)^{-1} = \frac{1}{2\pi i}\int (L_2+\beta-\frac{\omega}{2})^{-1}(L_2-\beta-\frac{\omega}{2})^{-1} \ud{\beta}
					\end{equation}
					
					The resolvant $L_1$ and $L_2$ are given by 
					\begin{align}
						(L_1-\omega)^{-1}h &=\frac{h}{\Delta_{-k}E(v_1)-\omega}-\frac{{\V}(k)}{\tilde{\e}(-k,\omega)}\frac{\Delta_{-k}f_0(v_1)}{\Delta_{-k}E(v_1)-\omega}\left<\frac{h}{\Delta_{-k}E(\tilde{v}_1)-\omega}\right>_{\tilde{v}_1}\\
						(L_2-\omega)^{-1}h &=\frac{h}{\Delta_{k}E(v_2)-\omega}-\frac{{\V}(k)}{\tilde{\e}(k,\omega)}\frac{\Delta_{k}f_0(v_2)}{\Delta_{k}E(v_2)-\omega}\left<\frac{h}{\Delta_{k}E(\tilde{v}_2)-\omega}\right>_{\tilde{v}_2}.
					\end{align} 
					where the (modified) dielectric constant $\tilde{\e}$ is defined by 
					\begin{equation}
						\tilde{\e}(k,\omega):=1+{\V(k)}\left< \frac{\Delta_kf_0(k_*)}{\Delta_kE(v_*)-\omega}\right>_{v_*}=1+{\V(k)}\int \frac{f_0(k_*) -f_0(v_*-\hbar k)}{\hbar k\cdot\left(v_*-\frac{\hbar k}{2}\right)-\omega}\ud{v_*}
					\end{equation}
					
					First we give an estimation of $\tilde{\e}$
					\begin{lemma}\label{lemma: borne epsilon}
						There exists a constant $C>0$ such that if $f_0$ satifies the bound  $\|\mathcal{V}\|_{L^1}\|\<v\>^{2d}\<\nabla\>^2f_0\|_{L^\infty}<1/C$,
						the function $\omega\mapsto\tilde{\e}(k,\omega)$ is holomorphic on $\{\omega,\,\Im \omega \neq 0\}$ and there exist a constant $\delta\in(0,1)$ independent of $\omega$ such that 
						$|\tilde{\e}(k,\omega)|>\delta$. In addition, 
						\begin{equation}|1-\tilde{\e}(k,\omega)|\lesssim\left\{ \begin{split}
							&{\|\<v\>^{2d}\<\nabla\>^2f_0\|_{L^\infty}/|\Im \omega|},\\
							&{\|\<v\>^{2d}\<\nabla\>^2f_0\|_{L^\infty}}/{|\Re \omega|}~{\rm if}~|\Re \omega|> 4\hbar^2 |k|^2,
						\end{split}\right.
						\end{equation}
						and in the same way
						\begin{equation}
							|\tfrac{\partial}{\partial\Re \omega}\tilde{\e}(k,\omega)|\lesssim \left\{ \begin{split}
								&{\|\<v\>^{2d}\<\nabla\>^3f_0\|_{L^\infty}/|\Im \omega \hbar k|},\\
								&{\|\<v\>^{2d}\<\nabla\>^3f_0\|_{L^\infty}}/{|\Re \omega\hbar k|}~{\rm if}~|\Re \omega|> 4\hbar^2 |k|^2.
							\end{split}\right.
						\end{equation}
					\end{lemma}
					\begin{proof}
						\step{1} We begin by treating the lower bound of $\tilde{\e}(k,\omega)$: 
						We denote $\tilde{\omega}:=\frac{\Re \omega}{\hbar^2|k|^2}$,
						\begin{equation}\label{eq:decompo e}
							\int \frac{f_0(v_*) -f_0(v_*-\hbar k)}{\hbar k \cdot(v_*-\frac{1}{2}\hbar k)-\omega}\ud{v_*}=\int \frac{f_0(v_*+\hbar k(\tfrac{1}{2}+\tilde{\omega})-f_0(v_*+\hbar k(-\tfrac{1}{2}+\tilde{\omega})) }{\hbar k\cdot v_*-i\Im \omega}\ud{v_*}
						\end{equation}
						Hence we can decompose it in real and imaginary parts
						\begin{align*}
							\Re \int \frac{f_0(v_*) -f_0(v_*-\hbar k)}{\hbar k \cdot(v_*-\frac{1}{2}\hbar k)-\omega}\ud{v_*}=&\int_{-\frac12}^{\frac12} \int \frac{(v_*\cdot \hbar k)\,\hbar k\cdot\nabla f\left(v_*+\hbar k\left(s+\tilde{\omega}\right)\right) }{(\hbar k\cdot v_*)^2+(\Im \omega)^2}\ud{v_*}\ud{s}\\
							\Im \int \frac{f_0(v_*) -f_0(v_*-\hbar k)}{\hbar k \cdot(v_*-\frac{1}{2}\hbar k)-\omega}\ud{v_*}=&\int_{-\frac12}^{\frac12} \ud{s}\int \frac{ \Im \omega \,\hbar k\cdot\nabla f\left(v_*+\hbar k\left(s+\tilde{\omega}\right)\right) }{(\hbar k\cdot v_*)^2+(\Im \omega)^2}\ud{v_*}
						\end{align*}
						
						We only need to bound the real part of $\tilde{\e}$ to get a bound from below. Using the orthogonal decomposition of $v_*$ as $v_*:= v_\para+v_\perp$, where $v_\perp\cdot k =0$, and denoting $\tilde{v}_s:=\hbar\left(s+\tilde{\omega}\right)k$, we get for any constant $a>0$
						\[\begin{split}
							\bigg|\Re\int \frac{f_0(v_*) -f_0(v_*-k)}{\hbar k \cdot(v_*-\frac{1}{2}\hbar k)-\omega}\ud{v_*}\bigg|&\leq \int_{-\frac12}^{\frac12} \int_{|\hbar k\cdot v_*|\geq a} \frac{\hbar|k|\left|\nabla f\left(v_*+\tilde{v}_s\right)\right|}{a}\ud{v_*}\ud{s} \\
							&+\int_{-\frac12}^{\frac12} \int_{|\hbar k\cdot v_\para|< a}  \left|\frac{\hbar v_\para\cdot k\,\hbar k\cdot\left(\nabla f\left(v_\para +v_\perp+\tilde{v}_s\right)-\nabla f\left(-v_\para +v_\perp+\tilde{v}_s\right)\right) }{(\hbar k\cdot v_\para)^2+(\Im \omega)^2}\right|\ud{v}_*\ud{s}\\
							&\lesssim\left(\frac{\hbar|k|}{a}+\frac{a}{\hbar|k|}\right) \|\<v\>^{2d}\<\nabla\>^2f_0\|_{L^\infty}
						\end{split}\]
						
						Choosing $a=\hbar|k|$, for some contant $c>0$,	
						\[ |\tilde{\e}(k,\omega)|\geq 1-c\|\mathcal{V}\|_{L^1}
						\|\<v\>^{2d}\<\nabla\>^2f_0\|_{L^\infty}>0.\]
						
						\step{2} We treat know the estimation of $|1-\tilde{\e}(k,\omega)|$. The bound with respect to $|\Im \omega|^{-1}$ can be deduced from our estimation of the lower bound of $|\tilde{\e}(k,\omega)|$.
						
						For the bound with respect to $\Re \omega$, we suppose now that $\Re \omega> 4\hbar^2 |k|$. Then $\tilde{v}_s\cdot\hat{k}$ is bigger than $\tfrac{\Re \omega}{4\hbar |k|}$ and 		
						\begin{align*}
							\bigg|\int \frac{f_0(v_*) -f_0(v_*-k)}{\hbar k \cdot(v_*-\frac{1}{2}\hbar k)-\omega}\ud{v_*}\bigg|&\leq \int_{-\frac12}^{\frac12} \int_{|\hbar k\cdot v_*|\geq \frac{\Re \omega}{4}} \frac{\hbar|k|\left|\nabla f\left(v_*+\tilde{v}_s\right)\right|}{\Re \omega/4}\ud{v_*}\ud{s} \\
							+\int_{-\frac12}^{\frac12}& \int_{|\hbar k\cdot v_\para|< \frac{\Re \omega}{4}}  \left|\frac{\hbar v_\para\cdot k\,\hbar k\cdot\left(\nabla f\left(v_\para +v_\perp+\tilde{v}_s\right)-\nabla f\left(-v_\para +v_\perp+\tilde{v}_s\right)\right) }{(\hbar k\cdot v_\para)^2+(\Im \omega)^2}\right|\ud{v}_*\ud{s}\\
							+\int_{-\frac12}^{\frac12}& \int_{|\hbar k\cdot v_\para|< \frac{\Re \omega}{4}}  \left|\frac{\Im \omega\,\hbar k\cdot\left(\nabla f\left(v_\para +v_\perp+\tilde{v}_s\right)+\nabla f\left(-v_\para +v_\perp+\tilde{v}_s\right)\right) }{(\hbar k\cdot v_\para)^2+(\Im \omega)^2}\right|\ud{v}_*\ud{s}\\
							&\lesssim\left(\hbar|k|+1\right) \frac{\|\<v\>^{2}\<\nabla\>^2f_0\|_{L^\infty}}{\Re \omega},
						\end{align*}
						where we use the estimation
						\begin{equation*}
							\frac{1}{1+|v_\para +v_\perp+\tilde{v}_s|^{2d}}\lesssim \frac{1}{(1+|v_\para +\tilde{v}_s|^{d})(1+|v_\perp|^{d})}\lesssim\frac{\hbar |k|}{\Re \omega(1+|v_\perp|^{d})}
						\end{equation*}
						and the identity
						\begin{equation*}
							\int_\mathbb{R} \frac{\hbar |k| \Im \omega \ud{x}}{(\hbar |k| x)^2+(\Im \omega)^2} = {\pi}.
						\end{equation*}
						
						We conclude by using the bound $\V(k) = O(\<k\>^{-1})$ (by Assumption \ref{ass:premier potentiel}).
						
						\step{3} Deriving \eqref{eq:decompo e} with respect to $\Re \omega$, one has
						\begin{align*}
							\tfrac{\partial}{\partial\Re \omega}\tilde{\e}(k,\omega)=\int \frac{k}{\hbar|k|^2}\cdot\frac{\nabla f_0(v_*+\hbar k(\tfrac{1}{2}+\tilde{\omega})-\nabla f_0(v_*+\hbar k(-\tfrac{1}{2}+\tilde{\omega})) }{\hbar k\cdot v_*-i\Im \omega}\ud{v_*}.
						\end{align*}
						The expected bound is obtained by using the same estimation than in Step 2.
					\end{proof}			
					
					Fix $\beta\in\mathbb{R}$ and $\omega\in\mathbb{C}\setminus\mathbb{R}$. We introduce now the function	
					\begin{align*}
						\mathcal{B}(v_1,k,\beta):=\left<\left(L_2+\beta-\frac{\omega}{2}\right)^{-1}\left(L_1-\beta-\frac{\omega}{2}\right)^{-1}A(v_1,v_2,k)\right>_{v_2}.
					\end{align*}
					We want to compute the limit as $\omega\to 0$ with $\Im \omega>|\Re \omega|$ of	$\frac{1}{2\pi i}\int_{\mathbb{R}}B(v_1,k)\ud\beta$. For $v_3\in\mathbb{{R}}^d$, we denote $v_3':=v_3+\hbar k$, and we apply the decomposition
					\begin{multline}\label{eq:B}
						\mathcal{B}(v_1,k,\beta)=\frac{\V(k)}{\tilde{\e}\left(k,\frac{\omega}{2}-\beta\right)}\left<\frac{f_0(v_1)f_0(v_2)-f_0(v_1')f_0(v_2')}{\left(\Delta_{-k}E(v_1)-\beta-\frac{\omega}{2}\right)\left(\Delta_{k}E(v_2)+\beta-\frac{\omega}{2}\right)}\right>_{\!\!\!v_2}\\
						+\frac{-\V(k)^2}{\tilde{\e}\left(k,\frac{\omega}{2}-\beta\right)\tilde{\e}\left(-k,\frac{\omega}{2}+\beta\right)}\frac{\Delta_{-k}f_0(v_1)}{\Delta_{-k}E(v_1)-\beta-\frac{\omega}{2}}\left<\frac{f_0(v_3)f_0(v_2)-f_0(v_3')f_0(v_2')}{\left(\Delta_{-k}E(k_3)-\beta-\frac{\omega}{2}\right)\left(\Delta_{k}E(v_2)+\beta-\frac{\omega}{2}\right)}\right>_{\!\!v_2,v_3}
					\end{multline}
					
					Using the identities
					\begin{gather}
						f_0(v_1)f_0(v_2)-f_0(v'_1)f_0(v'_2) = f_0(v_1)\Delta_kf_0(v_2)+\Delta_{-k}f_0(v_1)f_0(v_2'),\label{eq:decomposition prod croise}\\
						\V(k)\left\<\frac{\Delta_kf_0(v_2)}{\Delta_kE(v_2)+\beta-\frac{\omega}{2}}\right\>_{\!\! v_2}=\tilde{\e}(k,\tfrac{\omega}{2}-\beta)-1,\label{eq:reecriture constante dieelectrique}
					\end{gather}
					the first line of \eqref{eq:B} becomes
					\begin{equation*}
						\frac{\tilde{\e}\left(k,\frac{\omega}{2}-\beta\right)-1}{\tilde{\e}\left(k,\frac{\omega}{2}-\beta\right)}\frac{f_0(v_1)}{\Delta_{-k}E(v_1)-\beta-\frac{\omega}{2}}+\frac{\V(k)}{\tilde{\e}\left(k,\frac{\omega}{2}-\beta\right)}\frac{\Delta_{-k}f_0(v_1)}{\Delta_{-k}E(v_1)-\beta-\frac{\omega}{2}}\left<\frac{f_0(v'_2)}{\Delta_{k}E(v_2)+\beta-\frac{\omega}{2}}\right>_{\!\!v_2},
					\end{equation*}
					and the second line is decomposed as
					\begin{multline*}
						\frac{\V(k)\Delta_{-k}f_0(v_1)}{\Delta_{-k}E(v_1)-\beta-\frac{\omega}{2}}\bigg[
						\frac{1-\tilde{\e}\left(k,\frac{\omega}{2}-\beta\right)}{\tilde{\e}\left(k,\frac{\omega}{2}-\beta\right)\tilde{\e}\left(-k,\frac{\omega}{2}+\beta\right)}\left<\frac{f_0(v_3)}{\Delta_{-k}E(v_3)-\beta-\frac{\omega}{2}}\right>_{\!\!v_3}\\
						+\frac{1-\tilde{\e}\left(-k,\frac{\omega}{2}+\beta\right)}{\tilde{\e}\left(k,\frac{\omega}{2}-\beta\right)\tilde{\e}\left(-k,\frac{\omega}{2}+\beta\right)}\left<\frac{f_0(v_2')}{\Delta_{k}E(v_2)+\beta-\frac{\omega}{2}}\right>_{\!\!v_2}\bigg].
					\end{multline*}
					
					Hence, $\mathcal{B}(v_1,k,\beta)$ can be decompose into 3 peaces
					\begin{align*}
						&\mathcal{B}_1(v_1,k,\beta):=\frac{\tilde{\e}\left(k,\frac{\omega}{2}-\beta\right)-1}{\tilde{\e}\left(k,\frac{\omega}{2}-\beta\right)}\frac{f_0(v_1)}{\Delta_{-k}E(v_1)-\beta-\frac{\omega}{2}},\\
						&\mathcal{B}_2(v_1,k,\beta):=-\frac{\V(k)\Delta_{-k}f_0(v_1)}{\tilde{\e}\left(-k,\frac{\omega}{2}+\beta\right)(\Delta_{-k}E(v_1)-\beta-\frac{\omega}{2})}\left<\frac{f_0(v_3)}{\Delta_{-k}E(v_3)-\beta-\frac{\omega}{2}}\right>_{\!\!v_3},\\
						&\mathcal{B}_3(v_1,k,\beta):=\frac{\V(k)\Delta_{-k}f_0(v_1)}{(\Delta_{-k}E(v_1)-\beta-\frac{\omega}{2})\tilde{\e}\left(k,\frac{\omega}{2}-\beta\right)\tilde{\e}\left(-k,\frac{\omega}{2}+\beta\right)}\\
						&\hspace{6cm}\times\left(\left<\frac{f_0(v_2')}{\Delta_{k}E(v_2)+\beta-\frac{\omega}{2}}\right>_{\!\!v_2}-\left<\frac{f_0(v_2')}{\Delta_{k}E(v_2)+\beta+\frac{\omega}{2}}\right>_{\!\!v_2}\right).
					\end{align*}

					Note that that for any function $g$, $\Delta_{k}g(v_1-\hbar k)=-\Delta_{-k}g(v_1)$ and that $ \tilde{\e}(k,\bar{\omega})=\overline{ \tilde{\e}(k,{\omega})}$. The function $\beta\mapsto\tilde{\e}\left(k,\frac{\omega}{2}-\beta\right)^{-1}$  is holomorphic and bounded (thanks to Proposition \ref{lemma: borne epsilon}) in the half plane $\{\beta\in\mathbb{C},~\Im \beta<\Im \tfrac{\omega}{2}\}$. In the same way  $\beta\mapsto\tilde{\e}\left(-k,\frac{\omega}{4}+\beta\right)^{-1}$ is holomorphic in $\{\beta\in\mathbb{C}, \beta,~\Im \beta>-\Im \tfrac{\omega}{2}\}$. Using the residue Theorem, for any $R>\max(\Im \omega,4\hbar^2|k|^2)$ large enough,
					\[\begin{split}
					\int_{\mathbb{R}} \mathcal{B}_1(k,v_1,\beta) \frac{\ud{\beta}}{2i \pi} =& -\frac{f_0(v_1)(\tilde{\e}(k,\omega-\Delta_{-k}E(v_1))-1)}{\tilde{\e}(k,\omega-\Delta_{-k}E(v_1))}\\
					&+\int_{]-\infty,-R]\sqcup[-R,(-1+i)R]\sqcup[(-1+i)R,(1+i)R]\sqcup[(1+i)R,R]\sqcup[R,\infty[} B_1(k,v_1,\beta) \frac{\ud{\beta}}{2i \pi}
					\end{split}\]
					
					Using the estimation of Proposition \ref{lemma: borne epsilon}, the remaining integral is of order $O(R^{-1})$, and vanishes as $R\to \infty$. We can conclude that (using that \eqref{eq:decomposition prod croise} and that $\Delta_{-k}E(v_1)+\Delta_kE(v_2) = \hbar k\cdot(v_2-v_1-\hbar k)$)
					\[\begin{split}
					\lim_{\substack{\omega\to 0\\ \Im\omega>|\Re \omega|}}\int \big(\mathcal{B}_1(k,v_1,\beta)-\mathcal{B}_1(-k,&v_1',\beta)\big) \frac{\ud\beta}{2\pi i}=f_0(v_1)\left(\frac{1}{\tilde{\e}\left(k,i0-\Delta_{-k}E(v_1)\right)}-\frac{1}{\tilde{\e}\left(-k,i0-\Delta_{k}E(v_1')\right)}\right)\\
					&=\frac{2if_0(v_1)}{|\tilde{\e}(k,+i0-\Delta_{-k}E(v_1))|^2}\Im\left(1+\frac{{\V}(k)}{(2\pi)^d}\left<\frac{\Delta_kf(v_2)}{\hbar k\cdot(v_2-v_1-\hbar k)+i0}\right>_{v_2}\right)\\
					&=-\frac{2\pi i{\V}(k)f_0(v_1)}{|\tilde{\e}(k,+i0-\Delta_{-k}E(v_1))|^2}\left<\Delta_kf(v_2)\delta_{\hbar k\cdot(v_2-v_1-k)}\right>_{v_2}.
					\end{split}\]
					In the last line, we used the Sokhotskii-Plemelj formula: in the sens  of measure, one has
					\[\Im\frac{1}{k\cdot(v_2-v_1) + i0}= -\pi\delta_{k\cdot(v_2-v_1)}.\]
					
					In the same way, as $\mathcal{B}_2(v_1,k,\beta)$ is holomorphic on the half plane $\{\beta\in\mathbb{C},~\Im \beta<\Im \tfrac{\omega}{2}\}$, one has
					\[\lim_{\substack{\omega\to 0\\ \Im\omega>|\Re \omega|}}\int \big(\mathcal{B}_2(k,v'_1,\beta)-\mathcal{B}_2(-k,-v_1',\beta)\big)\frac{\ud\beta}{2\pi i}=0.\]
					
					We treat now $\mathcal{B}_{3}$. Note that  $\mathcal{B}_{3}(k,v_1,\beta))-\mathcal{B}_{3}(k,v_1',-\beta)$ is equal to 
					\begin{align*}
						&\frac{\V(k)}{\tilde{\e}\left(k,\frac{\omega}{2}-\beta\right)\tilde{\e}\left(-k,\frac{\omega}{2}+\beta\right)}\Bigg(\frac{\Delta_{-k}f_0(v_1)}{\Delta_{-k}E(v_1)-\beta-\frac{\omega}{2}}-\frac{\Delta_{-k}f_0(v_1)}{\Delta_{-k}E(v_1)-\beta+\frac{\omega}{2}}\Bigg)\\
						&\hspace{6cm}\times\Bigg(\left<\frac{f_0(v_2')}{\Delta_{k}E(v_2)+\beta+\frac{\omega}{2}}\right>_{v_2}-\left<\frac{f_0(v_2')}{\Delta_{k}E(v_2)+\beta-\frac{\omega}{2}}\right>_{v_2}\Bigg)\\
						=&\Bigg\<\frac{\V(k)\Delta_{-k}f_0(v_1)f_0(v_2')}{\tilde{\e}\left(k,\frac{\omega}{2}-\beta\right)\tilde{\e}\left(-k,\frac{\omega}{2}+\beta\right)}\Bigg(\frac{1}{\Delta_{-k}E(v_1)-\beta-\frac{\omega}{2}}-\frac{1}{\Delta_{-k}E(v_1)-\beta+\frac{\omega}{2}}\Bigg)
						\\
						&\hspace{6cm}\times\Bigg(\frac{1}{\Delta_{k}E(v_2)+\beta+\frac{\omega}{2}}_{v_2}-\frac{1}{\Delta_{k}E(v_2)+\beta-\frac{\omega}{2}}\Bigg)\Bigg\>_{v_2}\\
					\end{align*}
					
					Using Lemma \ref{lemma: borne epsilon}, we obtain the following bound
					\begin{gather*}
						\begin{split}
							\Bigg|\left<\frac{f_0(v_2')}{\Delta_{k}E(v_2)+\beta+\frac{\omega}{2}}\right>_{v_2}-\left<\frac{f_0(v_2')}{\Delta_{k}E(v_2)+\beta-\frac{\omega}{2}}\right>_{v_2}\Bigg|&\leq \left<\frac{f_0(v_2')|\omega|}{\left|(\Delta_kE(v_2)+\beta)^2+\omega^2/4\right|}\right>\\
							&\leq C\|\<v\>^{2}\<\nabla\>^2f_0\|_{L^\infty}\log\left(1+\frac{|\Re \omega|}{|\Im\omega|}\right)
						\end{split}\\
						\left|\frac{1}{2\pi i}\int \left(\mathcal{B}_{3}(v_1,k,\beta)-\mathcal{B}_{3}(v_1',-k)\right)\ud{\beta}\right|\leq C\V(k)\|\<v\>^{2}\<\nabla\>^2f_0\|_{L^\infty}^2 \log\left(1+\frac{|\Re \omega|}{|\Im\omega|}\right)^2.
					\end{gather*}
					
					We need now the following variant of the Sokhotskii-Plemelj formula :
					\begin{lemma}
						In the sens of distribution in $\mathcal{D}'(\mathbb{R}_\beta\times \mathbb{R}^d_{v_2})$, one have
						\begin{equation}
							\lim_{\substack{\omega\to 0\\ \Im\omega>|\Re \omega|}} \left(\frac{1}{\beta+\omega}-\frac{1}{\beta-\omega}\right)\left(\frac{1}{k(v_2-v_1)+\beta+\omega}-\frac{1}{k(v_2-v_1)+\beta-\omega}\right) = \delta_\beta\delta_{k\cdot(v_2-v_1)},
						\end{equation}
						with bounds that involves only the $0$-th and $1$-st order derivatives.
					\end{lemma}
					
					We deduce, using the $\mathcal{C}^1$ bound in Lemma \ref{lemma: borne epsilon}, that
					\begin{multline*}
						\lim_{\substack{\omega\to 0\\ \Im\omega>|\Re \omega|}}\frac{1}{2\pi i}\int \left(\mathcal{B}_{3}(v_1,k,\beta)-\mathcal{B}_3(v_1',-k,-\beta)\right)\ud{\beta}\\[-5pt]
						=\frac{2\pi i \V(k)\Delta_{-k}f_0(v_1)}{|\tilde{\e}\left(k,i0-\Delta_{-k}E(v_1)\right)|^2}\left< f_0(v_2-k)\delta_{\hbar k\cdot(v_2-v_1-\hbar k)}\right>_{v_2}
					\end{multline*}
					
					Finally,
					\begin{multline*}
						\lim_{\substack{\omega\to 0\\ \Im\omega>|\Re \omega|}}\left(\left<
						\Big(K_2+B-{\omega}\Big)^{-1}A(v_1,v_2,k)\right>_{v_2}-\left<
						\Big(K_2+B-\omega\Big)^{-1}A(v_1',v_2',-k)\right>_{v_2}\right)\\
						=\frac{2\pi i \V(k)}{|\tilde{\e}\left(k,i0-\Delta_{-k}E(v_1)\right)|^2}\left< (f_0(v_2')f_0(v_1')-f_0(v_1)f_0(v_2))\delta_{\hbar k\cdot(v_2-v_1-\hbar k)}\right>_{v_2},
					\end{multline*}
					and is uniformly bounded by $C\V(k)\|f\|^2$.
					
				\end{proof}
				
				We can now conclude the proof of Theorem \ref{thm:limite champ moyen}:
				
				First we have the identity
				\[Q^\hbar_{LB}(f)(v_1):=2\left<\frac{ \V(k)^2\delta_{\Delta_kE(v_2)+\Delta_{-k}E(v_1)}}{\hbar|\tilde{\e}\left(k,i0-\Delta_{-k}E(v_1)\right)|^2} (f_0(v_2')f_0(v_1')-f_0(v_1)f_0(v_2))\right>_{v_2,k}.\]
				
				Using dominated convergence theorem, 
				\begin{align*}
					{\mu}\int_0^\infty\!\left<\partial_t{f}_\mu(\tau_Nt,v_1)\varphi(t,v_1)\right>_{\!\!v_1}\!\! \ud{t} 
					\underset{N\to \infty}{\longrightarrow} & \int_{\mathbb{R}}\ud{\alpha}\left<\frac{\tilde{\varphi}(\alpha,v_1)}{1+i\alpha}Q_{LB}(f_0)(v_1)\right>_{v_1}\\
					&= \int_{0}^\infty\left<{\varphi}(t,v_1)Q^\hbar_{LB}(f_0)(v_1)\right>_{v_1}\ud{t}
				\end{align*}
				
				One can now use the estimation of Propositions \ref{prop:estimation une part} and \ref{prop:estimation grosse boite}.	For $\varphi:\mathbb{R}_t^+\times\mathbb{R}^d_{v_1}$ a smooth test function with compact support in time in $[0,T_*]$, and $\tau_N$ such that $\tfrac{N\tau_N}{\hbar \mu} = (\log N)^{1-\alpha}$ for some $\alpha\in(0,1)$,
				\begin{align*}
					&{\mu}\int_0^\infty \frac{(2\pi)^d}{L^d}\sum_{k\in\mathbb{Z}_L} \varphi\left(t,\hbar{k}\right)\partial_t \hat{F}_N^{1}(\tau_Nt,k,k) \ud{t}-	{\mu}\int_0^\infty\!\left<\partial_t{f}_\mu(\tau_Nt,v_1)\varphi(t,v_1)\right>_{\!\!v_1}\!\! \ud{t}   \\
					&=  O\left(\|\mu\partial_t (\hat{F}_N^{1}-\hat{F}_{\mu,L})\|_{L^\infty([0,T_*],\mathcal{L}^1)}\right)+O\left(\|\mu\partial_t ({f}_\mu-{f}_{\mu,L})\|_{L^\infty([0,T_*],E_1)}\right)\\
					&=O\left(N^{-1/4}+\frac{\exp\left({C\frac{T_*\tau_N}{\hbar}}\right)}{L/\hbar}\right)
				\end{align*}	
				for some constant $C$. Fixing $L\geq 2\hbar\log \tfrac{\tau_N}{\hbar}$, the preceding expression converges to $0$.
		
	\section{Bound of the dielectric constant $\e_\hbar(F,k,v_1)$}
		We begin by bounding the dielectric constant $\e_\hbar$. The bounds are similar to the one of $\e_0$ provided in \cite{DW}.
		
		\begin{prop}\label{prop:borne sur epsilone mieux}
			We fix the dimension $d\geq 2$. For any $\mathcal{V}$ satisfying Assumption \ref{ass: potentiel},
			\begin{enumerate}
				\item \emph{Non-degenerency at the Maxwellian}: For all $k,v\in\mathbb{R}^d$, \label{i}
				\begin{equation}
					|\e_0(M,k,v)|\simeq 1.
				\end{equation}
				
				\item Stability at fix $\hbar$: for any $\Phi,\Phi'$ measurable\label{ii}
				\begin{equation}\label{I}
					|\e_\hbar(\Phi,k,v)-\e_\hbar(\tilde{\Phi},k,v)|\lesssim  \|\<v\>^{\frac d2 }(\Phi-\tilde{\Phi})\|_{H^2}.
				\end{equation}
				If there is $f,\tilde{f}$ such that $\Phi=M+\sqrt{M}f$ and $\tilde{\Phi}=M+\sqrt{M}\tilde{f}$, 
				\begin{equation}
					|\e_\hbar(\Phi,k,v)-\e_\hbar(\tilde{\Phi},k,v)|\lesssim  \left\|f-\tilde{f}\right\|_2.
				\end{equation}
				
				\item Limit when $\hbar\to0$: for any $\Phi$ measurable\label{iii}
				\begin{equation}
					|\e_\hbar(\Phi,k,v)-\e_0(\Phi,k,v)|\lesssim \hbar\left\|\<v\>^{\frac d2 }\Phi\right\|_{H^3}.
				\end{equation}
				If there is $f$ such that $\Phi=M+\sqrt{M}f$,
				\begin{equation}
					|\e_\hbar(\Phi,k,v)-\e_0({\Phi},k,v)|\lesssim \hbar(1+\left\|f\right\|_3).
				\end{equation}
				
				\item Boundedness: For $(r,q,\ell)\in\mathbb{N}^3$, $\ell \leq 1$, and for any $\Phi$ measurable, denoting $\hat{k}:=\tfrac{k}{|k|}$,\label{iv}
				\begin{equation}
					|\nabla_{\hat{k}}^r\nabla_v^q\nabla_{|k|}^\ell\e_\hbar(\Phi,k,v)|\lesssim \<v\>^r\left\|\<v\>^{\frac d2+r}f\right\|_{H^{r+q+\ell+2}}
				\end{equation}
				If there is $f$ such that $\Phi=M+\sqrt{M}f$,
				\begin{equation}
					|\nabla_{\hat{k}}^r\nabla_v^q\nabla_{|k|}^\ell\e_\hbar(\Phi,k,v)|\lesssim \<v\>^r\left(1+\left\|f\right\|_{r+q+\ell+2}\right)
				\end{equation}
			\end{enumerate}
		\end{prop}
		
		\begin{proof}
			We will only prove the non-pertubativ forms of the estimations. The pertubativ forms can be deduced from
			\[\forall \ell,r\in\mathbb{N}\|\<v\>^{ r}(\sqrt{M} f)\|_{H^\ell}\lesssim_{r,\ell} \|f\|_\ell.\]
			
			\step{1} The proof of \eqref{i} has been performed in Lemma 2.1 of \cite{DW}.
			
			\step{2} Proof of \eqref{ii}.
			Setting $\hat{k}:=\tfrac{k}{|k|}$,
			\[\begin{split}
				\e_\hbar(\Phi,k,v)&=1+{\V(k)}\int_{-1}^0\ud{s}\int \frac{\hat{k}\cdot\nabla \Phi(v_*+s\hbar k)}{\hat{k}\cdot(v-v_*-\hbar k)-i0}\ud{v_*}=1+{\V(k)}\int_{1}^2\ud{s}\int \frac{\hat{k}\cdot\nabla \Phi(v_* -s\hbar k)}{\hat{k}\cdot(v-v_*)-i0}\ud{v_*}\end{split}\]
			
			Using Sobolev's inequalities, we deduce for any $s\in[1,2]$,
			\[\left\|\int \frac{\hat{k}\cdot\nabla \Phi(v_* -s\hbar k)}{\hat{k}\cdot(v-v_*)-i0}\ud{v_*}\right\|_{L^\infty}\lesssim \left\|\int \frac{\hat{k}\cdot\nabla \Phi(v_* -s\hbar k)}{\hat{k}\cdot(v-v_*)-i0}\ud{v_*}\right\|_{H^{1}(\mathbb{R}\hat{k})}\]
			
			Splitting the integral over $v_*\in <\hat{k}>\oplus<\hat{k}>^\perp$, and applying the usual bound on the Hilbert transform and trace operator, 
			\[\begin{split}
				\left\|\int \frac{\hat{k}\cdot\nabla \Phi(v_* -s\hbar k)}{\hat{k}\cdot(v-v_*)-i0}\ud{v_*}\right\|_{L^\infty}&\lesssim\left\|\int \hat{k}\cdot\nabla \Phi(v_*^\perp +\cdot)\ud{v_*^\perp}\right\|_{H^{1}(\mathbb{R}\hat{k})}\lesssim \|\<v\>^{\frac d2 +4}\Phi\|_{H^2}
			\end{split}\]
			
			\step{3} Proof of \eqref{iii}. We have		
			\[\e_\hbar(f,k,v)-\e_0({f},k,v) = {\V(k)}\int_{1}^2\ud{s}\int \frac{\hat{k}\cdot\nabla \left[f(v_* -s\hbar k)-f(v_*)\right]}{\hat{k}\cdot(v-v_*)-i0}\ud{v_*}.\]
			
			Using the same estimation as Step 2, 
			\begin{equation*}
				|\e_\hbar(f,k,v)-\e_0({f},k,v)| \lesssim  |\V(k)|\int_{1}^2\ud{s}\left\|\<v\>^{\frac d2 +4}\<\nabla\>^{2} \left[f(\cdot-s\hbar k)-\Phi\right]\right\|_{L^2}\\
				\lesssim \hbar\left\|\<v\>^{\frac d2 +4}\Phi\right\|_{H^3}.
			\end{equation*}
			
			\step{4} Proof of \eqref{iv}. As 
			\[\e_\hbar(\Phi,k,v) = 1+\V(|k|)\int_{1}^2\ud{s}\int \frac{\hat{k}\cdot\nabla^q\nabla \Phi(v-v_* -s\hbar |k|\hat{k})}{\hat{k}\cdot v_*-i0}\ud{v_*} \]We begin by the derivatives in $v$ and in $|k|$:
			\begin{gather*}
				\nabla_v^q \e_\hbar(\Phi,k,v)=\V(|k|)\int_{1}^2\int \frac{{k}\cdot\nabla^{q+1} \Phi(v-v_* -s\hbar k)}{{k}\cdot v_*-i0}\ud{v_*}\ud{s}\\
				\nabla_{|k|}\nabla_v^q \e_\hbar(\Phi,k,v)=\int_{1}^2\int\V'(|k|) \frac{\hat{k}\cdot\nabla^{q+1} \Phi(v-v_* -s\hbar k)}{\hat{k}\cdot v_*-i0}\ud{v_*}+ \hbar |k|\V(|k|)\frac{\hat{k}^{\otimes 2}:\nabla^{q+2} \Phi(v-v_* -s\hbar k)}{\hat{k}\cdot v_*-i0}\ud{v_*}\ud{s}.
			\end{gather*}
			Hence, without lost of generality, we can suppose that $\ell=q= 0$.
			
			We apply the Leibniz formula to 
			\[\nabla_{\hat{k}}^r\frac{\hat{k}\cdot\nabla^q \Phi(v-v_* -s\hbar k)}{\hat{k}\cdot v_*-i0}\]
			
			Fix $r_1,r_2\in\mathbb{N}$ such that $r_1+r_2=r$. One has
			\[\nabla_{\hat k}^{r_1}\hat k\cdot\nabla^{q+1} \Phi(v-v_* -s\hbar k) = (-s\hbar)^{r_1} \hat k\cdot\nabla^{q+r_1+1}\Phi(v-v_* -s\hbar k)+(-s\hbar)^{r_1-1} \nabla^{q+r}\Phi(v-v_* -s\hbar k)\]
			
			Denoting $v_*^\para := (\hat{k}\cdot v_* )\hat{k}$, we have
			\begin{align*}
				\nabla_{\hat k}^{r_2}\frac{1}{\hat{k}\cdot v_*-i 0}&= \sum_{r_3=0}^{r_2}\binom{r_2}{r_3}(-1)^{r_2}r_2!\frac{ \hat{k}^{\otimes (r_2-r_3)}\otimes_{\rm sym}(v_*^\perp)^{\otimes r_3}}{(\hat k\cdot v_*-i 0)^{r_3+1}}\\
				&=r_2!\sum_{r_3=0}^k\binom{r_2}{r_3}\frac{(-1)^{r_3}}{(r_2-r_3)!}\left[\hat{k}^{\otimes r3}:\nabla_{v_*}^{r_3}\frac{1}{\hat{k}\cdot v_*-i 0}\right]{ \hat{k}^{\otimes r_2-r_3}\otimes_{\rm sym}(v_*^\perp)^{\otimes r_3}}
			\end{align*}

			Using integration by part, one can transfer the $r_3$ derivative in $v_*$ into $F$. 
			
			\begin{multline*}
				\int \nabla_{v_*}^{r_3}\frac{1}{\hat k\cdot v_*-i 0}(v_*^\perp)^{\otimes r_3}\nabla^{q+r_1}\Phi(v-v_* -s\hbar k)\ud{v_*}=\int \frac{(v_*^\perp)^{\otimes r_3}\nabla^{q+r_1+r_3}f(v-v_* -s\hbar k)}{\hat k\cdot v_*-i 0} \ud{v_*}\\
				=O\left(\<v\>^r_3\left(1+\left\|\<v\>^{\frac d2+r_3}\Phi\right\|_{H^{q+r_1+r_3}}\right)\right)
			\end{multline*}
			Using that $\V$ respects Assumption \ref{ass: potentiel}, we can conclude the proof.
		\end{proof}

		We deduce from the preceding lemma the following bound
		\begin{corollary}
			There exist two constants $C_0>0$ and $\hbar_0>0$ such that for any $\hbar<\hbar_0$ and $\|f\|_3<C_0$, we have $|\e_\hbar(M+\sqrt{M} f,k,v)|\simeq 1$.
		\end{corollary}
	
	\section{Pertubative solutions of the Quantum Lenard--Balescu equation}\label{sec:cauchy theory}
		The next part is dedicated to the proof of Theorem \ref{thm:Cauchy theory}. The proof is based on the work of Duerinckx and Winter \cite{DW}, based on \cite{Guo}.
			
		\subsection{The discrete difference norm}
			
			In order to control the linearized collision operator $\mathcal{L}_\hbar$, we need to introduce an adapted norm, which looks like a Sobolev norm with wait.
			
			\begin{definition}
				We introduce $\mathcal{H}_\hbar$ the Hilbert space of norm $\|~\|_\hbar$ (with $\ud{V}:=\ud v_1\ud v_2 \ud k$)
				\begin{equation}
					\|g\|_\hbar^2 := \frac{2}{\hbar^{2}}\int \left[\left({g_1}-{g'_1}\right)^2M_2+g_1^2\left(\sqrt{M_2}-\sqrt{M'_2}\right)^2\right] \frac{ \V(k)^2\delta_{ k\cdot(v_2-v_1-\hbar k)}}{|\e_\hbar(M,k,v_1)|^2}\ud{V}.
				\end{equation}
				
				In the same way, we introduce for $\hbar = 0$ the space $\mathcal{H}_0$ with norm
				\begin{gather}
					\|g\|^2_0:=\int \nabla g_1 \left(\int B(v_1,v_2)M_2\ud{v_2}\right)\nabla g_1\ud{v_1} +\int  g_1v_1 \left(\int B(v_1,v_2)M_2\ud{v_2}\right) g_1v_1\ud{v_1}\\
					\label{eq:def de B}B(v_1,v_2):=B(M,v_1,v_2) = \int \frac{|\V(k)|^2k\otimes k}{|\e_0(M,k,v_1)|^2}\delta_{k\cdot(v_1-v_2)}\uud{k}.
					\end{gather}
			\end{definition}
			
			We begin with a general proposition:
			\begin{prop}\label{prop: equivalence dissipation}
				Fix $v_0\in \mathbb{R}^d$, $k_0>0$ and  $M_0>\tfrac{2}{k_0}$. Then, for any measurable function $g$, and any $s>0$
				\begin{equation}\label{eq: equivalence dissipation}
					\frac{1}{\hbar^2}\int_{|\hbar k|\leq 1} \frac{(g_1'-g_1)^2}{|k|\<k\>^{d+2+s}\<\hat{k}\cdot v_1 \>^{9/2}} \ud v_1 \uud k\\
					\lesssim_{M_0,v_0,k_0}\frac{1}{\hbar^2}\int (g_1'-g_1)^2 \ind_{\substack{M_0|\hat{k}\cdot (v_1-v_0) |\leq 1\\ M_0||k|-k_0|\leq 1}}\ud v_1 \uud k.
				\end{equation}
			\end{prop}
			\begin{proof}
				
				\step{1} We want to prove that 
				\begin{equation}\label{eq: equivalence dissipation i}
					\frac{1}{\hbar^2}\int_{|\hbar k|\leq 1} \frac{(g_1'-g_1)^2}{|k|\<k\>^{d+2+s}\<\hat{k}\cdot v_1 \>^{9/2}} \ud v_1 \uud k\\
					\lesssim\frac{1}{\hbar^2}\int_{M_0|k|\leq 1} \frac{(g_1'-g_1)^2}{|k|\<\hat{k}\cdot v_1 \>^{9/2}} \ud v_1 \uud k.
				\end{equation}
				
				Fix $R\in \mathbb{N}^*$. For $k\in \mathbb{R}^d$ with $\tfrac{R}{2M_0}\leq |k|<\tfrac{R+1}{2M_0}$, Cauchy-Schwartz inequality gives 
				\begin{align*}
					\frac{(g(v_1+\hbar k)-g(v_1))^2}{|k|\<k\>^{d+2+s}} \lesssim \frac{1}{R^{d+2+s}}\sum_{j=0}^{R-1} \frac{(g(v_1+\hbar(j+1) \frac{k}{R})-g(v_1+\hbar j \frac{k}{R}))^2}{| \frac{k}{R}|}
				\end{align*}
				
				We deduce that for $R\leq \tfrac{1}{\hbar}$
				\begin{align*}
					\int_{|\hbar k|\leq 1} \frac{(g_1'-g_1)^2}{|k|\<k\>^{d+3}\<\hat{k}\cdot v_1 \>^{9/2}} \ud v_1 \uud k  &\lesssim \sum_{j=0}^{R-1}\int_{\tfrac{R}{2M_0}\leq |k|<\tfrac{R+1}{2M_0}} \frac{(g(v_1+\hbar(j+1) \frac{k}{R})-g(v_1+\hbar j \frac{k}{R}))^2}{R^{d+2+s}| \frac{k}{R}|\<\hat{k}\cdot v_1 \>^{9/2}} \ud v_1 \uud k\\
					&\lesssim\frac{1}{R^{1+s}}\int_{\tfrac{1}{2M_0}\leq |k|<\tfrac{1}{M_0}} \frac{(g(v_1+\hbar k)-g(v_1))^2}{|k|\<\hat{k}\cdot v_1 \>^{9/2}} \ud v_1 \uud k.
				\end{align*}
				
				Summing with respect to $R\in \mathbb{N}\cap [0,\hbar ^{-1}]$, we deduce Estimation \eqref{eq: equivalence dissipation i}
				
				\step{2} We want to prove that 
				\begin{equation}\label{eq: equivalence dissipation ii}
					\frac{1}{\hbar^2}\int_{M_0|k|\leq 1} \frac{(g_1'-g_1)^2}{|k|\<\hat{k}\cdot v_1 \>^{9/2}} \ud v_1 \uud k\lesssim M_0^3\<v_0\>^{\frac 92}\frac{1}{\hbar^2}\int_{M_0|k|\leq 1} \frac{(g_1'-g_1)^2}{|k|} \ind_{M_0|\hat k \cdot(v_1-v_0)|\leq 1} \ud v_1 \uud k.
				\end{equation}
				
				Performing the change of variable $v_1\mapsto v_1+v_0$ and that $\<v_1\>\geq \<v_1+v_0\>/\<v_0\>$, we can suppose that $v_0=0$.
				
				We introduce the change of variable $k\mapsto(|k|,x,\sigma)\in\mathbb{R}^+\times[0,1]\times \mathbb{S}^{d-2}_{v_1^\perp}$, (where we denote $\mathbb{S}^{d-2}_{v^\perp_1} := \{\sigma\in\mathbb{S}^{d-1}|\sigma\cdot v_1=0\}$) defined by
				\begin{equation}\label{eq:changement de variable k}
					k = |k|\left( \sqrt{1-x^2} \sigma +x\tfrac{v_1}{|v_1|}\right),~|\sigma|=1=0, \text{ and }\sigma\cdot v_1=0.
				\end{equation}
				The Jacobian of this transformation is 
				\[\uud k = \frac{|k|^{d-1} \left(1-x^2\right)^{\frac{d-3}{2}}}{(2\pi)^d}\ud{|k|}\ud{x}\ud{\sigma}.\]
				
				We suppose now that $|k|<M_0$ and that $\tfrac{|\hat{k}\cdot v_1|}{3M_0}\in]R-1,R]$ for some $R\in\mathbb{N}\setminus\{0,1\}$. We define $k^R_1,k^R_2$ by
				\begin{gather*}
					k^R_{1}:=\frac{|k|}{2}\left(1+\frac{\sqrt{1-x^2}}{R\sqrt{1-{ x^2}/{R^2}}}\right)\left(\sqrt{1-{ x^2}/{R^2}} \sigma + \frac{x}{R}\hat{v_1}\right)\\
					k^R_{2}:=\frac{|k|}{2}\left(1-\frac{\sqrt{1-x^2}}{R\sqrt{1-{ x^2}/{R^2}}}\right)\left(-\sqrt{1-{ x^2}/{R^2}} \sigma + \frac{x}{R}\hat{v_1}\right).
				\end{gather*}
				One has $k^R_1+k^R_2 =\tfrac{1}{R} k$ and $\tfrac{|k_1^R|}{|k|},\tfrac{|k_2^R|}{|k|}\in[\frac{1}{4},1]$. In addition, the Jacobian of the application $k\mapsto k_i^R$ is 
				\begin{equation*}
					\uud k_i^R = \frac{1}{R}\left(\frac12-(-1)^i\frac{\sqrt{1-x^2}}{2R\sqrt{1-{ x^2}/{R^2}}}\right)^{d-1} \uud k =: \frac{1}{RJ_i^R(k)}\uud k,
				\end{equation*}
				where $J_i^R(k)\in[\tfrac{1}{4^{d-1}},1]$.
				
				We perform the decomposition
				\begin{multline}
					(g_1'-g_1)^2 \\ \leq R\sum_{j=0}^{R-1} \big(g(v_1+\hbar (\tfrac jR k+k_1^R+k_2^R))-g(v_1+\hbar(\tfrac jR k+k_1^R)\big)^2+\big(g(v_1+\hbar(\tfrac jR k+k_1^R))-g(v_1+\hbar \tfrac jR k)\big)^2
				\end{multline}
				Note that for any $j\in[0,R]$, we have
				\begin{equation*}
					\hbar\big|\tfrac jR k\big|\leq \hbar|k|\leq\tfrac{\hbar}{2M_0} \text{, and } \hbar\big|jR k +k_1^R\big| \leq \sqrt{2}|\hbar k|\leq \tfrac{\sqrt{2}\hbar}{2M_0} .
				\end{equation*}
				Then, denoting $\Omega_R := \{|k|<\tfrac{1}{2M_0},~\tfrac{R-1}{3M_0}<{|\hat{k}\cdot v_1|}\leq \tfrac{R}{3M_0}\}$ and performing the changes of variable $v_1\mapsto v_1+j(k_1^R+k_2^R)$ and $v_1\mapsto v_1+j(k_1^R+k_2^R)+k_1^R$, for $\hbar<(6M_0)^{-2}$, 
				\begin{multline*}
					\int_{\Omega_R} \frac{(g_1'-g_1)^2}{|k|\<\hat{k}\cdot v_1 \>^{9/2}} \ud v_1 \uud k\lesssim \frac{M_0^3}{R^{5/2}}\int_{|k|\leq \frac{1}{2M_0} } \left[(g(v_1+\hbar k_1^R)-g(v_1))^2+(g(v_1+\hbar k_2^R)-g(v_1))^2\right]\\
					\times\frac{\ind_{|\hat{k}\cdot v_1|\in\left[\frac{R-1}{6M_0},\frac{R}{2M_0}\right]}}{|k|}\ud v_1 \uud k.
				\end{multline*}
				Using the change of variable $k\mapsto k_1^R$, $k\mapsto k_2^R$, we deduce that 
				\begin{align*}
					\int_{\Omega_R} \frac{\V(k)^2(g_1'-g_1)^2}{|k|\<\hat{k}\cdot v_1 \>^{9/2}} \ud v_1 \uud k\lesssim \frac{M_0^3}{R^{3/2}}\int_{|k|\leq \frac{1}{2M_0} } \frac{\V(k)^2}{|k|}(g(v_1+\hbar k)-g(v_1))^2\ind_{|\hat{k}\cdot v_1|\leq \frac{1}{2M_0}}\ud v_1 \uud k.
				\end{align*}
				
				Finally,
				\begin{multline*}
					\int_{|k|\leq \frac{1}{2M_0}} \frac{\V(k)^2(g_1'-g_1)^2}{|k|\<\hat{k}\cdot v_1 \>^{9/2}} \ud v_1 \uud k= \sum_{R=1}^\infty\int_{\Omega_R} \frac{\V(k)^2(g_1'-g_1)^2}{|k|\<\hat{k}\cdot v_1 \>^{9/2}} \ud v_1 \uud k\\
					\lesssim \sum_{R=1}^\infty\frac{M_0^3}{R^{3/2}}\int_{|k|\leq \frac{1}{2M_0}} \frac{\V(k)^2}{|k|}(g(v_1+\hbar k)-g(v_1))^2\ind_{|\hat{k}\cdot v_1|\leq \frac{1}{2M_0}}\ud v_1 \uud k.
				\end{multline*}
				This conclude the proof of \eqref{eq: equivalence dissipation ii}.
				
				\step{3} We need to change the integration set of $k$: denoting $k = |k|\hat k$ and $k' = |k'| \hat k$, 
				\begin{align*}
					&\int_{2M_0|k|\leq 1} \frac{(g(v_1+\hbar k)-g(v_1))^2}{|k|}\ind_{|\hat{k}\cdot v_1|\leq \frac{1}{2M_0}} \ud v_1 \uud k  \\
					&\lesssim M_0\int_{\substack{2M_0|k|\leq 1\\2M_0||k'|-k_0|\leq 1}} \left[(g(v_1+\hbar k')-g(v_1))^2+(g(v_1+\hbar k)-g(v_1+\hbar k'))^2\right] \frac{\ind_{|\hat{k}\cdot v_1|\leq \frac{1}{2M_0}}}{|k|}\uud k \ud |k'| \ud v_1\\
					&\lesssim M_0\int_{2M_0|k|\leq 1}|k|^{d-2}\ud |k|\int_{\substack{M_0||k'|-k_0|\leq 1}} (g(v_1+\hbar k')-g(v_1))^2 \tfrac{|k'|^{d-1}}{k_0^{d-1}}\ind_{|\hat{k}\cdot v_1|\leq \frac{1}{2M_0}}\ud |k'|\ud \hat k \ud v_1\\
					&\lesssim \frac{M_0^{d}}{k_0^{d-1}}\int_{2M_0|k|\leq 1}\tfrac{|k|^{d-2}}{\<k\>^d}\ud |k|\int_{\substack{M_0||k|-k_0|\leq 1}} (g(v_1+\hbar k)-g(v_1))^2\ind_{|\hat{k}\cdot v_1|\leq \frac{1}{2M_0}} \ud |k|\ud \hat k \ud v_1
				\end{align*}
				
				This conclude the proof.
			\end{proof}
			
			\begin{corollary}\label{prop:form equivalent norm}
				We fix the dimension $d\geq 2$.
				
				There exists $\hbar_0>0$ such that for any $\hbar\in(0,\hbar_0)$, $\alpha\in(0,1/2]$ and $g$ a test function,
				\begin{equation}\label{eq:equivalence des normes}
					\frac{1}{\hbar^2}\int\left({g_1}-{g'_1}\right)^2\frac{\V^2(k)}{|k|}e^{-\alpha(v_1\cdot\hat{k})^2}\uud{k} \ud{v_1}+\|\<v\>^{-{1/2}}g\|^2\simeq\|g\|_\hbar.
				\end{equation}
				where the constant are independent of $\hbar$.
			\end{corollary}
			\begin{remark}
				It has been proved in \cite{DW} that
				\begin{equation}
					\|g\|_0^2 \simeq \left\|\<v\>^{-\frac{3}{2}}\tfrac{v}{|v|}\cdot\nabla g(v)\right\|^2+\left\|\<v\>^{-\frac{1}{2}}\left({\rm Id}-\tfrac{v\otimes v}{|v|^2}\right)\nabla g(v)\right\|^2+\left\|\<v\>^{-\frac{1}{2}}g(v)\right\|^2.
				\end{equation}
			\end{remark}
			\begin{proof}
				\step{1}  We want to prove the upper bound
				\begin{equation}
					\frac{1}{\hbar^{2}}\int g_1^2\left({M_2}^\alpha-{M'_2}^\alpha\right)^2 {  \V(k)^2\delta_{ k\cdot(v_2-v_1-\hbar k)}}\ud{V}\lesssim \|\<v\>^{-{1/2}}g\|^2
				\end{equation}	
				
				First, for $\alpha\in(0,\tfrac{1}{2})$,
				\begin{equation*}
					\int g_1^2\left({M_2}^\alpha-{M'_2}^\alpha\right)^2 \frac{ \V(k)^2\delta_{ k\cdot(v_2-v_1-\hbar k)}}{\hbar^2}\ud{V}
					=\int{g_1}^2\frac{C\V^2(k)}{\hbar^2|k|}\left(e^{-\alpha (v_1\cdot \hat{k}+\hbar k)^2}-e^{-\alpha (v_1\cdot \hat{k})^2}\right)^2\uud{k} \ud{v_1}
				\end{equation*}
				
				We use the change of variable introduce the change of variable $k\mapsto(|k|,x,\sigma)\in\mathbb{R}^+\times[0,1]\times \mathbb{S}^{d-2}_{v_1}$, introduce in \eqref{eq:changement de variable k}.
				
				\begin{align}\label{eq:2.30}
					\int&\frac{ \V^2(k)}{|k|}\left(e^{-\alpha (v_1\cdot \hat{k}+\hbar k)^2}-e^{-\alpha (v_1\cdot \hat{k})^2}\right)^2\uud{k}\\
					&= \int \tfrac{\V^2(k)|k|^{d-2}}{(2\pi)^d}\left(\int_0^12\alpha\hbar k(|v_1|x+s\hbar |k|)\exp\left(-\alpha(x|v_1|+s\hbar|k|)^2\right)\ud{s}\right)^2 (1-x^2)^{\frac{d-2}{2}}\ud{|k|}\ud{x}\nonumber\\
					&\lesssim \hbar^2\alpha^2 \int_0^1\int \V^2(k)|k|^d\int_{-1}^1\exp\left(-\alpha(x|v_1|+s\hbar|k|)^2\right)\ud{x}\ud{|k|}\ud{s}\lesssim \frac{\hbar^2\alpha^{3/2}}{<v_1>}.\nonumber
				\end{align}
				
				\step{2} We treat the lower bound
				\begin{equation}
					\frac{1}{\hbar^{2}}\int g_1^2\left({M_2}^\alpha-{M'_2}^\alpha\right)^2 { \V(k)^2\delta_{ k\cdot(v_2-v_1-\hbar k)}}\ud{V}\gtrsim \|\<v\>^{-{3/2}}|v|g\|^2-C\hbar^2\|\<v\>^{-{1/2}}g\|^2
				\end{equation}	
				for some constant $C>0$. We can restrict our self to the $k$ that verifies $|k|<1$.
				\begin{multline*}
					\int_{-1}^1\left(\int_0^1(|v_1|x+s\hbar |k|)e^{-\alpha(|v_1|x+s\hbar|k|)^2}\ud s\right)^2(1-x^2)^{\frac{d-2}{2}}\ud{x}\\
					\begin{split}
						&\gtrsim \int_{-1/2}^{1/2}\left(|v_1x|^2-\frac{\hbar^2 |k|}{4}\right) e^{-2\alpha\left(|v_1x|^2 +|\hbar k|^2\right)}\ud x\\
						&\gtrsim \frac{C}{|v_1|}\int_{-2|v_1|}^{2|v_1|}y^2e^{-2\alpha y^2}\ud y -O\left(\frac{\hbar^2}{\<v_1\>}\right)\geq \frac{C_\alpha |v_1|}{\<v_1\>^{3}} -O\left(\frac{|k|^2\hbar^2}{\<v_1\>}\right).
					\end{split}
				\end{multline*}
				We can conclude by integrating with respect to $k$.
				
				\step{3} We treat now the "differential part" of the norm: for $\alpha\in(0,1)$, there exits a constant $C>0$ such that
				\begin{equation*}
					\frac{2}{\hbar^2}\int\left({g_1}-{g'_1}\right)^2\V^2(k) ({M'_2})^{2\alpha}\delta_{k\cdot(v_2-v_1-\hbar k)}\ud{V} = \frac{C}{\hbar^2}\int\left({g_1}-{g'_1}\right)^2\frac{\V^2(k)}{|k|}e^{-\frac{\left(v_1\cdot{\hat{k}}\right)^2}{2}}\uud{k} \ud{v_1}.
				\end{equation*}

				We begin by cutting the large $k$ and $v_1$. Using Assumption\ref{ass: potentiel}, we have
				\begin{align*}
					\int \frac{\V^2(k)}{|k|}\exp\left(-\frac{\alpha\left(v_1\cdot{k}\right)^2}{|k|^2}\right)\ind_{|k|\geq \frac{1}{\hbar}}\uud{k}  &\lesssim \int_{\frac{1}{\hbar}}^\infty \int_{-1}^1{\V^2(k)}|k|^{d-2} \left(1-x^2\right)^{\frac{d-3}{2}} e^{-\alpha|v_1|^2 x^2}\ud{x}\ud|k|\\
					&\lesssim\frac{\hbar^4}{\<v_1\>}.
				\end{align*}

				As $\V$ respects Assumption \ref{ass: potentiel}, there exists $k_0,M_0,C_0>0$ such that $\V(k)^2>\tfrac{1}{C_0}\ind_{|k-k_0|>\frac{1}{M_0}}$. Hence, applying Proposition \ref{prop: equivalence dissipation},
				\begin{align*}
					\frac{1}{\hbar^2}\int_{|\hbar k|\leq 1}\left({g_1}-{g'_1}\right)^2\frac{\V^2(k)}{|k|}e^{-\frac{\left(v_1\cdot{\hat{k}}\right)^2}{2}}\uud{k} \ud{v_1}  &\lesssim \frac{1}{\hbar^2}\int\left({g_1}-{g'_1}\right)^2\ind_{\substack{M_0|k_0-|k||\leq 1\\|\hat k \cdot v_1|\leq 1 }}\uud{k} \ud{v_1}\\
					&\lesssim \frac{1}{\hbar^2}\int\left({g_1}-{g'_1}\right)^2\frac{\V^2(k)}{|k|}e^{-\alpha\left(v_1\cdot{\hat{k}}\right)^2}\uud{k} \ud{v_1}.
				\end{align*}
					
				\step{4} We prove the integrability near $0$:
				\[\int_{|v_1|<1} g(v_1)^2 \ud{v_1} \leq C \int_{\substack{2<|k|<3\\|v_1|<1}} (g(v_1+k)-g(v_1))^2\ud{v_1}+\int_{\substack{2<|k|<3\\|v_1|<1}} (g(v_1+k))^2\ud{v_1}\]
				
				Using the convex inequality, one can bound $\int_{|v_1|<1} g(v_1)^2 \ud{v_1}$ by (for some $C>0$)
				\begin{align*}
					&C\left\lfloor \hbar^{-1}\right\rfloor\sum_{n= 0}^{\left\lfloor \hbar^{-1}\right\rfloor-1}\int_{\substack{2<|k|<3\\|v_1|<1}} \left(g\left(v_1+\frac{(n+1)}{\left\lfloor \hbar^{-1}\right\rfloor}k\right)-g\left(v_1+  \frac{n}{\left\lfloor \hbar^{-1}\right\rfloor}k\right)\right)^2\ud{v_1}\uud{k}+C\|\<v\>^{-3/2}|v|g\|^2\\
					&\leq C\int_{\substack{2<|k|<3\\|v_1|<1+3\hbar}} \frac{\left(g\left(v_1+\frac{k}{\left\lfloor \hbar^{-1}\right\rfloor}\right)-g(v_1)\right)^2}{\left\lfloor \hbar^{-1}\right\rfloor^{-2}}\ud{v_1}\uud{k}+C\|\<v\>^{-3/2}|v|g\|^2\\
					&\leq \frac{1}{\hbar^2}\int\left({g_1}-{g'_1}\right)^2\frac{\V^2(k)}{|k|}e^{-\frac{1}{2}\left(v_1\cdot{\hat{k}}\right)^2}\uud{k} \ud{v_1} +C\|\<v\>^{-3/2}|v|g\|^2,
				\end{align*}
				where we use the change of variable $k\mapsto \tfrac{\hbar^{-1}}{\lfloor\hbar^{-1}\rfloor}k$ and that $2<|k|<3$, $|v_1|<4$.
				
				Hence 
				\[\|\<v\>^{-1/2}g\|^2\lesssim\int\frac{\left({g_1}-{g'_1}\right)^2}{\hbar^2}\frac{\V^2(k)}{|k|}e^{-\frac{1}{2}\left(v_1\cdot{\hat{k}}\right)^2}\uud{k} \ud{v_1} +C\|\<v\>^{-3/2}|v|g\|^2\]
				
				\step{5} Conclusion
				
				Using Step 1, for any $\alpha<\tfrac{1}{2}$,
				\[\|g\|_\hbar^2 \lesssim \frac{1}{\hbar^2}\int\left({g_1}-{g'_1}\right)^2\frac{\V^2(k)}{|k|}e^{-2\alpha(v_1\cdot\hat{k})^2}\uud{k} \ud{v_1} + \|\<v\>^{-1/2}g\|^2. \]
				
				Using Steps 2, 3 and 4, for any $\alpha<\tfrac{1}{2}$,
				\begin{align*}
					\frac{1}{\hbar^2}\int\left({g_1}-{g'_1}\right)^2\frac{\V^2(k)}{|k|}e^{-\alpha(v_1\cdot\hat{k})^2}\uud{k} \ud{v_1} + \|\<v\>^{-1/2}g\|^2 \lesssim&  \frac{1}{\hbar^2}\int\left({g_1}-{g'_1}\right)^2\frac{\V^2(k)}{|k|}e^{-\tfrac{1}{2}(v_1\cdot\hat{k})^2}\uud{k} \ud{v_1} \|g\|_\hbar \\
					&+\|\<v\>^{-3/2}|v|g\|^2 +\hbar^2\|\<v\>^{-1/2}g\|^2\\
					\lesssim &\|g\|_\hbar  +\hbar^2\|\<v\>^{-1/2}g\|^2.
				\end{align*}
				
				As the constant does not depend on $\hbar$, the estimation \eqref{eq:equivalence des normes} holds for $\hbar\leq\hbar_0$ small enough.
			\end{proof}
			
			\begin{prop}\label{prop: decroissance des normes}
				There exists a constant $C$ such that  for any test function $g\in L^2$,
				\begin{equation}
					0\leq \hbar<\hbar' \Rightarrow\|g\|_{\hbar'}\leq C\|g\|_{\hbar}.
				\end{equation}
			\end{prop}
			\begin{proof}
				We recall that 
				\[\|g\|^2_{\hbar}\simeq\frac{1}{\hbar^2}\int\left({g(v_1)}-{g(v_1+\hbar k)}\right)^2\frac{\V^2(k)}{|k|}e^{-\frac{(v_1\cdot \hat{k})^2}{2}}\uud{k} \ud{v_1} +\|\<v\>^{-1/2}g\|^2\]
				Hence we only need to focus on the first terms.
				
				\begin{align*}
					&\frac{1}{\hbar^2}\int\left({g(v_1)}-{g(v_1+\hbar k)}\right)^2\frac{\V^2(k)}{|k|}e^{-\frac{(v_1\cdot \hat{k})^2}{2}}\uud{k} \ud{v_1}\\
					&\leq \frac{2}{\hbar^2}\int\left[\left({g(v_1)}-{g(v_1+\frac{\hbar}{2} k)}\right)^2+\left(g(v_1+\frac{\hbar}{2}k)-{g(v_1+\hbar k)}\right)^2\right]\frac{\V^2(k)}{|k|}e^{-\frac{(v_1\cdot \hat{k})^2}{2}}\uud{k} \ud{v_1}\\
					&\leq\frac{1}{(\hbar/2)^2}\int\left({g(v_1)}-{g(v_1+\frac{\hbar}{2} k)}\right)^2\frac{\V^2(k)}{|k|}e^{-\frac{(v_1\cdot \hat{k})^2}{2}}\uud{k} \ud{v_1}
				\end{align*}
				For $q\in]\frac{1}{2},1]$,
				\begin{align*}
					\frac{1}{\hbar^2}\int\left({g(v_1)}-{g(v_1+\hbar k)}\right)^2\frac{\V^2(k)}{|k|}e^{-\left(v_1\cdot{\hat{k}}\right)^2}\uud{k} \ud{v_1}
					=&\frac{q^2}{(q\hbar)^2}\int\left({g(v_1)}-{g(v_1+q\hbar k)}\right)^2\frac{\V^2(qk)}{|qk|}e^{-\left(v_1\cdot{\hat{k}}\right)^2}\frac{\uud{k}}{q^d} \ud{v_1}\\
					\lesssim& \frac{1}{\hbar^2}\int\left({g(v_1)}-{g(v_1+q\hbar k)}\right)^2\frac{\V^2(k)}{|k|}e^{-\left(v_1\cdot{\hat{k}}\right)^2}\uud{k} \ud{v_1}
				\end{align*}
				Finally we use that for any $0<\hbar<\hbar'$, there exists a couple $(r,q)\in\mathbb{N}\times ]1/2,1]$ such that $\hbar = q2^{-r}\hbar'$.
				
				The case $\hbar=0$ can be deduce by taking the limit for smooth test functions.
			\end{proof}
			
		\subsection{Decomposition of the quadratic form $\< g,\mathcal{L}_\hbar g\>$}
			We can decompose $\<g,\mathcal{L}_\hbar g\>$ into three peaces:
			\begin{align}
				\int g\mathcal{L}_\hbar g =& \frac{1}{2\hbar^2}\int \left(g_1\sqrt{M_2}+g_2\sqrt{M_1}-g'_1\sqrt{M'_2}-g'_2\sqrt{M'_1}\right)^2\frac{\V(k)^2\delta_{k\cdot(v_2-v_1-\hbar k)}}{|\e(M,v_1,k)|^2}\ud{V}\nonumber\\
				=&\frac{1}{\hbar^2}\int \left[\left(g_1-g'_1\right)^2M_2+g_1'^2\left(\sqrt{M_2}-\sqrt{M_2'}\right)^2\right]\frac{\V(k)^2\delta_{k\cdot(v_2-v_1-\hbar k)}}{|\e_\hbar(M,v_1,k)|^2}\ud{V}\label{e:decompo quadra 1}\\
				&+\frac{2}{\hbar^2}\int \left[\left(g_1-g'_1\right)g_1'\sqrt{M_2}\left(\sqrt{M_2}-\sqrt{M_2'}\right)\right]\frac{\V(k)^2\delta_{k\cdot(v_2-v_1-\hbar k)}}{|\e_\hbar(M,v_1,k)|^2}\ud{V}\label{e:decompo quadra 2}\\
				&+\frac{1}{\hbar^2}\int \left(g_1\sqrt{M_2}-g'_1\sqrt{M'_2}\right)\left(g_2\sqrt{M_1}-g'_2\sqrt{M'_1}\right)\frac{\V(k)^2\delta_{k\cdot(v_2-v_1-\hbar k)}}{|\e_\hbar(M,v_1,k)|^2}\ud{V}\label{e:decompo quadra 3}
			\end{align}
			
			This decomposition will be necessary in the proof of Proposition \ref{prop:dissipation estimates}, which is the crucial step of the proof. One can identify \eqref{e:decompo quadra 1} with $\|g\|_\hbar^2/2$. The goal of this section is to study the remaining parts.
			
			We need to introduce a cutoff function.
			\begin{definition}\label{def: fonction de cutoff}
				Fix $\chi:\mathbb{R_+}\to [0,1]$ an increasing function, with $\chi(1) = 1$, $\chi(0)= 0$ and $\chi'$ is supported in $[1/3,2/3]$. We define $\chi_K(v):=\chi(K|v|)$.
			\end{definition}
			
			\begin{prop}\label{prop: traitement K}
				We fix the dimension $d\geq 2$. There exists a kernel $\mathcal{K}_\hbar(v_1,v_2)\in \mathcal{C}(\mathbb{R}^{2d}\setminus\{v_1=v_2\})$ such that
				\begin{equation}
					\frac{2}{\hbar^{2}}\int \left({g_2}{\sqrt{M_1}}-{g'_2}{\sqrt{M'_1}}\right) \frac{ \V(k)^2\delta_{ k\cdot(v_2-v_1-\hbar k)}}{|\e_\hbar(M,k,v_1)|^2}\sqrt{M_2}\uud{k}\ud{v_2}=\int\mathcal{K}_\hbar(v_1,v_2)g(v_2)\ud{v_2}
				\end{equation}
				with the bound
				\begin{equation}
					\left|\mathcal{K}_\hbar(v_1,v_2)\right|\leq \frac{C(1+|v_1|^2+|v_2|^2)\sqrt[13]{M_1M_2}}{|v_1-v_2|^3}.
				\end{equation}
				
				In addition, for almost all $(v_1,v_2)$, 
				\begin{equation}\label{eq:valeur de K limite}
					\lim_{\hbar\to0}\mathcal{K}_\hbar(v_1,v_2) = \mathcal{K}_0(v_1,v_2):=- \left(\nabla_1-\frac{v_1}{2}\right)\otimes\left(\nabla_2-\frac{v_2}{2}\right):\left(B(v_1,v_2)\sqrt{M_1M_2}\right).
				\end{equation}
				
				Finally, one can cutoff the singularity at $\{v_1=v_2\}$: for $\chi_K$ defined as is Definition \ref{def: fonction de cutoff},
				\begin{multline}
					\frac{2}{\hbar^{2}}\int \left({g_2}{\sqrt{M_1}}-{g'_2}{\sqrt{M'_1}}\right)  \frac{ \V(k)^2\delta_{ k\cdot(v_2-v_1-\hbar k)}}{|\e_\hbar(M,k,v_1)|^2}\sqrt{M_2}\uud{k}\ud{v_2}\\
					=\int\left(\chi_K(v_1-v_2)\mathcal{K}_\hbar(v_1,v_2)-\rho_K(v_1-v_2,\tfrac{v_1+v_2}{2})\right)g(v_2)\ud{v_2}+O\left(\left(\tfrac{1}{K}+\hbar K^3\right)\|g\|_\hbar\|h\|_\hbar\right).
				\end{multline}
					where we denote
				\begin{equation}\label{eq:def rhoK}
					\rho_K(v_1-v_2,\tfrac{v_1+v_2}{2}):=\frac{K\chi'(K|v_1-v_2|)}{|v_1-v_2|^2}\int_{<v_1-v_2>^{\perp}}\frac{|k|^2|\V|^2(k)\uud{k}}{|\e_0(M,k,\tfrac{v_1+v_2}{2})|^2}.
				\end{equation}
			\end{prop}
			
			\begin{remark}
				In the sens of distribution, the limit of \eqref{eq:def rhoK} is 
				\begin{itemize}
					\item for $d=3$,
					\[\rho_K(v_1-v_2,\tfrac{v_1+v_2}{2})\rightharpoonup \int\frac{|2\pi k\V(k)|^2}{|\e_0(M,k,v_1)|^2}\uud{k}~\delta_{v_1-v_2},\]
					\item for $d\geq 4$,
					\[\rho_K(v_1-v_2,\tfrac{v_1+v_2}{2})\rightharpoonup 0.\]
				\end{itemize}
				It is not convergent for $d=2$.
			\end{remark}
			
			\begin{proof}
				For $h,g\in L^2(\tfrac{\ud v_1}{\left\<v_1\right\>})$, we have		
				\begin{align*}
					\frac{4}{\hbar^2}\int& {h_1}{\sqrt{M_2}}\left({g_2}{\sqrt{M_1}}-{g'_2}{\sqrt{M'_1}}\right)\frac{ \V(k)^2\delta_{ k\cdot(v_2-v_1-\hbar k)}}{|\e_\hbar(M,k,v_1)|^2}\ud{V}\\
					&=\frac{2}{\hbar^2}\int \left(h_1\sqrt{M_2}-h_1'\sqrt{M'_2}\right)\left(g_2\sqrt{M_1}-g'_2\sqrt{M'_1}\right)\frac{ \V(k)^2\delta_{ k\cdot(v_2-v_1-\hbar k)}}{|\e_\hbar(M,k,v_1)|^2}\ud{V}\\
					&=\int h_1g_2 \,\mathcal{K}_\hbar(v_1,v_2)\ud{v_1}\ud{v_2},
				\end{align*}
				where, denoting $v_1'':= v_1-\hbar k$ and $v_2'':= v_2+\hbar k$, we define
				\begin{multline*}
					\mathcal{K}_\hbar(v_1,v_2):=\int\frac{2|\V(k)|^2}{\hbar^2}\Bigg(\frac{\sqrt{M_1M_2}\delta_{ k\cdot(v_2-v_1-\hbar k)}}{|\e_\hbar(M,k,v_1)|^2}-\frac{\sqrt{M''_1M_2}\delta_{ k\cdot(v_2-v_1)}}{|\e_\hbar(M,k,v''_1)|^2}\\
					-\frac{\sqrt{M_1M''_2}\delta_{ k\cdot(v_2-v_1)}}{|\e_\hbar(M,k,v_1)|^2}+\frac{\sqrt{M_1M_2}\delta_{ k\cdot(v_2-v_1+\hbar k)}}{|\e_\hbar(M,k,v''_1)|^2}\Bigg) \uud{k},
				\end{multline*}
				
				We begin by providing a naive bound of $\mathcal{K}_\hbar $. If $k\cdot (v_1-v_2) = 0$,
				\begin{align*}
					|v_1-\hbar k|^2+|v_2|^2 &\geq \tfrac{|v_1-v_2|^2}{2}+\tfrac{|v_1+v_2|}{2}+\tfrac{1}{2}(v_1+v_2)\cdot(v_2-v_1)-\hbar k\cdot(v_1-v_2)+|\hbar k|^2\\
					&\geq \tfrac{|v_1-v_2|^2}{2}+\tfrac{|v_1+v_2|}{4}+\tfrac{1}{2}(v_1+v_2)\cdot(v_2-v_1)\\
					&\geq \left(\tfrac 12 - \tfrac{1}{2\alpha}\right)|v_1-v_2|^2+\left(\tfrac14-\tfrac{\alpha}{8}\right)|v_1+v_2|^2,
				\end{align*}
				for any $\alpha>0$. Optimizing $\alpha =\tfrac{1+\sqrt{5}}{2}$, we obtain that $|v_1-\hbar k|^2+|v_2|^2\geq \tfrac{|v_1|^2+|v_2|^2}{6}$.
				
				Using that $|{\V}(k)|\leq C(1+|k|)^{d+3}$ (Assumption \ref{ass: potentiel}), Proposition \ref{prop:borne sur epsilone mieux} and the previous inequality, one can prove the trivial bound
				\[|\mathcal{K}_\hbar(v_1,v_2)|\lesssim \frac{\sqrt[12]{M_1M_2}}{\hbar^2|v_1-v_2|}.\]
				
				In order to obtain a better bound,  we split $\mathcal{K}_\hbar(v_1,v_2)$ into two pieces :
				\begin{align}
					\mathcal{K}_\hbar &=: \mathcal{K}_\hbar^<+\mathcal{K}_\hbar^>,\\
					\mathcal{K}_\hbar^>(v_1,v_2) &:=\int_{|k|\geq \frac{|v_1-v_2|}{4\hbar}}\frac{2|\V(k)|^2}{\hbar^2}\Bigg(\frac{\sqrt{M_1M_2}\delta_{ k\cdot(v_2-v_1-\hbar k)}}{|\e_\hbar(M,k,v_1)|^2}-\frac{\sqrt{M''_1M_2}\delta_{ k\cdot(v_2-v_1)}}{|\e_\hbar(M,k,v''_1)|^2}\\
					&\hspace{4cm} -\frac{\sqrt{M_1M''_2}\delta_{ k\cdot(v_2-v_1)}}{|\e_\hbar(M,k,v_1)|^2}+\frac{\sqrt{M_1M_2}\delta_{ k\cdot(v_2-v_1+\hbar k)}}{|\e_\hbar(M,k,v''_1)|^2}\Bigg) \uud{k} \nonumber
				\end{align}
				
				\step{1} Treatment of $\mathcal{K}_\hbar ^>$.
				
				Using that $|\V(k)|^2\lesssim \<k\>^{-d-3-\eta} $ and Proposition \ref{prop:borne sur epsilone mieux},
				\begin{equation*}
					|\mathcal{K}_\hbar^>(v_1,v_2)|\lesssim\frac{\sqrt[12]{M_1M_2}}{\hbar^2|v_1-v_2|}\left\<\frac{v_1-v_2}\hbar\right\>^{-4-\eta} \lesssim\frac{\sqrt[12]{M_1M_2}}{|v_1-v_2|^3}\left\<\frac{v_1-v_2}\hbar\right\>^{-1}.
				\end{equation*}
				
				In addition, we have that $\mathcal{K}_\hbar^>(v_1,v_2)\to0$.
			
				\step{2} Bound of $\mathcal{K}_\hbar ^<$.
				
				For $k\in<v_1-v_2>^\perp$, with $|k|\leq \frac{|v_1-v_2|}{\hbar}$ we introduce 
				\begin{align*}
					{\delta} k:= \left(\Lambda(\tfrac{\hbar |k|}{|v_2-v_1|})-1\right)k + \frac{\hbar |k|^2}{|v_2-v_1|^2}(v_2-v_1),~\Lambda(x):=\ind_{|x|\leq 1}\sqrt{1-x^2}.
				\end{align*}
				The application
				\begin{equation*}\label{eq:chang de var}
					\left\{\begin{array}{c}
						\left\{k,~k\cdot(v_2-v_1)=0, |k|\leq \frac{|v_1-v_2|}{\hbar}\right\}\to\left\{\tilde{k},~\tilde{k}\cdot(v_2-v_1-\hbar \tilde{k})=0\right\},\\
						k\mapsto \tilde{k} := k+{\delta} k
					\end{array}\right.
				\end{equation*}
				is a diffeomorphism with Jacobian
				\begin{equation*}
					\delta_{\tilde{k}\cdot(v_2-v_1-\hbar \tilde{k})} \uud{\tilde{k}} = \Lambda(\tfrac{\hbar |k|}{|v_2-v_1|})^{d-3}\uud k,
				\end{equation*}
				Note that $|k| = |\tilde{k}|$ and $|\delta k|\lesssim \tfrac{\hbar |k|^2}{|v_1-v_2|}$.
				
				We want to bound 
				\begin{equation}\label{eq:quasi derive seconde}
					\frac{\sqrt{M_1M_2}|\V(k)|^2\Lambda\left(\frac{2\hbar|k|}{|v_1-v_2|}\right)^{d-3}}{|\e_\hbar(M,k+\delta k,v_1)|^2}-\frac{\sqrt{M''_1M_2}|\V(k)|^2}{|\e_\hbar(M,k,v''_1)|^2}
					-\frac{\sqrt{M_1M''_2}|\V(k)|^2}{|\e_\hbar(M,k,v_1)|^2}+\frac{\sqrt{M_1M_2}|\V(k)|^2\Lambda\left(\frac{2\hbar|k|}{|v_1-v_2|}\right)^{d-3}}{|\e_\hbar(M,k-\delta(-k),v''_1-\hbar\delta (-k))|^2}
				\end{equation}

				
				Using Proposition \ref{prop:borne sur epsilone mieux} and Assumption \ref{ass: potentiel}, one can bound the derivative of $\frac{1}{|\e_\hbar(M,k,v_1)|^2}$: for $r,q\in\mathbb{N}$ with $r\leq 2$,
				\[|\nabla_{\hat{k}}^r\nabla^q_{v_1}\tfrac{1}{|\e_\hbar(M,k,v_1)|^2}|\lesssim \<v_1\>^q.\]
				
				 We introduce $k_s:= |k|(\cos(s\theta_k)\hat{k}+\sin(s\theta_k)\widehat{v_2-v_1})$ with $\theta_k := 2\arcsin \tfrac{|\delta k|}{2|k|}$. As $|\tfrac{\ud}{\ud s}\hat k_s|\lesssim \tfrac{|\delta k|}{|k|}$, we can apply a Taylor formula and
				\begin{align*}
					\frac{1}{|\e_\hbar(M,k+\delta k,v_1)|^2}&-\frac{1}{|\e_\hbar(M,k,v''_1)|^2}
					-\frac{1}{|\e_\hbar(M,k,v_1)|^2}+\frac{1}{|\e_\hbar(M,k-\delta(-k),v''_1-\hbar\delta (-k))|^2}\\
					&=-\int_0^1\int_0^1\frac{\ud}{\ud s}\frac{\ud}{\ud s'}\left(\frac{1}{|\e_\hbar(M,k_{-s},v_1+s'\hbar k_{-s})|^2}+\frac{1}{|\e_\hbar(M,k_{s-s'},v_1)|^2} \right)\ud s\ud s'\\
					&=O\left(\tfrac{|\delta k|^2}{|k|^2}\<k\>^2+|\hbar k|\tfrac{|\delta k|}{|k|}\<k\>^2\<v_1\>\right)\\
					&=O\left(\tfrac{\hbar^2\<k\>^4\<v_1\>^2}{|v_1-v_2|^2}\right).
				\end{align*}
			
				Using that $k\cdot(v_2-v_1)=0$, 
				\begin{multline*}
					\frac{\left(\sqrt{M''_1}-\sqrt{M_1}\right)\sqrt{M_2}}{|\e_\hbar(M,k,v''_1)|^2}+\frac{\sqrt{M_1}\left(\sqrt{M''_2}-\sqrt{M_2}\right)}{|\e_\hbar(M,k,v_1)|^2} = \frac{\hbar k\cdot\frac{v_1-v_2}{2}\sqrt{M_1M_2}}{|\e_\hbar(M,k,v_1)|^2} +O\left(\hbar^2|k|^2\sqrt[12]{M_1M_2} \right)\\
					= O\left(\hbar^2|k|^2\sqrt[12]{M_1M_2} \right)
				\end{multline*}
				
				Finally, for $\eta>0$ defined in Assumption \ref{ass: potentiel},
				\[\eqref{eq:quasi derive seconde} \lesssim \hbar^2\frac{\sqrt[12]{M_1M_2}}{|v_2-v_1|^2\<k\>^{d-1+\eta}}. \]
				Integrating with respect to $k$, this gives the bound
				\begin{equation}
					\left|\mathcal{K}_\hbar^{<}(v_1,v_2)\right|\lesssim \int \frac{\sqrt[12]{M_1M_2}}{|v_2-v_1|^2\<k\>^{d-1+\eta}} \delta_{(v_1-v_2)\cdot k}\ud k\lesssim\frac{\sqrt[13]{M_1M_2}}{|v_1-v_2|^3}.
				\end{equation}
				
				\step{3} Convergence of $\mathcal{K}^<_\hbar$.
				
				 It sufficient to prove for almost all $(v_1,v_2,k)$ the convergence of \eqref{eq:quasi derive seconde} when $\hbar\to 0$. Then, the limit  of $K^<_\hbar(v_1,v_2)$ is provided by the dominated convergence theorem. 
				
				The simple convergence of \eqref{eq:quasi derive seconde} is direct from a Taylor expansion, and the computation are similar than the one for {\it a priori} estimations (we will not detail them).
				
				In order to identify the limit, we proceed by duality. Let $g,f$ two smooth, compactly supported functions.
				
				\begin{align*}
					&\int h_1g_2 \mathcal{K}_\hbar(v_1,v_2)\ud{v_1}\ud{v_2}\\
					&=\frac{2}{\hbar^2}\int \left(h_1\sqrt{M_2}-h_1'\sqrt{M'_2}\right)\left(g_2\sqrt{M_1}-g'_2\sqrt{M'_1}\right)\frac{ \V(k)^2\delta_{ k\cdot(v_2-v_1-\hbar k)}}{|\e_\hbar(M,k,v_1)|^2}\ud{V}\\
					&=\frac{2}{\hbar^2}\int \left(h_1''\sqrt{M_2}-h_1\sqrt{M'_2}\right)\left(g_2\sqrt{M_1''}-g'_2\sqrt{M_1}\right)\frac{ \V(k)^2\delta_{ k\cdot(v_2-v_1)}}{|\e_\hbar(M,k,v''_1)|^2}\ud{V}\\
					&=-2\int k\cdot\left(\nabla h_1+\frac{v_2}{2}h_1\right)\left(\nabla g_2+\frac{v_1}{2}g_2\right)\cdot k\frac{ \V(k)^2\delta_{ k\cdot(v_2-v_1)}}{|\e_0(M,k,v_1)|^2}\sqrt{M_1M_2}\ud{V}+O(\hbar)\\
					&=-\int h_1 g_2 \left[\left(\nabla_1-\frac{v_1}{2}\right)\otimes\left(\nabla_2-\frac{v_2}{2}\right):\int\frac{ k\otimes k \V(k)^2\delta_{ k\cdot(v_2-v_1)}}{|\e_0(M,k,v_1)|^2}\sqrt{M_1M_2}\uud{k}\right]\ud{v_2}\ud{v_1}+O(\hbar).
				\end{align*}
				
				This conclude the proof of \eqref{eq:valeur de K limite}.

				\step{4} Cutoff of the singularity at $v_1=v_2$, estimation of $\mathcal{K}_\hbar^K$.
				
				Using the same strategy than for $\mathcal{K}_\hbar(v_1,v_2)$,
				\begin{multline*}
					\frac{2}{\hbar^2}\int \left(h_1\sqrt{M_2}-h_1'\sqrt{M'_2}\right)\left(g_2\sqrt{M_1}-g'_2\sqrt{M'_1}\right)\chi_K(v_1-v_2)\frac{ \V(k)^2\delta_{ k\cdot(v_2-v_1-\hbar k)}}{|\e_\hbar(M,k,v_1)|^2}\ud{V}\\
					=\int h_1g_2\left(\chi_K(v_1-v_2)\mathcal{K}_\hbar(v_1,v_2)\ud{v_1}\ud{v_2} + \mathcal{K}^K_\hbar(v_1,v_2)\right)\ud{v_1}\ud{v_2} 
				\end{multline*}
				with
				\begin{multline*}
					\mathcal{K}^K_\hbar(v_1,v_2) =\int\frac{2|\V(k)|^2}{\hbar^2}\Bigg(\frac{\sqrt{M_1M_2}}{|\e_\hbar(M,k,v''_1)|^2}[\chi_K(v_1-v_2-2\hbar k)-\chi_K(v_1-v_2)]\delta_{ k\cdot(v_2-v_1+\hbar k)}\\
					 -\left(\frac{\sqrt{M''_1M_2}}{|\e_\hbar(M,k,v''_1)|^2}+\frac{\sqrt{M_1M''_2}}{|\e_\hbar(M,k,v_1)|^2}\right)[\chi_K(v_1-v_2-\hbar k)-\chi_K(v_1-v_2)]\delta_{ k\cdot(v_2-v_1)}\Bigg) \uud{k},
				\end{multline*}
				
				We recall that $\nabla^k\chi_K = O(K^{k})$ and
				\begin{gather*}
					\nabla\chi_K	(v)=K\chi'(K|v|)\frac{v}{|v|} \text{ and }
					\nabla^2 \chi_K(v) = K^2\chi''(K|v|) \frac{v^{\otimes 2}}{|v|^2}+K\chi'(K|v|) \left(\frac{\rm Id}{|v|}-\frac{v^{\otimes 2}}{|v|^3}\right).
				\end{gather*}

				\noindent{\bf $\bullet$} If $k\cdot(v_2-v_1+\hbar k)= 0$, we have $|v_2-v_1|=|v_2-v_1-2\hbar k|$, and then $\chi_K(v_1-v_2-2\hbar k)=\chi_K(v_1-v_2)$.
				
				\noindent{\bf $\bullet$} If $k\cdot(v_2-v_1)= 0$, we have
				\begin{align*}
					\chi_K(v_1-v_2-\hbar k)-\chi_K(v_1-v_2)&=\left(-\hbar k\cdot\nabla\chi_K(v_1-v_2)+\frac{\hbar^2k^{\otimes 2}}{2}:\nabla^2\chi_K(v_1-v_2)\right)+O\left((\hbar K|k|)^3\right)\\
					&=\frac{K\chi'(K|v_1-v_2|)\hbar^2|k|^2}{2|v_1-v_2|} +O\left((\hbar K|k|)^3\right)
				\end{align*}
				
				Integrating with respect to $k$, and using Proposition \ref{prop:borne sur epsilone mieux},
				\begin{align}
					\mathcal{K}_\hbar^K(v_1,v_2)=&-2\int \frac{2|\V(k)|^2|k|^2\sqrt{M_1M_2}}{|\e_\hbar(M,k,v_1)|^2}\frac{K\chi'(K|v_1-v_2|)}{2|v_1-v_2|}\delta_{ k\cdot(v_2-v_1)}\uud{k}\nonumber\\
					&+\int \left(\delta_{ k\cdot(v_2-v_1)}+\delta_{ k\cdot(v_2-v_1+\hbar k)}\right)O\left(\hbar K^3|\V(k)|^2|k|^3\right)\sqrt{M_1M_2}\uud{k}\nonumber\\
					=& -\rho_K\left(v_1-v_2,\frac{v_1+v_2}{2}\right)+O\left(\tfrac{\hbar K^3\sqrt{M_1M_2}}{|v_1-v_2|}\right)
				\end{align}
			
				We have used that, as $k\cdot(v_1-v_2) = 0$,
				\begin{equation*}
					|\e_0(M,k,v_1)|^2 = \left|1+\V(k)\int\frac{k\cdot\nabla M(v_*)}{k\cdot(v_1-v_*)-i0}\ud v_*\right|^2 = \left|\e_0\left(M,k,\tfrac{v_1+v_2}{2}\right)\right|^2.
				\end{equation*}
				
				In order to bound the remainder, we note that for $g,h\in L^2(\tfrac{\ud v}{\<v\>})$,
				\begin{align*}
					\int |g(v_1)|\tfrac{\hbar K^3\sqrt{M_1M_2}}{|v_1-v_2|}|h(v_2)|\ud v_1 \ud v_2&\lesssim \hbar K^3 \int |\tfrac{g(v_1)}{\<v_1\>^{1/2}}| \tfrac{e^{-\frac{|v_1-v_2|^2}{8}}}{|v_1-v_2|}|\tfrac{h(v_2)}{\<v_2\>^{1/2}}|\ud v_1 \ud v_2\\
					&\lesssim \hbar K^3 \|\tfrac{g(v)}{\<v\>^{1/2}}\|_{L^2}\|\tfrac{h(v)}{\<v\>^{1/2}}\|_{L^2}\|\tfrac{1}{|v|}\exp\left(-\tfrac{|v|^2}{8}\right)\|_{L^1},
				\end{align*}
				which is bounded as we are in dimension $d\geq2$.
				
				\step{5} Estimation near $k\cdot (v_2-v_1)=0$.
				
				As we live in dimension $d\geq 2$, $\{k| \,k\cdot(v_2-v_1-\hbar k)=0\}$ is of dimension strictly bigger than $1$ and
				\begin{multline*}
					\left|\frac{1}{\hbar^2}\int \left(g_{1}\sqrt{M_2}-g'_{1}\sqrt{M'_2}\right)\left(g_{n,2}\sqrt{M_1}-g'_{2}\sqrt{M'_2}\right)(1-\chi_K(v_1-v_2))\frac{\V(k)^2\delta_{k\cdot(v_2-v_1-\hbar k)}}{|\e_{\hbar}(M,v_1,k)|^2}\ud{V}\right|\\
					\begin{split}
					&\lesssim\frac{1}{\hbar^2}\int_{|v_1-v_2|\leq \frac{1}{K}} \left(g_{1}\sqrt{M_2}-g'_{1}\sqrt{M'_2}\right)^2\V(k)^2\delta_{k\cdot(v_2-v_1-\hbar k)}\ud{V}\\
					&\lesssim\frac{1}{\hbar^2K}\int \left(g_{1}e^{-\frac{(\hat{k}\cdot v_1)^2}{2}}-g'_{1}e^{-\frac{(\hat{k}\cdot v_1-\hbar k)^2}{2}}\right)^2\frac{\V(k)^2}{|k|}\uud{k}\ud{v_1} = O\left(\frac{\|g\|_\hbar^2}{K}\right).
					\end{split}
				\end{multline*}
				This conclude the proof.
			\end{proof}
			
			\begin{prop}\label{prop:traitement tilde(K)}
				We fix the dimension $d\geq 2$. There exists a kernel $\tilde{\mathcal{K}}_0(v_1)$ such that
				\begin{equation}
					\frac{2}{\hbar^{2}}\int (g_1-g_1')g_1\sqrt{M_2}\left(\sqrt{M_2}-\sqrt{M_2'}\right) \frac{ \V(k)^2\delta_{ k\cdot(v_2-v_1-\hbar k)}}{|\e_\hbar(M,k,v_1)|^2}\ud{V}=\int\tilde{\mathcal{K}}_0(v_1)g^2(v_1)\ud{v_1}+O\left(\hbar\|g\|_\hbar^2\right)
				\end{equation}
				where
				\begin{equation}
					\tilde{\mathcal{K}}_0(v_1):=\frac{1}{2}\nabla_1\cdot\left(v_1\int k\otimes k \frac{ \V(k)^2\delta_{ k\cdot(v_2-v_1)}}{|\e_0(M,k,v_1)|^2}\uud{k}M_2\ud{v_2}\right)=O\left(\frac{1}{\<v_1\>^3}\right).
				\end{equation}
			\end{prop}
			
			\begin{proof}
				\begin{multline}
					\label{eq:2.33}\frac{2}{\hbar^{2}}\int (g_1-g_1')g_1\sqrt{M_2}\left(\sqrt{M_2}-\sqrt{M_2'}\right) \frac{ \V(k)^2\delta_{ k\cdot(v_2-v_1-\hbar k)}}{|\e_\hbar(M,k,v_1)|^2}\ud{V}\\
					=\frac{1}{2\hbar^{2}}\int\left[ (g_1^2-g_1'^2)\left({M_2}-{M_2'}\right)+ (g_1-g_1')^2\left(\sqrt{{M_2}}- \sqrt{M_2'}\right)^2\right]\frac{ \V(k)^2\delta_{ k\cdot(v_2-v_1-\hbar k)}}{|\e_\hbar(M,k,v_1)|^2}\ud{V}
				\end{multline}
				
				Using the change of variable $v_1\mapsto v_1-\hbar k$,
				\begin{equation*}
					\eqref{eq:2.33}=\int g_1^2\tilde{\mathcal{K}}_\hbar(v_1)\ud{v_1}+\frac{2}{\hbar^2}\int(g_1-g_1')^2\left(\sqrt{{M_2}}- \sqrt{M_2'}\right)^2\frac{ \V(k)^2\delta_{ k\cdot(v_2-v_1-\hbar k)}}{|\e_\hbar(M,k,v_1)|^2}\ud{V},
				\end{equation*}
				where we define
				\begin{equation*}
					\tilde{\mathcal{K}}_\hbar(v_1):=\frac{1}{2\sqrt{2\pi}\hbar^2}\int \frac{\V(k)^2}{|k|}\left(\frac{e^{-\frac{(v_1\cdot\hat k+\hbar|k|)^2}{2}}-e^{-\frac{(v_1\cdot\hat k+\hbar|k|)^2}{2}}}{|\e_\hbar(M,k,v_1)|^2}-\frac{e^{-\frac{(v_1\cdot\hat k+\hbar|k|)^2}{2}}-e^{-\frac{(v_1\cdot\hat k+2\hbar|k|)^2}{2}}}{|\e_\hbar(M,k,v_1-\hbar k)|^2}\right)\uud{k}.
				\end{equation*}
				
				\step{1} We denote
				\begin{equation}
					\varpi(k,v_1\cdot\hat{k},\hbar):=\frac{1}{\hbar^2|k|^2}\left(\frac{e^{-\frac{(v_1\cdot\hat k+\hbar|k|)^2}{2}}-e^{-\frac{(v_1\cdot\hat k+\hbar|k|)^2}{2}}}{|\e_\hbar(M,k,v_1)|^2}-\frac{e^{-\frac{(v_1\cdot\hat k+\hbar|k|)^2}{2}}-e^{-\frac{(v_1\cdot\hat k+2\hbar|k|)^2}{2}}}{|\e_\hbar(M,k,v_1-\hbar k)|^2}\right).
				\end{equation}
				Using a Taylor formula, and introducing $w:=v_1\cdot\hat{k}$,
				\begin{align}
					\varpi(k,w,\hbar)=&\int_{[0,1]^2}\frac{(w+(s+s')\hbar |k|)^2-1}{|\e_\hbar(M,k,w \hat{k})|^2}e^{-\frac{((w+(s+s')\hbar |k|)^2}{2}}\ud{s}\ud{s'}\nonumber\\
					&+\int_{[0,1]^2}\frac{(w+s\hbar |k|)}{|\e_\hbar(M,k,(w+s'\hbar|k|)\hat{k})|^4} \hat{k}\cdot\nabla_v|\e_\hbar(M,k,(w+s'\hbar|k|)\hat{k})|^2 e^{-\frac{(w+s\hbar |k|)^2}{2}}\ud{s}\ud{s'}.\nonumber
				\end{align}
				
				One have $\forall (k,w,\hbar)\in\mathbb{R}^d\times\mathbb{R}\times(0,\infty)$,
				\begin{gather}
					\label{v}|\varpi(k,w,\hbar)|\lesssim \int_0^2 e^{-\frac{(w+s\hbar |k|)^2}{4}}\ud{s},\\
					\label{vi}\varpi(k,w,0)=\left(\frac{w^2-1}{|\e_0(M,k,w\hat{k})|^2}+\frac{w\hat{k}\cdot\nabla_v|\e_0(M,k,w\hat{k})|^2}{|\e_0(M,k,w\hat{k})|^4}  \right)e^{-\frac{w^2}{2}}		=\frac{\ud}{\ud w}\left(\frac{-w}{|\e_0(M,k,w\hat{k})|^2}e^{-\frac{w^2}{2}}\right),\\
					\label{vii}|\varpi(k,w,\hbar)-\varpi(k,w,0)|\lesssim \hbar (1+|k|)\int_0^2 e^{-\frac{(w+s\hbar |k|)^2}{4}}\ud{s}~{\rm for}~\hbar|k|\leq1.
				\end{gather}
				This last estimate can be deduce from \ref{prop:borne sur epsilone mieux} and using the same method than before.
				
				Using the change of variable $k\mapsto(\sigma,x,|k|)$ introduce in \eqref{eq:changement de variable k}, we have
				\begin{equation}
					\tilde{\mathcal{K}}_\hbar(v_1)=\frac{|\mathbb{S}^{d-2}|}{2(2\pi)^{d+\frac12}}\int \V^2(k) |k|^{d}\left(1-x^2\right)^{\frac{d-3}{2}}\varpi(k,|v_1|x,\hbar)\ud{x}\ud{|k|}
				\end{equation}
				We introduce
				\begin{equation}	
					\tilde{\mathcal{K}}_0(v_1):=\frac{|\mathbb{S}^{d-2}|}{2 (2\pi)^{d+\frac12}}\int \V^2(k) |k|^{d}\left(1-x^2\right)^{\frac{d-3}{2}}\varpi(k,|v_1|x,0)\ud{x}\ud{|k|}.
				\end{equation}
				
				We split the difference $\tilde{\mathcal{K}}_\hbar(v_1)-\tilde{\mathcal{K}}_0(v_1)$ into two part
				\begin{align}
					&\frac{|\mathbb{S}^{d-2}|}{2(2\pi)^{d+\frac12}}\int_{|k|>|\hbar|^{-1}} \V^2(k) |k|^{d}\left(1-x^2\right)^{\frac{d-3}{2}}\left(\varpi(k,|v_1|x,\hbar)-\varpi(k,|v_1|x,0)\right)\ud{x}\ud{|k|}\label{eq:tilde(K)1}\\
					&+\frac{|\mathbb{S}^{d-2}|}{2(2\pi)^{d+\frac12}}\int_{|k|<|\hbar|^{-1}} \V^2(k) |k|^{d}\left(1-x^2\right)^{\frac{d-3}{2}}\left(\varpi(k,|v_1|x,\hbar)-\varpi(k,|v_1|x,0)\right)\ud{x}\ud{|k|}\label{eq:tilde(K)2}
				\end{align}
				
				We get the following bound: using \eqref{v},
				\[\begin{split}
				\eqref{eq:tilde(K)1} &\lesssim \int_0^2\int_{\substack{|k|>|\hbar|^{-1}}} \V^2(k) |k|^{d}e^{-\frac{(x|v_1|+s\hbar |k|)^2}{4}}\ud{x}\ud{|k|}\ud{s}\\
				&\lesssim\int_{\substack{|k|>|\hbar|^{-1}}} \V^2(k) |k|^{d}\ud{|k|}\int_{\mathbb{R}}e^{-\frac{(x|v_1|)^2}{2}}\ud{x}\lesssim \frac{\hbar}{<v_1>},
				\end{split}\]
				and using \eqref{vii}
				\[\eqref{eq:tilde(K)2} \lesssim \hbar\int_0^2\int \V^2(k) |k|^{d}(1+|k|)e^{-\frac{(x|v_1|+s\hbar |k|)^2}{4}}\ud{x}\ud{|k|}\ud{s}\lesssim \frac{\hbar}{<v_1>}.\]
				The second  line is obtained using \eqref{vii}. This gives the approximation results 
				\begin{equation}
					\mathcal{K}_\hbar(v_1) = \mathcal{K}_0(v_1) +O(\tfrac{\hbar}{\<v_1\>}).
				\end{equation}
				
				In order to bound $\tilde{\mathcal{K}}_0(v_1)$, we split it into two parts: 
				\begin{align}
					\tilde{\mathcal{K}}_0(v_1)=&\frac{|\mathbb{S}^{d-2}|}{2 (2\pi)^{d+\frac12}}\int_{|k|<|\hbar|^{-1}} \left(\left(1-x^2\right)^{\frac{d-3}{2}}-1\right) \V^2(k) |k|^{d}\varpi(k,|v_1|x,0)\ud{x}\ud{|k|}\label{eq:tilde(K)3}\\
					&+\frac{|\mathbb{S}^{d-2}|}{2 (2\pi)^{d+\frac12}}\int_{|k|<|\hbar|^{-1}} \V^2(k) |k|^{d}\varpi(k,|v_1|x,0)\ud{x}\ud{|k|}\label{eq:tilde(K)4}.
				\end{align}
				Using that for $|x|\leq \tfrac{1}{2}$, $(1-(1-x^2)^{(d-3)/2})\lesssim x^2$, and Inequality \eqref{v},
				\begin{multline*}
					\eqref{eq:tilde(K)3}\lesssim\int_{|x|\leq \frac{1}{2}} |\V(k)|^2 |k|^{d}x^2e^{-\frac{(x|v_1|)^2}{4}}\ud{x}\ud{|k|}+\int_{|x|\leq \frac{1}{2}} |\V(k)|^2 |k|^{d}\left|\left(1-x^2\right)^{\frac{d-3}{2}}-1\right|e^{-\frac{x^2|v_1|^2}{8}}\ud{x}\ud{|k|}\\
					\lesssim \int x^2e^{-\frac{x^2|v_1|^2}{8}}\ud{x}\lesssim\frac{1}{<v_1>^3}.
				\end{multline*}
				Using the equality \eqref{vi},
				\begin{equation*}
					\eqref{eq:tilde(K)4}=\frac{2}{\sqrt{2\pi}}\int_{|k|<1/\hbar}|\V(k)|^2|k|^d\int_{-1}^1\frac{\ud}{\ud w}\left(\frac{-w}{|\e_0(M,k,w\hat{k})|^2}e^{-\frac{w^2}{2}}\right)\Bigg\vert_{w=|v_1|x}\ud x\ud| k| \ud\sigma = O\left(\frac{e^{-\frac{|v_1|^2}{4}}}{\<v_1\>}\right).
				\end{equation*}
				We deduce that $\tilde{\mathcal{K}}_0(v_1)=O(\tfrac{1}{\<v_1\>^3})$.
				
				\step{2} We treat now the remaining part
				\begin{align*}\label{eq:remaining part}
					&\int(g_1-g_1')^2\left(\sqrt{{M_2}}- \sqrt{M_2'}\right)^2\frac{ \V(k)^2\delta_{ k\cdot(v_2-v_1-\hbar k)}}{\hbar^2|\e_\hbar(M,k,v_1)|^2}\ud{V}\\
					\lesssim&\int_0^1 \!\!\int(g_1-g_1')^2e^{-\frac{(v_1\cdot \hat{k}+s\hbar|k|)^2}{8}} |k|\V(k)^2\uud{k}\ud{v_1}\ud{s}\\
					\lesssim&\int_0^1\int_{|k|\geq\frac{1}{\sqrt{\hbar}}}\int g_1^2e^{-\frac{(v_1\cdot \hat{k}+s\hbar|k|)^2}{8}} |k|\V(k)^2\ud{v_1}\uud{k}\ud{s}+
					\int_0^1\int_{|k|\leq\frac{1}{\sqrt{\hbar}}}\int(g_1-g_1')^2e^{-\frac{(v_1\cdot \hat{k}+s\hbar|k|)^2}{8}} \frac{\V(k)^2}{\hbar |k|}\uud{k}\ud{v_1}\\
					\lesssim &\hbar\|g\|^2_\hbar.
				\end{align*}
				
				\step{3} We need to identify the limit. As in Step 3 of the proof of Proposition \ref{prop: traitement K}, we use a duality method. For $g$ a smooth test function with compact support,
				\begin{align*}
					\int g_1^2\tilde{K}(v_1)\ud{v_1}=&\frac{1}{2\hbar^{2}}\int (g_1^2-g_1'^2)\left({M_2}-{M_2'}\right)\frac{ \V(k)^2\delta_{ k\cdot(v_2-v_1-\hbar k)}}{|\e_\hbar(M,k,v_1)|^2}\ud{V}\\
					=&\frac{1}{2\hbar^{2}}\int (g''^2_1-g_1^2)\left({M_2}-{M_2'}\right)\frac{ \V(k)^2\delta_{ k\cdot(v_2-v_1)}}{|\e_\hbar(M,k,v_1)|^2}\ud{V}\\
					=&\frac{1}{2}\int k\cdot\nabla g^2_1 k\cdot v_2 M_2\frac{ \V(k)^2\delta_{ k\cdot(v_2-v_1)}}{|\e_0(M,k,v_1)|^2}\ud{V}+O(\hbar)\\
					=&-\frac{1}{2}\int g_1^2 \div_1\left(\left(\int k\otimes k \frac{ \V(k)^2\delta_{ k\cdot(v_2-v_1)}}{|\e_0(M,k,v_1)|^2}\uud{k}M_2\ud{v_2}\right)\cdot v_1\right)\ud{v_1}+O(\hbar)
				\end{align*}
				
				This conclude the proof.
			\end{proof} 
			
		\subsection{Dissipation estimates}
			We denote $w_0 = \sqrt{M}$, $(w_i)_{i\in[1,d]}=({v\cdot\gr{e}_i}\sqrt{M})$ (where $(\gr{e}_i)_i$ is an orthogonal bases of $\mathbb{R}^d$) and $w_{d+1}:=\frac{|v|^2-d}{\sqrt{2d}}\sqrt{M}$. We introduce $\pi_0$ the $L^2$-orthonormal on $<w_0,\cdots,w_{d+1}>$,
			\[\pi_0[g] := \sum_{i= 0}^{d+1}\left(\int w_i g\right) w_i.\]
			
			\begin{prop}\label{prop:dissipation estimates}
				We fix the dimension $d\geq 2$.
				
				There exists a constant $C>0$ independant of $\hbar$  and an $\hbar_0>0$ such that for any $\hbar<\hbar_0$, $g\in\mathcal{H}_\hbar$, we have the lower bound
				\begin{equation}
					C\int g\mathcal{L}_\hbar g \geq \|(1-\pi_0)g\|_\hbar^2.
				\end{equation}
			\end{prop}
			\begin{proof}
				We proceed by contradiction. Suppose that there exists a decreasing sequence $\hbar_n\to 0$ and a sequence $(g_n)_n$ such that $\pi_0 g_n=0$,  $\|g_n\|_{\hbar_n}=1$ and $\int g_n\mathcal{L}_{\hbar_n} g_n\to 0$.
				
				Using Proposition \ref{prop: decroissance des normes}, we deduce that up to an extraction, the sequence $(g_n)$ converge weakly to $g_\infty$ in all the  $\mathcal{H}_\hbar$, and there exists a constant $C>0$ such that
				\[\forall \hbar \in (0,\hbar_0),~\|g_\infty\|^2_\hbar <C.\]
				
				\begin{lemma}
					The sequence $g_n$ converges strongly in $L^2_{\rm loc}$ and 
					\begin{equation}
						\|g_\infty\|_0^2\leq 1.
					\end{equation}
				\end{lemma}
				\begin{proof}
					Fix $K>0$. We denote $B(K)$ the ball of radius $K$ and of center $0$.
					For any $\hbar\in(0,1)$, $k\in\overline{ B(1)}\setminus B(\tfrac23)$, and any test function $g$, one have 
					\begin{align*}
						\int_{B(K)}\frac{(g(v_1+\hbar k)-g(v_1))^2}{\hbar^2}\ud{v_1}
						\lesssim& \int_{\substack{k'\in \frac{k}{2}+B(\frac14)\\ v_1\in B(K)}}\frac{(g(v_1+\hbar {k'})-g(v_1))^2+(g(v_1+\hbar {k'})-g(v_1-\hbar k))^2}{\hbar^2}\ud{v_1}\uud{k'}\\
						\lesssim&\int_{\substack{k'\in B(\frac{3}{4})\setminus B(\frac1{12})\\ v_1\in B(K)}}\frac{(g(v_1+\hbar {k'})-g(v_1))^2}{\hbar^2}\ud{v_1}\uud{k'}\\
						\lesssim & e^{K^2}\frac{1}{\hbar^2}\int\left({g(v_1)}-{g(v_1+\hbar k)}\right)^2\frac{\V^2(k)}{|k|}e^{-\frac{(v_1\cdot \hat{{k}})^2}{2}}\uud{k'} \ud{v_1}\leq C_K
					\end{align*}
					for some constant $C_K$ depending on $K$. For any $k\in B(\tfrac{1}{3})$, $k = \hat{k} - (1-|k|)\hat{k}$. As $(1-|k|)\in[\tfrac{2}{3},1]$,
					\begin{align*}
						&\int_{B(K)}\frac{(g(v_1+\hbar k)-g(v_1))^2}{\hbar^2}\ud{v_1}\\
						\lesssim& \int_{ B(K)}\frac{(g(v_1+\hbar \hat{k})-g(v_1))^2+(g(v_1+\hbar \hat{k})-g(v_1-\hbar(1-|k|)\hat{k}))^2}{\hbar^2}\ud{v_1}\uud{k'}\\
						\lesssim&\int_{B(2K)}\frac{(g(v_1+\hbar \hat{k})-g(v_1))^2+(g(v_1+\hbar(1-|k|) \hat{k})-g(v_1))^2}{\hbar^2}\ud{v_1}\uud{k'}\lesssim 2C_K
					\end{align*}

					We deduce the uniform continuity of the sequence $(g_n)_n$: for any $\hbar>0$, take some $\hbar'\in(0,\hbar)$.
					\begin{equation*}
						\lim_{\substack{\hbar'\to0\\\hbar'\leq \hbar}}\sup_{\substack{|k| \leq \frac{\hbar'}3\\n\geq0}} \|g_n(k+\cdot)-g_n(\cdot)\|_{L^2(B(K))}\leq C_K\hbar + \lim_{\hbar'\to0} \sup_{\substack{|k| \leq \frac{\hbar'}3\\n,\,\hbar_n\geq \hbar}} \|g_n(k+\cdot)-g_n(\cdot)\|_{L^2(B(K))}\leq C_K\hbar.
					\end{equation*}
					One can apply Rellich-Kondrakov theorem. Hence the sequence $(g_n)_n$ converges strongly in $L^2(B(K))$.
					
					In addition, as $\|g_\infty\|_\hbar$ is uniformly bounded, we deduce that for any $k\in\mathbb{R}^d$,
					\[\|g_\infty(k+\cdot)-g_\infty\|_{L^2(B(K))}\leq |k|C_K,\]
					which is a criteria that proves $g_\infty\in H^1_{\rm loc}$, and $\frac{\left({g_\infty(\cdot+\hbar k)}-{g_\infty(\cdot)}\right)}{\hbar} $ converges weakly to $k\cdot\nabla g_\infty$ in $L^2(B(K))$.
					
					For any $K>0$,				
					\begin{align*}
						&\frac{2}{\hbar^{2}}\int_{|v_1|\leq K} \left[\left({g_{n,1}}-{g'_{n,1}}\right)^2M_2+g_{n,1}^2\left(\sqrt{M_2}-\sqrt{M'_2}\right)^2\right] \frac{ \V(k)^2\delta_{k\cdot(v_2-v_1-\hbar k)}}{|\e_\hbar(M,k,v_1)|^2}\uud{k}\ud{v_1}\\
						&=\frac{2}{\sqrt{2\pi}}\int_{|v_1|\leq K} \left[\left(\frac{{g_n(v_1+\hbar k)}-{g_n(v_1)}}{\hbar}\right)^2+\frac{g_{n,1}^2}{\hbar^2}\left(1-e^{-\frac{2\hbar v_1\cdot k +\hbar^2 |k|^2}{2}}\right)^2\right] \frac{ \V(k)^2e^{-\frac{(v_1\cdot\hat{k})^2}{2}}}{|\e_\hbar(M,k,v_1)|^2}\uud{k}\ud{v_1}\\
						&=\frac{2}{\sqrt{2\pi}}\int_{|v_1|\leq K} \left[\left(\frac{{g_n(v_1+\hbar k)}-{g_n(v_1)}}{\hbar}\right)^2+g_{n,1}^2\left(v_1\cdot k\right)^2\right]\frac{ \V(k)^2e^{-\frac{(v_1\cdot\hat{k})^2}{2}}}{|\e_0(M,k,v_1)|^2}\uud{k}\ud{v_1}+O\left(\hbar\|g_n\|^2_\hbar\right)
					\end{align*}
					
					One can take the infimum limit as $n\to0$ (as everything is convex, strongly continuous):
					\[\frac{2}{\sqrt{2\pi}}\int_{|v_1|\leq K} \left[\left(k\cdot\nabla g_\infty(v_1)\right)^2+g_{\infty}(v_1)^2\left(v_1\cdot k\right)^2\right]e^{-\frac{(v_1\cdot\hat{k})^2}{2}}\frac{ \V(k)^2}{|\e_0(M,k,v_1)|^2}\uud{k}\ud{v_1}\leq 1\]
					
					Finally we take the limit $K\to \infty$ by monotone convergence. This conclude the proof.
				\end{proof}
			
				In the following, we prove the convergence of $\int g_n\mathcal{L}_{\hbar_n} g_n$ to $0$. In dimension $3$, we will need a cut-off function $\chi_K$ introduce in Definition \ref{def: fonction de cutoff}. If the dimension is bigger than $4$, this cut-off function is not needed, and we consider $\chi_K \equiv 1$.
				
				We recall the decomposition of $\int g_n\mathcal{L}_{\hbar_n} g_n$: for any $K>0$,
				\begin{align*}
					\int g_n\mathcal{L}_{\hbar_n} g_n
					=&\frac{1}{{\hbar_n}^2}\int \left[\left(g_{n,1}-g'_{n,1}\right)^2M_2+g_{n,1}'^2\left(\sqrt{M_2}-\sqrt{M_2'}\right)^2\right]\frac{\V(k)^2\delta_{k\cdot(v_2-v_1-\hbar_n k)}}{|\e_{\hbar_n}(M,v_1,k)|^2}\ud{V}\\
					&+\frac{2}{\hbar_n^2}\int \left[\left(g_{n,1}-g'_{n,1}\right)g_{n,1}'\sqrt{M_2}\left(\sqrt{M_2}-\sqrt{M_2'}\right)\right]\frac{\V(k)^2\delta_{k\cdot(v_2-v_1-\hbar_n k)}}{|\e_{\hbar_n}(M,v_1,k)|^2}\ud{V}\\
					&+\frac{1}{\hbar_n^2}\int \left(g_{n,1}\sqrt{M_2}-g'_{n,1}\sqrt{M'_2}\right)\left(g_{n,2}\sqrt{M_1}-g'_{n,2}\sqrt{M'_1}\right)\frac{\V(k)^2\delta_{k\cdot(v_2-v_1-\hbar k)}}{|\e_\hbar(M,v_1,k)|^2}\ud{V}
				\end{align*}
				\begin{itemize}
					\item Using, the first line is equal to
					\[\frac{1}{{\hbar_n}^2}\int \left[\left(g_{n,1}-g'_{n,1}\right)^2M_2+g_{n,1}'^2\left(\sqrt{M_2}-\sqrt{M_2'}\right)^2\right]\frac{\V(k)^2\delta_{k\cdot(v_2-v_1-\hbar_n k)}}{|\e_{\hbar_n}(M,v_1,k)|^2}\ud V = \frac{\|g_n\|_{\hbar_n}}{2} = \frac12.\]
				
					\item For second line, using Proposition \ref{prop:traitement tilde(K)} and that
					\[\int_{|v_1|> K} g_n^2(v_1)\left|\tilde{\mathcal{K}}_{0}(v_1)\right|\ud{v_1} \lesssim\int_{|v_1|>K}\frac{g_n(v_1)^2}{\<v_1\>^3}\ud{v_1}=O(\tfrac{1}{K}),\]
					one have
					\begin{multline*}
						\frac{2}{\hbar_n^2}\int \left[\left(g_{n,1}-g'_{n,1}\right)g_{n,1}'\sqrt{M_2}\left(\sqrt{M_2}-\sqrt{M_2'}\right)\right]\frac{\V(k)^2\delta_{k\cdot(v_2-v_1-\hbar_n k)}}{|\e_{\hbar_n}(M,v_1,k)|^2}\ud{V}\\
						= \int_{|v_1|\leq K} g_n^2(v_1)\tilde{\mathcal{K}}_{0}(v_1)\ud{v_1}+O\left(\tfrac{1}{K}+\hbar_n\right).
					\end{multline*}
				
					\item For the third line, using Proposition \ref{prop: traitement K}, we perform the decomposition
					\begin{multline*}
						\frac{2}{\hbar_n^2}\int \left(g_{n,1}\sqrt{M_2}-g'_{n,1}\sqrt{M'_2}\right)\left(g_{n,2}\sqrt{M_1}-g'_{n,2}\sqrt{M'_1}\right)\frac{\V(k)^2\delta_{k\cdot(v_2-v_1-\hbar_n k)}}{|\e_{\hbar_n}(M,v_1,k)|^2}\ud{V}\\
						=\int_{\substack{|v_1|< K\\|v_2|< K}}\left(\chi_K(v_1-v_2)\mathcal{K}_{\hbar_n}(v_1,v_2)-\rho_K(v_1-v_2,\tfrac{v_1+v_2}{2})\right)g_n(v_1)g_n(v_2)\ud{v_1}\ud{v_2} +O\left(\tfrac{1}{K}+K^3\hbar_n\right).
					\end{multline*}
					We used that
					\begin{multline*}
						\int_{\substack{\{|v_1|\geq K\}}}\chi_K(v_1-v_2)\left(\left|\mathcal{K}_{\hbar_n}(v_1,v_2)\right|+\left|\rho_K(v_1-v_2,\tfrac{v_1+v_2}{2})\right|\right)\left|g_{n,1}\right|\left|g_{n,2}\right|\ud{v_1}\ud{v_2}\\
						\lesssim
						\int_{\substack{|v_1|\geq K\\|v_1-v_2|>\frac{1}{3K}}}\frac{\sqrt[5]{M_1M_2}}{|v_1-v_2|^3}\left|g_{n,1}\right|\left|g_{n,2}\right|\ud{v_1}\ud{v_2} = O(\tfrac{1}{K}).
					\end{multline*}
				\end{itemize}
				
				
				We can take now the limit $n\to \infty$, using that $g_n$ converges strongly to $g_\infty$ in $L^2_{\rm loc}$:
				\begin{align}
					0&=\frac{1}{2}+\frac{1}{2}\int_{\substack{|v_1|<K\\|v_2|< K}}\left(\chi_K(v_1-v_2)\mathcal{K}_{0}(v_1,v_2)-\rho_K(v_1-v_2,\tfrac{v_1+v_2}{2})\right) g_{\infty}(v_1)g_{\infty}(v_2)\ud{v_1}\ud{v_2}\nonumber \\
					&\hspace{8cm}+\int_{|v_1|\leq K} g_{\infty}(v_1)^2\tilde{\mathcal{K}}_{0}(v_1)\ud{v_1}+O(\tfrac{1}{K})\nonumber\\
					&\geq\frac{\|g_\infty\|_0^2}{2}+\frac{1}{2}\int\left(\chi_K(v_1-v_2)\mathcal{K}_{0}(v_1,v_2)-\rho_K(v_1-v_2,\tfrac{v_1+v_2}{2})\right) g_{\infty}(v_1)g_{\infty}(v_2)\ud{v_1}\ud{v_2} \label{eq:celle ou y a K}\\
					&\hspace{8.5cm}+\int g_{\infty}(v_1)^2\tilde{\mathcal{K}}_{0}(v_1)\ud{v_1}+O(\tfrac{1}{K}).\nonumber
				\end{align}
				
				One can perform an integration by part and
				\begin{multline*}
					\int g_{\infty}(v_1)^2\tilde{\mathcal{K}}_{0}(v_1)\ud{v_1} = \frac1{2} \int g_{\infty}(v_1)^2\nabla_1\cdot\left(v_1\int k\otimes k \frac{ \V(k)^2\delta_{ k\cdot(v_2-v_1)}}{|\e_0(M,k,v_1)|^2}\uud{k}M_2\ud{v_2}\right)\\
					=  -\int \nabla g_{\infty}(v_1)B(v_1,v_2)(g_\infty(v_1)v_1)M_2\ud{v_1} \ud{v_2.}
				\end{multline*}
				
				In the same way, using that $B(v_1,v_2)(v_1-v_2)=0$,
				\begin{align*}
					&\left(\nabla_1-\frac{v_1}{2}\right)\otimes\left(\nabla_2-\frac{v_2}{2}\right):\left(\chi_K(v_1-v_2) \sqrt{M_1M_2}B(v_1,v_2) \right)+\chi_K(v_1-v_2)\mathcal{K}_0(v_1-v_2)\\
					&=-\nabla^2\chi_K(v_1-v_2):B(v_1,v_2)\sqrt{M_1M_2}+\nabla\chi_K(v_1-v_2)\otimes\left(\nabla_2-\nabla_1-\tfrac{v_2-v_1}{2}\right):\left(\sqrt{M_1M_2}B(v_1,v_2)\right)\\
					&=\sqrt{M_1M_2}\left(-\nabla^2\chi_K(v_1-v_2):B(v_1,v_2)-\nabla\chi_K(v_1-v_2)\cdot\left(\left(\nabla_1-\nabla_2\right)\cdot B(v_1,v_2)\right)\right)
				\end{align*}
				
				Using that $\chi_K(v)= \chi(K|v|)$,
				\[\nabla^2\chi_K(v_1-v_2):B(v_1,v_2) = \frac{K\chi'(K|v_1-v_2|)}{|v_1-v_2|^2}\int_{<v_1-v_2>^\perp}\frac{|\V(k)|^2|k|^2\uud k}{|\e_0(M,\frac{v_1+v_2}{2},k)|^2} = \rho_K(v_1-v_2,\tfrac{v_1+v_2}{2}),\]
				where $\rho_K$ has been introduce in \eqref{eq:def rhoK}.
				
				\begin{align*}
					&\div_{v_1}\int  \frac{ \V(k)^2k\otimes k}{|\e_0(M,k,v_1)|^2}\delta_{ k\cdot(v_1-v_2)}\uud{k}\\
					=& \int k\cdot \nabla_{v_1} \left(\frac{ \V(k)^2}{|\e_0(M,k,v_1)|^2}\right)k\delta_{ k\cdot(v_1-v_2)}\uud{k}+\int \frac{ \V(k)^2|k|^2}{|\e_0(M,k,v_1)|^2}k\delta'_{ k\cdot(v_1-v_2)}\uud{k}\\
					=& \int k\cdot \nabla_{v_1} \left(\frac{ \V(k)^2}{|\e_0(M,k,v_1)|^2}\right)k\delta_{ k\cdot(v_1-v_2)}\uud{k}-\int\frac{v_1-v_2}{|v_1-v_2|^2}\cdot\nabla_k\left(\frac{ \V(k)^2|k|^2}{|\e_0(M,k,v_1)|^2}\right)k\delta_{ k\cdot(v_1-v_2)}\uud{k}\\
					&\hspace{8cm}-\frac{v_1-v_2}{|v_1-v_2|^2}\int\frac{ \V(k)^2|k|^2}{|\e_0(M,k,v_1)|^2}\delta_{ k\cdot(v_1-v_2)}\uud{k}.
				\end{align*}
				Using that $\nabla\chi_K(v)=K\chi'(K|v|)\tfrac{v}{|v|}$, one have
				\[-\nabla\chi_K(v_1-v_2)\cdot\left(\left(\nabla_1-\nabla_2\right)\cdot B(v_1,v_2)\right)=2\rho_K(v_1-v_2,\tfrac{v_1+v_2}{2}).\]
				Finally, 
				\begin{align*}
					\int&\left(\chi_K(v_1-v_2) \mathcal{K}_{0}(v_1,v_2)-\rho_K(v_1-v_2,\tfrac{v_1+v_2}{2})\right)g_{\infty}(v_1)g_{\infty}(v_2)\ud{v_1}\ud{v_2}\\
					=&-\int g_{\infty}(v_1)g_{\infty}(v_2)\left(\nabla_1-\frac{v_1}{2}\right)\otimes\left(\nabla_2-\frac{v_2}{2}\right):\left[\chi_K(v_1-v_2)B(v_1,v_2)\sqrt{M_1M_2}\right]\ud{v_1}\ud{v_2}\\
					=&-\int \left(\nabla_1+\frac{v_1}{2}\right)g_{\infty}(v_1)B(v_1,v_2)\left(\nabla_2+\frac{v_2}{2}\right)g_{\infty}(v_2)\sqrt{M_1M_2}\ud{v_1}\ud{v_2}\\
					&+\int\left(1-\chi_K(v_1-v_2)\right) \left(\nabla_1+\frac{v_1}{2}\right)g_{\infty}(v_1)B(v_1,v_2)\left(\nabla_2+\frac{v_2}{2}\right)g_{\infty}(v_2)\sqrt{M_1M_2}\ud{v_1}\ud{v_2}.
				\end{align*}
				
				One can bound the est using the Young inequality:
				\begin{multline*}
					\int\left|1-\chi_K(v_1-v_2)\right| \left|\left(\nabla_1+\frac{v_1}{2}\right)g_{\infty}(v_1)\right|\left|B(v_1,v_2)\right|\left|\left(\nabla_2+\frac{v_2}{2}\right)g_{\infty}(v_2)\sqrt{M_1M_2}\ud{v_1}\ud{v_2}\right|\\
					\lesssim \left\|\frac{1-\chi_K(v)}{|v|}\right\|_{L^1}\left\|\left(\nabla_2+\frac{v_2}{2}\right)g_{\infty}(v_2)\sqrt{M_2}\right\|^2_{L^2}=O(\tfrac{1}{K}).
				\end{multline*}
				
				Hence, taking the limit $K\to\infty$ in \eqref{eq:celle ou y a K},
				\begin{multline*}
					0\geq\frac{1}{2}\int \left[\left(\nabla_1+\frac{v_1}{2}\right)g_{\infty,1}\sqrt{M_2}-\left(\nabla_2+\frac{v_2}{2}\right)g_{\infty,2}\sqrt{M_1}\right]B(v_1,v_2)\\
					\left[\left(\nabla_1+\frac{v_1}{2}\right)g_{\infty,1}\sqrt{M_2}-\left(\nabla_2+\frac{v_2}{2}\right)g_{\infty,2}\sqrt{M_1}\right]\ud{v_1}\ud{v_2}
				\end{multline*}
				
				As this term is non-negative, it is zero. We deduce that (see \cite{DW}, Subset 2.2 of the proof of Proposition 2.5), $g_\infty\in<w_0,\cdots,w_{d+1}>$, which contradict the fact that $\pi_0(g_\infty)=0$. This conclude the proof.
			\end{proof}

			\begin{prop}\label{prop:control pi0}
				For any test function $g$
				\begin{eqnarray}
				\left\|\pi_0\left[\left(\nabla-\frac{v}{2}\right)g\right]\right\|_\hbar \lesssim \min(\|g\|_\hbar,\|g\|)
				\end{eqnarray}
			\end{prop}
			\begin{proof}
				We recall that 
				\[\pi_0[g] := \sum_{i= 0}^{d+1}\left(\int w_i g\right) w_i\]
				As $|w_i|= O(\sqrt[4]{M})$ and $\|w_i\|_\hbar\lesssim\|w_i\|_0<\infty$, we can conclude.
			\end{proof}

			\subsection{Non-linear estimation and Cauchy theory for pertubative solution of the quantum Lenard-Balescu equation}
			
			In the following, we denote $\mathcal{H}^r_\hbar$ the Hilbert space of norm
			\[\|g\|_{\hbar,r}^2:=\sum_{r'=0}^r\left\|\left(\nabla-\frac{v}{2}\right)^{r'}g\right\|^2_\hbar.\]
			
			The present section is dedicated to bound $\mathcal{L}_\hbar$ and $\mathcal{Q}_\hbar$.
			
			We denote  
			\begin{equation}\label{eq:def Lambda}
				\mathbf{H}_k(p)[v_1] := \int\frac{p(v_3)-p(v_3- k)}{ k \cdot(v_3-v_1- k)+i0}\ud{v_3}.
			\end{equation}
			
			\begin{prop}\label{prop:estimation terme quadratique}
				We fix the dimension $d\geq 2$. For any test functions $f$, $g$, $h$, and $p$, the  three following bounds hold:
				\begin{align}
				\label{eq:viii}\int \left|\DDelta\left({f}\sqrt{M_*}\right)\right|\left|\DDelta\left(gh_*\right)\right|\frac{\V^2(k)\delta_{k\cdot(v_2-v_1-\hbar k)}}{\hbar^2} &\ud{V}\lesssim\|f\|_\hbar\min(\|g\|_\hbar\|h\|_3,\|g\|_3\|h\|_\hbar),\\
				\label{eq:ix}\int \left|\DDelta\left({f}{\sqrt{M_*}}\right)\right|\left|\DDelta\left({gh_*}\right)\right|\left|\mathbf{H}_{\hbar k}(\sqrt{M}p)\right|\frac{\V^2(k)\delta_{k\cdot(v_2-v_1-\hbar k)}}{\hbar^2} &\ud{V}\lesssim\|f\|_\hbar\|p\|_\hbar\|g\|_2\|h\|_2,
				\end{align}
				\begin{equation}\label{eq:x}
					\left\|\V(k)\mathbf{H}_{\hbar k}(\sqrt{M}p)\right\|_{L^\infty_{k,v_1}}\lesssim \|p\|_{2}.
				\end{equation}
			\end{prop}
			\begin{proof}
				\step{1} We begin by \eqref{eq:viii}. Using the symmetry of the expression, it is sufficient to bound
				\begin{equation}\label{eq:premier terme quadratique}
					\frac{1}{\hbar^{2}}\int \left|f_1\sqrt{M_2}-f_1'\sqrt{M_2'}\right|\left|g_1h_2-g_1'h_2'\right| \V(k)^2\delta_{k\cdot(v_2-v_1-\hbar k)} \ud V.
				\end{equation}
				
				We can split it into four parts:
				\begin{align*}
					&\frac{1}{\hbar^{2}}\int \left|(f_1-f_1')\sqrt[4]{M_2}\right|\left|\sqrt[4]{M_2}g_1(h_2-h_2')\right| \V(k)^2\delta_{k\cdot(v_2-v_1-\hbar k)}\ud V\\
					& +\frac{1}{\hbar^{2}}\int \left|f_1'(\sqrt[4]{M_2'}-\sqrt[4]{M_2})\right|\left|(\sqrt[4]{M_2'}+\sqrt[4]{M_2})g_1(h_2-h_2')\right| \V(k)^2\delta_{k\cdot(v_2-v_1-\hbar k)} \ud V\\
					&+\frac{1}{\hbar^{2}}\int \left|(f_1-f_1')\sqrt[4]{M_2}\right|\left|\sqrt[4]{M_2}(g_1-g_1')h_2'\right| \V(k)^2\delta_{k\cdot(v_2-v_1-\hbar k)} \ud V\\
					&+\frac{1}{\hbar^{2}}\int \left|f_1'(\sqrt[4]{M_2}-\sqrt[4]{M_2'})\right|\left|(\sqrt[4]{M_2}+\sqrt[4]{M'_2})(g_1-g_1')h_2'\right| \V(k)^2\delta_{k\cdot(v_2-v_1-\hbar k)} \ud V.
				\end{align*}
				
				Applying the Cauchy--Schwartz inequality to each term of the sum, we obtain
				\[\eqref{eq:premier terme quadratique}\lesssim A_1(f)(A_2(g,h)+A_2(h,g))\]
				where we denote
				\begin{gather*}
					A_1(f) := \int e^{\frac{|v_2\cdot k|^2}{4|k|^2}}\left[(f_1-f_1')^2M_2+(f_1')^2(\sqrt{M_2'}-\sqrt{M_2})^2\right] \frac{\V(k)^2\delta_{k\cdot(v_2-v_1-\hbar k)}\ud{V}}{\hbar^2}.\\
					A_2(g,h):= \int e^{-\frac{|v_2\cdot k|^2}{4|k|^2}}g^2_1(h_2-h_2')^2 \frac{\V(k)^2\delta_{k\cdot(v_2-v_1-\hbar k)}\ud V}{\hbar^2}.
				\end{gather*}
				
				For a fix $k\in \mathbb{R}^d$, we introduce the notation $L^p_\para(L^q_\perp):=L^p(<\hat{k}>,L^q(<\hat{k}>^\perp))$. Using Proposition \ref{prop:form equivalent norm} and usual Sobolev injection theorem, these terms can be bound by
				\begin{gather*}
					A_1(f) \lesssim \|f\|_\hbar^2,\\
					A_2(g,h) \lesssim \|\nabla h\|^2_{L^\infty_\para(L^2_\perp)}\int g^2_1 e^{-\frac{|v_2\cdot k|^2}{4|k|^2}}|k|^2\V(k)^2\ud k\ud{v_1}\leq \|h\|_2^2\|g\|_\hbar^2.
				\end{gather*}
				This conclude the proof of \eqref{eq:viii}.
				
				\step{2}  We treat now \eqref{eq:ix}. We introduce the orthogonal decomposition $v = v^\para+v^\perp$, where $v^\para$ lives in $\mathbb{R}k$. Remark that $\mathbf{H}_{\hbar k}(\sqrt{M}p)(v_1) = \mathbf{H}_{\hbar k}(\sqrt{M}p)(v_1^\para)$.
				
				We bound first $\mathbf{H}_{\hbar k}$: using the $L^2\to L^2$ bound on the Hilbert transform,
				\[\begin{split}
				\int\frac{\V^2(k)}{|k|} |\mathbf{H}_{\hbar k}(\sqrt{M}p)(v_1^\para)|^2&\ud{v^\para_1}\lesssim\int\frac{\V^2(k)}{|k|} \left|\int_{<k^\perp>} (p_1\sqrt{M_1}-p_1'\sqrt{M'_1})\ud v_1^\perp\right|^2\ud{v^\para_1}\uud{k}\\
				&\lesssim\int\frac{\V^2(k)}{|k|} \left|\int_{<k^\perp>} (p_1e^{-\frac{|v_1^{|\hspace{-1pt}|}|^2}{4}}-p_1'e^{-\frac{|v_1^{|\hspace{-1pt}|}+\hbar k|^2}{4}})e^{-\frac{|v_1^\perp|^2}{2}}\ud v_1^\perp\right|^2\ud{v^\para_1}\ud k\\
				&\lesssim\int\frac{\V^2(k)}{|k|}  (p_1e^{-\frac{|v_1^{|\hspace{-1pt}|}|^2}{4}}-p_1'e^{-\frac{|v_1^{|\hspace{-1pt}|}+\hbar k|^2}{4}})e^{-\frac{|v_1^\perp|^2}{2}}\ud v_1 \uud k\\
				&\lesssim \hbar^2 \|p\|_\hbar^2
				\end{split}\]
				Hence, 
				\begin{align*}
					&\int \left[\mathbf{H}_{\hbar k}(\sqrt{M}p)(h_1g_2-h_1'g_2')\right]^2 \frac{\V(k)^2\delta_{k\cdot(v_2-v_1-\hbar k)}}{\hbar^2}\ud V\\
					&\lesssim \int\left[ \mathbf{H}_{\hbar k}\sqrt{M}p)(v_1^\para,k)
					(h(v_1^\para+v_1^\perp)g(v_1^\para+\hbar k+v_2^\perp))-(h(v_1^\para+\hbar k+v_1^\perp)g(v_1^\para+v_1^\perp))\right]^2 \frac{\V(k)^2}{|k|\hbar^2}\ud k\ud{v_1^\para}\ud{v^\perp_2}\ud{v^\perp_1}\\
					&\lesssim \|p\|_\hbar^2 \|g\|^2_{L^\infty_\para(L^2_\perp)}\|h\|^2_{L^\infty_\para(L^2_\perp)}\\
					&\lesssim\left(\|p\|_\hbar\|g\|_1\|h\|_1\right)^2.
				\end{align*}
				
				Applying the Cauchy--Schwartz inequality, 
				\begin{equation*}
					\eqref{eq:ix}\lesssim A_1(f)^{\frac{1}{2}}\left(\int \left[\mathbf{H}_{\hbar k}(\sqrt{M}p)(h_1g_2-h_1'g_2')\right]^2 \tfrac{\V(k)^2\delta_{k\cdot(v_2-v_1-\hbar k)}}{\hbar^2}\ud V\right)^{\frac{1}{2}}\lesssim\|f\|_\hbar\|p\|_\hbar\|g\|_1\|h\|_1.
				\end{equation*}
				
				\step{3} For any test function $p$, using the orthogonal decomposition $v_1 = v^\para_1+v_1^\perp$, $v_1^\para\in\mathbb{R}\hat{k}$, using trace inequality,
				\begin{align*}
					\left\|\V(k)\mathbf{H}_{\hbar k}(\sqrt{M}p)\right\|_{L^\infty_{k,v_1}}&\lesssim\left\|\V(k)\nabla_{v_1}\mathbf{H}_{\hbar k}(\sqrt{M}p)\right\|_{L^\infty_{k}(L^2(\ud v_1^{|\hspace{-1pt}|})}=\left\|\V(k)\mathbf{H}_{\hbar k}\left(\sqrt{M}\left(\nabla-\frac{v}{2}\right)p\right)\right\|_{L^\infty_{k}(L^2(\ud v_1^{|\hspace{-1pt}|})}\\
					&\lesssim\left\|\frac{\V(k)}{\hbar|k|}\int\left(\left(\nabla-\frac{v}{2}\right)p(v_*+\hbar k)-\left(\nabla-\frac{v}{2}\right)p(v_*)\right)\ud{v_*^\perp}\right\|_{L^\infty(\uud{k},L^2(\ud{v_1\cdot\hat{k}}))}\lesssim \left\|p\right\|_2.
				\end{align*}
			\end{proof}
			
			\begin{prop}\label{prop:controle dans les sobolev des termes quadratique}
				We fix the dimension $d\geq 2$.	Denoting $\tilde{r}:= \max(5,r)$, for $f,g,h$ three tests functions
				\begin{equation}
					\left|\left\<\left(-\nabla-\tfrac{v}{2}\right)^r\left(\nabla-\tfrac{v}{2}\right)^rf, \mathcal{Q}_\hbar(g,h)\right\>\right|
					\lesssim \left(1+\|h\|_{\tilde{r}-1}\right)^{r+1}\left(\|g\|_{\hbar,r}\|h\|_{\tilde{r}-1}+\|g\|_{\hbar,r}\|h\|_{\tilde{r}-1}\right)\|\left\|f\right\|_{\hbar,r},
				\end{equation}
				\begin{equation}
					\left\<\left(-\nabla-\tfrac{v}{2}\right)^r\left(\nabla-\tfrac{v}{2}\right)^rf, \mathcal{L}_\hbar g\right\> = \left\<\left(\nabla-\frac{v}{2}\right)^rf ,\mathcal{L}_\hbar \left(\nabla-\tfrac{v}{2}\right)^rg\right\>+O\left(\|g\|_{\hbar,{r}-1}\|\left\|f\right\|_{\hbar,r}\right),
				\end{equation}
				\begin{multline}
					\left|\left\<\left(-\nabla-\frac{v}{2}\right)^r\left(\nabla-\frac{v}{2}\right)^rf \left[\mathcal{Q}_\hbar(g,h)-\mathcal{Q}_\hbar(\tilde{g},\tilde{h})\right]\right\>\right|\\
					\lesssim\left\|f\right\|_{\hbar,r} \left(1+\|h\|_{\tilde{r}-1}+\|\tilde{h}\|_{\tilde{r}-1}\right)^{r+1}\Big(\|g-\tilde{g}\|_{\hbar,r}+\left(\|g\|_{\tilde{r}-1}+\|\tilde{g}\|_{\tilde{r}-1}\right)\|h-\tilde{h}\|_{\hbar,r}\\
					+\left(\|h\|_{\hbar,r}+\|\tilde{h}\|_{\hbar,r}\right)\|g-\tilde{g}\|_{\tilde{r}-1}+\left(\|g\|_{\hbar,r}+\|\tilde{g}\|_{\hbar,r}\right)\|h-\tilde{h}\|_{\tilde{r}-1}\Big).
				\end{multline}
			\end{prop}
			\begin{proof}
				As $(\nabla_1+\nabla_2) g(v_1') = \nabla g(v_1')$, we has
				\begin{gather*}
					\DDelta\left(\frac{\left(\nabla+\frac{v}{2}\right)g}{\sqrt{M}}\right)=\DDelta\left(\nabla\frac{g}{\sqrt{M}}\right) = (\nabla_1+\nabla_2)\DDelta\left(\frac{g}{\sqrt{M}}\right),\\
					(\nabla_1+\nabla_2)\DDelta\left({gh_*}{\sqrt{MM_*}}\right)
					= \DDelta\left(\left[\left(\nabla-\frac{v}{2}\right)gh_*+g\left(\nabla_*-\frac{v_*}{2}\right)h_*\right]\sqrt{MM_*}\right).
				\end{gather*}
				In addition,
				\[(\nabla_1+\nabla_2)\delta_{k\cdot (v_2-v_1-\hbar k)} = \delta_{k\cdot (v_2-v_1-\hbar k)}\left(- k+ k\right)=0.\]
				
				Hence, for $f,g,h,F = M+\sqrt{M}p$ measurable functions
				\begin{align*}
					&(-1)^r \int\DDelta\left(\frac{\left(\nabla+\frac{v}{2}\right)^r\left(\nabla-\frac{v}{2}\right)^rf}{\sqrt{M}}\right)\frac{\V(k)^2|v_1-v_2|^{d-2}}{|\e(F,k,v_1)|^2}\DDelta\left({gh_*}\right)\sqrt{M_1M_2}\delta_{k\cdot (v_2-v_1-\hbar k)}\ud V\\
					=&(-1)^r\int(\nabla_1+\nabla_2)^r\DDelta\left(\frac{\left(\nabla-\frac{v}{2}\right)^rf}{\sqrt{M}}\right)\frac{\V(k)^2|v_1-v_2|^{d-2}}{|\e(F,k,v_1)|^2}\DDelta\left({gh_*}\right)\sqrt{M_1M_2}\delta_{k\cdot (v_2-v_1-\hbar k)}\ud V\\
					=&\sum_{r_1+r_2+r_3=r}\binom{r}{r_1,r_2,r_3} \int\DDelta\left(\frac{\left(\nabla-\frac{v}{2}\right)^rf}{\sqrt{M}}\right)\V(k)^2|v_1-v_2|^{d-2}\frac{\partial^{r_1}|\e(F,k,v_1)|^{-2}}{(\partial v_1)^{r_1}}\\
					&\hspace{6cm}\times\DDelta\left(\left(\nabla-\frac{v}{2}\right)^{r_2}g\left(\nabla_*-\frac{v_*}{2}\right)^{r_3}h_*\right)\sqrt{M_1M_2}\delta_{k\cdot (v_2-v_1-\hbar k)}\ud V
				\end{align*}
				
				Using the Faà di Bruno formula, and $F = M+\sqrt{M} p$
				\begin{equation*}
					\frac{\partial^{r_1}|\e(F,k,v_1)|^{-2}}{(\partial v_1)^{r_1}}= \sum_{k,l = 0}^{r_1} \sum_{\substack{\gamma_1+\cdots+ \gamma_k\\+ \tilde{\gamma}_1+\cdots+\tilde{\gamma}_l=r_1\\ \gamma_i,\tilde{\gamma_i}\geq 1}} \frac{(-1)^{k+l}r_1!k!l!}{(\gamma_1+\cdots \gamma_k)!(\tilde{\gamma}_1+\cdots+\tilde{\gamma}_l)!} \frac{\prod_{j= 1}^k\frac{\partial^{\gamma_j}\e_\hbar(F,k,v_1)}{(\partial v_1)^{{\gamma}_j}}\prod_{j= 1}^l\frac{\partial^{\tilde{\gamma}_j}\overline{\e_\hbar(F,k,v_1)}}{(\partial v_1)^{\tilde{\gamma}_j}}}{\e_\hbar(F,k,v_1)^{k+1}\overline{\e_\hbar(F,k,v_1)}^{l+1}}
				\end{equation*}
				
				We recall that (for $\mathbf{H}_k$ defining in \eqref{eq:def Lambda})
				\begin{multline*}
					\frac{\partial^{\gamma_j}\e(F,k,v)}{(\partial v_1)^{\gamma_j}}-\frac{\partial^{\gamma_j}\e(\tilde{F},k,v)}{(\partial v_1)^{\gamma_j}} = {{\V}(k)}\int\frac{\nabla^{\gamma_j}\left[\sqrt{M}(p-\tilde{p})\right](v_*)-\nabla^{\gamma_j}\left[\sqrt{M}(p-\tilde{p})\right](v_*-\hbar k)}{\hbar k\cdot(v-v_*-\hbar k)+i0}\ud{v_*}\\
					={\V(k)\mathbf{H}_{\hbar k}\left(\sqrt{M}\left(\nabla-\frac{v}{2}\right)^{\gamma_j}(p-\tilde{p})\right)}
				\end{multline*}
				
				We deduce that
				\begin{multline*}
					\left|\frac{\partial^{r}|\e(F,k,v_1)|^{-2}}{(\partial v_1)^{\gamma_j}}-\frac{\partial^{r}|\e(\tilde{F},k,v_1)|^{-2}}{(\partial v_1)^{r}}\right|\\
					\lesssim_r |\V(k)|\left(1+\left\|p\right\|_{\max(4,r-1)}+\left\|\tilde{p}\right\|_{\max(4,r-1)}\right)^r\left(\sum_{s=0}^2\left|\mathbf{H}_{\hbar k}\left(\sqrt{M}\left(\nabla-\frac{v}{2}\right)^{r-r}(p-\tilde{p})\right)\right|+\left\|\tilde{p}-p\right\|_{r-1}\right).
				\end{multline*}
				
				Using that $r_1+r_2+r_3 = r$, and Proposition \ref{prop:estimation terme quadratique} denoting $\tilde{r} := \max\{5,r\}$
				\begin{multline*}
					\int\DDelta\left(\frac{\left(\nabla-\frac{v}{2}\right)^r\left(-\nabla-\frac{v}{2}\right)^rf}{\sqrt{M}}\right)\frac{\V(k)^2}{\hbar^2}\\
					\times\left[|\e(F,k,v_1)|^{-2}-|\e(\tilde{F},k,v_1)|^{-2}\right]\DDelta\left(\frac{gh_*}{\sqrt{MM_*}}\right)M_1M_2\delta_{k\cdot(v_2-v_1-\hbar k)}\ud V\\		
					\lesssim\left(1+\left\|p\right\|_{\tilde{r}-1}+\left\|\tilde{p}\right\|_{\tilde{r}-1}\right)^{r+1}\left\|f\right\|_{\hbar,r}\left(\left\|g\right\|_{\tilde{r}-1}
					\left\|h\right\|_{r,\hbar}+\left\|h\right\|_{\tilde{r}-1}
					\left\|g\right\|_{r,\hbar}+\left\|g\right\|_{\tilde{r}-1}
					\left\|h\right\|_{\tilde{r}-1}\|p-\tilde{p}\|_{r,\hbar}\right)
				\end{multline*}
				
				One can conclude by choosing suitable $p$ and $h$.
			\end{proof}
			
			We introduce the space $ \mathcal{C}^0_b(\mathbb{R}^+,\mathcal{H}_r)\cap L^2(\mathbb{R}^+,\mathcal{H}_\hbar^r)$, and define the sets
			\[\mathcal{E}_{r,\hbar}(\kappa,\delta)=\left\{h\in\mathcal{H}^r_\hbar\Big\vert \|h(t=0)\|_r\leq \kappa{\delta},~\|h\|_{L^\infty_t(\mathcal{H}^r)}+\|h\|_{L^2(\mathcal{H}_\hbar^r)}\leq \delta\right\}\]
			\[\mathcal{E}_{r,\hbar}(\kappa,\delta,h_0)=\left\{h\in\mathcal{E}_\hbar(\delta),~h(t=0)=h_0\right\}\]
			where $\kappa_r\in(0,1)$ is fix constant depending only on $r$ and $\mathcal{V}$, and $\|h_0\|_r\leq \kappa_r{\delta}$. The constant $\kappa_r$ will be fixed later.
			
			\begin{prop}\label{prop:condition picard}
				We fix the dimension $d\geq 2$. We define $\mathcal{F}_\hbar$ the application which associate to $h\in\mathcal{E}_{r,\hbar}$ the solution $g$ of the equation\footnote{The solution always exists as $\mathcal{L}_\hbar$ and $\mathcal{Q}_\hbar(\cdot,h(t))$ are both bounded in $L^2(\ud v)\to L^2(\ud v)$.}
				\begin{equation}
					\left\{\begin{split}
					\partial_t g(t) +\mathcal{L}_\hbar g(t) &= \mathcal{Q}_\hbar(g(t),h(t)),\\
					g(t=0)&=h(t=0).
					\end{split}\right.
				\end{equation}
				
				Then for $r>5$, there exist three constants $\kappa_r\in(0,1)$, $\delta_0>0$, $\hbar_0>0$ such that for $\kappa\in(0,\kappa_r)$, $\delta\in(0,\delta_0)$, and $\hbar\in(0,\hbar_0)$,
				\begin{itemize}
					\item the sets $\mathcal{E}_{r,\hbar}(\kappa,\delta)$ and $\mathcal{E}_{r,\hbar}(\kappa,\delta,h_0)$ are stable under the action of $\mathcal{F}_\hbar$,
					\item the application $\mathcal{F}_\hbar$ is continuous on $\mathcal{E}_{r,\hbar}(\kappa,\delta)$ and contractive on $\mathcal{E}_{r,\hbar}(\kappa,\delta,h_0)$.
				\end{itemize}
			\end{prop}
			\begin{proof}
				\step{1} The stability of $\mathcal{E}(\delta)$.
				
				Fix $h\in\mathcal{E}_{r,\hbar}(\kappa,\delta)$ and $g:=\mathcal{F}_\hbar(g)$. We have the following energy inequality.
				\begin{align*}
					\tfrac{\ud}{\ud t}\big\|\big(\nabla-\tfrac{v}{2}\big)^{r'}g(t)\big\|^2 &=2\big\< \big(-\nabla-\tfrac{v}{2}\big)^{r'}\big(\nabla-\tfrac{v}{2}\big)^{r'}g(t),-\mathcal{L}_\hbar g(t)+\mathcal{Q}_\hbar(g(t),h(t))\big\>\\
					\leq& -2\big\< \big(\nabla-\tfrac{v}{2}\big)^{r'}g(t)\mathcal{L}_\hbar\big(\nabla-\tfrac{v}{2}\big)^{r'}g(t)\big\>+C\|g(t)\|_{\hbar,r'-1}^2\\
					&+C\big(1+\|h(t)\|_{\tilde{r}-1}\big)^{r'+1}\big(\|g(t)\|_{\hbar,r'}\|h(t)\|_{r'-1}+\|g(t)\|_{r'-1}\|h(t)\|_{\hbar,r'}\big)\|\big\|g(t)\big\|_{\hbar,r}\\
					\leq& -\frac{2}{C}\big\|\big(\nabla-\tfrac{v}{2}\big)^{r'}g(t)\big\|^2_{\hbar}+C\|g(t)\|_{r'-1}^2\\
					&+C\big(1+\|h(t)\|_{\tilde{r}-1}\big)^{r'+1}\big(\|g(t)\|_{\hbar,r'}\|h(t)\|_{r'-1}+\|g(t)\|_{r'-1}\|h(t)\|_{\hbar,r'}\big)\|\big\|g(t)\big\|_{\hbar,r}
				\end{align*}
				where we use that $\pi_0(g(t))= 0$, and Propositions \ref{prop:dissipation estimates} and \ref{prop:control pi0}.
				
				We deduce that
				\begin{multline*}
					\tfrac{\ud}{\ud t}\big\|\big(\nabla-\tfrac{v}{2}\big)^{r'}g(t)\big\|^2+\big(\frac{1}{C}-C\big(1+\|h(t)\|_{\tilde{r}-1}\big)^{r+1}\big(\|h(t)\|_{\tilde{r}-1}+\|g(t)\|_{\tilde{r}-1}\big)\big)\big\|\big(\nabla-\tfrac{v}{2}\big)^{r'}g(t)\big\|_{\hbar}^2\\
\leq C\big(1+\|h(t)\|_{\tilde{r}-1}\big)^{r'+1}\|g(t)\|_{\tilde{r}-1}\big\|\big(\nabla-\tfrac{v}{2}\big)^{r'}h(t)\big\|^2_{\hbar}+C\|g(t)\|_{\hbar,r'-1}^2
				\end{multline*}
				Summing on $r'$,
				\begin{multline*}
					\tfrac{\ud}{\ud t}\sum_{r'=0}^r\tfrac{2^{r'}}{C^{r'}}\Big\|\Big(\nabla-\tfrac{v}{2}\Big)^{r'}g(t)\Big\|^2
					+\Big(\tfrac{1}{2C}-C\Big(1+\|h(t)\|_{\tilde{r}-1}\Big)^{r+1}\Big(\|h(t)\|_{\tilde{r}-1}+\|g(t)\|_{\tilde{r}-1}\Big)\Big)\sum_{r'=0}^r\tfrac{2^{r'}}{C^{r'}}\Big\|\Big(\nabla-\tfrac{v}{2}\Big)^{r'}g(t)\Big\|_{\hbar}^2\\
					\leq C\Big(1+\|h(t)\|_{\tilde{r}-1}\Big)^{r'+1}\|g(t)\|_{\tilde{r}-1}\sum_{r'=0}^r\tfrac{2^{r'}}{C^{r'}}\Big\|\Big(\nabla-\tfrac{v}{2}\Big)^{r'}h(t)\Big\|^2_{\hbar}
				\end{multline*}
				Let $T>0$ such that for all $t\in[0,T)$, $\|h(t)\|_r\leq \delta$. One can integrate the previous inequality: for $t\in[0,T)$
				\[\|g(t)\|^2_{r}+\left(\frac{1}{2C}-2C2^{r+1}\delta\right)\int_0^t\|g(s)\|_{\hbar,r}^2\ud{s}\leq \frac{C^r}{2^r}\left(\|g(0)\|_r^2+C2^{r+1}\delta \int_0^t\|g(s)\|_{\hbar,r}^2\ud{s}\right)\]
				
				Finally, for $\delta<\delta_0:=(4C^{2r+2})^{-1}$, $\kappa_r = C^{-(r+1)/2}$, we obtain
				\[\|g(t)\|^2_{r}+\frac{1}{4C}\int_0^t\|g(s)\|_{\hbar,r}^2\ud{s}\leq \frac{1}{4C} \delta^2\]
				
				By a boot-strap argument, we deduce that $g\in\mathcal{E}_\hbar(\delta)$.
				
				\step{2} Continuity in time of $g$:
				
				As $\|g(t)\|_{r,\hbar}$ is a $L^2(\mathbb{R}^+)$, it is bounded for almost all $t'\in\mathbb{R}^+$.
				\begin{align*}
					\frac{\ud}{\ud t}\bigg\|\left(\nabla-\frac{v}{2}\right)^{r'}&\left(g(t)-g(t')\right)\bigg\|^2  \leq -2\int \left(-\nabla-\frac{v}{2}\right)^{r'}\left(\nabla-\frac{v}{2}\right)^{r'}\left(g(t)-g(t')\right)\mathcal{L}_\hbar\left(g(t)-g(t')\right)\\
					&-2\int \left(-\nabla-\frac{v}{2}\right)^{r'}\left(\nabla-\frac{v}{2}\right)^{r'}\left(g(t)-g(t')\right)\mathcal{L}_\hbar g(t')\\
					&+C\left(1+\|h(t)\|_{\tilde{r}-1}\right)^{r'+1}\left(\|g(t)\|_{\hbar,r'}\|h(t)\|_{r'-1}+\|g(t)\|_{r'-1}\|h(t)\|_{\hbar,r'}\right)\|\left\|g(t)-g(t')\right\|_{\hbar,r}\\
					&\leq -\frac{1}{2C}\left\|\left(\nabla-\frac{v}{2}\right)^{r'}\left(g(t)-g(t')\right)\right\|^2_\hbar +4C^2\|g(t')\|^2_{\hbar,r} +C\|g(t)-g(t')\|^2_{\hbar,r'-1}\\
					&+\tilde{C} \delta^2\left(\|h(t)\|_{\hbar,r}^2+\|h(t)\|_{\hbar,r}^2\right)\\
				\end{align*}
				
				Summing on $r'\in[0,r]$, we get
				\[\frac{\ud}{\ud t}\left\|g(t)-g(t')\right\|_r^2  \lesssim \|g(t')\|^2_{\hbar,r}+\delta^2\left(\|h(t)\|_{\hbar,r}^2+\|h(t)\|_{\hbar,r}^2\right).\]
				We deduce that $g$ is continuous for almost all $t'\in \mathbb{R}^+$, and so on $\mathbb{R}^+$.
				
				\step{3} Continuity and contractility of $\mathcal{F}_\hbar$.
				\begin{align*}
					\frac{d}{dt}\bigg\|\left(\nabla-\frac{v}{2}\right)^{r'}&\left(g(t)-\tilde{g}(t)\right)\bigg\|^2 =-2\int \left(-\nabla-\frac{v}{2}\right)^{r'}\left(\nabla-\frac{v}{2}\right)^{r'}\left(g(t)-\tilde{g}(t)\right)\mathcal{L}_\hbar\left(g(t)-\tilde{g}(t)\right)\\
					+&2\int \left(-\nabla-\frac{v}{2}\right)^{r'}\left(\nabla-\frac{v}{2}\right)^{r'}\left(g(t)-\tilde{g}(t)\right)\left(\mathcal{Q}_\hbar(g(t),h(t))-\mathcal{Q}_\hbar(\tilde{g}(t),\tilde{h}(t))\right)\\
					\leq -&\frac{1}{C}\left\|\left(\nabla-\frac{v}{2}\right)^{r'}\left(g(t)-\tilde{g}(t)\right)\right\|_{\hbar}+C\left\|\left(g(t)-\tilde{g}(t)\right)\right\|^2_{\hbar,r'-1}\\
					+&C\left(1+\|h(t)\|_{\tilde{r}}+\|\tilde{h}(t)\|_{\tilde{r}}\right)^{r+1}\Big(\left(\|h\|_{\tilde{r}}+\|\tilde{h}\|_{\tilde{r}}\right)\|g(t)-\tilde{g}(t)\|_{\hbar,r}+\|g-\tilde{g}\|_{\tilde{r}}\left(\|h\|_{\hbar,r}+\|\tilde{h}\|_{\hbar,r}\right)\\
					+&\left(\|g\|_{\hbar,r}+\|\tilde{g}\|_{\hbar,r}\right)\|h(t)-\tilde{h}(t)\|_{\tilde{r}}+\left(\|g\|_{\tilde{r}}+\|\tilde{g}\|_{\tilde{r}}\right)\|h(t)-\tilde{h}(t)\|_{\hbar,r}\Big)\left\|g(t)-\tilde{g}(t)\right\|_{\hbar,r}.
				\end{align*}
				
				Summing on $r'\in[0,r]$ and using the polarization inequalities, and that $g,\tilde{g},h,\tilde{h}\in\mathcal{E}_\hbar(\delta)$
				\begin{multline*}
					\frac{d}{dt}\|g(t)-\tilde{g}(t)\|_{r}^2+\|g(t)-\tilde{g}(t)\|_{\hbar,r}^2\lesssim_r \left(\|h(t)\|^2_{\hbar,r}+\|\tilde{h}(t)\|^2_{\hbar,r}\right)\|g(t)-\tilde{g}(t)\|_{r}^2 + \delta^2\|h(t)-\tilde{h}(t)\|_{\hbar,r}^2\\
					+\left(\|g\|^2_{\hbar,r}+\|\tilde{g}\|^2_{\hbar,r}\right)\|h(t)-\tilde{h}(t)\|^2_{{r}}
				\end{multline*}
				
				We begin by treat the part $\|g(t)-\tilde{g}(t)\|_{r}^2$. Using the Gronwall's inequalities,
				\begin{align*}
					\|g(t)-\tilde{g}(t)\|_{r}^2 \leq& \|g(0)-\tilde{g}(0)\|_{r}^2e^{C\int_0^t\|h\|^2_{\hbar,r}+\|\tilde{h}\|^2_{\hbar,r}}\\
					+&\int_0^t2e^{C\int_s^t\|h\|^2_{\hbar,r}+\|\tilde{h}\|^2_{\hbar,r}}\left(\delta^2\|h(s)-\tilde{h}(s)\|^2_{{\hbar,r}}+\left(\|g(s)\|^2_{\hbar,r}+\|\tilde{g}(s)\|^2_{\hbar,r}\right)\|h(s)-\tilde{h}(s)\|^2_{{r}}\right)\ud{s}\\
					\leq& C\|h(0)-\tilde{h}(0)\|_{r}^2 +C\delta^2 \left[\|h-\tilde{h}\|^2_{L^\infty_t(\mathcal{H}^r)}+\|h-\tilde{h}\|^2_{L^2_t(\mathcal{H}^r_\hbar)}\right].
				\end{align*}
				
				One get now the estimation of $\|g(t)-\tilde{g}(t)\|_{\hbar,r}^2$:
				\begin{align*}
					\int_0^\infty \|g(t)-\tilde{g}(t)\|_{\hbar,r}^2\ud{t} &\lesssim\int_0^\infty  \left[\left(\|h(t)\|^2_{\hbar,r}+\|\tilde{h}(t)\|^2_{\hbar,r}\right)\left(\|h(0)-\tilde{h}(0)\|_{r}^2 +\delta^2 \|h-\tilde{h}\|_{\hbar,rL^2_t(\mathcal{H}^r_\hbar)\cap L^\infty_t(\mathcal{H}^r)}^2\right)\right.\\
					&\hspace{40pt}\left. + \delta^2\|h(t)-\tilde{h}(t)\|_{\hbar,r}^2 +\left(\|g\|^2_{\hbar,r}+\|\tilde{g}\|^2_{\hbar,r}\right)\|h(t)-\tilde{h}(t)\|^2_{{r}}\right]\ud{t}\\
					&\lesssim \delta^2 \left[\|h-\tilde{h}\|^2_{L^\infty_t(\mathcal{H}^r)}+\|h-\tilde{h}\|^2_{L^2_t(\mathcal{H}^r_\hbar)}\right].
				\end{align*}
				
				We conclude that 
				\[\|g-\tilde{g}\|_{L^2_t(\mathcal{H}^r_\hbar)\cap L^\infty_t(\mathcal{H}^r)}\lesssim \|h(0)-\tilde{h}(0)\|_r+\delta\|h-\tilde{h}\|_{L^2_t(\mathcal{H}^r_\hbar)\cap L^\infty_t(\mathcal{H}^r)}.\]
				This conclude the proof.
			\end{proof}
			
			We can now proof Theorem \ref{thm:Cauchy theory}: if $\|g_0\|<\kappa_r\delta$, one can perform a Picard's fix point argument in $\mathcal{E}_\hbar(\kappa,\delta,h_0)$. In addition, $M(v)+\sqrt{M}(v) g(t,v)$ is non-negative for any $(t,v)$ by a maximum principle on the equation \eqref{eq:equation de Lenard Balescu Quantique}.

		\subsection{Semi-classical (or grazing collision) limit for perturbative solution}\label{sec:semi classical limit}		
			We recall that for $g,h$ two test functions, and $H:=M+\sqrt{M}h$,
			
			\begin{gather*}
				\mathcal{L}_0 g(v_1):=\left(\nabla-\frac{v_1}{2}\right)\cdot\int B(M,v_1,v_2) \left[\sqrt{M_2}\left(\nabla_1+\frac{v_1}{2}\right)g_1-\sqrt{M_1}\left(\nabla_2+\frac{v_2}{2}\right)g_2\right]\sqrt{M_2} \ud{v_2}\\
				\begin{split}
				\mathcal{Q}_0(g,h)&(v_1):=-\left(\nabla-\frac{v_1}{2}\right)\cdot\int B(H,v_1,v_2) \left[h_2\nabla g_1+g_2\nabla h_1-h_1\nabla g_2-g_1\nabla h_2\right]\sqrt{M_2} \ud{v_2}\\
				+&\left(\nabla-\frac{v_1}{2}\right)\cdot\int \left(B(H,v_1,v_2)-B(M,v_1,v_2)\right) \left[\sqrt{M_2}\left(\nabla_1+\frac{v_1}{2}\right)g_1-\sqrt{M_1}\left(\nabla_2+\frac{v_2}{2}\right)g_2\right]\sqrt{M_2} \ud{v_2}\nonumber
				\end{split}\\
				B(v_1,v_2,H):=\int \frac{\V(k)^2\delta_{ k\cdot(v_2-v_1)}}{|\e_0(H,k,v_1)|^2}k\otimes k \uud{k}
			\end{gather*}
			
			\begin{prop}\label{prop:convergence operateur}
				We fix the dimension $d\geq 2$. For $f,g$ and $h$ three test functions,
				\begin{align}
				\left|\left\< f,\left(\mathcal{L}_\hbar-\mathcal{L}_0\right)\!g\right\> \right|&\lesssim \hbar\|f\|_\hbar\|g\|_5\\
				\left|\left\< f,\left(\mathcal{Q}_\hbar (g,h)-\mathcal{Q}_0(g,h)\right)\right\>\right| &\lesssim
				\hbar\|f\|_\hbar\|g\|_5 \|h\|_5(1+\|h\|_5)^4 \end{align}
			\end{prop}
			
			\begin{proof}
				It is sufficient treat for any $f,g,h,p$ satisfying 
				\[\|g\|_5+\|h\|_5 <\infty,~\|p\|_5\ll 1,\]
				the operator $\tilde{Q}_\hbar(p,h,g)$ such that
				\begin{equation}
					\tilde{Q}_\hbar(p,h,g)(v_1):=2\int\left(g_1h_2-g_1'h_2'\right)\frac{\sqrt{M_2}\V(k)^2\delta_{k\cdot(v_2-v_1-\hbar k)}}{\hbar^2|\e_\hbar(F_p,k,v_1)|^2}\uud{k}\ud{v_2} .
				\end{equation}
				
				We want to prove that there exists a limiting operator $\tilde{Q}$ such that for any test function $f$,
				\begin{equation}
					\int f_1\tilde{Q}_\hbar(p,h,g)(v_1)\ud{v_1} =\int f_1\tilde{Q}_0(p,h,g)_1 \ud{v_1}+O\left(\hbar \|f\|_\hbar(1+\|p\|_5)^4\|h\|_5\|g\|_5\right).
				\end{equation}
				
				\step{1} We begin by the identification of $\tilde{Q}_0$.
				\begin{multline*}
					\int f_1\tilde{Q}_\hbar(p,h,g)(v_1)\ud{v_1} =\int \left(f_1\sqrt{M_2}-f_1'\sqrt{M_2'}\right)\left(g_1h_2-g_1'h_2'\right)\frac{\V(k)^2\delta_{k\cdot(v_2-v_1-\hbar k)}}{\hbar^2|\e_\hbar(F_p,k,v_1)|^2}\ud{V}\\
					=\int \left(f_1\sqrt{M_2}-f_1'\sqrt{M_2'}\right)\left(h_2\nabla g_1-g_1\nabla h_2\right)\cdot\frac{k\V(k)^2\delta_{k\cdot(v_2-v_1-\hbar k)}}{\hbar|\e_0(F_p,-k,v_2)|^2}\ud{V}+\mathfrak{R}_{1,\hbar} +\int_0^1 \mathfrak{R}_{2,\hbar}^s\ud{s},
				\end{multline*}
				where, denoting $v_{1,s}:=v_1+s\hbar k$ and $v_{2,s}:=v_2-s\hbar k$, we define the remaining term as
				\begin{gather*}
					\mathfrak{R}_{1,\hbar}:=\int\left(f_1\sqrt{M_2}-f_1'\sqrt{M_2'}\right)\left(g_1h_2-g_1'h_2'\right)\frac{\V(k)^2\delta_{k\cdot(v_2-v_1-\hbar k)}}{\hbar^2}\left(\tfrac1{|\e_\hbar(F_p,k,v_1)|^2}-\tfrac1{|\e_0(F_p,k,v_1)|^2}\right)\ud{V}\\
					\mathfrak{R}_{2,\hbar}^s:=(1-s)\int \Big(f_1\sqrt{M_2}-f_1'\sqrt{M_2'}\Big)\Big(h_{2,s}\nabla^2 g_{1,s}-2\nabla g_{1,s}\nabla h_{2,s} +g_{1,s}\nabla^2h_{2,s}\Big):\frac{k^{\otimes 2}\V(k)^2\delta_{k\cdot(v_2-v_1-\hbar k)}}{|\e_0(F_p,-k,v_2)|^2}\ud{V}.
				\end{gather*}
				
				Using the change of variable $(v_1,v_2)\mapsto(v_1,v_2')$ and $(v_1,v_2)\mapsto(v_1',v_2)$
				\begin{align*}
					&\int \left(f_1\sqrt{M_2}-f_1'\sqrt{M_2'}\right)\left(h_2\nabla g_1-g_1\nabla h_2\right)\cdot\frac{k\V(k)^2\delta_{k\cdot(v_2-v_1-\hbar k)}}{\hbar|\e_0(F_p,-k,v_2)|^2}\ud{V}\\
					&=\int f_1\left( \frac{\sqrt{M''_2}\left(h''_2\nabla g_1-g_1\nabla h''_2\right)}{\hbar|\e_0(F_p,-k,v''_2)|^2}-\frac{\sqrt{M'_2}\left(h_2\nabla g''_1-g''_1\nabla h_2\right)}{\hbar|\e_0(F_p,-k,v_2)|^2}\right)\cdot k\delta_{k\cdot(v_2-v_1)}\V(k)^2\ud{V}\\
					&=\int f_1k\cdot\left(\nabla_1+\nabla_2-\tfrac{v_2}{2}\right) \frac{\sqrt{M_2}\left(h_2\nabla g_1-g_1\nabla h_2\right)}{|\e_0(F_p,-k,v_2)|^2}\cdot k\delta_{k\cdot(v_2-v_1)}\V(k)^2\ud{V} + \int_0^1 \mathfrak{R}^s_{3,\hbar},
				\end{align*}
				where
				\begin{equation*}
					\mathfrak{R}^s_{3,\hbar} := (1-s)\int f_1\frac{\ud^2}{\ud s^2}\bigg(\frac{\sqrt{M_{2,(1-2s)}}\left(h_{2,-s}\nabla g_{1,-s}-g_{1,-s}\nabla h_{2,-s}\right)}{|\e_0(F_p,-k,v_{2,-s})|^2}\bigg)\cdot k\delta_{k\cdot(v_2-v_1)}\V(k)^2\uud{k}\ud{v_1}\ud{v_2}.
				\end{equation*}
				
				Finally, using that $k\cdot v_1=k\cdot v_2$,
				\begin{align*}
					&\int f_1k\cdot\left(\nabla_1+\nabla_2-\tfrac{v_2}{2}\right) \frac{\sqrt{M_2}\left(h_2\nabla g_1-g_1\nabla h_2\right)}{|\e_0(F_p,-k,v_2)|^2}\cdot k\delta_{k\cdot(v_2-v_1)}\V(k)^2\ud{V} \\
					=&\int \tfrac{f_1}{\sqrt{M_1}}k\cdot\left(\nabla_1+\nabla_2\right) \bigg[\frac{\sqrt{M_2 M_1}\left(h_2\nabla g_1-g_1\nabla h_2\right)}{|\e_0(F_p,-k,v_2)|^2}\cdot k\delta_{k\cdot(v_2-v_1)}\bigg]\V(k)^2\ud{V}\\
					=&\int f_1 M_1^{-\frac12}\nabla_1\cdot\bigg(\int  \left(h_2\nabla g_1-g_1\nabla h_2\right)\frac{k\otimes k\V(k)^2\delta_{k\cdot(v_2-v_1)}}{|\e_0(F_p,-k,v_2)|^2}\sqrt{M_2 M_1}\uud{k}\ud{v_2}\bigg)\ud{v_1},
				\end{align*}
				where we use that $\left(\nabla_1+\nabla_2\right)\delta_{k\cdot(v_2-v_1)} = 0$. We deduce that 
				\[\tilde{Q}_0(p,h,g)_1 = M_1^{-1/2}\nabla_1\cdot\bigg(\int  \left(h_2\nabla g_1-g_1\nabla h_2\right)\frac{k\otimes k\V(k)^2\delta_{k\cdot(v_2-v_1)}}{|\e_0(F_p,-k,v_2)|^2}\sqrt{M_2 M_1}\uud{k}\ud{v_2}\bigg).\]
				
				\step{2} We estimate now the remaining parts $\mathfrak{R}_{1,\hbar}$, $\mathfrak{R}^s_{2,\hbar}$ and $\mathfrak{R}^s_{3,\hbar}$. Using the orthogonal decomposition $v=v^\para+v^\perp$ with $v^\para \in\mathbb{R}\hat{k}$ and Proposition \ref{prop:borne sur epsilone mieux} and \ref{prop:estimation terme quadratique}, we obtain that
				\begin{align}
					\int \Big|h_{2,s}\nabla^2 g_{1,s}-2\nabla g_{1,s}\nabla h_{2,s} +g_{1,s}&\nabla^2h_{2,s}\Big|^2|k|^{4}|\V(k)|^2\delta_{k\cdot(v_2-v_1-\hbar k)}\ud V\\
					&\qquad\lesssim \int |k|^3|\V(k)|^2\|g\|^2_2\|h\|^2_{L^\infty_\para(H^2_\perp)}\ud k \lesssim O(\|g\|_3^2\|h\|_3^2),\nonumber\\
					\int \left(f_1\sqrt{M_2}-f_1'\sqrt{M_2'}\right)^2&\frac{\V(k)|^2\delta_{k\cdot(v_2-v_1-\hbar k)}}{\hbar^2}\ud V\lesssim \|f\|_\hbar^2,\\
					\int \left(g_1h_2-g_1'h_2'\right)^2\bigg(\tfrac1{|\e_\hbar(F_p,k,v_1)|^2}-&\tfrac1{|\e_0(F_p,k,v_1)|^2}\bigg)^2\frac{\V(k)^2\delta_{k\cdot(v_2-v_1-\hbar k)}}{\hbar^2}\ud V\\
					\lesssim \hbar\int |k||\V(k)|^2&\|g\|^2_1\|h\|^2_{L^\infty_\para(H^1_\perp)}(1+\|p\|_2)\ud k \lesssim O(\hbar\|g\|_3^2\|h\|_3^2(1+\|p\|_2)).\nonumber
				\end{align}
				
				Using Cauchy--Schwartz inequality, we deduce that $\mathfrak{R}_{1,\hbar}+\mathfrak{R}^s_{2,\hbar}=O(\hbar\|f\|_\hbar\|g\|_3\|h\|_3(1+\|p\|_2))$. The third error term can be treated in the same way, and $\mathfrak{R}^s_{3,\hbar}=O(\hbar\|f\|_\hbar\|g\|_4\|h\|_4(1+\|p\|_4))$. This conclude the proof.
			\end{proof}
			
			\begin{proof}[Proof of Theorem \ref{thm:Cauchy theory}]
				First, as $\|g_0\|_r<\delta$ and $\hbar<\hbar_0$, one can applied Theorem \ref{thm:Cauchy theory}, and for any $\hbar<\hbar_0$, there exists $g_\hbar$ solution of \eqref{eq:QLB-bis}. In addition one have the following bound: 
				\[\sup_{t\geq 0}\|g_\hbar(t)\|_r^2+\int_0^\infty\|g_\hbar(t)\|_{\hbar,r}^2\ud{t}\lesssim \delta^2.\]
				
				We deduce that the family $(g_\hbar)_{\hbar<\hbar_0}$ is weakly compact in $\mathcal{C}^0_b(\mathbb{R}^+,\mathcal{H}^r)\cap L^2(\mathcal{H}_\hbar^r)$ for all $\hbar>0$. Hence, up to the extraction of a subsequence $\hbar'$, $(g_{\hbar'})$ converges weakly to some function $g_\infty\in\mathcal{C}^0_b(\mathbb{R}^+,\mathcal{H}^r)\cap L^2(\mathcal{H}_0^r)$.
				
				We need to identify the limit. In \cite{DW}, the authors show the existence of $g_\infty \in\mathcal{C}^0_b(\mathbb{R}^+,\mathcal{H}^r)\cap L^2(\mathcal{H}_0^r)$, the unique solution of \eqref{eq:CLB-bis}-equation, with initial data $g_\infty(t=0):=g_0$\footnote{In reality, their results is a little bit weaker, but one can easily adapt their proof.}.
				
				We perform again an energy estimation:
				\begin{align*}
					\frac{\ud}{\ud{t}}\|g_\hbar(t)-g_\infty(t)\|^2 =& -2\big\< g_\hbar(t)-g_\infty(t),\mathcal{L}_\hbar(g_\hbar(t)-g_\infty(t))\big\>-2\big\<g_\hbar(t)-g_\infty(t),(\mathcal{L}_0-\mathcal{L}_\hbar)g_\infty(t)\big\>\\
					&+2\big\<g_\hbar(t)-g_\infty(t),\mathcal{Q}_\hbar(g_\hbar(t),g_\hbar(t))-\mathcal{Q}_\hbar(g_\infty(t),g_\infty(t))\big\>\\
					&+2\big\<g_\hbar(t)-g_\infty(t),\mathcal{Q}_\hbar(g_\infty(t),g_\infty(t))-\mathcal{Q}_\infty(g_\infty(t),g_\infty(t))\big\>.
				\end{align*}
				
				Using Proposition \ref{prop:dissipation estimates}, \ref{prop:estimation terme quadratique} and \ref{prop:convergence operateur}, and using the same strategy than in the same way than in the proof of Proposition \ref{prop:condition picard}, 
				\[\frac{\ud}{\ud{t}}\|g_\hbar(t)-g_\infty(t)\|^2 +(1-\delta)\|g_\hbar(t)-g_\infty(t)\|^2_\hbar \lesssim \hbar\|g_\hbar(t)-g_\infty(t)\|_\hbar \|g_\infty\|_5\left(1+\|g_\infty\|_5\right)^4.\]
				Hence
				\[\frac{\ud}{\ud{t}}\|g_\hbar(t)-g_\infty(t)\|^2 +(1-2\delta)\|g_\hbar(t)-g_\infty(t)\|^2_\hbar \lesssim \frac{\hbar^2}{\delta} \|g_\infty\|^2_5\left(1+\|g_\infty\|_5\right)^8.\]
				
				As $r\geq 5$, we can applied the Gronwall inequality, and obtain
				\[\|g_\hbar(t)-g_\infty(t)\|\leq \hbar t \|g_\infty\|_{L^2_t(\mathcal{H}^5_\hbar)\cap L^\infty_t(\mathcal{H}^5)}(1+\|g_\infty\|_{L^2_t(\mathcal{H}^5_\hbar)\cap L^\infty_t(\mathcal{H}^5)})^4\]
				
				This conclude the proof.		
			\end{proof}
		
	\section{non-pertubative solutions of the Quantum Lenard--Balescu equation}\label{sec:cas d=2 ou 3}		
		\subsection{Preliminary}
		In order to simplify the presentation, we we set the following definition:
		\begin{definition}~~
			
			\begin{itemize}
				\item Fix $r\in\mathbb{N}$ and $\hbar \geq 0$. We define $\mathcal{B}_\hbar^r(\mathbb{C})$ (respectively $\mathcal{B}_\hbar^r(\mathbb{R})$) the set of function $\varpi:\mathbb{R}_k^d\times\mathbb{R}^d_v\to \mathbb{C}$ such that, 
				\begin{equation}\label{eq:symetry de eta}
					\begin{array}{c}
						k\cdot(v_2-v_1-\hbar k)=0\Rightarrow\varpi(k,v_1) = \overline{\varpi(-k,v_2)}\\
						\varpi(k,v_1)= \overline{\varpi(-k,v_1+\hbar k}).
					\end{array}
				\end{equation}
				and such that 
				\begin{equation}
					\|\varpi\|_{\mathcal{B}^r} =\max_{\substack{i+j+\ell\leq r,j\leq 2,\ell\leq 1}}\|\tfrac{1}{\<v_1\>^j}\nabla_{v_1}^i\nabla_{\hat{k}}^j\nabla^\ell_{|k|}\varpi\|_{L^{\infty}}
				\end{equation}
				is finite.
				
				\item For $g,h$ two functions $\mathbb{R}_v^d\to \mathbb{R}$, $\varpi\in\mathcal{B}_\hbar^0(\mathbb{R})$, we define
				\begin{equation}
					Q_\hbar(f,g,\varpi)(v_1) :=  \frac{1}{\hbar^2}\int \left(f_1'g_2'-f_1g_2\right)\V(k)^2\varpi(k,v_1)\delta_{k\cdot(v_2-v_1-\hbar k)}\ud k \ud v_2.
				\end{equation}
				
				\item For any function $p\in L^2(\<v\>^d \ud v)$, we recall the definition
				\begin{equation}
					\mathbf{H}_k(p)[v_1] := \int\frac{p(v_3)-p(v_3- k)}{ k \cdot(v_3-v_1- k)+i0}\ud{v_3}.
				\end{equation}
				
				\item We denote $L_\para^q(L_\perp^p):=L^q(\mathbb{R}\hat{k},L^p(<k>^\perp))$.
			\end{itemize}
		\end{definition}
	
		The Quantum Lenard-Balescu equation can be written as (we linearize around the initial data)
		\begin{gather}
			\partial_t \varphi_\hbar = Q_\hbar\left(\varphi_\hbar(t),\varphi_\hbar(t),\tfrac{1}{|\e_\hbar(\varphi_\hbar(t))|^2}\right),
			\varphi_\hbar(t=0)=\Phi_0,\\
			\e_{\hbar}(\varphi,k,v_1) = 1+\V(k) \gr{H}_{\hbar k}(\varphi)(v_1).
		\end{gather}
		One can easily check that if $\varphi$ is smooth enough, $\e_{\hbar}(\varphi)$ is in $\mathcal{B}^0(\mathbb{C})$ (as the conditions \eqref{eq:symetry de eta} are verified). 
		\medskip
		
		\paragraph{Control of the discrete derivative}
		As in the perturbative case, the collision operator allows to control an additional discrete derivative term, given by the following quantity: denoting $L_\para^q(L_\perp^p):=L^q(\mathbb{R}\hat{k},L^p(<k>^\perp))$
		\begin{definition}
			We introduce the energy term 
			\begin{equation}
				A_\hbar(f)=\frac{1}{\hbar^2}\int \V(k)^2(f_1'-f_1)^2 \Phi_0(v_2) \delta_{k\cdot(v_2-v_1-\hbar k)}\ud V
			\end{equation} 
		\end{definition}
		One can easy show that this is bound by a Sobolev norm:
		\begin{lemma}
			For any mesurable function $f:\mathbb{R}^d\to \mathbb{R}$,
			\begin{gather}
				A_\hbar(f) \lesssim \|f\|_{H^1}^2,~
				A_\hbar(f)\lesssim \frac{1}{\hbar^2}\|f\|_{L^2}.\label{eq:estim bete A1}
			\end{gather}
		\end{lemma}
		
		We begin by proving this result of equivalence of some norms:

		\begin{prop}\label{cor:equivalence des normes}
			Fix $\Phi_0$ such that there exists $v_0\in \mathbb{R}^d$ and $M>0$ such that $\Phi_0>\tfrac{1}{M}\ind_{M|v-v_0|<1}$.
			
			Then for any $f,g\in L^2(\mathbb{R})$, with $\|\<v\>^{\frac d2+4}g\|_{H^1}$ finite, $\<k\>^{\frac {d+3+\delta}2}\mathcal{U}(k)\in L^{\infty}(\mathbb{R}^+)$, $\delta>0$, $s\in[-1,1]$, we have
			\begin{gather}
				\begin{split}\frac{1}{\hbar^2}\bigg|\int \mathcal{U}(k)^2(f_1'-f_1)^2 &g_2 \delta_{k\cdot(v_2-v_1-s\hbar k)}\ud V \bigg|\\
				&\leq\frac{\|\<v\>^{\frac d2+4}g\|_{H^1}\|\<k\>^{\frac {d+3+\delta}2}\mathcal{U}\|_{L^\infty}}{C_{\Phi_0}}\left[\frac{1}{\hbar^2}\int \frac{(f_1'-f_1)^2}{\<k\>^{d+5}\<v_1\cdot \hat{k}\>^\frac92} \ud v_1 \uud k+\|f\|^2_{L^2}\right]\label{eq: equivalence dissipation1},\end{split}\\
				\frac{1}{\hbar^2}\int \frac{\mathcal{U}(k)^2(f_1'-f_1)^2}{|k|\<v_1\cdot \hat{k}+s\hbar|k|\>^\frac92} \ud v_1 \uud k\leq\frac{\|\<k\>^{\frac d2+4}\mathcal{U}\|_{L^\infty}}{C_{\Phi_0}}\left[A_\hbar(f) +\|f\|^2_{L^2}\right]\label{eq: equivalence dissipation2}
			\end{gather}
		\end{prop}
		\begin{proof}
			We begin by integrating with respect to $v_2$:
			\[\bigg|\int g_2\delta_{k\cdot(v_2-v_1-s\hbar k)} \ud v_2\bigg|\lesssim\left(\int \tfrac{\delta_{k\cdot(v_2-v_1-s\hbar k)}}{\<v_2\>^{d+8}}\ud v_2\right)^{\frac12}\|\<v\>^{\frac{d}{2}+4}g\|_{L^\infty_\para(L^2_\perp)}\lesssim \tfrac{\|\<v\>^{\frac d2+4}g\|_{H^1}}{\<v_1\cdot \hat{k}+s\hbar|k|\>^\frac92},\]
			where we use the trace inequality. Hence, \eqref{eq: equivalence dissipation1} is equivalent to \eqref{eq: equivalence dissipation2}.
			
			Then, using that for $|\hbar k|<1$, $\<v_1\cdot\hat{k}+s\hbar k\>\geq \<v_1\cdot\hat{k}\>$, we obtain that 
			\[\frac{1}{\hbar^2}\int \frac{\mathcal{U}(k)^2(f_1'-f_1)^2}{|k|\<v_1\cdot \hat{k}+s\hbar|k|\>^\frac92} \ud v_1 \ud k\lesssim \frac{1}{\hbar^2}\int_{\hbar |k|\leq 1} \frac{\mathcal{U}(k)^2(f_1'-f_1)^2}{|k|\<v_1\cdot \hat{k}+s\hbar|k|\>^\frac92} \ud v_1 \ud k+ \int_{\hbar |k|>1} \frac{\mathcal{U}(k)^2}{\hbar^2|k|} \uud k ~\|f\|^2_{L^2}. \]
			Using Proposition \ref{prop: equivalence dissipation} and the bound on $\mathcal{U}$, one can conclude.
		\end{proof}
		
		\paragraph{Basic integral bounds}
		The following lemma will be used to bound integrals that often appear in this section.
		\begin{prop}\label{lem:borne general integral}
			Fix $(s_j)_{j\leq 3}\in[-1,1]^3$, $f,g,h,p$ four measurable functions, denoting $(v_1^s,v_2^s):=(v_1+s\hbar k,v_2-s\hbar k)$, 
			\begin{align}
				\int f(v_{1,s_1})g(v_{1,s_2}) h(v_{2,s_3}) \<k\>^{\frac{d}{2}}\V(k)^2\delta_{k\cdot(v_2-v_1)}&\ud V\lesssim \|f\|_{L^2}\min(\|g\|_{L^2}\|\<v\>^{\frac d2}h\|_{H_1},\|g\|_{H^1}\|\<v\>^{\frac d2}h\|_{L^2}),\label{xvii}\\
				\int f(v_{1,s_1})g(v_{1,s_2}) h(v_{2,s_3}) \mathbf{H}_{\hbar k}(p)(v_1)|k|^{r}\V(k)^2&\delta_{k\cdot(v_2-v_1)}\ud V \label{xviii}\\
				\lesssim \|f\|_{L^2}&\|g\|_{H^1}\|\<v\>^{\frac d2+4}h\|_{H^1}(A_{\hbar}(\<v\>^{\frac d2+4}p)+\|\<v\>^{\frac d2+4}p\|_{L^2}),\nonumber\\
				\int (f(v_1)-f(v_1'))g(v_{1,s_2}) h(v_{2,s_3}) \<k\>^{\frac d2}\V(k)^2&\delta_{k\cdot(v_2-v_1)}\ud V\lesssim (A_\hbar(f)^{\frac12}+\|f\|_{L^2})\|g\|_{L^2}\|\<v\>^{\frac d2+4}h\|_{H^1}.\label{xix}
			\end{align}
		\end{prop}
		\begin{proof}
			\step{1} Proof of \eqref{xvii}.
			
			Using a Cauchy-Schwartz inequality,
			\begin{multline*}
				\int f(v_{1,s_1})g(v_{1,s_2}) h(v_{2,s_3}) \<k\>^{\frac{d}{2}}\V(k)^2\delta_{k\cdot(v_2-v_1)}\ud V\\
				\lesssim \left(\int \tfrac{f(v_{1,s_1})^2}{\<v_{2,s_3}\>^d\<k\>^{d+2\eta}}\delta_{k\cdot(v_2-v_1)}\ud V\right)^{\frac12}\left(\int \frac{g(v_1^{s_2})^2 (\<v_{2,s_3}\>^{d/2}h(v_{2,s_3}))^2}{\<k\>^{2d+6}} \delta_{k\cdot(v_2-v_1)}\ud V\right)^{\frac12}
			\end{multline*}
			Integrating first with respect to $v_2$ and then to $k$, 
			\begin{align*}
				\left(\int\tfrac{f(v_{1,s_1})^2}{\<v_{2,s_3}\>^d\<k\>^{d+2\eta}}\delta_{k\cdot(v_2-v_1)}\ud V\right)^{\frac12}\lesssim \left(\int\tfrac{f(v_{1,s_1})^2}{|k|\<k\>^{d+2\eta}}\uud k\ud v_1 \right)^{\frac12}\lesssim \|f\|_{L^2}.
			\end{align*}
			In the same way, using Assumption \ref{ass: potentiel bis}
			\begin{align*}
				\left(\int \frac{g(v_1^{s_2})^2 (\<v_{2,s_3}\>^dh(v_{2,s_3}))^2}{\<k\>^{2d+6}} \delta_{k\cdot(v_2-v_1)}\ud V\right)^{\frac12}\lesssim\left( \int\frac{g(v_1^{s_2})^2 }{\<k\>^{2d+6}} \|\<v\>^{\frac d2}h\|^2_{L^\infty_\para(L^2_\perp)}\ud v_1\uud k\right)^{\frac 12}\lesssim \|g\|_{L^2}\|\<v\>^{\frac{d}{2}}h\|_{H_1},
			\end{align*}
			where use the trace inequality on $h$. We can obtain the second estimation by integrating first with respect to $v_1$ and then with respect to $k,v_2$.
			
			\step{2} Proof of \eqref{xviii}.
			
			We perform the same decomposition, one only need to bound 
			\[\int \frac{g(v_1^{s_2})^2 (\<v_{2,s_3}\>^{d/2}h(v_{2,s_3}))^2}{\<k\>^{2d+6}}\gr{H}_{\hbar k}(p)(v_1) \delta_{k\cdot(v_2-v_1)}\ud V.\]
			Using the $L^2$ bound of the Hilbert transform and integrating with respect to $v_1^\perp$ and $v_2^\perp$, denoting $v_3:=v^\perp_3+v^\para$
			\begin{align*}
				\int& \frac{g(v_1^{s_2})^2 (\<v_{2,s_3}\>^{d/2}h(v_{2,s_3}))^2}{\<k\>^{2d+6}}\gr{H}_{\hbar k}(p)(v_1) \delta_{k\cdot(v_2-v_1)}\ud V\\
				&\quad\lesssim \int \frac{1}{|k|\<k\>^{2d+6}}\|g\|^2_{L^\infty_\para(L^2_\perp)}\|\<v\>^{\frac d2}h\|^2_{L^\infty_\para(L^2_\perp)}\left(\int \tfrac{p(v_3)-p(v_3+\hbar k)}\hbar \ud v_3^\perp\right)^2\ud v^\para \uud k\\
				&\quad\lesssim\|g\|^2_{H^1}\|\<v\>^{\frac d2}h\|^2_{H^1}\int \frac{\<v_3\>^d(p(v_3)-p(v_3+\hbar k))^2}{\hbar ^2|k|\<k\>^{2d+6}} \ud v_3 \uud k\\
				&\quad\lesssim\|g\|^2_{H^1}\|\<v\>^{\frac d2}h\|^2_{H^1}\int \frac{(\<v_3\>^{\frac{d+5}{2}}p(v_3)-\<v'_3\>^{\frac{d+5}{2}}p(v_3'))^2+(\<v_3\>^{\frac{d+5}{2}}-\<v'_3\>^{\frac{d+5}{2}})^2p(v_3')^2}{\hbar^2|k|\<k\>^{2d+6}\<v_3\>^{5}} \ud v_3 \uud k.
			\end{align*}
			Using that $(\<v_3\>^{\frac{d+5}{2}}-\<v'_3\>^{\frac{d+5}{2}})^2\lesssim |\hbar k|^2 (\<v'_3\>^{d+3}+\<\hbar k\>^{d+3})$ and using Proposition \ref{cor:equivalence des normes},
			\begin{align*}
				\int& \frac{g(v_1^{s_2})^2 (\<v_{2,s_3}\>^{d/2}h(v_{2,s_3}))^2}{\<k\>^{2d+6}}\gr{H}_{\hbar k}(p)(v_1) \delta_{k\cdot(v_2-v_1)}\ud V \lesssim \|g\|^2_{H^1}\|\<v\>^{\frac d2}h\|^2_{H^1} \big(A_\hbar(\<v\>^{\frac{d+5}{2}}p)+\|\<v\>^{\frac{d+5}{2}}p\|_{L^2}\big).
			\end{align*}
			
			\step{3} Proof of \eqref{xix}.
			
			Using a Cauchy-Schwartz inequality,
			\begin{multline*}
				\int (f(v_1)-f(v_1'))g(v_{1,s_2}) h(v_{2,s_3}) \<k\>^{\frac{d}{2}}\V(k)^2\delta_{k\cdot(v_2-v_1)}\ud V\\
				\lesssim \left(\int \tfrac{(f(v_1)-f(v_1'))^2}{\<v_{2,s_3}\>^d\<k\>^{2d+6}}\delta_{k\cdot(v_2-v_1)}\ud V\right)^{\frac12}\left(\int \frac{g(v_1^{s_2})^2 (\<v_{2,s_3}\>^{d/2}h(v_{2,s_3}))^2}{\<k\>^{d+2\eta}} \delta_{k\cdot(v_2-v_1)}\ud V\right)^{\frac12}.
			\end{multline*}
			Integrating first with respect to $v_2$ and then to $k$
			\begin{multline*}
				\int (f(v_1)-f(v_1'))g(v_{1,s_2}) h(v_{2,s_3}) \<k\>^{\frac{d}{2}}\V(k)^2\delta_{k\cdot(v_2-v_1)}\ud V\\
				\lesssim \left(\int \tfrac{(f(v_1)-f(v_1'))^2}{\<v_{2,s_3}\>^{d+8}\<k\>^{2d+6}}\delta_{k\cdot(v_2-v_1)}\ud V\right)^{\frac12}\left(\int \frac{g(v_1^{s_2})^2 (\<v_{2,s_3}\>^{\frac d2+4}h(v_{2,s_3}))^2}{\<k\>^{d+2\eta}} \delta_{k\cdot(v_2-v_1)}\ud V\right)^{\frac12}.
			\end{multline*}
			We can conclude by applying Proposition \ref{cor:equivalence des normes} for the right hand side and the same strategy than in the first step for the right hand side.
		\end{proof}
		
		\subsection{Cancellation lemmas}
			The main purpose of this section is to take advantages of the cancellation
			\[\int \V(k)^2 k \delta_{k\cdot(v_2-v_1-\hbar k)}\uud k = C_{\V}\hbar \tfrac{v_2-v_1}{|v_2-v_1|^2}\]
			for some constant $ C_{\V}>0$. This is due to the symmetry $k\mapsto -k$ inside the integral 
			\[\int \V(k)^2 k \delta_{k\cdot(v_2-v_1)}\uud k =0.\]
			When we add an asymmetry $\varpi\in\mathcal{B}^2(R)$ (which will be a derivative of the dielectric constant $\tfrac{1}{|\e_\hbar|^2})$, the symmetry $k\mapsto -k$ remains almost true, with $\varpi(-k,v_1) = \varpi(k,v_1-\hbar k)$. Hence,
			\begin{multline*}
				\int \V(k)^2 k \varpi (k,v_1)\delta_{k\cdot(v_2-v_1)}\uud k =\frac{1}{2}\int \V(k)^2 k (\varpi (k,v_1)-\varpi (k,v_1-\hbar k))\delta_{k\cdot(v_2-v_1)}\uud k \\
				 =\frac{\hbar}{2}\int \V(k)^2 k \cdot\nabla_{v_1}\varpi (k,v_1)\delta_{k\cdot(v_2-v_1)}\uud k +O(\hbar^2).
			\end{multline*}
			The following lemma exploit this idea.
			
			\begin{prop}\label{prop:sum of the collision kernel}
				Consider $f,g,h\in L^2(\mathbb{R}^d)$, and $\varpi\in \mathcal{B}^1(\mathbb{R})$. Then
				\begin{equation}
					\frac{1}{\hbar}\int f_1g_1h_2\varpi(k,v_1)\V(k)^2k\delta_{k\cdot(v_2-v_1-\hbar k)}\ud V \lesssim \|\varpi\|_{\mathcal{B}^1}(\|f\|_{L^2}+A_\hbar(f)^{\frac12}) \|g\|_{H^1}\|\<v\>^{\frac d2+1}h\|_{L^2}.
				\end{equation}
			\end{prop}
				\begin{proof}
					We begin by performing the decomposition
					\begin{multline*}
						\frac{1}{\hbar}\int f_1g_1h_2\varpi(k,v_1)\V(k)^2k\delta_{k\cdot(v_2-v_1-\hbar k)}\ud V \\
						=\int \tfrac{f_1g_1-f_1'g_1'}{\hbar}h_2\varpi(k,v_1)\V(k)^2k\delta_{k\cdot(v_2-v_1-\hbar k)}\ud V+\frac{1}{\hbar}\int f'_1g'_1h_2\varpi(-k,v_2)\V(k)^2k\delta_{k\cdot(v_2-v_1-\hbar k)}\ud V   
					\end{multline*}
					
					Using Proposition \ref{lem:borne general integral}, the left hand side is bounded by $O(\|\varpi\|_{\mathcal{B}^0}(\|f\|_{L^2}+A_\hbar(f)^\frac12)\|g\|_{H^1}\|h\|_{L^2})$.
					
					For the left hand side, using the change of variable $v_1\mapsto v_1'$ and the symmetry $k\mapsto-k$,
					\begin{multline*}
						\frac{1}{\hbar}\int f'_1g'_1h_2\varpi(-k,v_2)\V(k)^2k\delta_{k\cdot(v_2-v_1-\hbar k)}\ud V=\frac{1}{\hbar}\int f_1g_1h_2\varpi(-k,v_2)\V(k)^2k\delta_{k\cdot(v_2-v_1)}\ud V\\
						=\frac{1}{2\hbar}\int f_1g_1h_2(\varpi(-k,v_2)-\varpi (-k,v_2+\hbar k))\V(k)^2k\delta_{k\cdot(v_2-v_1)}\ud V			
					\end{multline*}
					Using that
					\[\frac{1}{\hbar}\int (\varpi(-k,v_2)-\varpi (-k,v_2+\hbar k))\V(k)^2k\delta_{k\cdot(v_2-v_1)}\uud k=O\left(\tfrac{\|\varpi\|_{\mathcal{B}^1}}{|v_1-v_2|}\right),\]
					and applying the Hardy-Littelwood and Sobolev inequalities, if $d\geq 3$,
					\begin{align*}
						\bigg|\frac{1}{\hbar}\int f_1g_1h_2\varpi(k,v_1)\V(k)^2k\delta_{k\cdot(v_2-v_1-\hbar k)}\ud V\bigg|&\lesssim \|\varpi\|_{\mathcal{B}^1} \|f\|_{L^2}\|g\|_{L^{\frac{2d}{d-2}}} \|h\|_{L^1}\\
						&\lesssim \|\varpi\|_{\mathcal{B}^1} \|f\|_{L^2}\|g\|_{H^1} \|\<v\>^{\frac{d}{2}+1}h\|_{L^2}.
					\end{align*}
					If $d=2$,
					\begin{align*}
						\bigg|\frac{1}{\hbar}\int f_1g_1h_2\varpi(k,v_1)\V(k)^2k\delta_{k\cdot(v_2-v_1-\hbar k)}\ud V\bigg|\lesssim \|\varpi\|_{\mathcal{B}^1} \|f\|_{L^2}\|g\|_{L^2} \|h\|_{L^2}.
					\end{align*}
					This conclude the proof.
				\end{proof}

			This proposition is an extension of the previous one in the case where the asymmetric term $\gr{H}_{\hbar k}(p)(v_1)\varpi(k,v_1)$ is not regular. The idea is to transform the discrete derivative in $v_1$ that appear in $\gr{H}_{\hbar k}(p)(v_1)$ into a derivative in $k$ (it is similar to the proof of estimation (2.26) of \cite{DW}).

			\begin{prop}\label{lem:lemme technique}
				Fix $\varpi\in \mathcal{B}_\hbar^1(\mathbb{C})$, $f,g,h,p\in L^2(\mathbb{R}^d)$. Then for any $K\geq 1$,
				\begin{multline}
					\left|\int \frac{\V(k)^2k}{\hbar} \Re (\varpi(k,v_1)\gr{H}_{\hbar k}(p)(v_1))f_1g_1 h_2  \delta_{k\cdot(v_1-v_2)}\ud V\right|\\
					\lesssim \|\varpi\|_{\mathcal{B}^1}\left(K^4\|f\|_{L^2}+K^{-1}A_\hbar(f)\right)\|g\|_{H^2}\|\<v\>^{\frac d2+4}h\|_{H^2}\left(A_\hbar (\<v\>^{\frac d2+4}p)+\|\<v\>^{\frac d2+4}p\|_{L^2}\right).
				\end{multline}
			\end{prop}
			\begin{proof}
				Fix $\varphi:\mathbb{R_+}\to [0,1]$ an increasing function, with $\chi(1) = 1$, $\chi(0)= 0$ and $\chi'$ is supported in $[1/3,2/3]$.
				
				Using the orthogonal decomposition $v = v^\para+v^\perp$ with $v^\para:=\hat{k}\cdot v \hat{k}$, we define in the following
				\[\gr{H}_{\hbar k}^K(p)(v_1):=\int\chi(K|v_1^\perp-v_3^{\perp}|)\tfrac{p(v_3+\hbar k)-p(v_3)}{\hbar k\cdot(v_3-v_1)+i0}\ud v_3,~\tilde{\gr{H}}_{\hbar k}^K(p)=\gr{H}_{\hbar k}(p)-\gr{H}_{\hbar k}^K(p).\]
				
				\step{1} Treatment of $\gr{H}_{\hbar k}^K$:
				
				Using the change of variable $v_2\mapsto v_2-\hbar k$, the symmetry $k\mapsto -k$, and that for $k\cdot(v_1-v_2)=0$, $\varpi(k,v_1)=\varpi(k,v_2)$, the first term is equal to 
				\begin{align}
					&\int \frac{\V(k)^2k}{\hbar} \Re\big(\varpi(k,v_1)\gr{H}^K_{\hbar k}(p)(v_1)\big)f_1h_2 g_1\delta_{k\cdot(v_2-v_1)}\ud V\nonumber\\
					&=\int \frac{\V(k)^2k}{2\hbar}\Re\big(\varpi(k,v_2)\gr{H}^K_{\hbar k}(p)(v_1)-\varpi(-k,v_2)\gr{H}^K_{-\hbar k}(p)(v_1)\big)f_1g_1h_2\delta_{k\cdot(v_2-v_1)}\ud V\nonumber\\
					&=\Re\int \frac{\V(k)^2k}{2\hbar}\big(\varpi(k,v_2)-\varpi(k,v_2-\hbar k)\big)\gr{H}^K_{\hbar k}(p)(v_1-\hbar k)f_1g_1h_2\delta_{k\cdot(v_2-v_1)}\ud V\label{x}\\
					&+\Re\int \frac{\V(k)^2k}{2\hbar}\varpi(k,v_2) \big(\gr{H}^K_{\hbar k}(p)(v_1)-\gr{H}^K_{\hbar k}(p)(v_1-\hbar k)\big)f_1g_1h_2\delta_{k\cdot(v_2-v_1)}\ud V\label{xi}
				\end{align}
				Using that $\|\nabla_{v_1} \varpi\|_{L^\infty}$ is bounded, \eqref{x} can be treated by using Proposition \ref{lem:borne general integral}. The main difficulty is the treatment of \eqref{xi}. The main idea is to transform the discrete derivative in $v_1$ of $\big(\gr{H}_{\hbar k}(p)(v_1)-\gr{H}_{\hbar k}(p)(v_1-\hbar k)\big)\big)$ into a derivative in $k$ and $v_2$ (via integration by part).
				
				Using that $\div_v \tfrac{v}{|v|^2}=\tfrac{d-2}{|v|^2}$ and that $\nabla_k|v^\perp| = -\tfrac{\hat{k}\cdot v}{|k||v^\perp|}v^\perp$ one has
				\begin{align*}
					&|k|^2\div_k\left(\tfrac{\chi(K|v^\perp|)}{ k\cdot(v-s\hbar k)+i0}\tfrac{v-s\hbar k}{|v-s\hbar k|^2}\right)\\
					=&-\tfrac{\hbar|k|^2\chi(K|v^\perp|)(v-2s\hbar k)}{(\hbar k\cdot(v-s\hbar k)+i0)^2}\tfrac{v-s\hbar k}{|v-s\hbar k|^2}+(d-2)\tfrac{\hbar|k|^2\chi(K|v^\perp|)}{\hbar k\cdot(v-s\hbar k)+i0}\tfrac{-s\hbar }{|v-s\hbar k|^2}+\tfrac{|k|^2}{ k\cdot(v-s\hbar k)+i0}\tfrac{v-s\hbar k}{|v-s\hbar k|^2}\cdot v^\perp\tfrac{K\chi'(K|v^\perp|)\hat{k}\cdot v}{|k||v^\perp|}\\
					=&-\tfrac{\hbar^2|k|^2\chi(K|v^\perp|)}{(\hbar k\cdot(v-s\hbar k)+i0)^2}-(d-3)\tfrac{\chi(K|v^\perp|)}{\hat{k}\cdot(v-\hbar k)+i0}\tfrac{s\hbar|k|}{|v-s\hbar k|^2}-\tfrac{(\hat{k}\cdot v)K\chi'(K|v^\perp|)}{ \hat{k}\cdot(v-s\hbar k)+i0}\tfrac{|v^\perp|}{|v-s\hbar k|^2}\\
					=&-\tfrac{\ud}{\ud s}\tfrac{\chi(K|v^\perp|)}{\hbar k\cdot(v-s\hbar k)+i0}-\tfrac{(d-3)\chi(K|v^\perp|)+K\chi'(K|v^\perp|)|v^\perp|}{\hat{k}\cdot(v-\hbar k)+i0}\tfrac{s\hbar|k|}{|v-s\hbar k|^2}-\tfrac{K\chi'(K|v^\perp|)|v^\perp|}{|v-s\hbar k|^2}.
				\end{align*}
				
				Denoting
				\begin{gather}
					\psi_{\hbar k}^1(v)=\int_0^1\tfrac{\chi(K|v^\perp|)}{\hat k\cdot (v-s\hbar k)+i0}\tfrac{v-s\hbar k}{|v-s\hbar k|^2}\ud s\\
					\psi_{\hbar k}^2(v)=\int_0^1\tfrac{K\chi'(K|v^\perp|)|v^\perp|}{|v-s\hbar k|^2}\ud s\\
					\psi_{\hbar k}^3(v)= \int_0^1\left(\tfrac{s((d-3)\chi(K|v^\perp|)+K\chi'(K|v^\perp|)|v^\perp|)}{(\hat{k}\cdot (v-s\hbar k)+i0)|v-s\hbar k|^2}-\tfrac{(s-1)((d-3)\chi(K|v^\perp|)+K\chi'(K|v^\perp|)|v^\perp|)}{(\hat{k}\cdot (v-(s-1)\hbar k)+i0)|v-(s-1)\hbar k|^2}\right)\ud s,\label{eq:def varphi 2}
				\end{gather}
				we deduce that 
				\begin{align*}
					&\gr{H}^K_{\hbar k}(p)(v_1)-\gr{H}^K_{\hbar k}(p)(v_1-\hbar k)\big) = \int p(v_1-v) \Big(\tfrac{\chi(K|v^\perp|)}{\hbar k\cdot (v-s\hbar k) + i 0}-\tfrac{2\chi(K|v^\perp|)}{\hbar k\cdot v + i 0}+\tfrac{\chi(K|v^\perp|)}{\hbar k\cdot (v+s\hbar k) + i 0}\Big)\ud v\\
					&=\int p(v_1-v ) \int_0^1\tfrac{\ud}{\ud s}\left(\tfrac{\chi(K|v^\perp|)}{\hbar k\cdot(v-(s-1)\hbar k)+i0}-\tfrac{\chi(K|v^\perp|)}{\hbar k\cdot(v-s\hbar k)+i0}\right)\ud s\ud v\\
					&=-|k|^2 \div_k\left[\tfrac{1} {|k|}\psi^1_{\hbar k}\ast_v (p(\cdot)-p(\cdot-\hbar k))\right]\!(v_1)-\left[\psi^2_{\hbar k}\ast_v (p(\cdot)-p(\cdot-\hbar k))\right]\!(v_1) + \hbar |k| \left[\psi^3_{\hbar k}\ast_v p\right]\!(v_1)
				\end{align*}
				
				Hence
				\begin{align*}
					\eqref{xi}=&\Re\int \tfrac{|k|\V(k)^2k}{2}\varpi(k,v_2)f_1g_1h_2 \left[\psi^3_{\hbar k}\ast_v p\right](v_1)\delta_{k\cdot(v_2-v_1)}\ud V\\
					&+\Re\int  \tfrac{\V(k)^2k}{2}\varpi(k,v_2)f_1g_1h_2 \left[ \psi^2_{\hbar k}\ast_v \tfrac{p(\cdot)-p(\cdot-\hbar k)}{\hbar}\right](v_1)\delta_{k\cdot(v_2-v_1)}\ud V\\
					&+\Re\int \tfrac{1}{|k|}\nabla_k\left[ \tfrac{|k|^2\V(k)^2k}{2}\varpi(k,v_2)f_1g_1h_2 \delta_{k\cdot(v_2-v_1)}\right]\cdot\left[ \psi^1_{\hbar k}\ast_v \tfrac{p(\cdot)-p(\cdot-\hbar k)}{\hbar}\right](v_1)\ud V
				\end{align*}
				
				Using that $\nabla_k \delta_{k\cdot (v_2-v_1)} =\left[\tfrac{k}{|k|^2}\cdot \nabla_{v_2} \delta_{k\cdot (v_2-v_1)}\right] (v_2-v_1)$, we obtain that 
				\begin{align*}
					&\Re\int \tfrac{1}{|k|}\nabla_k\left[ \tfrac{|k|^2\V(k)^2k}{2}\varpi(k,v_2)f_1g_1h_2 \delta_{k\cdot(v_2-v_1)}\right]\cdot\left[ \psi^1_{\hbar k}\ast_v \tfrac{p(\cdot)-p(\cdot-\hbar k)}{\hbar}\right](v_1)\uud k\ud v_1\ud v_2\\
					&=\Re\int \tfrac{1}{|k|}\nabla_k\cdot\left[ \tfrac{|k|^2\V(k)^2k}{2} \varpi(k,v_2) \right] f_1 g_1 h_2\big[ \psi^1_{\hbar k}\ast_v \tfrac{p(\cdot)-p(\cdot-\hbar k)}{\hbar}\big](v_1)\delta_{k\cdot(v_2-v_1)}\ud V\\
					&-\Re\int \tfrac{\V(k)^2k}{2|k|}f_1g_1 k\cdot \nabla_{v_2}\left[ h_2\varpi(k,v_2)\right] (v_2-v_1)\cdot\big[ \psi^1_{\hbar k}\ast_v \tfrac{p(\cdot)-p(\cdot-\hbar k)}{\hbar}\big](v_1)\delta_{k\cdot(v_2-v_1)}\ud V .\\
				\end{align*}
				
				As we do not want any new weight in $v_1$, we have to care about the new factor $v_1$ that appears in the last line:
				\begin{align*}
					v_1\cdot\big[ \psi^1_{\hbar k}\ast_v \tfrac{p(\cdot)-p(\cdot-\hbar k)}{\hbar}\big](v_1)&= \int \tfrac{p(v_1-v)-p(v_1-v-\hbar k)}{\hbar}(v_1-v) \cdot\int_0^1\tfrac{\chi(K|v^\perp|)}{\hat k\cdot (v-s\hbar k)+i0}\tfrac{v-s\hbar k}{|v-s\hbar k|^2}\ud s\ud v\\
					&+\int  \tfrac{p(v_1-v)-p(v_1-v-\hbar k)}{\hbar} \int_0^1\underset{=\frac{\chi(K|v^\perp|)}{\hat k\cdot(v-s\hbar k)}+\frac{s\hbar\chi(K|v^\perp|)}{|v-s\hbar k|^2}} {\underbrace{\tfrac{\chi(K|v^\perp|)v}{\hat k\cdot (v-s\hbar k)+i0}\cdot\tfrac{v-s\hbar k}{|v-s\hbar k|^2}}}\ud s\ud v\\
					=\psi^1_{\hbar k}\dot{\ast}_v \Big(&\tfrac{p(v)-p(v-\hbar k)}{\hbar}v\Big)(v_1)+ \psi^4_{\hbar k}\ast_v \Big(\tfrac{p(v)-p(v-\hbar k)}{\hbar}\Big)(v_1)+|k|\psi^5_{\hbar k}\ast_v \Big(\tfrac{p(v)-p(v-\hbar k)}{\hbar}\Big)(v_1)
				\end{align*}
				where we have denoted  $\dot{\ast}$  convolution product with scalar product ($f\dot{\ast}g(v) := \int f(v_*)\cdot g(v-v_*)\ud v_*$), and
				\begin{gather}
					\psi^4_{\hbar k}(v):= \int_0^1 \tfrac{\chi(K|v^\perp|)}{\hat k\cdot(v-s\hbar k)+i0}\ud s,\qquad\psi^5_{\hbar k}(v):= \int_0^1\tfrac{s\hbar\chi(K|v^\perp|)}{|v-s\hbar k|^2}\ud s.
				\end{gather}
				
				Integrating with respect to $v_2$ and using the bound on $\varpi$ and $\V$, \eqref{xi} is bounded by
				\begin{align*}
					&\int  \tfrac{1}{\<k\>^{\frac{3}{2}d+2+\eta}}\|\varpi\|_{\mathcal{B}^1}\|h\|_{L^\infty_\para(L^1_\perp)}\int\left|f(v_1)\right|\left|g(v_1)\right|\left|\psi^3_{\hbar k}\ast_v p\right|(v_1) \ud v_1 \uud k\\
					&+\int \tfrac{1}{\<k\>^{\frac{3}{2}d+3+\eta}}\|\varpi\|_{\mathcal{B}^1}\|\<v\>\<\nabla\>h\|_{L^\infty_\para(L^1_\perp)}\int\left|f(v_1)\right|\left| g(v_1)\right|
					\sum_{j\in\{1,2,4,5\}}\left|\psi^j_{\hbar k}\ast_v \big(\tfrac{p(v)-p(v-\hbar k)}{\hbar}\big)\right|(v_1) \ud v_1 \uud k\\
					&+\int \tfrac{1}{\<k\>^{\frac{3}{2}d+3+\eta}}\|\varpi\|_{\mathcal{B}^1}\|\<v\>\<\nabla\>h\|_{L^\infty_\para(L^1_\perp)}\int\left|f(v_1)\right|\left| g(v_1)\right|
					\left|\psi^1_{\hbar k}\dot{\ast}_v \big(\tfrac{p(v)-p(v-\hbar k)}{\hbar}v\big)\right|(v_1) \ud v_1 \uud k
				\end{align*}
				
				The terms $\psi_{\hbar k}^i$ are controlled by the following lemma:
				\begin{lemma}\label{lem:convolution}
					Fix $k\in\mathbb{R}^d$. There exists a constant $C>0$ such that for any function $f$ and any $K>1$, for any $d\geq2$, one has
					\begin{gather}
						\big\|\tfrac{\chi(K|v^\perp|)}{\hat k\cdot v +i0}\ast f\big\|_{L^{\infty}_{\perp}(L^2_\para)}\lesssim\ \left\|f\right\|_{L^{1}_{\perp}(L^2_\para)},~
						\big\|\tfrac{1-\chi(K|v^\perp|)}{\hat k\cdot v +i0}\ast f\big\|_{L^{\infty}_{\perp}(L^2_\para)}\lesssim K^{-\tfrac{1}{2}}\left\|f\right\|_{L^{2}},
					\end{gather}
					for any $d\geq 3$, 	for any $p\in[ 1,\tfrac{d-1}{d-2}]$ and $q=\tfrac{p(d-1)}{d-1-(d-2)p}$,
					\begin{gather} 
						\big\|\tfrac{\chi(K|v^\perp|)}{|v|^2}\ast f\big\|_{L^{q}_{\perp}(L^2_\para)}+\big\|\tfrac{\chi(K|v^\perp|)v}{(\hat k\cdot v +i0)|v|^2}\ast f\big\|_{L^{q}_{\perp}(L^2_\para)}+\big\|\tfrac{\chi(K|v^\perp|)|v^\perp|}{(\hat k\cdot v +i0)|v|^2}\ast f\big\|_{L^{q}_{\perp}(L^2_\para)}\leq C \left\|f\right\|_{ L^p_\perp( L^{2}_{\para})},
					\end{gather}
					for any $d\geq 4$, 	for any $p\in[ 1,\tfrac{d-1}{d-3}]$ and $q=\tfrac{p(d-1)}{d-1-(d-3)p}$,
					\begin{gather}
						\big\|\tfrac{\chi(K|v^\perp|)}{(\hat k\cdot v +i0)|v|^2}\ast f\big\|_{L^{q}_{\perp}(L^2_\para)}\leq C \left\|f\right\|_{L^p_\perp (L^{2}_{\para})}
					\end{gather}
					for $d=2$ and for any $p,q\in [1,\infty]$ 
					\begin{gather} 
						 \text{if }q<p,~\big\|\tfrac{\chi(K|v^\perp|)}{|v|^2}\ast f\big\|_{L^{q}_{\perp}(L^2_\para)}+\big\|\tfrac{\chi(K|v^\perp|)v}{(\hat k\cdot v +i0)|v|^2}\ast f\big\|_{L^{p}_{\perp}(L^2_\para)}\leq CK^{\tfrac{1}{q}-\tfrac{1}{p}} \left\|f\right\|_{ L^q_\perp( L^{2}_{\para})},\\
						 \text{if }q<p,~\big\|\tfrac{\chi(K|v^\perp|)}{(\hat k\cdot v +i0)|v|^2}\ast f\big\|_{L^{q}_{\perp}(L^2_\para)}\leq CK^{1+\tfrac{1}{q}-\tfrac{1}{p}}\left\|f\right\|_{L^p_\perp (L^{2}_{\para})}
					\end{gather}
				\end{lemma}
				We postpone the proof in Appendix \ref{sec:borne convolution}.
				
				Using this lemma with the maximal admissible $p\in[1,2]$, and that  $\left\|f\right\|_{L^p_\perp (L^{2}_{\para})}\leq \|\<v\>^{\frac{d}{2}}f\|_{L^2}$ for $p\in[1,2]$, one has 
				\begin{gather*}
					\left\|\psi^3_{\hbar k}\ast_vp\right\|_{L^{q(d)}_{\perp}(L^2_\para)+L^{\infty}_{\perp}(L^2_\para)}\lesssim K^2\big\|\<v\>^{\frac{d}{2}}p\big\|_{L^2},\\
					\sum_{j\in\{1,2,4,5\}}\left\|  \chi^i_{\hbar k}\ast_v \big(\tfrac{p(v)-p(v-\hbar k)}{\hbar}\big)\right\|_{L^{\infty}_{\perp}(L^2_\para)} + \left\|\chi^1_{\hbar k}\dot{\ast}_v \big(\tfrac{p(v)-p(v-\hbar k)}{\hbar}v\big)\right\|_{L^{\infty}_{\perp}(L^2_\para)}\lesssim K \big\|\<v\>^{\frac{d}{2}+1} \tfrac{p(v)-p(v-\hbar k)}{\hbar}\big\|_{L^2}.
				\end{gather*}
				where $q(2)=q(3)=2$, $q(4)=6$ and $q(d)=\infty$ for $d\geq 5$.

				Using Sobolev injections (in the following, $\mathcal{C}^0_\para$ is the set of function continuous with the norm $L^\infty$, which is a separable Banach space), 
				\begin{gather}
					\forall d\geq 2,~\| f\|_{L^2_\perp(\mathcal{C}^0_\para)}\lesssim \| f\|_{L^2_\perp(H^1_\para)}\leq \|f\|_{H^1},\\
					d\leq 3,~\|f\|_{\mathcal{C}^0} \lesssim \|f\|_{H^2},\\
					d=4,~\|f\|_{L^3_{\perp}(\mathcal{C}_\para^0)} \lesssim \|f\|_{L^3_{\perp}(H^1)}\lesssim\|f\|_{H^2},
				\end{gather}
				and trace estimates
				\begin{equation}
					\forall d\geq 3, ~\|h\|_{L^\infty_\para(L^1_\perp)}\lesssim \|\<v\>^{\frac{d-1}{2}}h\|_{L^\infty_\para(L^2_\perp)}\lesssim\|\<v\>^{\frac{d-1}{2}}h\|_{H^1}.
				\end{equation}

				Applying Hölder inequality, we deduce that
				\begin{multline*}
					|\eqref{xi}|\lesssim\int  \tfrac{K}{\<k\>^{2d+2+\eta}}\|\varpi\|_{\mathcal{B}^1}\|\<v\>^{\frac{d}{2}+1}h\|_{H^2}\|g\|_{H^2}\|f\|_{L^2}\|\<v\>^{\frac{d}{2}}p\|_{L^2} \uud k\\
					+\int \tfrac{K^2}{\<k\>^{2d+2+\eta}}\|\varpi\|_{\mathcal{B}^1}\|\<v\>^{\frac{d}{2}+1}h\|_{H^2}\|g\|_{H^2}\|\|f\|_{L^2}\|\<v\>^{\frac{d}{2}+1}\tfrac{p(\cdot)-p(\cdot-\hbar k)}{\hbar}\|_{L^2} \uud k.
				\end{multline*}
				
				It remains to bound $\int \tfrac{1}{\<k\>^{2d+2+\eta}}\|\<v\>^{\frac{d}{2}+1}\tfrac{p(\cdot)-p(\cdot-\hbar k)}{\hbar}\|_{L^2} \uud k$. By Cauchy-Schwartz inequality,
				\begin{align*}
					&\int \tfrac{1}{\<k\>^{2d+2+\eta}}\|\<v\>^{\frac{d}{2}+1}\tfrac{p(\cdot)-p(\cdot-\hbar k)}{\hbar}\|_{L^2} \uud k 
					\lesssim \left(\int \tfrac{\uud k}{\<k\>^{d+2\eta}}\right)^{\frac{1}{2}}\left(\int \tfrac{\<v_1\>^{{d+2}}}{\hbar^2\<k\>^{3d+4}} (f(v_1')-f(v_1))^2\uud k\ud v_1\right)^{\frac{1}{2}}\\
					&\lesssim\left(\int \tfrac{(\<v'_1\>^{\frac {d+3}2}f(v_1')-\<v_1\>^{\frac {d+3}2}f(v_1))^2}{\hbar^2\<k\>^{3d+4}\<v_1\>^{5}} \uud k\ud v_1\right)^{\frac{1}{2}}+\left(\int \tfrac{(\<v'_1\>^{\frac {d+3}2}-\<v_1\>^{\frac {d+3}2})^2f(v'_1)^2}{\hbar^2\<k\>^{3d+4}\<v_1\>^{5}} \uud k\ud v_1\right)^{\frac{1}{2}}.
				\end{align*}
				
				It remains to bound 
				\[\tfrac{(\<v'_1\>^{\frac {d+3}2}-\<v_1\>^{\frac {d+3}2})^2}{\hbar^2\<k\>^{3d+4}\<v_1\>^{5}}\lesssim\tfrac{|k|^2(\<v'_1\>^{d+3}+\<v_1\>^{d-3})}{\<k\>^{3d+4}} \lesssim\tfrac{\<v'_1\>^{d+3}}{\<k\>^{d+2}}+\tfrac{\<\hbar k\>^{d-3}}{\<k\>^{2d+2}},  \]
				and applying Proposition \ref{cor:equivalence des normes}, 
				\begin{equation}
					\int \tfrac{1}{\<k\>^{2d+2+\eta}}\|\<v\>^{\frac{d+1}{2}}\tfrac{p(\cdot)-p(\cdot-\hbar k)}{\hbar}\|_{L^2} \uud k
					\lesssim A_{\hbar}(\<v\>^{\frac d2+4}p)+\|\<v\>^{\frac d2+4}p\|_{L^2.}
				\end{equation}
				
				This concludes the proof that
				\begin{multline}
					\int \frac{\V(k)^2k}{\hbar} \Re\big(\varpi(k,v_1)\gr{H}^K_{\hbar k}(p)(v_1)\big)f_1h_2 g_1\delta_{k\cdot(v_2-v_1)}\ud V\\
					\lesssim \|\varpi\|_{\mathcal{B}^1}\|f\|_{L^2}\|g\|_{H^2}\|\<v\>^{\frac{d}{2}+1}h\|_{H^2}(K^2\|\<v\>^{\frac{d}{2}+1}p\|_{L^2}+KA_{\hbar}(\<v\>^{\frac{d}{2}+1}p)^{\frac12}).
				\end{multline}
				
				\step{2} Treatment of $\tilde{\gr{H}}_{\hbar k}^K$:
				
				Using the symmetry $(v_1,v_2,k)\mapsto(v_1',v_2',-k)$,
				\begin{multline*}
					\int \frac{\V(k)^2k}{\hbar} \Re\big(\varpi(k,v_1)\gr{H}^K_{\hbar k}(p)(v_1)\big)f_1g_1h_2\delta_{k\cdot(v_2-v_1)}\ud V \\
					= \int \frac{\V(k)^2k}{2\hbar} \Re\big(\varpi(k,v_1)\tilde{\gr{H}}^K_{\hbar k}(p)(v_1)\big)(f_1g_1h_2-f'_1g'_1h'_2) \delta_{k\cdot(v_2-v_1)}\ud V.
				\end{multline*}
				
				Using the decomposition 
				\[\tfrac{f_1g_1h_2-f'_1g'_1h'_2}{\hbar}=\tfrac{f_1-f_1'}{\hbar}g_1h_2+f'_1\tfrac{g'_1-g_1}{\hbar}h_2+f_1'g_1'\tfrac{h_2-h_2'}{\hbar},\]
				using that $\hat{k}\cdot v_1=\hat{k}\cdot v_2$, denoting $v_{1,s}:=v_1+s\hbar k$ and integrating with respect to $v_2$, 
				\begin{align*}
					&\bigg|\int \tfrac{\V(k)^2k}{\hbar} \Re\big(\varpi(k,v_1)\tilde{\gr{H}}^K_{\hbar k}(p)(v_1)\big)f_1g_1h_2 \delta_{k\cdot(v_2-v_1)}\uud k\ud v_1\ud v_2\bigg|\\
					&\lesssim \int_0^1\int \tfrac{\|\varpi\|_{\mathcal{B}^0}}{\<k\>^{2d+1+\eta}\<\hat{k}\cdot v_1\>^{\frac94}}\big(\big|\tfrac{f_1'-f_1}{\hbar}\big|+|f'_1|\big)\big(\big|\nabla g(v_{1,s})\big|+|g_1|+|g'_1|\big)\big|\tilde{\gr{H}}^K_{\hbar k}(p)(v_1)\big|\|\<v\>^{\frac94}\<\nabla\>h\|_{L^\infty_\para(L^1_\perp)}\uud k\ud v_1\\
					&\lesssim \int \tfrac{\|\varpi\|_{\mathcal{B}^0}\|\<\nabla\>g\|_{L^2(\mathcal{C}^0_b)}\|\<v\>^{\frac d2+4}h\|_{H^2}}{|k|\<k\>^{2d+2+\eta}}\big(\big\|\tfrac{f(v+\hbar k)-f(v)}{\hbar\<\hat{k}\cdot v\>^{9/4}}\big\|_{L^2}+|k|\|f\|_{L^2}\big)\|\tfrac{1-\chi(K|v^\perp|)}{\hat{k}\cdot v+i0}\ast\tfrac{p(\cdot+\hbar k)-p(\cdot)}{\hbar}\|_{L^\infty_\perp(L^2_\para)}\uud k\\
					&\lesssim \int \tfrac{\|\varpi\|_{\mathcal{B}^0}\|g\|_{H^2}\|\<v\>^{\frac d2+4}h\|_{H^2}}{K^{\frac12}|k|\<k\>^{2d+2+\eta}}\big(\big\|\tfrac{f(v+\hbar k)-f(v)}{\hbar\<\hat{k}\cdot v\>^{9/4}}\big\|_{L^2}+|k|\|f\|_{L^2}\big)\|\tfrac{p(\cdot+\hbar k)-p(\cdot)}{\hbar}\|_{L^2}\uud k.
				\end{align*}
				
				Using Cauchy-Schwartz inequality,
				\begin{align*}
					\int \tfrac{1}{|k|\<k\>^{2d+2+\eta}}&\big\|\tfrac{f(v+\hbar k)-f(v)}{\hbar\<\hat{k}\cdot v\>^{9/4}}\big\|_{L^2}\|\tfrac{p(\cdot+\hbar k)-p(\cdot)}{\hbar}\|_{L^2}\uud k\\
					&\leq \left(\int\tfrac{(f(v+\hbar k)-f(v))^2}{\hbar |k|\<k\>^{d+3+\eta}\<\hat{k}\cdot v\>^{9/2}}\uud k\ud v\right)^{\frac{1}{2}} \left(\int\tfrac{(p(v+\hbar k)-p(v))^2}{\hbar |k|\<k\>^{3d+1+\eta}}\uud k\ud v\right)^{\frac{1}{2}}\\
					&\leq A_{\hbar}(f)^{\frac{1}{2}}\left(\int\tfrac{(\<v+\hbar k\>^{\frac94}p(v+\hbar k)-\<v\>^{\frac94}p(v))^2}{\hbar |k|\<k\>^{3d+1+\eta}\<v\>^{\frac92}}+\tfrac{p(v+\hbar k)^2(\<v+\hbar k\>^{\frac94}-\<v\>^{\frac94})^2}{\hbar |k|\<k\>^{3d+1+\eta}\<v\>^{\frac92}}\uud k\ud v\right)^{\frac{1}{2}}.
				\end{align*}
				
				As  $\tfrac{|\<v+\hbar k\>^{\frac94}-\<v\>^{\frac94}|^2}{\hbar |k|\<k\>^{3d+1+\eta}\<v\>^{\frac92}}\leq \tfrac{\<v\>^{\frac52}}{\<k\>^{3d-5/2}}$ and $d\geq 2$, 
				\begin{equation*}
					\int \tfrac{1}{|k|\<k\>^{2d+2+\eta}}\big\|\tfrac{f(v+\hbar k)-f(v)}{\hbar\<\hat{k}\cdot v\>^{9/4}}\big\|_{L^2}\|\tfrac{p(\cdot+\hbar k)-p(\cdot)}{\hbar}\|_{L^2}\uud k\lesssim (A_\hbar(f)^{\frac12}+\|f\|_{L^2})(A_\hbar(\<v\>^{\frac94}p)^{\frac12}+\|\<v\>^{\frac94}p\|_{L^2}),
				\end{equation*}
				and in the same way
				\begin{align*}
					\int \tfrac{1}{\<k\>^{2d+2+\eta}}&\|f\|_{L^2}\|\tfrac{p(\cdot+\hbar k)-p(\cdot)}{\hbar}\|_{L^2}\uud k\lesssim \|f\|_{L^2}(A_\hbar(\<v\>^{\frac94}p)^{\frac12}+\|\<v\>^{\frac94}p\|_{L^2}).
				\end{align*}
				
				We deduce that 
					\begin{multline*}
					\bigg|\int \tfrac{\V(k)^2k}{\hbar} \Re\big(\varpi(k,v_1)\tilde{\gr{H}}^K_{\hbar k}(p)(v_1)\big)f_1g_1h_2 \delta_{k\cdot(v_2-v_1)}\uud k\ud v_1\ud v_2\bigg|\\
					\lesssim  \tfrac{\|\varpi\|_{\mathcal{B}^0}\|g\|_{H^2}\|\<v\>^{\frac d2+4}h\|_{H^2}}{K^{\frac12}}(A_\hbar(f)^{\frac12}+\|f\|_{L^2})(A_\hbar(\<v\>^{\frac94}p)^{\frac12}+\|\<v\>^{\frac94}p\|_{L^2}).
				\end{multline*}
			\end{proof}

			The following proposition is the counter part of Lemma 2.6 of \cite{HLPZ}, and will be use to transfer regularity from one term to an other.
			
			\begin{prop}\label{prop:lemme d'annulation}
				Consider $f,g,h\in L^2(\mathbb{R}^d)$, and $\varpi\in \mathcal{B}^2(\mathbb{R})$. Then for any $K\geq 1$,
				\begin{align*}
					\frac{1}{\hbar^2}\int (f'_1g'_1&-f_1g_1)h_2\varpi(k,v_1)\V(k)^2\delta_{k\cdot(v_2-v_1-\hbar k)}\ud V \\
					&\lesssim \|\varpi\|_{\mathcal{B}^2}\left((K^4\|f\|_{L^2}+K^{-1}A_{\hbar}(f)\right)\nonumber \\
					&\times\min\left(\|g\|_{H^2}\left(\|\<v\>^{\frac d2+4}h\|_{L^2}+A_\hbar(\<v\>^{\frac d2+4}h)^{\frac12}\right),\left(\|g\|_{L^2}+A_{\hbar}(g)^{\frac12}\right)\|\<v\>^{\frac d2+4}h\|_{H^2}\right).\nonumber
				\end{align*}
			\end{prop}
			\begin{proof}
				Fix $\chi:\mathbb{R_+}\to [0,1]$ an increasing function, with $\chi(1) = 1$, $\chi(0)= 0$ and $\chi'$ is supported in $[1/3,2/3]$. We perform the decomposition 
				
				Using the symmetry $(v_1,v_2,k)\mapsto(v_1',v_2',-k)$, 
				\begin{align}
					\frac{1}{\hbar^2}\int\V(k)^2& (f'_1g'_1-f_1g_1)h_2\varpi(k,v_1)\delta_{k\cdot(v_2-v_1-\hbar k)}\ud V=
					\int\V(k)^2 \tfrac{f'_1g'_1-f_1g_1}{\hbar}\tfrac{h_2-h_2'}{\hbar}\varpi(k,v_1)\delta_{k\cdot(v_2-v_1-\hbar k)}\ud V\nonumber\\
					&=\int\V(k)^2(1-\chi(K|v_1-v_2|)) \tfrac{f'_1g'_1-f_1g_1}{\hbar}\tfrac{h_2-h_2'}{\hbar}\varpi(k,v_1)\delta_{k\cdot(v_2-v_1-\hbar k)}\ud V\label{eq:morceau sans cutoff1}\\
					&+\int\V(k)^2\chi(K|v_1-v_2|) \tfrac{f'_1g'_1-f_1g_1}{\hbar}\tfrac{h_2-h_2'}{\hbar}\varpi(k,v_1)\delta_{k\cdot(v_2-v_1-\hbar k)}\ud V.\label{eq:morceau avec cutoff1}
				\end{align}
				
				\step{1} We begin by the bound of \eqref{eq:morceau sans cutoff1}. Using Young inequality,
				\begin{align*}
					\bigg|\int&\V(k)^2(1-\chi(K|v_1-v_2|)) \tfrac{f'_1g'_1-f_1g_1}{\hbar}\tfrac{h_2-h_2'}{\hbar}\varpi(k,v_1)\delta_{k\cdot(v_2-v_1-\hbar k)}\ud V\bigg|\\	
					&\lesssim \|\varpi\|_{\mathcal{B}^0}\int\V(k)^2\ind_{K|v_1-v_2|\leq 1} \big(\big|g_1\tfrac{f'_1-f_1}{\hbar}\big|+\big|f'_1\tfrac{g'_1-g_1}{\hbar}\big|\big)\big|\tfrac{h_2-h_2'}{\hbar}\big|\delta_{k\cdot(v_2-v_1-\hbar k)}\ud V\\
					&\lesssim \|\varpi\|_{\mathcal{B}^0}\bigg(\int\V(k)^2\tfrac{\ind_{K|v^\perp_1-v^\perp_2|\leq 1}}{\<v_2\cdot \hat{k}\>^{9/2}} \big(\big|\tfrac{f'_1-f_1}{\hbar}\big|^2+|f'_1|^2\big)\delta_{k\cdot(v_2-v_1-\hbar k)}\ud V\bigg)^{\frac{1}{2}}\\
					&\qquad\qquad\qquad\qquad\times\bigg(\int\V(k)^2\<v_2\cdot \hat{k}\>^{9/2} \big(\big|\tfrac{g'_1-g_1}{\hbar}\big|^2+|g_1|^2\big)\delta_{k\cdot(v_2-v_1-\hbar k)}|\tfrac{h_2-h_2'}{\hbar}\big|^2\ud V\bigg)^{\frac{1}{2}}.
				\end{align*}
				
				Using the same strategy than in the proof of Proposition \ref{lem:borne general integral},
				\begin{multline*}
					\bigg|\int\V(k)^2(1-\chi(K|v_1-v_2|)) \tfrac{f'_1g'_1-f_1g_1}{\hbar}\tfrac{h_2-h_2'}{\hbar}\varpi(k,v_1)\delta_{k\cdot(v_2-v_1-\hbar k)}\ud V\bigg|\lesssim \|\varpi\|_{\mathcal{B}^2} K^{\frac12}(\|f\|_{L^2}+A_\hbar (f)^{\frac12})\\
					\times\min\left(\|g\|_{H^1}(\|\<v\>^{\frac d2+4}h\|_{L^2}+A_\hbar (\<v\>^{\frac d2+4} h)^{\frac12}),(\|g\|_{L^2}+A_\hbar (g)^{\frac12})\|\<v\>^{\frac d2+4} h\|_{H^1}\right).
				\end{multline*}
				
				\step{2} Estimation of \eqref{eq:morceau avec cutoff1}. Denoting $(v_1''):=v_1-\hbar k$, and using that $|v_1-v_2-\hbar k|=|v_1-v_2+\hbar k|$ if $k\cdot(v_1-v_2)=0$ and $|v_1-v_2-2\hbar k|=|v_1-v_2|$ if $k\cdot(v_1-v_2+\hbar k)=0$,
				\begin{multline}\label{eq:morceau avec cutoff}
					\int\chi(K|v_1-v_2|) \tfrac{f'_1g'_1-f_1g_1}{\hbar}\tfrac{h_2-h_2'}{\hbar}\varpi(k,v_1)\delta_{k\cdot(v_2-v_1-\hbar k)}\V(k)^2\ud V\\
					=\int\chi(K|v_1-v_2|) f_1g_1h_2\left[\tfrac{ \varpi(k,v''_1)}{\hbar^2}\delta_{k\cdot(v_2-v_1+\hbar k)}-\tfrac{\varpi(k,v''_1)+\varpi(k,v_1)}{\hbar^2}\delta_{k\cdot(v_2-v_1)}+\tfrac{\varpi(k,v_1)}{\hbar^2}\delta_{k\cdot(v_2-v_1-\hbar k)}\right]\V(k)^2\ud V\\
					-\int f_1g_1h_2\tfrac{\chi(K|v_1-v_2|)-\chi(K|v_1-v_2-\hbar k|)}{\hbar^2}(\varpi(k,v''_1)+\varpi(k,v_1))\delta_{k\cdot(v_2-v_1)}\V(k)^2\ud V.
				\end{multline}

				We need the following estimation:
				\begin{lemma}
					Consider $\varpi(k,v_1)\in\mathcal{B}_\hbar^2(\mathbb{R})$. There exists some constant $C>0$ such that 
					\begin{multline}
						\bigg|\int \V(k)^2\left[\tfrac{ \varpi(k,v''_1)}{\hbar^2}\delta_{k\cdot(v_2-v_1+\hbar k)}-\tfrac{\varpi(k,v''_1)+\varpi(k,v_1)}{\hbar^2}\delta_{k\cdot(v_2-v_1)}+\tfrac{\varpi(k,v_1)\delta_{k\cdot(v_2-v_1-\hbar k)}}{\hbar^2}\right]\uud k\bigg|\\
						\leq \|\varpi\|_{\mathcal{B}^2}\left(\tfrac{C}{|v_2-v_1|^2}+\tfrac{C}{|v_2-v_1|^3}\right).
					\end{multline}
				\end{lemma}
				
				\begin{proof}
					For $k\in<v_1-v_2>^\perp$, with $|k|\leq \frac{|v_1-v_2|}{\hbar}$ we introduce (as in the Step 2 of the proof of Proposition \ref{prop: traitement K})
					\begin{align*}
						\bar{\delta} k:= \left(\Lambda(\tfrac{\hbar |k|}{|v_2-v_1|})-1\right)k + \frac{\hbar |k|^2}{|v_2-v_1|^2}(v_2-v_1),~\Lambda(x):=\ind_{|x|\leq 1}\sqrt{1-x^2}.
					\end{align*}
					The application
					\begin{equation*}
						\left\{\begin{array}{c}
							\left\{k,~k\cdot(v_2-v_1)=0, |k|\leq \frac{|v_1-v_2|}{\hbar}\right\}\to\left\{\tilde{k},~\tilde{k}\cdot(v_2-v_1-\hbar \tilde{k})=0\right\},\\
							k\mapsto \tilde{k} := k+\bar{\delta} k
						\end{array}\right.
					\end{equation*}
					is a diffeomorphism with Jacobian
					\begin{equation*}
						\delta_{\tilde{k}\cdot(v_2-v_1-\hbar \tilde{k})} \ud{\tilde{k}} = \Lambda(\tfrac{\hbar |k|}{|v_2-v_1|})^{d-3}\uud k,
					\end{equation*}
					and $|k| = |\tilde{k}|$.
					
					We introduce $k_s:= |k|(\cos(s\theta_k)\hat{k}+\sin(s\theta_k)\widehat{v_2-v_1})$ with $\theta_k := 2\arcsin \tfrac{|\delta k|}{2|k|}$.
					
					We deduce that,
					\begin{align}
						\int &\left[\tfrac{ \varpi(k,v''_1)}{\hbar^2}\delta_{k\cdot(v_2-v_1+\hbar k)}-\tfrac{\varpi(k,v''_1)+\varpi(k,v_1)}{\hbar^2}\delta_{k\cdot(v_2-v_1)}+\tfrac{\varpi(k,v_1)}{\hbar^2}\delta_{k\cdot(v_2-v_1-\hbar k)}\right]\V(k)^2\uud k\nonumber\\
						&=\int  \left[\tfrac{\varpi(k_0,v_1)+\varpi(k_0,v_1-\hbar k_0)}{\hbar^2}-\tfrac{(\varpi(k_1 ,v_1)+\varpi(k_{-1},v_1-\hbar k_{-1})}{\hbar^2}\Lambda(\tfrac{\hbar |k|}{|v_2-v_1|})^{d-3}\right]\V(k)^2\delta_{k\cdot(v_2-v_1)}\uud k\nonumber\\
						\label{eq:K1}&=\int \left[\varpi(k,v_1)+\varpi(k,v_1-\hbar k)\right]\tfrac{1-\Lambda\big(\frac{\hbar |k|}{|v_2-v_1|}\big)^{d-3}}{\hbar^2}\V(k)^2\delta_{k\cdot(v_2-v_1)}\uud k\\
						\label{eq:K2}&+\int \tfrac{\varpi(k_0,v_1)+\varpi(k_0,v_1-\hbar k_0)-\varpi(k_1,v_1)-\varpi(k_{-1},v_1-\hbar k_{-1})}{\hbar^2}\Lambda(\tfrac{\hbar |k|}{|v_2-v_1|})^{d-3}\V(k)^2\delta_{k\cdot(v_2-v_1)}\uud k
					\end{align}
					
					Using the bound on $\varpi$ and that $\V$ checks Assumption \ref{ass: potentiel bis},
					\begin{align*}
						|\eqref{eq:K1}|&\lesssim \frac{\|\varpi\|_{\mathcal{B}^0}}{\hbar^2} \int \left|\Lambda(\tfrac{\hbar |k|}{|v_1-v_2|})^{d-3}-1\right|\frac{\delta_{k\cdot(v_1-v_2)}}{|k|^{d}}\uud k\\
						&\lesssim \frac{\|\varpi\|_{\mathcal{B}^0}}{\hbar^2} \int \frac{\left|\Lambda(|u|)^{d-3}-1\right|}{|u|^{d}}\frac{\hbar}{|v_2-v_1|}\delta_{\frac{|v_1-v_2|}{\hbar}u\cdot(v_1-v_2)} \ud u\\
						&\lesssim \frac{\|\varpi\|_{\mathcal{B}^0}}{|v_2-v_1|^3} \int_0^\infty \frac{\left|\Lambda(|u|)^{d-3}-1\right|}{|u|^2}\ud |u|.
					\end{align*}
					As $d\geq 2$, this is bounded.

					On the other hand,
					\begin{align*}
						\eqref{eq:K2} = \int\int_0^1\int_0^1 \tfrac{\ud^2}{\ud s \ud s'}\tfrac{\varpi(k_{s-s'},v_1- s'\hbar k_{s-s'})}{\hbar^2}\Lambda(\tfrac{\hbar |k|}{|v_2-v_1|})^{d-3}\V(k)^2\delta_{k\cdot(v_2-v_1)}\uud k.
					\end{align*}
					Hence, using that $\theta_k\lesssim \tfrac{|\delta k|}{|k|}\lesssim \tfrac{\hbar |k|}{|v_1-v_2|}$, 
					\begin{align*}
						|\eqref{eq:K2}| &\lesssim \int \|\varpi\|_{\mathcal{B}^2} \left(\tfrac{|k|^2}{|v_1-v_2|}+\tfrac{|k|^2}{|v_1-v_2|^2}\right)\Lambda(\tfrac{\hbar |k|}{|v_2-v_1|})^{d-3}\V(k)^2\delta_{k\cdot(v_2-v_1)}\uud k\\
						&\lesssim \|\varpi\|_{\mathcal{B}^2} \left(\tfrac{1}{|v_1-v_2|^2}+\tfrac{1}{|v_1-v_2|^3}\right).
					\end{align*}
					This conclude the proof.
				\end{proof}
				
				\begin{lemma}
					There exists some constant $C>0$ such that 
					\begin{equation}
						\int \V(k)^2\left|\tfrac{\chi(K|v_1-v_2|)-\chi(K|v_1-v_2-\hbar k|)}{\hbar ^2}\right|\delta_{k\cdot (v_2-v_1)}\uud k
						\leq C \tfrac{K^2}{|v_1-v_2|}.
					\end{equation}
				\end{lemma}
				\begin{proof}
					We begin by computing the derivative
					\begin{gather*}
						\nabla \chi(K|v|) = K\chi'(K|v|)\tfrac{v}{|v|}\\
						\nabla^2 \chi(K|v|) = K^2\chi''(K|v|)\tfrac{v\otimes v}{|v|^2}+K\chi'(K|v|)\tfrac{\rm Id}{|v|}-K\chi'(K|v|)\tfrac{v\otimes v}{|v|^3} = O(K^2),
					\end{gather*}
					where we use that $\chi'$ is supported in $[1/3,2/3]$.
					
					Hence, applying a Taylor expansion, as $(v_1-v_2)\cdot k = 0$,
					\begin{align*}
						\tfrac{\chi(K|v_1-v_2|)-\chi(K|v_1-v_2-\hbar k|)}{\hbar ^2} = \int_0^1 (1-s)\nabla^2\chi(K|v_1-v_2-s\hbar k|):k^{\otimes 2}\ud s = O(K^2|k|^2).
					\end{align*}
					
					We conclude by integrating with respect to $k$.					
				\end{proof}
				
				Applying the two previous lemmas, decomposition \eqref{eq:morceau avec cutoff} and Hardy-Littlewood inequality,
				\begin{align*}
					\bigg|\int\chi(K|v_1-v_2|) \tfrac{f'_1g'_1-f_1g_1}{\hbar}\tfrac{h_2-h_2'}{\hbar}\varpi(k,v_1)\delta_{k\cdot(v_2-v_1-\hbar k)}\V(k)^2\ud V\bigg|
					&\lesssim K^2\|\varpi\|_{\mathcal{B}^2} \int \tfrac{1}{|v_1-v_2|} f_1g_1h_2 \ud v_1\ud v_2\\
					&\lesssim K^{2}\|\varpi\|_{\mathcal{B}^2}  \|f\|_{L^2}\|g\|_{L^2}\|h\|_{L^{\frac{d}{d-1}}}\\
					&\lesssim K^{2}\|\varpi\|_{\mathcal{B}^2}  \|f\|_{L^2}\|g\|_{L^2}\|\<v\>^{\frac d2+4}h\|_{L^2}.
				\end{align*}
				
				This conclude the proof.
			\end{proof}
						
		\subsection{Non-Linear estimates}

			\begin{prop}
				Consider $\varpi\in\mathcal{B}_\hbar^0(\mathbb{R})$, $f,g,h:\mathbb{R}^d\to \mathbb{R}$ three measurable functions, for any $K>1$, 
				\begin{gather}
					\left|\left\<f,Q_\hbar(g,h,\varpi)\right\>\right|\lesssim \frac{1}{\hbar^2}\|\varpi\|_{\mathcal{B}^0}\|f\|_{L^2}\|g\|_{L^2}\|\<v\>^{\frac d2+4}h\|_{H^1}\label{eq:estim non lin 1}\\ \left|\left\<f,Q_\hbar(g,h,\varpi)\right\>\right|\lesssim \|\varpi\|_{\mathcal{B}^0}A_{\hbar}(f)^{\frac12}\min\left(\|g\|_{H_1}\|\<v\>^{\frac d2+4}h\|_{H^2},\|g\|_{H_2}\|\<v\>^{\frac d2+4}h\|_{H^1}\right)\label{eq:estim non lin 2}\\
					\left|\left\<f,Q_\hbar(g,h,\varpi)\right\>\right|\lesssim \|\varpi\|_{\mathcal{B}^2} \Big(K^4\|f\|_{L^2}+K^{-1}A_{\hbar}(f)\Big)\qquad\qquad\qquad\qquad\qquad\label{eq:estim non lin 3}\\
					\qquad\qquad\qquad\qquad\qquad\qquad\qquad\qquad\qquad\qquad\times\Big(\|\<v\>^{\frac d2+4}h\|_{L^2}+A_\hbar(\<v\>^{\frac d2+4}h)^{\frac12}\Big)\|g\|_{H^3}.\nonumber
				\end{gather}
				
				Consider $\varpi\in\mathcal{B}_\hbar^1(\mathbb{C})$, $f,g,h,p:\mathbb{R}^d\to \mathbb{R}$ three measurable functions, then, for any $K>1$,
				\begin{multline}
					\left|\left\<f,Q_\hbar(g,h,\Re(\varpi \gr{H}(p))\right\>\right|\\
					\lesssim \|\varpi\|_{\mathcal{B}^1}\Big(K^4\|f\|_{L^2}+K^{-1}A_{\hbar}(f)\Big)\|g\|_{H_3}\|\<v\>^{\frac d2+4}h\|_{H^3}\left(A_{\hbar}(\<v\>^{\frac d2 +4}p)+\|\<v\>^{\frac d2 +4}p\|_{L^2}\right)\label{eq:estim non lin 4}.
				\end{multline}
			\end{prop}
			\begin{proof}
				\step{1} Proof of \eqref{eq:estim non lin 1}. It is a direct application of Proposition \ref{lem:borne general integral}.
				
				\step{2}  Proof of \eqref{eq:estim non lin 2}. Using the symmetry $(v_1,v_2,k)\mapsto(v_1',v_2',-k))$,
				\begin{align*}
					\left\<f,Q_\hbar(g,h,\varpi)\right\> &= \int \frac{\V(k)^2}{\hbar^2}(f_1'-f_1) g_1 h_2 \varpi(k,v_1)\ud V\\
					&= \int \frac{\V(k)^2}{\hbar^2}(f_1'-f_1) \left[\tfrac{g_1+g_1'}{2} h_2 +\tfrac{g_1-g_1'}{2} h_2\right]\varpi(k,v_1)\ud V\\
					&= \int \frac{\V(k)^2}{\hbar^2}(f_1'-f_1) \left[\tfrac{g_1+g_1'}{2} \tfrac{h_2-h_2'}{2} +\tfrac{g_1-g_1'}{2} h_2\right]\varpi(k,v_1)\ud V\\
					&= \int_0^1\int \frac{\V(k)^2}{\hbar}(f_1'-f_1) \left[\tfrac{g_1+g_1'}{2} \tfrac{k\cdot\nabla h(v_2+(s-1))}{2} -\tfrac{k\cdot\nabla g(v_1+s\hbar k)}{2} h_2\right]\varpi(k,v_1)\ud V\ud s.
				\end{align*}
				We conclude by applying Proposition \ref{lem:borne general integral}.
				
				\step{3} We treat now \eqref{eq:estim non lin 3}:
				\begin{align}
					\big\<f,Q_\hbar&(g,h,\varpi)\big\> =\int \tfrac{f'_1-f_1}{\hbar^2}g_1h_2\V(k)^2\varpi(k,v_1)\delta_{k\cdot(v_2-v_1-\hbar k)}\ud V\nonumber\\
					&=\int \tfrac{f'_1g_1'-f_1g_1}{\hbar^2}h_2\V(k)^2\varpi(k,v_1)\delta_{k\cdot(v_2-v_1-\hbar k)}\ud V +\int f'_1\tfrac{g_1-g_1'}{\hbar^2}h_2\V(k)^2\varpi(k,v_1)\delta_{k\cdot(v_2-v_1-\hbar k)}\ud  V\nonumber\\
					&=\int \tfrac{f'_1g_1'-f_1g_1}{\hbar^2}h_2\V(k)^2\varpi(k,v_1)\delta_{k\cdot(v_2-v_1-\hbar k)}\ud V+\tfrac{1}{\hbar}\int h_2f'_1\nabla g'_1\cdot k\V(k)^2\varpi(k,v_1)\delta_{k\cdot(v_2-v_1-\hbar k)}\ud V\label{viii}\\
					&+\int_{0}^1s\int h_2f'_1\nabla^2 g(v_1+s\hbar k): k^{\otimes2}\V(k)^2\varpi(k,v_1)\delta_{k\cdot(v_2-v_1-\hbar k)}\ud V\label{ix}.
				\end{align}
				The term \eqref{viii} is bounded by Propositions \ref{prop:lemme d'annulation} and \ref{prop:sum of the collision kernel}, and the term \eqref{ix} is treated by applying Proposition \ref{lem:borne general integral}.

				\step{4} Proof of \eqref{eq:estim non lin 4}. Applying a Taylor expansion,
				\begin{multline*}
					\left\<f,Q_\hbar(g,h,\Re(\varpi \gr{H}(p)))\right\>= \int \tfrac{\V(k)^2}{\hbar}f_1(h'_2\nabla g_1-g_1\nabla h_2')\cdot k\Re\big(\varpi(k,v_1)\gr{H}_{\hbar k}(p)(v_1)\big)\delta_{k\cdot(v_2-v_1-\hbar k)}\ud v\\
					+\int_0^1\int \V(k)^2f_1((1-s)h_2\nabla^2 g(v_1+s\hbar k)+sg_1\nabla^2 h(v_2-\hbar k)):k^{\otimes 2}\\
					\times\Re\big(\varpi(k,v_1)\gr{H}_{\hbar k}(p)(v_1)\big)\delta_{k\cdot(v_2-v_1-\hbar k)}\ud V\ud s,
				\end{multline*}
				
				Using Lemma \ref{lem:lemme technique} to treat the first and Proposition \ref{lem:borne general integral} to treat  the second term, we obtain the expected estimation.
			\end{proof}

			The following estimation provides the regularity gain provided by the dissipation.
			\begin{prop}\label{prop:recuperation de coerciviter}
				Fix $\varpi\in\mathcal{B}_\hbar^2(\mathbb{R})$ such that $\inf_{k,v_1}\tfrac{1}{\varpi(k,v_1)}\geq \frac{1}{C_0}$ for some $C_0>1$.
				
				There exists a constant  $C_{\Phi_0}>1$ depending only on $\Phi_0$ such that for any measurable functions $f,g$, for any $K>1$,
				\begin{multline}\label{eq:recuperation de coerciviter}
					\left\<Q_\hbar(f,g,\varpi), f\right\> +\frac{1}{C_0}A_\hbar(f)\\
					\lesssim \|\varpi\|_{\mathcal{B}^2}\left(A_\hbar(f)\left(\|\<v\>^{\tfrac{d}{2}+4}(\Phi_0-g)\|_{H^2}+K^{-1}\|\<v\>^{\tfrac{d}{2}+4}g\|_{H^2}\right)+K^4\|f\|_{L^2}^2\<v\>^{\tfrac{d}{2}+4}g\|_{H^2}\right).			
				\end{multline}
			\end{prop}
			\begin{proof}
				Using the change of variable $(v_1,v_2,k)\mapsto(v_1',v_2',-k)$
				\begin{align*}
					\big\<Q_\hbar(f,g,\varpi), f\big\>
					=&\frac{1}{\hbar^2}\int f_1 (g'_2f'_1-g_2f_1) \V(k)^2 \varpi(k,v_1) \delta_{k\cdot (v_2-v_1-\hbar k)}\ud V\\
					=&\frac{1}{\hbar^2}\int (f'_1-f_1) g_2f_1 \V(k)^2 \varpi(k,v_1) \delta_{k\cdot (v_2-v_1-\hbar k)}\ud V\\
					=&-\frac{1}{2\hbar^2}\int (f'_1-f_1)^2 g_2 \V(k)^2 \varpi(k,v_1) \delta_{k\cdot (v_2-v_1-\hbar k)}\ud V\\
					&+\frac{1}{\hbar^2}\int ((f'_1)^2-(f_1)^2) g_2 \V(k)^2 \varpi(k,v_1) \delta_{k\cdot (v_2-v_1-\hbar k)}\ud V
				\end{align*}
				
				The last term is bounded by Proposition \ref{prop:lemme d'annulation}.
				
				Using Proposition \ref{cor:equivalence des normes}, there exists a constant $C>1$ such that
				\begin{align*}
					\frac{1}{\hbar^2}\int (f'_1-f_1)^2 g_2 &\V(k)^2 \varpi(k,v_1) \delta_{k\cdot (v_2-v_1-\hbar k)}\ud V \\
					\leq&\frac{1}{\hbar^2}\int (f'_1-f_1)^2 \Phi_0(v_2) \V(k)^2 \varpi(k,v_1) \delta_{k\cdot (v_2-v_1-\hbar k)}\ud V\\
					&+\frac{1}{\hbar^2}\int (f'_1-f_1)^2 \left|\Phi_0(v_2)-g(v_2)\right| \V(k)^2 \varpi(k,v_1) \delta_{k\cdot (v_2-v_1-\hbar k)}\ud V\\
					\leq& -\frac{1}{C_0}A_\hbar(f)+C\|\varpi\|_{\mathcal{B}^2}(A_\hbar(f)+\|f\|^2_{L^2})\|\<v\>^{\tfrac{d}{2}+4}(g-\Phi_0)\|_{H^2}.
				\end{align*}
				This conclude the proof.
			\end{proof}
			
			\begin{prop}
				 Consider $\varpi\in \mathcal{B}_\hbar^1(\mathbb{R})$, and $f,g,h\in L^2\mathbb{R}^d$.  We have				
				\begin{gather}\label{eq:estim comutation 1}
					\big|\big\<f,Q_\hbar(\<v\>^mg,h,\varpi)-\<v\>^mQ_\hbar(g,h,\varpi)\big\>\big|\lesssim  \|\varpi\|_{\mathcal{B}^0}\|f\|_{L^2} \|g\|_{H^2}\|\<v\>^{\frac d2+4}h\|_{L^2},\\
					\big|\big\<f,Q_\hbar(\<v\>^mg,h,\varpi)-\<v\>^mQ_\hbar(g,h,\varpi)\big\>\big|\lesssim  \|\varpi\|_{\mathcal{B}^0}\|f\|_{L^2} (A_\hbar(\<v\>^{m-1}g)^{\frac12}+\|\<v\>^{m-1}g\|_{L^2})\|\<v\>^{\frac d2+4}h\|_{H^1}\label{eq:estim comutation 2}.
				\end{gather}
				
				Consider $\varpi\in \mathcal{B}_\hbar^1(\mathbb{C})$, and $f,g,h,p\in L^2\mathbb{R}^d$.  We have for any $K>1$
				\begin{multline}\label{eq:estim comutation 4}
					\left|\left\<f,Q_\hbar(\<v\>^mg,h,\Re(\varpi\gr{H}_{\hbar k}(p)))-\<v\>^mQ_\hbar(g,h,\Re(\varpi\gr{H}_{\hbar k}(p)))\right\>\right|\\
					\lesssim \|\varpi\|_{\mathcal{B}^1}\Big(K^4\|f\|_{L^2}+K^{-1}A_\hbar(f)\Big)\|\<v\>^{m-1}g\|_{H^2}\|\<v\>^{\frac d2+4}h\|_{H^2}\left(A_{\hbar}(\<v\>^{\frac d2+4}p)^{\frac12}+\|\<v\>^{\frac d2+4}p\|_{L^2}\right).
				\end{multline}
			\end{prop}
			
			\begin{proof}
			Perform the decomposition
			\begin{multline*}
				\left\<f,Q_\hbar(\<v\>^mg,h,\varpi)-\<v\>^mQ_\hbar(g,h,\varpi)\right\>=\frac{1}{\hbar^2}\int f'_1g_1h_2(\<v_1\>^m-\<v_1'\>^m)\V(k)^2\varpi(k,v_1)\delta_{k\cdot(v_2-v_1-\hbar k)}\ud V\\
				= J_1 + J_2+J_3
			\end{multline*}
			where
			\begin{gather}
				J_1:=\frac{m}{\hbar}\int f'_1g'_1h_2\<v'_1\>^{m-2}v'_1\cdot k \V(k)^2\varpi(k,v_1)\delta_{k\cdot(v_2-v_1-\hbar k)}\ud V\\
				J_2 := \frac{m}{\hbar}\int f'_1h_2(g_1\<v_1\>^{m-2}v_1-g'_1\<v'_1\>^{m-2}v'_1)\cdot k \V(k)^2\varpi(k,v_1)\delta_{k\cdot(v_2-v_1-\hbar k)}\ud V\\
				J_3:=\int_0^1(1-s)\int f'_1g_1h_2 W_m(v_1+s\hbar k):k^{\otimes 2} \V(k)^2\varpi(k,v_1)\delta_{k\cdot(v_2-v_1-\hbar k)}\ud V
			\end{gather}
			where $W_m(v):=m\<v\>^{m-4}(\<v\>^2 + (m-2)v^{\otimes 2})$. Using that $|\hbar k|\leq |v_1-v_2|\leq |v_1|+|v_2|$,
			\[|W_m(v_1+s\hbar k)|\lesssim \<v_1\>^{m-2}+\<\hbar k\>^{m-2}\lesssim \<\hbar k\>^{\frac d2-2}(\<v_1\>^{m-2}+\<v_2\>^{m-\frac d2}).\]
			
			We deduce that 
			\begin{equation*}
				|J_3|\lesssim \int \left|f'_1\right|\left|\<v_1\>^{m-2}g_1\right|\left|\<v_2\>^{m-\frac d2}h_2\right|  \frac{|\varpi(k,v_1)|}{\<k\>^{\frac{3}{2}d+3+\eta}}\delta_{k\cdot(v_2-v_1-\hbar k)}\ud V.
			\end{equation*}
			
			\step{1} Proof of \eqref{eq:estim comutation 1} and \eqref{eq:estim comutation 2}. Applying Proposition \ref{lem:borne general integral},
			\begin{gather*}
				J_2
				\lesssim \|\varpi\|_{\mathcal{B}_0}\|f\|_{L^2}\min((A_\hbar(\<v\>^{m-1}g)^{\frac12}+\|\<v\>^{m-1}g\|_{L^2})\|\<v\>^{\frac d2+4}h\|_{H^1},\|\<v\>^{m-1}g\|_{H_2}\|\<v\>^{\frac d2+4}h\|_{L^2}),\\
				J_3\lesssim \|\varpi\|_{\mathcal{B}^0}\|f\|_{L^2} \min\left(\|\<v\>^{m}g\|_{H^1}\|\<v\>^{m}h\|_{L^2},\|\<v\>^{m}g\|_{L^2}\|\<v\>^{m}h\|_{H^1}\right).
			\end{gather*}
			
%
%
			Using the symmetry $(v_1,v_2,k)\mapsto(v_1',v_2',-k)$, 
			\begin{align*}
				J_1=\frac{m}{2\hbar^2}\int (h_2f'_1g'_1\<v'_1\>^{m-2}v'_1-h'_2f_1g_1\<v_1\>^{m-2}v'_1)\V(k)^2\varpi(k,v_1)\delta_{k\cdot(v_2-v_1-\hbar k)}\ud V
			\end{align*}
			Hence,
			\begin{align*}
				|J_1|&\lesssim \|\varpi\|_{\mathcal{B}^0}\int \left|\tfrac{h_2-h_2'}{\hbar}\right| \left|f'_1\right|\left|g'_1\<v'_1\>^{m-1}\right| \V(k)^2\delta_{k\cdot(v_2-v_1-\hbar k)} \ud V\\
				&+\|\varpi\|_{\mathcal{B}^0}\int \left|h'_2\right| \left|\tfrac{f'_1-f_1}{\hbar}\right|\left|g'_1\<v'_1\>^{m-1}\right| \V(k)^2\delta_{k\cdot(v_2-v_1-\hbar k)} \ud V\\
				&+\|\varpi\|_{\mathcal{B}^0}\int \left|h'_2\right| \left|f'_1\right|\left|\tfrac{g_1\<v_1\>^{m-2}v_1-g'_1\<v'_1\>^{m-2}v_1'}{\hbar}\right|\V(k)^2\delta_{k\cdot(v_2-v_1-\hbar k)} \ud V.
			\end{align*}
			We can conclude by applying Proposition \ref{lem:borne general integral}.
%
%
			
			\step{2}  For $\tilde{\varpi}\in\mathcal{B}^2(\mathbb{C})$ and $p\in L^2(\mathbb{R}^d)$, we denote $\varpi := \Re(\tilde{\varpi} \gr{H}(p))$. We can perform the same decomposition 
			\begin{align*}
			\Big\<f,Q_\hbar(\<v\>^mg,h,\Re(\tilde\varpi\gr{H}_{\hbar k}(p)))-\<v\>^mQ_\hbar(g,h,\Re(\tilde\varpi\gr{H}_{\hbar k}(p))\Big\>=J_1+J_2+J_3
			\end{align*}
			
			Performing a Taylor expansion of $(g_1\<v_1\>^{m-2}v_1-g'_1\<v'_1\>^{m-2}v'_1)$ and applying Proposition \ref{lem:borne general integral},
			\begin{equation*}
				|J_2+J_3|\leq \|\tilde\varpi\|_{\mathcal{B}^0} \|f\|_{L^2}\|\<v\>^mg\|_{H^2}\|\<v\>^mh\|_{H^2}(A_{\hbar}(p)^{\frac 12}+\|p\|_{L^2}).
			\end{equation*}
						
			Using  Proposition \ref{lem:lemme technique}, for any $K>1$
			\[J_1\lesssim \|\tilde\varpi\|_{\mathcal{B}^2}\Big(K^4\|f\|_{L^2}+K^{-1}A_{\hbar}(f)^{\tfrac{1}{2}}\Big)\|\<v\>^{m-1}g\|_{H^2}\|\<v\>^{\frac d2+4}h\|_{H^2}\left(A_{\hbar}\<v\>^{\frac d2+4}(p)^{\frac12}+\|\<v\>^{\frac d2+4}p\|_{L^2}\right).\]
			
			This conclude the proof.			
			\end{proof}

			In order to obtain a short time existence (via a Picard fix point argument), we need the following Lipschitz stability estimates on the collision operator:
			\begin{prop}\label{prop:estim lip}
				Fix $N\geq 4$, $m\geq \frac d2+4$, and $\alpha_1,\alpha_2,\alpha_3\in\mathbb{N}^d$ such that $|\sum_i\alpha_i|\leq N$. Fix $\Phi_0$ such that $\|\<v\>^m \Phi_0\|_{H^N}$ and $\|\tfrac{1}{|\e_\hbar(\Phi_0|)}\|_{L^\infty}$ are bounded. There exists a constant $C_{\Phi_0}>0$ (depending on $\Phi_0$) such that for any $f,g,h$,
				\begin{multline}\label{eq:estim lip}
					\left\<f, \<v\>^m \left(Q_\hbar(\partial^{\alpha_1} g,\partial^{\alpha_2} g,\partial^{\alpha_3}_v\tfrac{1}{|\e_{\hbar}(g)|^2})-Q_\hbar(\partial^{\alpha_1} h,\partial^{\alpha_2} h,\partial^{\alpha_3}_v\tfrac{1}{|\e_{\hbar}(h)|^2})\right)\right\>\\
					\leq \frac{1}{\hbar^2}\max_{F\in\{g,h\}}\left(\frac{C_{\Phi_0}}{\left(1-C_{\Phi_0}\|\<v\>^{\ell}(F-\Phi_0)\|_{H^N}\right)^{2N+2}}\right)\|f\|_{H_N}\|\<v\>^m(g-h)\|_{H_N}.
				\end{multline}
			\end{prop}
			\begin{proof}
				We recall the Faà di Bruno formula,
				\begin{align*}
					\nabla^{|\alpha_3|}_{v_1}\tfrac{1}{|\e_\hbar(g,k,v_1)|^2}= \sum_{k,l = 0}^{|\alpha_3|} \sum_{\substack{\gamma_1+\cdots+ \gamma_k\\+ \tilde{\gamma}_1+\cdots+\tilde{\gamma}_l=|\alpha_3|\\ \gamma_i,\tilde{\gamma_i}\geq 1}} \frac{(-1)^{k+l}r_1!k!l!}{(\gamma_1+\cdots \gamma_k)!(\tilde{\gamma}_1+\cdots+\tilde{\gamma}_l)!} \frac{\prod_{j= 1}^k\nabla_{v_1}^{\gamma_j}\e_\hbar(g,k,v_1)\prod_{j= 1}^l\nabla_{v_1}^{\tilde{\gamma}_j}\overline{\e_\hbar(g,k,v_1)}}{\e_\hbar(g,k,v_1)^{k+1}\overline{\e_\hbar(g,k,v_1)}^{l+1}}
				\end{align*}
				
				As $\nabla^{\ell}_{v_1}\e(g,k,v_1) = \V(k)\gr{H}_{\hbar k}\left(\nabla^\ell g\right)(v_1)$, there exists three functions $(\varpi^{\alpha_3}_\ell(g))_{-1\leq \ell\leq 1}$ depending on $\V(k)$ and $(\gr{H}_{\hbar k}(\nabla^rg))_{r\leq N-2}$
				\begin{equation*}
					\partial_{v_1}^{\alpha_3}\tfrac{1}{|\e_\hbar(g,k,v_1)|^2} = \hat{k}^{\otimes \alpha_3}:\nabla^{|\alpha_3|}_{v_1}\tfrac{1}{|\e_\hbar(g,k,v_1)|^2}
					=: \tilde{\varpi}^{\alpha_3}(g,k,v_1)+ \sum_{\ell=0}^1\Re\left(\varpi^{\alpha_3}_{\ell}(g,k,v_1)\gr{H}_{\hbar k} (\nabla^{N-\ell}g)\right).
				\end{equation*}
				
				One can easily check that $\tilde{\varpi}^{\alpha_3}(g)$ and $(\varpi^{\alpha_3}_\ell(g))_{\ell\leq 3}$ verify the conditions \eqref{eq:symetry de eta} and $\varpi^{\alpha_3}_\ell(g)$ vanishes if $|\alpha_3|< N-\ell$. In addition, as $\tilde{\varpi}^{|\alpha_3|}(g)$ depends only on $(\gr{H}_{\hbar k}(\nabla^rg))_{r\leq N-2}$, using a variation of the proof of Proposition \ref{prop:borne sur epsilone mieux} we have
				\begin{gather}
					\|\tilde{\varpi}^{\alpha_3}(g)\|_{\mathcal{B}^0} \leq C_{\Phi_0} \frac{1+\|\<v\>^{\frac d2}g\|^{|\alpha_3|}_{H^N}}{\left(1-C_{\Phi_0}\|\<v\>^{\frac d2}(g-\Phi_0)\|_{H^{2}}\right)^{2|\alpha_3|+2}}\\
					\|\tilde{\varpi}^{\alpha_3}(g)-\tilde\varpi^{\alpha_3}_j(h)\|_{\mathcal{B}^0} \leq \max_{F\in\{g,h\}}\left(\frac{C_{\Phi_0}\left(1+\|\<v\>^{\frac d2}F\|^{|\alpha_3|}_{H^N}\right)}{\left(1-C_{\Phi_0}\|\<v\>^{\frac d2}(F-\Phi_0)\|_{H^{2}}\right)^{2|\alpha_3|+2}}\right) \|\<v\>^{\frac d2+4}(g-h)\|_{H^N}.
				\end{gather}
				In the same way, as $\varpi_\ell(g)$ only depends on the $(\gr{H}_{\hbar k}(\nabla^rg))_{r\leq 2}$, we have
				\begin{gather}
					\|\varpi^{\alpha_3}_j(g)\|_{\mathcal{B}^2} \leq C_{\Phi_0} \frac{1+\|\<v\>^{\frac d2}g\|^{|\alpha_3|}_{H^4}}{\left(1-C_{\Phi_0}\|\<v\>^{\frac d2}(g-\Phi_0)\|_{H^{2}}\right)^{2|\alpha_3|+2}}\\
					\|\varpi^{\alpha_3}_j(g)-\varpi^{\alpha_3}_j(h)\|_{\mathcal{B}^2} \leq \max_{F\in\{g,h\}}\left(\frac{C_{\Phi_0}\left(1+\|\<v\>^{\frac d2}F\|^{|\alpha_3|}_{H^4}\right)}{\left(1-C_{\Phi_0}\|\<v\>^{\frac d2}(F-\Phi_0)\|_{H^{2}}\right)^{2|\alpha_3|+2}}\right) \|\<v\>^{\frac d2+4}(g-h)\|_{H^4}.
				\end{gather}
				
				We obtain the decomposition
				\begin{equation}\label{eq:decomposition des derive de epsilon}
					Q_\hbar(\partial^{\alpha_1}g,\partial^{\alpha_2}g,\partial^{\alpha_3}\tfrac{1}{|\e_{\hbar}(g)|^2}) =Q_\hbar(\partial^{\alpha_1}g,\partial^{\alpha_2}g,\tilde{\varpi}(g))+\sum_{\ell=0}^1 Q_\hbar(\partial^{\alpha_1}g,\partial^{\alpha_2}g,\Re(\varpi_\ell^{\alpha_3}(g)\gr{H}(\nabla^{N-\ell}g)).
				\end{equation}
				
				Using that \ref{eq:estim bete A1} and that $Q_\hbar$ is trilinear and
				\begin{itemize}
					\item \eqref{eq:estim non lin 2} and \eqref{eq:estim comutation 2} if $\max(|\alpha_1|,|\alpha_2|)\leq N-1$ (remark that $|\alpha_1|=|\alpha_2|=N-1$ is impossible since $N\geq 4$),
					\item \eqref{eq:estim non lin 1} and \eqref{eq:estim comutation 2} if $|\alpha_1|=N$ (then $|\alpha_2| \leq N-2$ and $|\alpha_3|\leq N-4$),
					\item \eqref{eq:estim non lin 3} and \eqref{eq:estim comutation 1} if $|\alpha_2|=N$ (then $|\alpha_1| \leq N-2$ and $|\alpha_3|\leq N-4$),
				\end{itemize}
				we obtain
				\begin{multline*}
					\left|\left\<f, \<v\>^l \left(Q_\hbar(\partial^{\alpha_1} g,\partial^{\alpha_2} g,\tilde{\varpi}^{\alpha_3}(g))-Q_\hbar(\partial^{\alpha_1} h,\partial^{\alpha_2} h,\tilde{\varpi}^{\alpha_3}(h))\right)\right\>\right|\\
					\leq \frac{1}{\hbar^2}\max_{F\in\{g,h\}}\left(\frac{C_{\Phi_0}}{\left(1-C_{\Phi_0}\|\<v\>^{\ell}(F-\Phi_0)\|_{H^N}\right)^{2N+2}}\right)\|\<v\>^lf\|_{H_N}\|\<v\>^l(g-h)\|_{H_N}.
				\end{multline*}
				
				If $\varpi_\ell(g)$ is non-zero, one has $|\alpha_3|\geq N-1$. We deduce that $|\alpha_1|\leq 1\leq N-3$ and $|\alpha_2|\leq 1\leq N-3$. Hence, one can apply \eqref{eq:estim non lin 4} and \eqref{eq:estim comutation 4} and obtain 
				\begin{multline*}
					\sum_{\ell=0}^1\left|\left\<f, \<v\>^l \left(Q_\hbar(\partial^{\alpha_1} g,\partial^{\alpha_2} g,\Re\left(\varpi_\ell^{\alpha_3}(g)\gr{H}(\nabla^{N-\ell}g)\right))-Q_\hbar(\partial^{\alpha_1} h,\partial^{\alpha_2} h,\Re\left(\varpi_\ell^{\alpha_3}(h)\gr{H}(\nabla^{N-\ell}h)\right)\right)\right\>\right|\\
					\leq \frac{1}{\hbar^2}\max_{F\in\{g,h\}}\left(\frac{C_{\Phi_0}}{\left(1-C_{\Phi_0}\|\<v\>^{\ell}(F-\Phi_0)\|_{H^N}\right)^{2N+2}}\right)\|\<v\>^lf\|_{H_N}\|\<v\>^l(g-h)\|_{H_N}.
				\end{multline*}
				This concludes the proof.
			\end{proof}
			
		\subsection{Cauchy theory for non-pertubative solution of the quantum Lenard-Balescu equation}
			The goal of the present section is to prove Theorem \ref{thm:Cauchy theory bis}. We split the proof into two steps: first we prove a local existence result, and in a a second time, we will provide some energy estimation.
			\begin{prop}\label{prop:existence local}
				Fix $N\geq 4$, $m\geq \frac d2+4$. Fix $\Phi_0$ such that $\|\<v\>^m \Phi_0\|_{H^{N+1}}$ and $\|\tfrac{1}{|\e_0(\Phi_0)|}\|_{L^\infty}$ are bounded. There exist two constants $T>0$ and $\delta>0$ depending on $\Phi_0$ but independent from $\hbar$ such that for any $\varphi_0$ with $\|\<v\>^m \varphi_0\|_{H^N}<\delta$, there exists a unique solution $\Phi_0+\varphi$  of the equation \eqref{eq:equation de Lenard Balescu Quantique} with initial data $\Phi_0+\varphi_0$. In addition, $\<v\>^m\varphi\in\mathcal{C}([0,\tfrac{T}{\hbar^2}],H^N)$.
			\end{prop}
			\begin{proof}
				First, as $\|\tfrac{1}{\e_0(\Phi_0)}\|_{L^\infty}$ and $\|\<v\>^m \Phi_0\|_{H^{N+1}}$ are bounded, from Proposition \ref{prop:borne sur epsilone mieux}, for some positive constant $C>0$
				\[{|\e_\hbar(\Phi_0,k,v_1)| }\geq \left\|\tfrac{1}{|\e_0(\Phi_0|)}\right\|_{L^\infty}^{-1}-C\hbar \|\<v\>^m \Phi_0\|_{H^{N+1}},\]
				which is bigger than $\|\tfrac{1}{2\e_0(\Phi_0)}\|_{L^\infty}$ for $\hbar$ small enough.
				
				We recall the Quantum Lenard-Balescu equation 
				\begin{equation*}
					\partial_t \varphi(t) = Q_\hbar\left(\varphi(t),\varphi(t),\tfrac{1}{|\e_\hbar(\varphi(t))|^2}\right),~~
					\varphi(t=0)=\Phi_0+\varphi_0.
				\end{equation*}
				
				We recall the constant $C_{\Phi_0}$ defined in Proposition \ref{prop:estim lip}. We set 
				\[T:=\frac{1}{3^N 2^{2N+3}C_{\Phi_0}},~\delta := \frac{1}{4C_{\Phi_0}}.\]
				
				We will use a Picard fix-point argument: defining $\3 g\3  = \sup_{t\in [0,\frac{T}{\hbar^2}]}\|\<v\>^m g(t,v)\|_{H_N(\ud v)}$, it is sufficient to prove that the application
				\begin{align*}
					\mathcal{F}:g\mapsto\left[(t,v)\mapsto\Phi_0+\varphi_0+\int_0^tQ_\hbar\left(g(s),g(s),\tfrac{1}{|\e_\hbar(g(s))|^2}\right)  \right]
				\end{align*}
				is stable and contractant on 
				\[\mathcal{E}:=\left\{g|\<v\>^m g\in\mathcal{C}([0,\tfrac{T}{\hbar^2}],H^N),~g(0,v)=\Phi_0+\varphi_0,~\3 g-\Phi_0\3\leq 2\delta \right\}.\]
				
				We only show the continuity estimates (stability is performed in the same way): for $g,h\in \mathcal{E}$, for all $t\in [0,\tfrac{T}{\hbar^2}]$, $\alpha\in \mathbb{N}^d$ with $|\alpha|\leq N$
				\begin{multline*}
					\|\<v\>^m\partial^\alpha(\mathcal{F}g(t)-\mathcal{F}h(t))\|_{L^2}^2\\=\int_0^t\Big\<\<v\>^l\partial^\alpha(\mathcal{F}g(t)-\mathcal{F}h(t)),
					\<v\>^m\partial^\alpha\Big(Q_\hbar\left(\Phi_0+g(s),\Phi_0+g(s),\tfrac{1}{|\e_\hbar(\Phi_0+g(s))|^2}\right)\\-Q_\hbar\left(\Phi_0+h(s),\Phi_0+h(s),\tfrac{1}{|\e_\hbar(\Phi_0+h(s))|^2}\right)\Big)\Big\>\ud s.
				\end{multline*}
				
				Using that
				\begin{equation*}
					\partial_v^{\alpha}Q_\hbar(f,g,\varpi) = \sum_{\alpha_1+\alpha_2+\alpha_3 = \alpha}\binom{\alpha}{\alpha_1,\alpha_2,\alpha_3}Q_\hbar(\partial_v^{\alpha_1}f,\partial_v^{\alpha_2}g,\partial_v^{\alpha_3}\varpi),
				\end{equation*}
				and \eqref{eq:estim lip}, we obtain that 
				\begin{equation*}
					\|\<v\>^m\partial^\alpha(\mathcal{F}g(t)-\mathcal{F}h(t))\|_{L^2}^2 \leq \frac{3^N 2^{2N+2}C_{\Phi_0}}{\hbar^2} \|\<v\>^l\partial^\alpha(\mathcal{F}g(t)-\mathcal{F}h(t))\|_{L^2}^2\int_0^t \|\<v\>^m (g(s)-h(s))\|_{H^N}\ud s
				\end{equation*}
				and finally, passing to the supremum in  $t$ and $\alpha$,
				\begin{equation*}
					\3\mathcal{F}g-\mathcal{F}h\3 \leq T3^N 2^{2N+2}C_{\Phi_0}\3 g-h\3\leq \tfrac{1}{2} \3 g-h\3.
				\end{equation*}
			\end{proof}
			
			We prove now that the validity time does not depend on $\hbar$. We only need a long time energy estimate:
			\begin{prop}
				Fix $N\geq 4$, $m\geq \frac d2+4$. Fix $\Phi_0$ such that $\|\<v\>^m \Phi_0\|_{H^{N+1}}$ and $\|\tfrac{1}{|\e_0(\Phi_0)}\|_{L^\infty}$ are bounded. There exist a constant $\tilde{T}>0$  depending on $\Phi_0$ but independent from $\hbar$ such that any solution $\Phi_0+\varphi(t)$  of \eqref{eq:equation de Lenard Balescu Quantique} with initial data $\Phi_0$ satisfies
				\[\sup_{t\in[0,\tilde{T}]} \| \<v\>^m(\varphi(t)-\Phi_0)\|_{H^{N}}\leq \delta,\]
				where $\delta$ is defined in Proposition \ref{prop:existence local}.
			\end{prop}
			\begin{proof}
				Without loss of generality, one can suppose that $\delta <1$. We will show that 
				\[\tilde{T}:=\sup\left\{t>0,~\forall s\leq t, ~\| \<v\>^m (\varphi(t)-\Phi_0)\|_{H^{N}}\leq \delta\leq 1\right\}.\]
				
				We recall the Equation \eqref{eq:equation de Lenard Balescu Quantique}:
				\begin{equation*}
					\partial_t \varphi(t) = Q_\hbar\left(\varphi(t),\varphi(t),\tfrac{1}{|\e_\hbar(\varphi(t))|^2}\right),~
					\varphi(t=0)=\Phi_0.
				\end{equation*}
				
				We deduce that for $\alpha\in \mathbb{N}^d$ with $|\alpha|\leq N$,
				\begin{multline*}
					\tfrac12\tfrac{\ud}{\ud t} \|\<v\>^m \partial^\alpha(\varphi(t)-\Phi_0)\|_{L^2}^
					2\\
					=\sum_{\alpha_1+\alpha_2+\alpha_3 = \alpha}\binom{\alpha}{\alpha_1,\alpha_2,\alpha_3} \left\<\<v\>^m \partial^\alpha(\varphi(t)-\Phi_0),\<v\>^mQ_\hbar\left(\partial^{\alpha_1}\varphi(t),\partial^{\alpha_2}\varphi(t),\partial^{\alpha_3}\tfrac{1}{|\e_\hbar(\varphi(t))|^2}\right)\right\>.
				\end{multline*}
				
				Using the same decomposition of $Q_\hbar\left(\partial^{\alpha_1}\varphi(t),\partial^{\alpha_2}\varphi(t),\partial^{\alpha_3}\tfrac{1}{|\e_\hbar(\varphi(t))|^2}\right)$ defined in \eqref{eq:decomposition des derive de epsilon}, we obtain
				\begin{multline*}
					\tfrac{1}{2}\tfrac{\ud}{\ud t} \|\<v\>^m \partial^\alpha(\varphi(t)-\Phi_0)\|_{L^2}^
					2\\
					=\sum_{\substack{\alpha_1+\alpha_2+\alpha_3 = \alpha}}\binom{\alpha}{\alpha_1,\alpha_2,\alpha_3} \bigg[\left\<\<v\>^m \partial^\alpha(\varphi(t)-\Phi_0),\<v\>^mQ_\hbar\left(\partial^{\alpha_1}\varphi(t),\partial^{\alpha_2}\varphi(t),\tilde{\varpi}^{\alpha_3}(\varphi(t))\right)\right\>\\
					+\sum_{\substack{\ell\in\{0,1\}\\|\alpha_3|\geq N-\ell-1}}\left\<\<v\>^m \partial^\alpha\varphi(t),\<v\>^mQ_\hbar\left(\partial^{\alpha_1}(\varphi(t)-\Phi_0),\partial^{\alpha_2}\varphi(t),\Re\left[\varpi^{\alpha_3}_{\ell}(\varphi(t))\gr{H}_{\hbar k} (\nabla^{N-\ell}\varphi(t))\right]\right)\right\>\bigg].
				\end{multline*}
				
				We split the set of the $(\alpha_1,\alpha_2,\alpha_3,\ell)$ into four peaces, 
				
				We begin by the case  $\alpha_1 = \alpha$. 
				
				Note that $\tilde{\varpi}^0:=\tfrac{1}{|\e_\hbar(\varphi(t))|^2}$. Using Propositions \ref{prop:borne sur epsilone mieux} and \ref{prop:recuperation de coerciviter}, there exists a constant $C_{\Phi_0}$ depending only on $\Phi_0$ (we recall that $\delta$ only depend on $\Phi_0$) such that $\forall t\in[0,\tilde{T}[$ and $K>1$
					\begin{align*}
						\Big\<\<v\>^m \partial^\alpha(\varphi(t)&-\Phi_0),Q_\hbar\Big(\<v\>^m\partial^{\alpha}(\varphi(t)-\Phi_0),\varphi(t),\tfrac{1}{|\e_\hbar(\varphi(t))|^2}\Big)\Big\> +\frac{1}{C_{\Phi_0}}A_\hbar(f)\\
						&\lesssim A_\hbar\left(\<v\>^m \partial^\alpha(\varphi(t)-\Phi_0)\right)\left(\|\<v\>^{\tfrac{d}{2}+4}(\varphi(t)-\Phi_0)\|_{H^2}+\tfrac 1K\|\<v\>^{\tfrac{d}{2}+4}\varphi(t)\|_{H^2}\right)\\
						&\qquad\qquad+K^4\|\<v\>^m \partial^\alpha(\varphi(t)-\Phi_0)\|^2_{L^2}\|\<v\>^{\tfrac{d}{2}+4}\varphi(t)\|_{H^2}\\
						&\lesssim A_\hbar(\<v\>^\ell \partial^\alpha(\varphi(t)-\Phi_0))\left(\delta +\tfrac 1K\right)+ K^4.
					\end{align*}
					
					Applying \eqref{eq:estim non lin 2}, we get  $\forall t\in[0,\tilde{T}[$, $K>1$
					\begin{align*}
						&\Big\<\<v\>^m \partial^\alpha(\varphi(t)-\Phi_0),Q_\hbar\Big(\<v\>^m\partial^{\alpha}\Phi_0,\varphi(t),\tfrac{1}{|\e_\hbar(\varphi(t))|^2}\Big)\Big\>\\
						&\lesssim  A_\hbar(\<v\>^m \partial^\alpha(\varphi(t)-\Phi_0))^{\frac12}\left(A_\hbar(\<v\>^m \partial^\alpha\Phi_0)^{\frac12}+\|\<v\>^m \partial^\alpha\Phi_0\|_{L^2}\right)\|\<v\>^{\tfrac{d}{2}+4}\varphi(t)\|_{H^2}\\
						&\lesssim \tfrac{1}{K}A_\hbar(\<v\>^\ell \partial^\alpha(\varphi(t)-\Phi_0)) + K.
					\end{align*}
					
					Using \eqref{eq:estim comutation 2}, one has $\forall t\in[0,\tilde{T}[$, $K>1$
					\begin{align*}
						\bigg\<\<v\>^m \partial^\alpha&(\varphi(t)-\Phi_0),\<v\>^mQ_\hbar\Big(\partial^{\alpha}\varphi(t),\varphi(t),\tfrac{1}{|\e_\hbar(\varphi(t))|^2}\Big)- Q_\hbar\Big(\<v\>^m\partial^{\alpha}\varphi(t),\varphi(t),\tfrac{1}{|\e_\hbar(\varphi(t))|^2}\Big)\bigg\>\\
						&\lesssim \|\<v\>^{m} \partial^\alpha(\varphi(t)-\Phi_0)\|_{L^2}\Big(A_{\hbar}(\<v\>^{m-1} \partial^\alpha\varphi(t))^{\frac12}+\|\<v\>^{\ell-1} \partial^\alpha\varphi(t)\|_{L^2}\Big)\|\<v\>^{\frac d2+4}\varphi(t)\|_{H^1},\\
						&\lesssim \tfrac{1}{K} A_{\hbar}(\<v\>^{m} \partial^\alpha(\varphi(t)-\Phi_0)) + K.
					\end{align*}
					
					Finally, $\forall t\in[0,\tilde{T}[$,
					\begin{multline*}
						\Big\<\<v\>^m \partial^\alpha(\varphi(t)-\Phi_0),\<v\>^mQ_\hbar\Big(\partial^{\alpha}\varphi(t),\varphi(t),\tfrac{1}{|\e_\hbar(\varphi(t))|^2}\Big)\Big\>+\frac{1}{C_{\Phi_0}}A_\hbar(\<v\>^m \partial^\alpha(\varphi(t)-\Phi_0))\\
						\lesssim  A_\hbar\left(\<v\>^\ell \partial^\alpha(\varphi(t)-\Phi_0)\right)\left(\delta+\tfrac 1K\right) + K^4.
					\end{multline*} 
					
					In the same, using
					\begin{itemize} 
						\item estimations \eqref{eq:estim non lin 3} and \eqref{eq:estim comutation 1} for $\alpha_2 = \alpha$, 
						\item \eqref{eq:estim non lin 2} and \eqref{eq:estim comutation 2} if $\max(|\alpha_1|,|\alpha_2|)\leq N-1$ (remark that $|\alpha_1|=|\alpha_2|=N-1$ is impossible since $N\geq 4$),
				\end{itemize}
				we obtain for any $K>1$
				\begin{multline*}
					\Big\<\<v\>^m \partial^\alpha(\varphi(t)-\Phi_0),\<v\>^mQ_\hbar\Big(\partial^{\alpha_1}\varphi(t),\partial^{\alpha_2}\varphi(t),\tilde{\varpi}^{\alpha_3}(\varphi(t))\Big)\Big\>\\
					\lesssim \tfrac{1}{K}A_\hbar(\<v\>^m \partial^\alpha(\varphi(t)-\Phi_0))+ A_\hbar(\<v\>^m \partial^\alpha(\varphi(t)-\Phi_0))^{\frac12} + K^4.
				\end{multline*}
				
				If $\tilde{\varpi}^{\alpha_3}_\ell$ is non-zero, one has $|\alpha_3|\geq N-1$. We deduce that $|\alpha_1|\leq 1\leq N-3$ and $|\alpha_2|\leq 1\leq N-3$. Hence, one can apply \eqref{eq:estim non lin 4} and \eqref{eq:estim comutation 4} and obtain 
				\begin{multline*}
					\left|\left\<\varphi(t)-\Phi_0, \<v\>^l Q_\hbar\left(\partial^{\alpha_1} \varphi(t),\partial^{\alpha_2} \varphi(t),\Re\left[\varpi_\ell^{\alpha_3}(\varphi(t))\gr{H}(\nabla^{N-\ell}\varphi(t))\right]\right)\right\>\right|\\
					\lesssim (\tfrac{1}{K}+\delta )A_\hbar(\<v\>^m \partial^\alpha(\varphi(t)-\Phi_0))^{\frac12} + K^4.
				\end{multline*}
				
				Summing on $(\alpha_1,\alpha_2,\alpha_3,\ell)$, we obtain 
				\begin{align*}
					\tfrac{1}{2}\tfrac{\ud}{\ud t} \|\<v\>^m \partial^\alpha(\varphi(t)-\Phi_0)\|_{L^2}^
					2 +\tfrac{1}{C_{\Phi_0}}A_\hbar(\<v\>^m \partial^\alpha(\varphi(t)-\Phi_0))&\lesssim A_\hbar\left(\<v\>^\ell \partial^\alpha(\varphi(t)-\Phi_0)\right)\left(\delta+\tfrac 1K\right)
					+  K^4.
				\end{align*}
				for $K>1$, $\delta<1$. Choosing $K, \tfrac{1}{\delta}$ large enough , we deduce that for any $t\in[0,\tilde{T}[$
				\begin{equation}
					 \|\<v\>^m \partial^\alpha(\varphi(t)-\Phi_0)\|_{L^2}\leq \tilde{C}_{\Phi_0} \sqrt{t}
				\end{equation}
				for some large constant $\tilde{C}_{\Phi_0}$ depending on $\Phi_0$. This concludes the proof.
			\end{proof}
			
		\subsection{Semi-classical (or grazing collision) limit for non-perturbative solution}		
			For $g,h$ two functions $\mathbb{R}_v^d\to \mathbb{R}$, $\varpi\in\mathcal{B}_\hbar^0(\mathbb{R})$ (with $\hbar \geq 0$), we define
			\begin{equation}
				Q_0(f,g,\varpi)(v_1) :=  \div_{v_1}\int k \V(k)^2\varpi(k,v_1)k\cdot\left(\nabla f_1g_2-f_1\nabla g_2\right)\delta_{k\cdot(v_2-v_1)}\uud k \ud v_2.
			\end{equation}
			
			The classical Lenard-Balescu equation can be written as (we linearize around the initial data)
			\begin{gather}
				\partial_t \varphi_0 = Q_0\left(\varphi_0(t),\varphi_0(t),\tfrac{1}{|\e_0(\varphi_0(t))|^2}\right),
				\varphi_0(t=0)=\Phi_0.
			\end{gather}
			
			\begin{prop}
				For any $g,h\in H^3(\<v\>^{d+8}\ud v)$, $\varpi_\hbar\in \mathcal{B}_\hbar^3(\mathbb{R})$ and $\varpi_0\in \mathcal{B}_0^3(\mathbb{R})$. Then
				\begin{equation}\label{eq:estim op semi classique}
					\left|\<v\>^m Q_\hbar(g,h,\varpi_\hbar)-\<v\>^mQ_0(g,h,\varpi_0)\right|\leq \|f\|_{L^2}\|\<v\>^mg\|_{H^3}\|\<v\>^{m}h\|_{H^3} (\hbar\|\varpi_\hbar\|_{\mathcal{B}^2}+\|\varpi_\hbar-\varpi_0\|_{\mathcal{B}^1})
				\end{equation}
			\end{prop}
			\begin{proof}
				Performing the change of variable $v_2\mapsto v_2'$ and a Taylor expension, for any $f\in L^2(\mathbb{R}^d)$, (denoting $v_2'':=v_2+\hbar k$),
				\begin{align*}
					&\big\<f,\<v\>^m Q_\hbar(g,h,\varpi_\hbar)\big\> = \int f_1\<v_1\>^m \tfrac{g_1'h_2-g_1h_2''}{\hbar^2}\V(k)^2\varpi(k,v_1)\delta_{k\cdot(v_1-v_2)}\ud V\\
					&=\int f_1\<v_1\>^m \Big(\tfrac{h_2\nabla g_1-g_1\nabla h_2}{\hbar}\cdot k+\tfrac{h_2\nabla^2 g_1-g_1\nabla^2 h_2}{2}: k^{\otimes 2}\Big)\V(k)^2\varpi_\hbar(k,v_1)\delta_{k\cdot(v_1-v_2)}\ud V+\hbar\mathfrak{R}_1\\
					&=\int f_1\<v_1\>^m \Big(\tfrac{h_2\nabla g_1-g_1\nabla h_2}{2}\cdot k (\varpi_\hbar(k,v_1)-\underset{=\varpi_\hbar(k,v_1-\hbar k)}{\varpi_\hbar(-k,v_1))}+\tfrac{h_2\nabla^2 g_1-g_1\nabla^2 h_2}{2}: k^{\otimes 2}\varpi_\hbar(k,v_1)\Big)\V(k)^2\delta_{k\cdot(v_1-v_2)}\ud V+\hbar\mathfrak{R}_1	\\
					&=\int f_1\<v_1\>^m \Big(\tfrac{h_2\nabla g_1-g_1\nabla h_2}{2}\otimes\nabla_{v_1}\varpi_\hbar(k,v_1)+\tfrac{h_2\nabla^2 g_1-g_1\nabla^2 h_2}{2}\varpi(k,v_1)\Big): k^{\otimes 2}\V(k)^2\delta_{k\cdot(v_1-v_2)}\ud V+\hbar\big(\mathfrak{R}_1+\mathfrak{R}_2\big),
				\end{align*}
				where, denoting $(v_{1,s},v_{2,s}):=(v_1+s\hbar k,v_2-s\hbar k)$
				\begin{gather*}
					\mathfrak{R}_1 := \int_0^1\tfrac{(1-s)^2}{2}\int f_1\<v_1\>^m\tfrac{h_2\nabla^3g(v_{1,s})-g_1\nabla^3h(v_{2,s})}{2}:k^{\otimes 3}\varpi_\hbar(k,v_1)\V(k)^2\delta_{k\cdot(v_1-v_2)}\ud V,\\
					\mathfrak{R}_2 := \int_0^1(1-s)\int f_1\<v_1\>^m\tfrac{h_2\nabla g(v_1)-g_1\nabla h(v_2)}{2}\otimes\nabla^2\varpi_\hbar(k,v_1-s\hbar k):k^{\otimes 3}\V(k)^2\delta_{k\cdot(v_1-v_2)}\ud V.
				\end{gather*}
				
				Using that $k\cdot (v_2-v_{1,s}-s\hbar k)=0$, one has $|s\hbar k|\leq |v_{1,s}-v_2|$ and $\<v_1\>^m\leq \<v_{1,s}\>^m+\<s\hbar k\>^m\leq \<s\hbar k\>^{\frac{d}{2}}(\<v_{1,s}\>^m+\<v_2\>^{m-\frac d2})$. Hence,
				\begin{align*}
					|\mathfrak{R}_1| \lesssim\|\varpi\|_{\mathcal{B}^0}\int_0^1\int |f_1|\Big(\<v_2\>^{m-\frac d2}|h_2||\<v_{1,s}\>^m\nabla^3g(v_{1,s})|+\<v_1\>^m|g_1||\nabla^3h(v_{2,s})|\<k\>^{\frac{d}{2}+ 3}\V(k)^2\delta_{k\cdot(v_1-v_2)}\ud V
				\end{align*}
				
				Applying Proposition \ref{lem:borne general integral}, 
				\[|\mathfrak{R}_1 |+|\mathfrak{R}_2|\lesssim \|f\|_{L^2}\|\varpi\|_{\mathcal{B}^2}\|\<v\>^mg\|_{H^3}\|\<v\>^{\frac d2+4}\|_{H^3}.\]
				
				On the other hand, using that $(\nabla_{v_1}+\nabla_{v_2})\delta_{k\cdot(v_1+v_2)}=0$,
				\begin{align*}
					\big\<f,&\<v\>^m Q_0(g,h,\varpi_0)\big\> =\int f_1\<v_1\>^m \nabla_1\cdot\big(k \tfrac{h_2\nabla g_1-g_1\nabla h_2}{2}\cdot k\V(k)^2\varpi_0(k,v_1)\delta_{k\cdot(v_1-v_2)}\big)\ud V\\
					&=\int f_1\<v_1\>^m (\nabla_1+\nabla_2)\cdot\big(k \tfrac{h_2\nabla g_1-g_1\nabla h_2}{2}\cdot k\V(k)^2\varpi_0(k,v_1)\delta_{k\cdot(v_1-v_2)}\big)\ud V\\
					&=\int f_1\<v_1\>^m\big(\tfrac{h_2\nabla g_1-g_1\nabla h_2}{2}\otimes\nabla_{v_1}\varpi_0(k,v_1)+ \tfrac{h_2\nabla g_1-g_1\nabla h_2}{2}\varpi_0(k,v_1)\big):k^{\otimes 2}\V(k)^2\delta_{k\cdot(v_1-v_2)}\big)\ud V
				\end{align*}
				
				Applying again Proposition \ref{lem:borne general integral}, we can conclude.
			\end{proof}
			
			We can finally prove the theorem.
			\begin{proof}
				We fix $\Phi_0\in H^{N+1}(\<v\>^{2m}\ud v)$ and $\|\tfrac{1}{\e_0(\Phi_0)}\|$ are bounded. Thanks to Theorem \ref{thm:Cauchy theory bis}, there exists $T>0$ and $\hbar_0$ such that for any $\hbar <\hbar_0$, there exist $\varphi_\hbar(t)\in\mathcal{C}_b^0(H^N(\<v\>^{2m}\ud v))$ mild solution of \eqref{eq:equation de Lenard Balescu Quantique}, with  $\sup_{\hbar\in]0,\hbar_0[}\sup_{t\in[0,T]}\|g_\hbar(t)\|_r^2$ bounded. In \cite{DW}, the authors show the existence of $\varphi_\infty \in\mathcal{C}_b^0(H^N(\<v\>^{2m}\ud v))$ mild solution of \eqref{eq:Lenard--Balescu classique} with initial data $\varphi_\infty(t=0):=g_0$.
				
				We perform again an energy estimation:
				\begin{align*}
					\frac12\frac{\ud}{\ud{t}}\|&\<v\>^m(\varphi_\hbar(t)-g_\infty(t))\|^2 = \big\<\<v\>^m(\varphi_\hbar(t)-g_\infty(t)),\<v\>^mQ_\hbar (\varphi_\hbar(t)-\varphi_\infty(t),\varphi_\hbar(t),\tfrac{1}{|\e_\hbar(\varphi_\hbar(t))|^2})\big\>\\
					&+\big\<\<v\>^m(\varphi_\hbar(t)-g_\infty(t)),\<v\>^mQ_\hbar (\varphi_\infty(t),\varphi_\hbar(t)-\varphi_\infty(t),\tfrac{1}{|\e_\hbar(\varphi_\hbar(t))|^2})\big\>\\
					&+\big\<\<v\>^m(\varphi_\hbar(t)-g_\infty(t)),\<v\>^mQ_\hbar (\varphi_\infty(t),\varphi_\infty(t),\tfrac{1}{|\e_\hbar(\varphi_\hbar(t))|^2}-\tfrac{1}{|\e_\hbar(\varphi_\infty(t))|^2})\big\>\\
					&+\big\<\<v\>^m(\varphi_\hbar(t)-g_\infty(t)),\<v\>^m\big(Q_\hbar (\varphi_\infty(t),\varphi_\infty(t),\tfrac{1}{|\e_\hbar(\varphi_\infty(t))|^2})-Q_0 (\varphi_\infty(t),\varphi_\infty(t),\tfrac{1}{|\e_0(\varphi_\infty(t))|^2})\big)\big\>.
				\end{align*}
				
				The first line is bounded by \eqref{eq:recuperation de coerciviter} and \eqref{eq:estim comutation 2}, the second line by \eqref{eq:estim non lin 3} and \eqref{eq:estim comutation 2}, and the fourth by \eqref{eq:estim op semi classique}. In order to treat the third line, we remark that 
				\[\tfrac{1}{|\e_\hbar(\varphi_\hbar(t))|^2}-\tfrac{1}{|\e_\hbar(\varphi_\infty(t))|^2} = -\Re\left[ \tfrac{\V(k)^2(\overline{\e_\hbar(\varphi_\hbar(t))}+\overline{\e_\hbar(\varphi_\infty(t)))}}{|\e_\hbar(\varphi_\hbar(t)|^2|\e_\hbar(\varphi_\hbar(t)|^2}\gr{H}_{\hbar k}(\varphi_\hbar(t)-\varphi_{\infty}(t))\right],\]
				thus one can apply \eqref{eq:estim non lin 4} and \eqref{eq:estim comutation 4}. Finally, as $\|\<v\>^m\varphi_\hbar\|_{L^{\infty}([0,T],H^N)}$ and $\|\<v\>^m\varphi_\hbar\|_{L^{\infty}([0,T],H^N)}$ are both bounded, for any $t\in[0,T]$,
				\begin{multline*}
					\frac{\ud}{\ud{t}}\|\<v\>^m(\varphi_\hbar(t)-g_\infty(t))\|^2 +A_\hbar(\<v\>^m(\varphi_\hbar(t)-g_\infty(t))) \\
					\lesssim \|\<v\>^m(\varphi_\hbar(t)-g_\infty(t))\|^2+ \|\<v\>^m(\varphi_\hbar(t)-g_\infty(t))\|\left(A_\hbar(\<v\>^m(\varphi_\hbar(t)-g_\infty(t)))^{\frac12}+\hbar\right)
				\end{multline*}
				Applying Young inequality,
				\begin{align*}
					\frac{\ud}{\ud{t}}\|\<v\>^m(\varphi_\hbar(t)-g_\infty(t))\|^2\lesssim\|\<v\>^m(\varphi_\hbar(t)-g_\infty(t))\|^2+\hbar^2
				\end{align*}
				and by Gronwall inequality, there exists a constant $C$ such that for all $t\in[0,T]$,
				\begin{equation}
					\|\<v\>^m(\varphi_\hbar(t)-g_\infty(t))\|^2\leq \hbar^2 C(e^{Ct}-1).
				\end{equation}
				This concludes the proof.
			\end{proof}
			
			\appendix
			
			\section{Bound on some convolution operators}\label{sec:borne convolution}
			The goal of this section is to provide a proof of Lemma \ref{lem:convolution}:
			\begin{lemma}
				Fix $k\in\mathbb{R}^d$. Denoting $L_\perp^p(L_\para^q):=L^p(<k>^\perp,L^q(\mathbb{R}\hat{k}))$. There exists a constant $C>0$ such that for any function $f$ and any $K>1$, for any $d\geq2$, one has
				\begin{gather}
					\big\|\tfrac{\chi(K|v^\perp|)}{\hat k\cdot v +i0}\ast f\big\|_{L^{\infty}_{\perp}(L^2_\para)}\lesssim\ \left\|f\right\|_{L^{1}_{\perp}(L^2_\para)},~
					\big\|\tfrac{1-\chi(K|v^\perp|)}{\hat k\cdot v +i0}\ast f\big\|_{L^{\infty}_{\perp}(L^2_\para)}\lesssim K^{-\tfrac{1}{2}}\left\|f\right\|_{L^{2}},\label{xii}
				\end{gather}
				for any $d\geq 3$, 	for any $p\in[ 1,\tfrac{d-1}{d-2}]$ and $q=\tfrac{p(d-1)}{d-1-(d-2)p}$,
				\begin{gather} \label{xiii}
					\big\|\tfrac{\chi(K|v^\perp|)}{|v|^2}\ast f\big\|_{L^{q}_{\perp}(L^2_\para)}+\big\|\tfrac{\chi(K|v^\perp|)v}{(\hat k\cdot v +i0)|v|^2}\ast f\big\|_{L^{q}_{\perp}(L^2_\para)}+\big\|\tfrac{\chi(K|v^\perp|)|v^\perp|}{(\hat k\cdot v +i0)|v|^2}\ast f\big\|_{L^{q}_{\perp}(L^2_\para)}\leq C \left\|f\right\|_{ L^p_\perp( L^{2}_{\para})},
				\end{gather}
				for any $d\geq 4$, 	for any $p\in[ 1,\tfrac{d-1}{d-3}]$ and $q=\tfrac{p(d-1)}{d-1-(d-3)p}$,
				\begin{gather}\label{xiv}
					\big\|\tfrac{\chi(K|v^\perp|)}{(\hat k\cdot v +i0)|v|^2}\ast f\big\|_{L^{q}_{\perp}(L^2_\para)}\leq C \left\|f\right\|_{L^p_\perp (L^{2}_{\para})}
				\end{gather}
				for $d=2$ and for any $p,q\in [1,\infty]$ with $p>q$
				\begin{gather} 
					\big\|\tfrac{\chi(K|v^\perp|)}{|v|^2}\ast f\big\|_{L^{p}_{\perp}(L^2_\para)}+\big\|\tfrac{\chi(K|v^\perp|)v}{(\hat k\cdot v +i0)|v|^2}\ast f\big\|_{L^{q}_{\perp}(L^2_\para)}\leq CK^{\tfrac{1}{q}-\tfrac{1}{p}} \left\|f\right\|_{ L^p_\perp( L^{2}_{\para})},\label{xv}\\
					\big\|\tfrac{\chi(K|v^\perp|)}{(\hat k\cdot v +i0)|v|^2}\ast f\big\|_{L^{q}_{\perp}(L^2_\para)}\leq CK^{1+\tfrac{1}{q}-\tfrac{1}{p}} \left\|f\right\|_{L^p_\perp (L^{2}_{\para})}\label{xvi}
				\end{gather}
			\end{lemma}
			
			\begin{proof}
				We define in the following 
				\begin{gather*}
					\varphi_1(v) := \tfrac{\chi(K|v^\perp|)}{v^\para+i0},~\varphi_2(v) := \tfrac{1-\chi(K|v^\perp|)}{v^\para+i0},~\varphi_3(v) := \tfrac{\chi(K|v^\perp|)}{|v^\para|^2+|v^\perp|^2},\\
					\varphi_4(v) := \tfrac{\chi(K|v^\perp|)v^\perp}{(v^\para+i0)(|v^\para|^2+|v^\perp|^2)},~\varphi_5(v) := \tfrac{\chi(K|v^\perp|)|v^\perp|}{(v^\para+i0)(|v^\para|^2+|v^\perp|^2)},~\varphi_6(v) := \tfrac{\chi(K|v^\perp|)}{(v^\para+i0)(|v^\para|^2+|v^\perp|^2)}.
				\end{gather*}
				
				We denote $\check{\varphi}(\xi,v^\perp):=\int e^{i\xi v^\para}\varphi(v^\para,v^\perp)\ud v^\para$ the Fourier transform with respect to the first variable. We recall that, denoting $\sg(\xi):=(1-2\ind_{\xi<0})$, the Fourier transform of $\tfrac{1}{v^\para+i0}$ is $-2i\pi\ind_{\xi>0}$, the Fourier transform of $\tfrac{1}{1+|v^\para|^2}$ is $\pi \exp(-|\xi|)$ and the Fourier transform of $\tfrac{1}{(v^\para+i0)(1+|v^\para|^2)}$ is $i\pi\sg(\xi) (1-\exp(-|\xi|))$. Hence,
				\begin{gather*}
					|\check{\varphi}_2(\xi,v^\perp)|\lesssim 1-\chi(K|v^\perp|)\text{ and }\forall i\in\{1,3,4,5,6\},~|\check{\varphi}_i(\xi,v^\perp)|\lesssim \tfrac{\chi(K|v^\perp|)}{|v^\perp|^{r_i}},
				\end{gather*}
				where $r_1=0$, $r_3 = r_4 = r_5 = 1$ and $r_6 =2$. 
				
				\begin{itemize}
					\item \emph{Proof of \eqref{xii}.} Using Parseval inequality, for any $v^\perp\in <k>^\perp$,
					\begin{multline*}
						\int\big|\tfrac{\chi(K|v^\perp|)}{\hat k\cdot v +i0}\ast f\big|^2\ud v^\para \leq 2\pi \int \bigg|\int \check{f}(\xi,u^\perp)\chi(K|v^\perp-u^\perp|)\ud u^\perp\bigg|^2\ud \xi\leq 2\pi  \bigg\|\int |\check{f}(\cdot,u^\perp)|\ud u^\perp\bigg\|_{L^2_\xi}^2\\
						\leq 2\pi \bigg(\int \|\check{f}(\cdot,u^\perp)\|_{L^2_\xi}\ud u^\perp\bigg)^2
						\leq 2\pi  \|\check{f}(\cdot,u^\perp)\|_{L^1_\perp L^2_\para}^2.
					\end{multline*}
					
					In the same way, 
					\begin{multline*}
						\int\big|\tfrac{1-\chi(K|v^\perp|)}{\hat k\cdot v +i0}\ast f\big|^2\ud v^\para \leq 2\pi \int \bigg|\int \check{f}(\xi,u^\perp)(1-\chi(K|v^\perp-u^\perp|))\ud u^\perp\bigg|^2\ud \xi\\
						\leq2\pi\int (1-\chi(K|u^\perp|))^2\ud u^\perp \int |\check{f}(\xi,u^\perp)|^2\ud u^\perp\ud \xi
						\lesssim 2 K^{d-1}\pi \|f\|_{L^2}.
					\end{multline*}
					
					\item \emph{Proof of \eqref{xiii} and \eqref{xiv}.} Using Parseval inequality, for  $i\in\{3,4,5\}$
					\begin{multline*}
						\big\|\varphi_i\ast f(v)\big\|_{L^2_\para}
						=\bigg\|\int\check{\varphi}_i(\xi,v^\perp-\tilde{v}^\perp) \check{f}(\xi,\tilde{v}^\perp)\ud \tilde{v}^\perp \bigg\|_{L^2_\xi}
						\lesssim \bigg\|\int\tfrac{ \check{f}(\xi,\tilde{v}^\perp)}{|v^\perp-\tilde{v}^\perp|^{r_i}}\ud \tilde{v}^\perp \bigg\|_{L^2_\xi}
						\lesssim \int\tfrac{\|{f}(\cdot,\tilde{v}^\perp)\|_{L^2_\para}}{|v^\perp-\tilde{v}^\perp|^{r_i}}\ud\tilde{v}^\perp.
					\end{multline*}
					Applying the Hardy-Littelwood inequality, $\big\|\varphi_i\ast_v f\big\|_{L^q_\perp(L^2_\para)}\lesssim\big\|\check{f}\big\|_{L^q_\perp(L^2_\para)}$ if
					\[\frac{r_i}{d-1}+\frac{1}{p}=1+\frac{1}{q}.\]
					
					\item \emph{Proof of \eqref{xv} and \eqref{xvi}.} Using Parseval inequality, for  $i\in\{3,4,5\}$
					\begin{align*}
						\big\|\varphi_i\ast f(v)\big\|_{L^2_\para}
						=\bigg\|\int\check{\varphi}_i(\xi,v^\perp-\tilde{v}^\perp) \check{f}(\xi,\tilde{v}^\perp)\ud \tilde{v}^\perp \bigg\|_{L^2(\ud \xi)}
						\lesssim \int\tfrac{\chi(K|v^\perp-\tilde{v}^\perp|)}{|v^\perp-\tilde{v}^\perp|^{r_i}}\|{f}(\cdot,\tilde{v}^\para)\|_{L^2_\para} \ud \tilde{v}^\perp.
					\end{align*}
					
					Using that for $\alpha>1$, $\|\tfrac{\chi(K|v^\perp|}{|v^\perp|^{r_i}}\|_{L^\alpha(<k>^\perp)}=O(K^{r_i-1+\tfrac{1}{\alpha}})$ and the Young's inequality, we obtain the expected result.
				\end{itemize}
			\end{proof}

	{\large\textbf{Acknowledgment}}: The authors thanks Mitia Duerinckx, Nicolas Rougerie and Matthieu Ménard for stimulating and fruitful discussions. The author acknowledges financial support from the European Union (ERC, PASTIS, Grant Agreement n$^\circ$101075879) and from ANR NO-LIMIT.\footnote{{Views and opinions expressed are however those of the author only and do not necessarily reflect those of the European Union or the European Research Council Executive Agency. Neither the European Union nor the granting authority can be held responsible for them.}}
	
	\bibliographystyle{alpha}

\end{document}